\documentclass[a4paper, 11pt]{amsart}

\usepackage[T1]{fontenc}

\usepackage{xcolor}

\usepackage{enumerate, amsmath, amsfonts, amssymb, amsthm, mathrsfs, wasysym, graphics, graphicx, url, hyperref, hypcap, xargs, multicol, pdflscape, multirow, hvfloat, array, ae, aecompl, pifont, mathtools, a4wide, float, blkarray, overpic, nicefrac, tikz-cd}

\tikzcdset{scale cd/.style={every label/.append style={scale=#1},
    cells={nodes={scale=#1}}}}
\usepackage[all]{xy}

\usepackage[normalem]{ulem}

\hypersetup{colorlinks=true, citecolor=blue, linkcolor=blue}

\usepackage[noabbrev,capitalise]{cleveref}

\usepackage{braket}

\usepackage[margin=25mm]{geometry}
\usepackage{stmaryrd}

\usepackage{quiver}

\newtheorem{theorem}{Theorem}[section]
\newtheorem{corollary}[theorem]{Corollary}
\newtheorem{proposition}[theorem]{Proposition}
\newtheorem{lemma}[theorem]{Lemma}

\newtheorem*{theorem*}{Theorem}

\theoremstyle{definition}
\newtheorem{definition}[theorem]{Definition}
\newtheorem{example}[theorem]{Example}
\newtheorem{remark}[theorem]{Remark}

\newtheorem{notation}[theorem]{Notation}
\newtheorem{assumption}[theorem]{Assumption}

\newtheorem{fact}[theorem]{Fact}
\newtheorem{condition}[theorem]{Condition}
\newtheorem{problem}[theorem]{Problem}

\newtheorem*{definition*}{Definition}

\newcommand{\darkblue}{\color{blue!75!black}}
\newcommand{\defn}[1]{\textsl{\darkblue #1}}

\usepackage{adjustbox}

\newcommand{\id}{\mathrm{id}}         
\newcommand{\se}{\subseteq}           
\newcommand{\ppr}{^{\prime}}          
\newcommand{\Ker}{\mathrm{Ker}}       
\newcommand{\co}{\colon}              
\newcommand{\uas}{^{\ast}}            
\newcommand{\sas}{_{\ast}}            
\newcommand{\Ext}{\mathrm{Ext}}       

\newcommand{\EQ}{\Leftrightarrow}     
\newcommand{\dra}{\dashrightarrow}  
\newcommand{\lra}{\longrightarrow}    

\newcommand{\afr}{\mathfrak{a}}  
\newcommand{\sfr}{\mathfrak{s}}  

\newcommand{\Ebb}{\mathbb{E}}        
\newcommand{\Fbb}{\mathbb{F}}        
\newcommand{\Zbb}{\mathbb{Z}}        
\newcommand{\zb}{\mathbb{Z}}         

\newcommand{\Csc}{\mathscr{C}}   
\newcommand{\Dsc}{\mathscr{D}}   
\newcommand{\Esc}{\mathscr{E}}   
\newcommand{\Fsc}{\mathscr{F}}   
\newcommand{\Tsc}{\mathscr{T}}   

\newcommand{\Acal}{\mathcal{A}} 
\newcommand{\Bcal}{\mathcal{B}} 
\newcommand{\Ccal}{\mathcal{C}} 
\newcommand{\Dcal}{\mathcal{D}} 
\newcommand{\Ical}{\mathcal{I}} 
\newcommand{\Ncal}{\mathcal{N}} 
\newcommand{\Pcal}{\mathcal{P}} 
\newcommand{\Rcal}{\mathcal{R}} 
\newcommand{\Scal}{\mathcal{S}} 
\newcommand{\Tcal}{\mathcal{T}} 
\newcommand{\Wcal}{\mathcal{W}} 
\newcommand{\Xcal}{\mathcal{X}} 
\newcommand{\Ycal}{\mathcal{Y}} 

\newcommand{\und}{\underline}   
\newcommand{\ovl}{\overline}    
\newcommand{\ov}{\overset}      
\newcommand{\ti}{\times}        

\newcommand{\cat}{\mathscr{C}}

\newcommand{\rc}{\mathcal{R}}

\newcommand{\uc}{\mathcal{U}}
\newcommand{\cA}{{\mathcal A}}

\newcommand{\ct}{{\mathcal T}}

\newcommand{\add}{\operatorname{add}}

\newcommand{\infl}{\rightarrowtail}
\newcommand{\defl}{\twoheadrightarrow}
\newcommand{\pr}{\operatorname{pr}}
\newcommand{\copr}{\operatorname{copr}}

\newcommand{\kzero}[1]{K_0(#1)}
\newcommand{\kzerosp}[1]{K_0^\text{sp}(#1)}
\newcommand{\ind}{\operatorname{ind}}
\newcommand{\coind}{\operatorname{coind}}

\newcommand{\susp}{\Sigma}
\newcommand{\al}{\alpha}         
\newcommand{\vp}{\varphi}        
\newcommand{\vph}{\varphi}
\newcommand{\kap}{\kappa}        
\newcommand{\lam}{\lambda}       

\newcommand{\om}{\omega}         
\newcommand{\thh}{\theta}        
\newcommand{\del}{\delta}        
\newcommand{\ep}{\varepsilon}    
\newcommand{\eps}{\varepsilon}

\newcommand{\Sig}{\Sigma}        
\newcommand{\Om}{\Omega}         

\newcommand{\bsm}{\begin{smallmatrix}}
\newcommand{\esm}{\end{smallmatrix}}

\numberwithin{equation}{section}

\newcommand{\proj}{\operatorname{proj}} 
\newcommand{\inj}{\operatorname{inj}} 
\newcommand{\soc}{\operatorname{soc}} 
\newcommand{\Fac}{\operatorname{Fac}} 
\newcommand{\topp}{\mathrm{top}} 
\newcommand{\modd}{\operatorname{mod}}  
\newcommand{\End}{\mathrm{End}}   
\newcommand{\thick}{\mathrm{thick}} 

\newcommand{\blossom}{\text{\ding{96}}} 
\newcommand{\bls}{^\blossom} 

\newcommand{\CR}{\ovl{\C}_{\Rcal}}        
\newcommand{\Rint}{{}^{\perp_1}\Rcal\cap\Rcal^{\perp_1}}      
\newcommand{\confX}{X\infl 0\defl Y}      

\newcommand{\Hom}{\mathrm{Hom}}      
\newcommand{\Mod}{\mathrm{Mod}}      
\newcommand{\Cok}{\operatorname{Cok}}      
\newcommand{\C}{\Csc}               
\newcommand{\E}{\Ebb}               
\newcommand{\ush}{^\sharp}           
\newcommand{\ssh}{_\sharp}           

\newcommand{\Proj}{\mathrm{Proj}}    
\newcommand{\Inj}{\mathrm{Inj}}      
\newcommand{\op}{^\mathrm{op}}       

\newcommand{\gldim}{\operatorname{gl.dim}}
\newcommand{\Cone}{\operatorname{Cone}}
\newcommand{\Fib}{\operatorname{Fib}}

\newcommand{\CEs}{(\C,\E,\sfr)}           
\newcommand{\EbbI}{\Ebb_{\mathrm{I}}} 
\newcommand{\EbbII}{\Ebb_{\mathrm{I\!I}}}  
\newcommand{\EbbPB}{\Ebb_{\mathrm{PB}}}    
\newcommand{\EbbE}{\Ebb_{\mathrm{emb}}}   

\newcommand{\suf}{\mathbf{i}}  
\newcommand{\sufe}{\mathbf{i}=\mathrm{I},\mathrm{I}\!\mathrm{I}}  
\newcommand{\Irm}{\mathrm{I}}
\newcommand{\IIrm}{\mathrm{I}\!\mathrm{I}}

\newcommand{\cCP}{\operatorname{cCP}}
\newcommand{\Tilt}{\operatorname{Tilt}}
\newcommand{\Cot}{\operatorname{Co-}t}

\newcommand{\Fm}{F^{\bullet}}  

\newcommand{\domdim}{\operatorname{dom.dim}}
\newcommand{\codomdim}{\operatorname{codom.dim}}

\title[Hereditary extriangulated categories]{Hereditary extriangulated categories: Silting objects, mutation, negative extensions}

\author{Mikhail Gorsky}
\address{Fakult\"at f\"ur Mathematik, Universit\"at Wien, Oskar-Morgenstern-Platz 1, 1090 Wien, Austria}
\email{mikhail.gorskii@univie.ac.at}
\author{Hiroyuki Nakaoka}
\address{Graduate School of Mathematics, Nagoya University, Furocho, Chikusaku, Nagoya 464-8602, Japan}
\email{nakaoka.hiroyuki@math.nagoya-u.ac.jp}
\author{Yann Palu}
\address{LAMFA, Universit\'e Picardie Jules Verne, 80039 Amiens, France}
\email{yann.palu@u-picardie.fr}

\dedicatory{Dedicated to Bernhard Keller on the occasion of his 60th birthday}

\begin{document}

\maketitle



\begin{abstract}
 In this article, we initiate the study of hereditary extriangulated categories. Many important categories arising in representation theory in connection with various theories of mutation are hereditary extriangulated. Special cases include homotopy categories of 2-term complexes with projective components, which are related to silting mutation, cluster categories (with relevant relative extriangulated structures) where cluster tilting mutation take place, and more general triangulated categories with relative structures connected with relative tilting theory.

 We prove that there is a theory of irreducible mutation for maximal rigid objects and subcategories
 in 
 hereditary extriangulated categories of dominant dimension 1, which we shall call 0-Auslander.
 We show that many notions of maximality for rigid subcategories become equivalent in 0-Auslander extriangulated categories.
 Applied to the examples above, this recovers 2-term silting mutation in triangulated categories, cluster tilting mutation, and relative rigid mutation.
 By constructing suitable extriangulated categories, we also recover $\tau$-tilting mutation for gentle algebras and flips for their non-kissing facets.
 Combined with results by Adachi--Tsukamoto and Pauksztello--Zvonareva, our mutation also provides mutation for intermediate co-$t$-structures, extending the known construction of mutation of bounded intermediate co-$t$-structures.
 
 In spirit of our earlier work, we study negative extensions in hereditary extriangulated categories. We give sufficient conditions for the existence of universal balanced negative extensions. We explicitly compute certain, a priori non-universal, versions of negative extensions 
 in 0-Auslander categories 
 constructed from triangulated categories with rigid subcategories. 
 We discuss examples where these two constructions give the same $\delta$-functors and where they disagree. In both constructions, the $\delta$-functors are concentrated in degrees $[-2, 1].$
 \end{abstract}

\tableofcontents

\section{Introduction}

Flips of triangulations are the combinatorial shadows of \emph{mutations} in cluster algebras.
In order to categorify cluster algebras and the key notion of mutation, Aslak Buan, Bethany Marsh, Markus Reineke, Idun Reiten and Gordana Todorov introduced cluster tilting theory in~\cite{BuanMarshReinekeReitenTodorov}.
They proved that in a cluster category some specific objects categorifying clusters (the so-called \emph{cluster tilting objects}) have a well behaved theory of mutation.
Mutation of cluster tilting objects was quickly adapted or generalized to various settings: to cluster tilting in some abelian or exact categories defined from preprojective algebras~\cite{GeissLeclerclSchroer-survey}, to $\tau$-tilting in module categories of finite-dimensional algebras~\cite{AdachiIyamaReiten}, to 2-term silting in homotopy categories of complexes with projective components over a finite-dimensional algebra~\cite{AiharaIyama,Kimura}, to relative cluster tilting theory in triangulated categories~\cite{YangZhu,YangZhouZhu-Ghost}...

The flip of triangulations can also be generalized to a flip for dissections (or partial triangulations), that turned out to be the flip of so-called \emph{non-kissing facets}~\cite{McConville,BrustleDouvilleMousavandThomasYildirim,PaluPilaudPlamondon-nonkissing} in a different guise~\cite{PaluPilaudPlamondon-surfaces}.

Our aim in this paper is to advertise the ubiquity of hereditary extriangulated categories in this context. We develop a theory of mutation for silting objects in certain hereditary extriangulated categories that encompasses all mutations mentioned above, while permitting elementary proofs.

The setting we work in is that of \emph{0-Auslander extriangulated categories}, whose prototypical examples are given by homotopy categories of two-term complexes of projectives.

\begin{definition} \label{def:intro_0-Auslander}
 An extriangulated category $\CEs$ is 0-Auslander if
 \begin{itemize}
  \item it has enough projectives;
  \item it is hereditary: $\forall X\in\Csc$, $\exists$ $P_1\infl P_0\defl X\dashrightarrow$ where $P_0,P_1$ are projective;
  \item it has dominant dimension at least 1: for any projective object $P$, there is a conflation $P\infl Q\defl I$ where $I$ is injective and $Q$ is projective-injective.
 \end{itemize}
It is reduced if moreover its only projective-injective object is 0 (up to isomorphism).
\end{definition}

This definition turns out to be self-dual: one can ask for the existence of enough injectives, injective dimension to be at most 1, and codominant dimension to be at least 1. Such categories can be thought of as extriangulated analogues of Auslander-regular rings of global dimension at most 1.\footnote{We thank Osamu Iyama for this comment. Our choice of the name ``0-Auslander'' instead of ``1-Auslander'' is motivated by the notation in higher dimensional Auslander-Reiten theory.  In particular, in \cite{IyamaJasso}, certain additive categories $\cA$ are called \emph{d-Auslander dualizing R-varieties}, for a positive integer $d$, if they satisfy $\gldim \cA \leq d + 1 \leq \domdim \cA$.} 
In 0-Auslander categories, several candidate definitions of tilting objects turn out to be equivalent.

\begin{proposition}[particular case of \cref{thm: equivalent versions of tilting}]
\label{prop:equivDefIntro}
 Let $\CEs$ be a Krull--Schmidt, Hom-finite, reduced 0-Auslander extriangulated category with $P=P_1\oplus\cdots\oplus P_n$ a basic projective generator, and let $R=R_1\oplus\cdots\oplus R_k$ be a basic rigid object (i.e. $\Ebb(R,R)=0$).
 Then $k\leq n$ and the following are equivalent:
 \begin{enumerate}[(a)]
  \item $R$ is maximal rigid: If $R\oplus X$ is rigid, then $X\in\add R$;
  \item $R$ is silting: The smallest thick (that is, strictly full, closed under direct summands, and satisfying the 2-out-of-3 property for $\sfr$-triangles) subcategory of $\Csc$ containing $R$ is $\Csc$ itself. 
  \item $R$ is tilting: There is an $\sfr$-triangle $P\infl R^0\defl R^1\dashrightarrow$ with $R^0,R^1\in\add R$;
  \item $R$ is cotilting: There is an $\sfr$-triangle $R^1\infl R^0\defl I\dashrightarrow$ where $R^0,R^1\in\add R$ and $I$ is an injective cogenerator;
  \item $R$ is $\Ebb$-tilting: If $\Ebb(R,X)=0=\Ebb(X,R)$ then $X\in\add R$;
  \item $R$ is complete: We have $k=n$.
 \end{enumerate}
\end{proposition}

Some examples of reduced 0-Auslander extriangulated categories are given by:
\begin{itemize}
 \item Cluster categories endowed with a fixed cluster tilting object, and the maximal relative extriangulated structure making it projective. In this case, the silting objects are precisely the cluster tilting objects.
 \item The homotopy category of complexes with projective components and concentrated in degrees -1 and 0 over a finite-dimensional algebra. In this case, the silting objects are the 2-term silting objects. 
 \item Certain well-chosen subcategories of triangulated categories with a fixed rigid object and endowed with a suitable extriangulated structure. In this case, the silting objects are the relative tilting objects.
 \item Certain exact categories constructed from gentle algebras. In this case, the silting objects categorify accordion dissections, or non-kissing facets.
 \item The category of finite-dimensional representations of a Dynkin quiver of type $A$, endowed with a certain exact structure introduced by Thomas Br\"ustle, Eric Hanson, Sunny Roy and Ralf Schiffler. In this case, the silting objects are the maximal almost rigid modules of~\cite{BarnardGunawanMeehanSchiffler}.
\end{itemize}

Our next main result shows that there is a well-behaved theory of irreducible mutation for silting subcategories (only stated here for silting objects for simplicity) 
in 0-Auslander extriangulated categories, thus generalizing and unifying the mutation theories or flips in all four examples above.

\begin{theorem}[particular case of \cref{Th:mutation}]
\label{theorem:mutationIntro}
 Let $\CEs$ be a Hom-finite, Krull--Schmidt, reduced 0-Auslander extriangulated category with a projective generator, and let $R$ be a basic silting object.
 Write $R=R'\oplus X$ with $X$ indecomposable.
 Then, there is (up to isomorphism) a unique indecomposable object $Y$, non-isomorphic to $X$, such that $R'\oplus Y$ is silting.
 Moreover, there is precisely one exchange $\sfr$-triangle
 \[
 X \overset{i}{\infl} E\overset{p}{\defl} Y \dashrightarrow \hspace{7pt} \text{ or } \hspace{10pt}
 Y \overset{i'}{\infl} E'\overset{p'}{\defl} X \dashrightarrow,
 \]
where $E,E'$ belong to $\add R'$, $i,i'$ are left $R'$-approximations and $p,p'$ are right $R'$-approximations.
\end{theorem}

We note that \cref{theorem:mutationIntro} also applies to non-reduced 0-Auslander extriangulated categories, but there, projective-injective summands cannot be mutated (and thus behave like coefficients in cluster algebras). Note also that the last sentence of \cref{theorem:mutationIntro} means that in 0-Auslander categories the distinction between left and right mutations is very transparent.  

In order to prove \cref{prop:equivDefIntro,theorem:mutationIntro}, we are led to studying Bongartz completions, indices and coindices (this permits us to define categorical $c$-vectors as in~\cite{JorgensenYakimov}) and some form of Iyama--Yoshino reduction~\cite{IyamaYoshino}.

We show that extended cohearts of co-$t$-structures \cite{Pauksztello} (also known as weight structures \cite{Bondarko}) on triangulated categories, as considered by \cite{PauksztelloZvonareva}, are reduced 0-Auslander extriangulated categories. 
By combining results of Adachi--Tsukamoto \cite{AdachiTsukamoto} relating tilting subcategories with complete cotorsion pairs with one of the main theorems of Pauksztello--Zvonareva \cite{PauksztelloZvonareva}, we then deduce from \cref{prop:equivDefIntro} a theory of irreducible mutations of intermediate co-$t$-structures. This recovers the previously known construction of irreducible mutations of intermediate bounded co-$t$-structures \cite{AiharaIyama, KoenigYang, BrustleYang} and commutes with restrictions to and extensions from the thick subcategory generated by the coheart of the original co-$t$-structures.

The remainder of the paper
concerns negative extensions in hereditary extriangulated categories. Long exact sequences of homomorphisms and positive extensions in extriangulated categories a priori do not continue naturally in the negative direction. In \cite{GNP1} we suggested a definition of negative extensions in extriangulated categories (see also~\cite{AdachiEnomotoTsukamoto}): we defined covariant and contravariant versions of negative extensions in extriangulated categories satisfying certain natural conditions and proved their respective functoriality and universality. The universality means that these bifunctors are covariant and contravariant left derived functors of the functor $\Hom$.

In this paper, we study bivariant, or \emph{balanced}, versions of negative extensions in hereditary extriangulated categories. In \cite{GNP1} we gave sufficient conditions on covariant and contravariant negative extensions to agree. In this case, they give a bivariant $\delta$-functor satisfying natural universal properties in both arguments. These conditions are satisfied for exact categories and for the   homotopy category of 2-term complexes with projective components over a finite-dimensional algebra. The requirement on the covariant and contravariant versions to agree is nonetheless very restricitve: the cluster category of type $A_2$ with the maximal relative extriangulated structure making some cluster tilting object projective is a basic example of a 0-Auslander category where they disagree. 

We give a fairly natural construction of a bivariant connected sequence of bifunctors $(\EbbPB^{\bullet},\del\ssh,\del\ush)$, with $\EbbPB^i = \Ebb^i$ for $i \ge 0$ and $\EbbPB^{-i} = 0$ for $i \ge 3$, for essentially small hereditary extriangulated categories having enough projectives and enough injectives. 

\begin{theorem}[see \cref{PropEquivAcycBiProl} for more details]
 The bivariant connected sequence of functors $(\EbbPB^{\bullet},\del\ssh,\del\ush)$ is a bivariant $\delta$-functor if and only if a certain 4-term complex is exact at the middle terms. In that case, $(\EbbPB^{\bullet},\del\ssh,\del\ush)$ satisfies the relevant universal property. Its components $\EbbPB^{-n}$ can thus be seen as universal balanced negative extensions.
\end{theorem}

For a reduced 0-Auslander category arising from a rigid subcategory of a triangulated category, we define another bivariant $\delta-$functor $(\EbbE^{\bullet},\del\ssh,\del\ush)$ (see \cref{def: conn_EbbE}),  a priori depending on the ambient triangulated category.  We expect that whenever such a  category admits universal balanced negative extensions, these are given by $\EbbE^{\bullet}$, and so $\EbbE^{\bullet}$ is the correct replacement of $\EbbPB$ in this setting. These two constructions agree for cluster categories with relative structures corresponding to acyclic or self-injective cluster tilting objects, but differ in general. The simplest example when they differ comes from a non-acyclic cluster seed of type $A_4$.




\vspace{0.3cm}

\noindent {\bf Further directions.} 
\begin{itemize}
\item[(i)] In work in progress of M.G. and Y.P. with Xin Fang, Pierre-Guy Plamondon, and Matthew Pressland, certain further properties and examples of 0-Auslander categories are studied. In particular, Higgs categories coming from ice quivers with potential, as recently defined by Yilin Wu \cite{Wu}, are shown to be 0-Auslander when endowed with suitable relative extriangulated structures.

\item[(ii)] The definition of $\EbbE^{-1}(-,-)$ is quite nontrivial: it is given by the quotient of certain morphism space in the ambient triangulated category by an ideal of morphisms factoring through morphisms between objects in certain subcategories (see \cref{def:negativeComputation}). The fact that it makes sense in this context to consider quotients by morphisms factoring through certain {\it morphisms} rather than through certain {\it objects} seems fairly unexpected. A conjectural explanation, of sorts, will be given in the same work in progress as in (i). We also wonder if this definition can be deduced or explained naturally for extriangulated categories admitting enhancements, either via exact $\infty$-categories and their stable hulls \cite{barwick2015exact, NakaokaPalu2, Klemenc}, or via differential graded categories as in the ongoing Ph.D. project of Xiaofa Chen.  

\item[(iii)] It is natural to ask which of the results of the present paper hold, up to suitable modifications, for extriangulated categories of higher global and (co)dominant dimension, which could be called d-Auslander extriangulated. 
Similarly, one can also consider the following generalisation.
Extriangulated categories are a special case of {\it $n$-exangulated categories} introduced by Herschend, Liu, and H.N. \cite{HerschendLiuNakaoka-I}: they correspond to $n$ being equal to $1$. Projective and injective objects in $n$-exangulated categories are defined for general $n$, and one can adapt \cref{def:intro_0-Auslander} to this generality. Some implications in \cref{prop:equivDefIntro} and parts of \cref{theorem:mutationIntro} extend to the framework of extriangulated categories of higher global and (co)dominant dimension, and some generalise to $n$-exangulated analogues of 0-Auslander extriangulated categories. Both of these generalisations will be treated in work in progress
joint with Nicholas Williams.
\end{itemize}

\vspace{0.3cm}

\noindent {\bf Organisation of the paper.}
The article is organised as follows:
Reminders on extriangulated categories are gathered in \cref{ssection:extricats}.
In \cref{section:hereditary}, we discuss equivalent definitions of hereditary extriangulated categories and behaviour of those under natural operations. In \cref{section:0-Auslander}, we introduce 0-Auslander extriangulated categories and prove some of their basic properties. We list many examples of such categories in \cref{ssection:examples0-Auslander}.
\cref{section:tilting}, forming the heart of the paper, is devoted to studying tilting subcategories in 0-Auslander extriangulated categories.
We prove that various possible notions of maximality for rigid subcategories are in fact equivalent in \cref{ssection:equivDef}.
For a better understanding of complete rigid objects (\cref{subsection:complete rigid}), we prove a version of Bongartz completion in \cref{ssection:Bongartz} and define indices, coindices and categorical $c$-vectors in \cref{ssection:indices}.
\cref{ssection:reduction,ssection:tilting and reduction} introduce reduction for hereditary extriangulated categories, which is a crucial tool in \cref{ssection:mutation} for proving \cref{theorem:mutationIntro}.
In \cref{section:ExtendedCohearts}, we apply our results on mutation of tilting subcategories in 0-Auslander extriangulated categories to extended cohearts of co-$t$-structures in triangulated categories. We explain why this recovers mutation for intermediate co-$t$-structures.
In \cref{section:gentle}, we construct, starting from any finite-dimensional gentle algebra, a 0-Auslander extriangulated category of walks. 
This allows us to prove that flips of non-kissing facets in the sense of McConville~\cite{McConville} are instances of mutations of tilting objects in 0-Auslander extriangulated categories.
\cref{section:negative} is concerned with negative extensions in hereditary extriangulated categories. We give explicit computations in a large class of examples (\cref{ssection:computation}), construct a candidate bivariant connected sequence of functors for any hereditary extriangulated category, and find a necessary and sufficient condition so that it is acyclic and universal in \cref{ssection:universal}.


\section*{Acknowledgements}


The authors would like to thank Lidia Angeleri H\"ugel, Thomas Br\"ustle, Xin Fang, Sira Gratz, Osamu Iyama, Bernhard Keller, Rosanna Laking, Matthew Pressland, Ralf Schiffler, Salvatore Stella, Nicholas Williams, Dong Yang and Alexandra Zvonareva for many interesting discussions related to this project. We are grateful to Pierre-Guy Plamondon for his comments on a preliminary version of this paper. 
Y.P. would like to express his gratitude to the organizers of the 2022 Abel Symposium, of the CATS21 conference, of the research school \emph{New developments in representation theory arising from cluster algebras}, of the Oberwolfach workshop \emph{Cluster Algebras and Its Applications} and of the FD seminar where part of this work was presented. 

H.N. and Y.P. were partially supported by the French ANR grant SC3A~(ANR-15-CE40-0004-01). M.G. and Y.P. were partially supported by the French ANR grant CHARMS~(ANR-19-CE40-0017). Parts of this work were done when M.G. and Y.P. participated at the Junior Trimester Program ``New Trends in Representation Theory'' at the Hausdorff Institute for Mathematics in Bonn.  This work is a part of a project that has received funding from the European Research Council (ERC) under the European Union’s Horizon 2020 research and innovation programme (grant agreement No. 101001159). Parts of this work were done during stays of M.G. at the University of Stuttgart, and he is very grateful to Steffen Koenig for the hospitality. 

\section{Hereditary extriangulated categories}
\label{section:hereditary}

In this section, we discuss several equivalent definitions and some properties of  extriangulated categories of (positive) global dimension 1. Basics on extriangulated categories are recalled in \cref{ssection:extricats}. We assume all extriangulated categories to be essentially small and linear over a commutative ring $R$ with $1$.
Recall from \cref{ssection:extricats} the notion of cone and that of cocone (or fiber).

\subsection{Equivalent definitions of hereditary extriangulated categories}

In \cite[Section 3]{GNP1}, we defined higher positive extensions $\Ebb^i, i \geq 1$ via a version of Yoneda's theory. Let $\CEs$ be an extriangulated category, $C, A \in \cat$. The underlying set of the $R$-module $\Ebb^2(C, A)$ admits the following explicit description first sketched in \cite[Remark 5.10]{NakaokaPalu}:

\begin{align*}
    \Ebb^2(C, A) & = \Big(\coprod_{M\in\cat}\Big(\Ebb(M, A)\ti \Ebb (C, M)\Big)\Big)/
\sim,
\end{align*}

\noindent where the equivalence relation $\sim$  is generated by 
\[ (\kap,f_{\ast}(\lam))\sim (f^{\ast}(\kap),\lam),\]
for arbitrary $f\in\cat(M,N),\kap\in \Ebb(N, A),\lam\in \Ebb(C, M)$. An element of $\Ebb^2(C, A)$ is $0$ if and only if it is equivalent to a pair $(0, \lam),$ if and only if it is equivalent to a pair $(\kap, 0).$

For any $\sfr$-triangle $A\ov{x}{\lra}B\ov{y}{\lra}C\ov{\del}{\dra}$, we have a pair of long exact sequences \cite[Theorem 3.5]{GNP1}:

\begin{equation}
\label{les:1}
\cat(-,A)\to\cat(-,B)\to\cat(-,C)\to\Ebb(-,A)\to \Ebb(-,B) \to \Ebb(-,C)\to\Ebb^2(-,A) \to \cdots
\end{equation}
in $\Mod\cat$, and
\begin{equation}
\label{les:2}
\cat(C,-)\to\cat(B,-)\to\cat(A,-)\to\Ebb(C,-)\to \Ebb(B,-) \to \Ebb(A,-)\to\Ebb^2(C,-) \to \cdots
\end{equation}
in $\cat\Mod$.

If $\CEs$ has enough projectives and enough injectives, another construction of the same higher extension groups had been introduced earlier in \cite{LiuNakaoka}, \cite{HLNII} by using (co)resolutions.

\begin{proposition} \label{hered_equiv}
\begin{enumerate}
    \item The following conditions on an extriangulated category $(\Csc, \Ebb, \mathfrak{s})$ are equivalent.

\begin{itemize}
    \item [(i)] The functor $\Ebb(X, -)$ is right exact for any $X \in \Csc;$
    \item[(ii)] The functor $\Ebb(-, Y)$ is right exact for any $Y \in \Csc;$
    \item[(iii)] The bifunctor $\Ebb^2(-, -)$ is zero;
    \item[(iv)] Let $\delta\in\Ebb(Z,Y)$ and $\delta'\in\Ebb(X,Z)$ be realized respectively by
\[
Y\overset{f}{\longrightarrow}U\overset{f'}{\longrightarrow}Z\ \text{ and } \ Z\overset{d}{\longrightarrow}T\overset{e}{\longrightarrow}X.
\]
Then there exist an object $V\in\Csc$, a commutative diagram in $\Csc$
\begin{equation}
\label{diag:hereditary}
    \begin{aligned}
\xy
(-18,6)*+{Y}="0";
(-6,6)*+{U}="2";
(6,6)*+{Z}="4";
(18,6)="6";
(-18,-6)*+{Y}="10";
(-6,-6)*+{V}="12";
(6,-6)*+{T}="14";
(18,-6)="16";
(-6,-18)*+{X}="22";
(6,-18)*+{X}="24";
(-6,-28)="32";
(6,-28)="34";
{\ar^{f} "0";"2"};
{\ar^{f'} "2";"4"};
{\ar@{=} "0";"10"};
{\ar@[blue]_{g} "2";"12"};
{\ar^{d} "4";"14"};
{\ar@[blue]^{h} "10";"12"};
{\ar@[blue]^{h'} "12";"14"};
{\ar@[blue]_{g'} "12";"22"};
{\ar^{e} "14";"24"};
{\ar@{=} "22";"24"};
{\ar@{}|{} "0";"12"};
{\ar@{}|{} "2";"14"};
{\ar@{}|{} "12";"24"};
{\ar@{-->}^\delta "4";"6"};
{\ar@{-->}@[blue]^{\delta''} "14";"16"};
{\ar@{-->}@[blue]^{\delta'''} "22";"32"};
{\ar@{-->}^{\delta'} "24";"34"};
{\ar@{}|\circlearrowright "0";"12"};
{\ar@{}|\circlearrowright "2";"14"};
{\ar@{}|\circlearrowright "12";"24"};
\endxy
\end{aligned}
\end{equation}
 with $\delta''\in\Ebb(T,Y)$ realized by $Y\overset{h}{\longrightarrow}V\overset{h'}{\longrightarrow}T$ and 
$\delta'''\in\Ebb(X,U)$ realized by $U\overset{g}{\longrightarrow}V\overset{g'}{\longrightarrow}X$,
which satisfy the following compatibilities.
  \begin{enumerate}[(i)]
    \item $f'_\ast\delta''' = \delta'$,
    \item $d^\ast\delta''=\delta$,
    \item $f_\ast\delta''=e^\ast\delta'''$. 
  \end{enumerate}
    
  \item[(v)]  Let $f$ be a composition $f = i \circ d,$ where $i: Z \to T$ is an inflation with cone $X,$ $d: U \to Z$ is a deflation with cocone $Y$. Then $f$ admits a decomposition as $f = d' \circ i',$ where $d': V \to T$ is a deflation with cone $X,$ $i': U \to V$ is an inflation with cocone $Y$, and, moreover, $U \overset{\tiny{\begin{bmatrix}-d \\ i'\end{bmatrix}}}\rightarrowtail Z \oplus V \overset{\tiny{\begin{bmatrix} i & d'\end{bmatrix}}}\twoheadrightarrow T$ is a conflation.
\end{itemize}

\item If $(\Csc, \Ebb, \mathfrak{s})$ has enough projectives, then these conditions are satisfied if and only if any object admits a projective resolution of length 1: For each $X\in\Csc$, there is a conflation $P_1\infl P_0\defl X$, with $P_0,P_1$ projective.

\item Dually, if $(\Csc, \Ebb, \mathfrak{s})$ has enough injectives, then these conditions are satisfied if and only if any object admits a injective coresolution of length 1: For each $X\in\Csc$, there is a conflation $X\infl I^0\defl I^1$, with $I^0,I^1$ injective.
\end{enumerate}
\end{proposition}

\begin{proof}

The implications (iii) $\Rightarrow$ (i), (iii) $\Rightarrow$ (ii) follow from the long exact sequences (\ref{les:1}) and (\ref{les:2}). 

The implications (i) $\Rightarrow$ (ii) and (i) $\Rightarrow$ (iv) follow from the axiom (ET4) as follows. Let $A\ov{x}{\lra}B\ov{y}{\lra}C\ov{\del}{\dra}$ be an arbitrary $\sfr$-triangle. Consider an arbitrary $X \in \cat$ and an arbitrary $\sfr$-triangle of the form $X\ov{f}{\lra}Y\ov{g}{\lra}A\ov{\theta}{\dra}$.
For (ii), it is enough to show that for each such data, there exists $\nu \in \Ebb(B, X)$ such that $x^\ast \nu = \theta$. 
By the assumption (i), there exists $\tau \in \Ebb(C, Y)$ such that $g_{\ast} \tau = \del$. Consider any realization of $\tau$, it has the form $Y\ov{t}{\lra}Z\ov{u}{\lra}C\ov{\tau}{\dra}$. By (ET4), there exists a diagram of the form

\[
\xy
(-18,6)*+{X}="0";
(-6,6)*+{Y}="2";
(6,6)*+{A}="4";
(18,6)="6";
(-18,-6)*+{X}="10";
(-6,-6)*+{Z}="12";
(6,-6)*+{B'}="14";
(18,-6)="16";
(-6,-18)*+{C}="22";
(6,-18)*+{C}="24";
(-6,-28)="32";
(6,-28)="34";
{\ar^{f} "0";"2"};
{\ar^{g} "2";"4"};
{\ar@{=} "0";"10"};
{\ar_{t} "2";"12"};
{\ar^{x'} "4";"14"};
{\ar^{} "10";"12"};
{\ar^{} "12";"14"};
{\ar_{u} "12";"22"};
{\ar^{y'} "14";"24"};
{\ar@{=} "22";"24"};
{\ar@{}|{} "0";"12"};
{\ar@{}|{} "2";"14"};
{\ar@{}|{} "12";"24"};
{\ar@{-->}^\theta "4";"6"};
{\ar@{-->}^{\mu} "14";"16"};
{\ar@{-->}^{\tau} "22";"32"};
{\ar@{-->}^{g_{\ast} \tau} "24";"34"};
{\ar@{}|\circlearrowright "0";"12"};
{\ar@{}|\circlearrowright "2";"14"};
{\ar@{}|\circlearrowright "12";"24"};
\endxy
\]
such that $(x')^{\ast} \mu = \theta$. Since $g_{\ast} \tau = \del,$ there exists $b \in \cat(B, B')$ making the diagram 

\[
\xy
(-12,6)*+{A}="0";
(0,6)*+{B}="2";
(12,6)*+{C}="4";
(-12,-6)*+{A}="10";
(0,-6)*+{B'}="12";
(12,-6)*+{C}="14";
{\ar^{x} "0";"2"};
{\ar^{y} "2";"4"};
{\ar@{=} "0";"10"};
{\ar^{b} "2";"12"};
{\ar@{=} "4";"14"};
{\ar_{x'} "10";"12"};
{\ar_{y'} "12";"14"};
{\ar@{}|{} "0";"12"};
{\ar@{}|{} "2";"14"};
{\ar@{}|\circlearrowright "0";"12"};
{\ar@{}|\circlearrowright "2";"14"};
\endxy
\]
commute. Then $b^{\ast} \mu \in \Ebb(B, X)$ satisfies $x^{\ast} (b^{\ast} \mu) = (x')^{\ast}(\mu) = \theta,$ hence we can take $\nu = b^{\ast} \mu.$
Moreover, the isomorphism $b$ allows to replace $B'$ by $B$ in the first diagram above, thus giving (iv).

The implications (ii) $\Rightarrow$ (i) and (ii) $\Rightarrow$ (iv) are proved dually by using \rm (ET4)$\hspace*{0pt}^{\mathrm{op}}$.

(iv) $\Rightarrow$ (iii): Each element in $\Ebb^2(X, Y)$ can be represented by a pair $(\delta, \delta')$ which can in turn be realized by a pair of  $\sfr$-triangles \[
Y\overset{f}{\longrightarrow}U\overset{f'}{\longrightarrow}Z\ \text{ and } \ Z\overset{d}{\longrightarrow}T\overset{e}{\longrightarrow}X.
\] By the assumption, they can be completed to a diagram of the form (\ref{diag:hereditary}). Then the class $(\delta, \delta') = (\delta, f'_\ast\delta''')$ is equivalent to $(f'^{\ast}\delta, \delta')$. But $f'^{\ast}\delta = 0$ by \cite[Lemma 3.2]{NakaokaPalu}, hence $(\delta, \delta') \sim (0, \delta') = 0.$

(iv)  $\Leftrightarrow$ (v) follows from \cite[Proposition 1.20]{LiuNakaoka} by using the diagram (\ref{diag:hereditary}).

Parts (2) and (3) are immediate consequences of \cite[Corollary 3.21]{GNP1} and constructions in \cite[Section 3.1]{HLNII} and \cite[Section 5.1]{LiuNakaoka}.
\end{proof}

\begin{remark}
By \cite[Proposition 3.17]{GNP1}, conditions in \cref{hered_equiv}.(1) are also equivalent to the following: for all $X, Y \in \Csc,$ every $\alpha \in \Ebb^2(X, Y)$ admits a trivialization. In fact, condition (iv) is essentially a more explicit rephrasing of this condition.
\end{remark}

\begin{definition}
An extriangulated category is called \defn{hereditary} if it satisfies the equivalent conditions of \cref{hered_equiv}.
\end{definition}

\newpage
\subsection{Operations on hereditary extriangulated categories}

\subsubsection{Extension-closed subcategories}

\begin{lemma}
\label{lemma:ext-closed hereditary}
An extension-closed subcategory of a hereditary extriangulated category, considered with the induced extriangulated structure, is hereditary.
\end{lemma}

\begin{proof}
It is straightforward to check that if the condition (iv) in \cref{hered_equiv}.(1) holds for an extriangulated category, it holds for each of its extension-closed subcategories.
\end{proof}

\subsubsection{Relative theories}

Given a relative structure $(\Csc, \Fbb, \mathfrak{s}_{\Fbb})$ of an extriangulated structure $(\Csc, \Ebb, \mathfrak{s})$, one may ask whether the latter being hereditary implies the former being hereditary, or vice versa. The following easy examples explain why neither implication is true already for exact structures. Nontrivial examples of hereditary relative structures of non-hereditary extriangulated structures will appear throughout the paper.

\begin{example} \emph{The larger structure is hereditary, the smaller is not.}

Consider the category $\modd k A_3.$ It is hereditary when considered with its abelian exact structure. The relative exact structure whose class of conflations consists of all short exact sequences on which $\Hom(\begin{smallmatrix}3\\2\end{smallmatrix}, -)$ becomes exact is no longer hereditary. This is clear from the Auslander-Reiten quiver of the category. The Auslander-Reiten triangles ending at $2$ and at $3$ give conflations, and the class in $\widetilde{\Ext}^2(3, 1)$ defined by their Yoneda product does not vanish (here we denote by $\widetilde{\Ext}^i(-, -)$ (higher) extension groups in this structure).

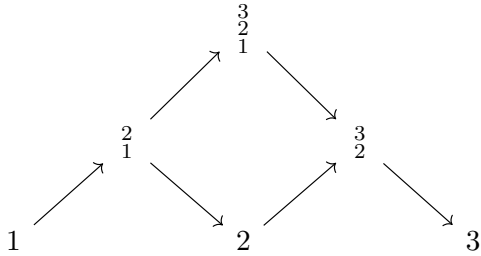
\[\begin{tikzcd}
	&& {\begin{smallmatrix}3\\2\\1\end{smallmatrix}} \\
	& {\begin{smallmatrix}2\\1\end{smallmatrix}}  && {\begin{smallmatrix}3\\2\end{smallmatrix}}\\
	1 && 2 && 3
	\arrow[from=3-1, to=2-2]
	\arrow[from=2-2, to=1-3]
	\arrow[from=1-3, to=2-4]
	\arrow[from=2-4, to=3-5]
	\arrow[from=2-2, to=3-3]
	\arrow[from=3-3, to=2-4]
\end{tikzcd}\]
\end{example}

\begin{example}\emph{The smaller structure is hereditary, the larger is not.}

Every additive category admits a split exact structure, which is by definition hereditary. Thus, any abelian category of global dimension at least 2 provides such an example: its maximal exact structure is the abelian one, its minimal exact structure is the split one.
\end{example}

\subsubsection{Reductions of hereditary extriangulated categories}
\label{ssection:reduction}




\begin{definition}\label{Rem: extriangulated reduction}
Let $\Rcal\se\C$ be any rigid full additive subcategory, closed under isomorphisms and under taking direct summands. Put $\CR=\frac{\Rint}{[\Rcal]}$.
Since all objects in $\Rcal$ are projective-injective in $\Rint$, the bifunctor $\Ebb$, restricted to $\Rint$, descends to the quotient and we have an induced extriangulated category $(\CR,\ovl{\E},\ovl{\sfr})$ (see~\cite[Proposition 3.30]{NakaokaPalu}).
\end{definition}

\begin{remark}
The case when $\Rcal$ is the full subcategory whose objects are the projective-injectives in $\Csc$ was considered (for non-necessarily hereditary extriangulated categories) in~\cite[Proposition 3.30]{NakaokaPalu} and will often be used in our proofs.
\end{remark}

The following is an immediate consequence of the definitions and of~\cref{lemma:ext-closed hereditary}.

\begin{proposition}\label{Prop: reductions of hereditary extricats}
The following holds.
\begin{enumerate}
\item If $\C$ is hereditary, then so is $\CR$.
\item For any $X,X\ppr\in\Rint$, the following are equivalent.
\begin{itemize}
\item[{\rm (i)}] $\E(\add(\Rcal\oplus X),\add(\Rcal\oplus X\ppr))=0$.
\item[{\rm (ii)}] $\E(X,X\ppr)=0$.
\item[{\rm (iii)}] $\ovl{\E}(X,X\ppr)=0$.
\end{itemize}
\end{enumerate}
\end{proposition}

\subsubsection{Localization of hereditary categories}
\label{subsubsec:localization}

In this subsection, we assume all categories to be small. Let us recall some definitions and results from \cite{NakaokaOgawaSakai}.

A full additive subcategory $\mathcal{N} \subseteq \mathcal{C}$ is 
\defn{thick} 
if it is closed under isomorphisms and direct summands and satisfies the 2-out-of-3 property for conflations. In particular, $\mathcal{N}$ is closed under extensions in $\mathcal{C}$, so it carries an induced extriangulated structure.

We associate with a thick $\mathcal{N} \subseteq \mathcal{C}$ two sets of morphisms:

\[\mathcal{L}_{\mathcal{N}} = \{f \in \mathcal{M} \;|\; f \mbox{ is an inflation with Cone($f$)} \in \mathcal{N}\};
\]
\[\mathcal{R}_{\mathcal{N}} = \{f \in \mathcal{M} \;|\; f \mbox{ is a deflation with CoCone($f$)} \in \mathcal{N}\}.
\]

$\mathscr{S}_\mathcal{N}$ is defined to be the set of all finite compositions of morphisms in $\mathcal{L}_{\mathcal{N}}$ and $\mathcal{R}_{\mathcal{N}}$. The following lemma shows that in the hereditary case, this set has an easier description.

\begin{lemma} \label{lem:s_n:hered}
Let $\mathcal{N}$ be a thick subcategory of a hereditary extriangulated category. Then we have 
$$\mathscr{S}_{\mathcal{N}} = \mathcal{R}_{\mathcal{N}} \circ \mathcal{L}_{\mathcal{N}}.$$
\end{lemma}

\begin{proof}
It is enough to show that  $\mathcal{L}_{\mathcal{N}} \circ \mathcal{R}_{\mathcal{N}} \subseteq \mathcal{R}_{\mathcal{N}} \circ \mathcal{L}_{\mathcal{N}}$. This follows from (v) in \cref{hered_equiv}, applied to $i \in \mathcal{L}_{\mathcal{N}}, d \in \mathcal{R}_{\mathcal{N}}.$
\end{proof}

Let $\mathcal{N} \subseteq \Csc$ be a thick subcategory of an extriangulated category such that $\mathscr{S}_\mathcal{N}$ satisfies the conditions of \cite[Theorem 3.5]{NakaokaOgawaSakai} (this is the case for example when $\Ncal$ is biresolving, or when $\Csc$ is exact or triangulated and $\Ncal$ is percolating, see~\cite[Section 4]{NakaokaOgawaSakai}). Then the main result of \cite{NakaokaOgawaSakai} states that the localization $\Csc/\mathcal{N} \overset\cong\to \Csc[\mathscr{S}_{\mathcal{N}}^{-1}]$ carries a canonical induced extriangulated structure $(\Csc/\mathcal{N}, \widetilde{\Ebb}, \widetilde{\sfr})$. Its morphisms, defining bifunctor $\widetilde{\Ebb}$, and $\sfr$-triangles are all described via a suitable calculus of fractions. In particular, each $\sfr$-triangle realizing an element in $\widetilde{\Ebb}(C, A)$ can be represented by a diagram of the form $A \overset{a}{\longrightarrow} X\overset{f}{\rightarrowtail}B\overset{f'}{\twoheadrightarrow}Y \overset{c}{\longrightarrow} C $, with $X\overset{f}{\rightarrowtail}B\overset{f'}{\twoheadrightarrow}Y$ being a conflation in $\CEs$ and with $a, c \in \mathscr{S}_{\mathcal{N}}$.

\begin{theorem}
Let $(\Csc, \Ebb, \mathfrak{s})$ be a hereditary extriangulated category and let $\mathcal{N} \subseteq \Csc$ be a thick subcategory such that $\mathscr{S}_\mathcal{N}$ satisfies the conditions of \cite[Theorem 3.5]{NakaokaOgawaSakai}. Then the induced extriangulated structure $(\Csc/\mathcal{N}, \widetilde{\Ebb}, \widetilde{\sfr})$ is also hereditary.
\end{theorem}

\begin{proof}
An element in $\widetilde{\Ebb}^2(M, N)$ can be represented by a pair of $\sfr$-triangles in $(\Csc/\mathcal{N}, \widetilde{\Ebb}, \widetilde{\sfr})$ which together form a diagram of the form 
\[
N \overset{a}{\longrightarrow} B\overset{b}{\rightarrowtail}C\overset{c}{\twoheadrightarrow}D \overset{d}{\longrightarrow} E  \overset{e}{\longrightarrow} F\overset{f}{\rightarrowtail}G\overset{g}{\twoheadrightarrow}H \overset{h}{\longrightarrow} M,
\]
where
$B\overset{b}{\rightarrowtail}C\overset{c}{\twoheadrightarrow}D$ and $F\overset{f}{\rightarrowtail}G\overset{g}{\twoheadrightarrow}H$ are conflations in $\CEs$, $a, d, e, h \in  \mathscr{S}_{\mathcal{N}}$. By using \cref{lem:s_n:hered}, the composition $D \overset{d}{\longrightarrow} E  \overset{e}{\longrightarrow} F$ can be rewritten as $D \overset{d'}{\longrightarrow} E'  \overset{e'}{\longrightarrow} F$, with $d' \in \mathcal{L}_{\mathcal{N}}, e' \in \mathcal{R}_{\mathcal{N}}$. 


By applying \cref{hered_equiv}.(v) three times in $\Csc$, we construct the following commutative diagram:
\[\begin{tikzcd}
	B & C & D & E \\
	B & X & {E'} & F \\
	B & Z & Y & G \\
	& H & H & H,
	\arrow[tail, from=1-1, to=1-2]
	\arrow[two heads, from=1-2, to=1-3]
	\arrow["\sim", from=1-3, to=1-4]
	\arrow["\wr", from=1-4, to=2-4]
	\arrow["\wr", tail, from=1-3, to=2-3]
	\arrow["\sim", two heads, from=2-3, to=2-4]
	\arrow[Rightarrow, no head, from=1-1, to=2-1]
	\arrow[Rightarrow, no head, from=2-1, to=3-1]
	\arrow[Rightarrow, no head, from=4-2, to=4-3]
	\arrow[Rightarrow, no head, from=4-3, to=4-4]
	\arrow[tail, from=2-1, to=2-2]
	\arrow[two heads, from=2-2, to=2-3]
	\arrow["\wr", tail, from=1-2, to=2-2]
	\arrow[tail, from=3-1, to=3-2]
	\arrow[two heads, from=3-2, to=3-3]
	\arrow["\sim", two heads, from=3-3, to=3-4]
	\arrow[tail, from=2-2, to=3-2]
	\arrow[two heads, from=3-2, to=4-2]
	\arrow[tail, from=2-3, to=3-3]
	\arrow[two heads, from=3-3, to=4-3]
	\arrow[tail, from=2-4, to=3-4]
	\arrow[two heads, from=3-4, to=4-4]
\end{tikzcd}\]
where we abbreviate by $X \overset\sim\to Y$ morphisms belonging to $\mathscr{S}_{\mathcal{N}}$,
thus showing that the category $(\Csc/\mathcal{N}, \widetilde{\Ebb}, \widetilde{\sfr})$ is hereditary.
\end{proof}


\section{0-Auslander extriangulated categories}
\label{section:0-Auslander}

After a short discussion on the notion of dominant dimension (\cref{ssection: domdim}), we introduce the notion of a 0-Auslander extriangulated category (\cref{ssection:0-Auslander}) and illustrate it with many examples arising from representation theory (\cref{ssection:examples0-Auslander}), the most basic being the homotopy category $K^{[-1,0]}(\proj \Lambda)$ of 2-term complexes with projective components over a finite-dimensional algebra and cluster categories with the largest relative extriangulated structure making some cluster tilting object projective.
We fix an extriangulated category $(\Csc,\Ebb,\mathfrak{s})$.

\subsection{(Co)dominant dimensions}
\label{ssection: domdim}

The main result of this subsection is \cref{Thm:Auslander:equiv}. We first prove some useful lemmata.


\begin{lemma}[\cite{LiuNakaoka}]
Let $Y\infl P\defl X\dashrightarrow$ be an $\sfr$-triangle with $P$ projective in $\Csc$.
Then $\Ebb^2(X,-) \cong \Ebb(Y,-)$.
\end{lemma}

\begin{proof}
This follows from the long exact sequence associated with the $\sfr$-triangle by noticing that $\Ebb(P,-)=0$ and $\Ebb^2(P,-)=0$.
\end{proof}

The following two immediate consequences will often be used in the article. 
\begin{corollary}\label{lemma: E2=0}
Assume that there is an $\sfr$-triangle $P_1\infl P_0\defl X\dashrightarrow$ with $P_0, P_1$ projective in $\Csc$.
Then for any object $Y\in\Csc$, we have $\mathbb{E}^2(X,Y)=0$.
\end{corollary}

\begin{corollary}\label{lemma: Sigma P injective}
Let $P\infl I\defl \Sigma P\dashrightarrow$ be an $\sfr$-triangle with $I$ injective.
If $\mathbb{E}^2(-,P)=0$, then $\Sigma P$ is injective.
\end{corollary}

\begin{lemma}\label{lemma: pdim=1 iff idim=1}
Assume that, for each projective object $P$, we have $\mathbb{E}^2(-,P)=0$ and there exists an $\sfr$-triangle $P\infl I\defl \Sigma P\dashrightarrow$ with $I$ injective. If an object $X\in\Csc$ is the cone of an inflation between projective objects, then it is the fiber of a deflation between injective objects.
\end{lemma}

\begin{proof}
Let $P_1\infl P_0\defl X$ be a conflation with $P_0,P_1$ projectives.
Fix two conflations $P_1\infl I\defl \Sigma P_1$ and $P_0\infl J\defl \Sigma P_0$ with $I,J$ injectives.
By successively forming the diagrams:
\[\begin{tikzcd}
	{P_1} & {P_0} & \,X \\
	I & {X\oplus I} & \,X \\
	{\Sigma P_1} & {\Sigma P_1}
	\arrow[tail, from=1-1, to=1-2]
	\arrow[two heads, from=1-2, to=1-3]
	\arrow[tail, from=1-1, to=2-1]
	\arrow[two heads, from=2-1, to=3-1]
	\arrow[tail, from=2-1, to=2-2]
	\arrow[no head, from=1-3, to=2-3]
	\arrow[no head, from=3-1, to=3-2]
	\arrow[tail, from=1-2, to=2-2]
	\arrow[two heads, from=2-2, to=2-3]
	\arrow[two heads, from=2-2, to=3-2]
	\arrow[shift left=1, no head, from=1-3, to=2-3]
	\arrow[shift right=1, no head, from=3-1, to=3-2]
\end{tikzcd}
\hspace{2cm}
\begin{tikzcd}
	{P_0} & {X\oplus I} & {\Sigma P_1} \\
	J & {\Sigma P_1\oplus J} & {\Sigma P_1} \\
	{\Sigma P_0} & {\Sigma P_0}
	\arrow[tail, from=1-1, to=1-2]
	\arrow[two heads, from=1-2, to=1-3]
	\arrow[tail, from=1-1, to=2-1]
	\arrow[two heads, from=2-1, to=3-1]
	\arrow[tail, from=2-1, to=2-2]
	\arrow[no head, from=1-3, to=2-3]
	\arrow[no head, from=3-1, to=3-2]
	\arrow[tail, from=1-2, to=2-2]
	\arrow[two heads, from=2-2, to=2-3]
	\arrow[two heads, from=2-2, to=3-2]
	\arrow[shift left=1, no head, from=1-3, to=2-3]
	\arrow[shift right=1, no head, from=3-1, to=3-2]
\end{tikzcd}\]
we constuct a conflation
$X\oplus I\infl \Sigma P_1\oplus J\defl \Sigma P_0$,
where $\Sigma P_0, \Sigma P_1$ are injective by \cref{lemma: Sigma P injective}.
Because $I$ is injective, it splits from the conflation and $X$ is the fiber of a deflation between injectives.
\end{proof}

\begin{definition}
Let $(\Csc, \Ebb, \mathfrak{s})$ be an extriangulated category with enough projectives. We define its \defn{dominant dimension} $\domdim(\Csc, \Ebb, \mathfrak{s})$ to be the largest integer $n$ such that for any projective object $P$, there exist $n$ $\sfr$-triangles
$$P \infl I_0 \defl M_1 \dashrightarrow ;$$
$$M_1 \infl I_1 \defl M_2 \dashrightarrow ;$$
$$\cdots$$
$$M_{n-1} \infl I_{n-1} \defl M_n \dashrightarrow ,$$
with $I_k$ being projective-injective for all $0 \le k \le n-1$. If such an $n$ does not exist, we let $\domdim(\Csc, \Ebb, \mathfrak{s})=\infty.$
Dually, let $(\Csc, \Ebb, \mathfrak{s})$ be an extriangulated category with enough injectives. We define its \defn{codominant dimension} $\codomdim(\Csc, \Ebb, \mathfrak{s})$ to be the largest integer $n$ such that for any injective object $I$, there exist $n$ $\sfr$-triangles
$$N_1 \infl P_0 \defl I \dashrightarrow ;$$
$$N_2 \infl P_1 \defl N_1 \dashrightarrow ;$$
$$\cdots$$
$$N_n \infl P_{n-1} \defl N_{n-1} \dashrightarrow ,$$
with $P_k$ being projective-injective for all $0 \le k \le n-1$. If such an $n$ does not exist, we let $\codomdim(\Csc, \Ebb, \mathfrak{s})$ equal $\infty.$
\end{definition}

\begin{proposition} \label{Thm:Auslander:equiv}
Let $(\Csc, \Ebb, \mathfrak{s})$ be an extriangulated category. The following conditions are equivalent.

\begin{itemize}
    \item[(i)] $(\Csc, \Ebb, \mathfrak{s})$ has enough projectives and $\mathrm{pd}(\Csc, \Ebb) \leq 1 \leq \mathrm{dom. dim}(\Csc, \Ebb).$
    \item[(ii)] $(\Csc, \Ebb, \mathfrak{s})$ has enough injectives and $\mathrm{id}(\Csc, \Ebb) \leq 1 \leq \mathrm{codom. dim}(\Csc, \Ebb).$
\end{itemize}
\end{proposition}

Explicitly, the conditions in (i) mean that every object of $\Csc$ is the cone of an inflation between projectives, and that each projective object $P$ admits an inflation $P \infl Q,$ with $Q$ being projective-injective (using heredity, it then follows from the long exact sequence that the cone is injective). Dually, the conditions in (ii) mean that every object of $\Csc$ is the cocone of a deflation between injectives, and that each injective object $I$ admits an inflation $Q \defl I,$ with $Q$ being projective-injective.

\begin{proof}
We only prove the implication (i) $\Rightarrow$ (ii), since the proof of the converse result is dual. Assume that (i) holds. The facts that $\Csc$ has enough injectives and $\mathrm{id}(\Csc,\Ebb) \leq 1$ follow from \cref{lemma: pdim=1 iff idim=1}. It remains to prove that $1 \leq \mathrm{codom. dim}(\Csc).$ Let $I$ be an injective object in $\Csc$. By the assumption, we have a conflation $P_1 \infl P_0 \defl I,$ with $P_0, P_1$ projective in $\Csc$. We also have a conflation $P_0 \infl Q \infl \Sigma P_0,$ with $Q$ projective-injective. Since $\mathrm{pd}(\Csc, \Ebb) \leq 1$, we know by \cref{hered_equiv} that the category is hereditary and, in particular, $\mathbb{E}^2(-,P)=0$. Thus, by \cref{lemma: Sigma P injective}, $\Sigma P_0$ is injective. By (ET4), we can complete these two conflations to a diagram 

\[\begin{tikzcd}
	{P_1} & {P_0} & \,I \\
	{P_1} & {Q} & \,R \\
	& {\Sigma P_0} & {\Sigma P_0}
	\arrow[tail, from=1-1, to=1-2]
	\arrow[two heads, from=1-2, to=1-3]
	\arrow[no head, from=1-1, to=2-1]
	\arrow[two heads, from=2-3, to=3-3]
	\arrow[tail, from=2-1, to=2-2]
	\arrow[tail, from=1-3, to=2-3]
	\arrow[no head, from=3-2, to=3-3]
	\arrow[tail, from=1-2, to=2-2]
	\arrow[two heads, from=2-2, to=2-3]
	\arrow[two heads, from=2-2, to=3-2]
	\arrow[shift left=1, no head, from=1-1, to=2-1]
	\arrow[shift right=1, no head, from=3-2, to=3-3]
\end{tikzcd}\]

Since $I$ is injective, the rightmost vertical conflation splits. Thus, we get a deflation $Q \defl I \oplus \Sigma P_0.$ By composing it with the projection onto the first factor, which is also a deflation, we obtain a deflation $Q \defl I$ thanks to $\mathrm{(ET4)}^{\mathrm{op}}$. Since $I$ was an arbitrary injective object, this proves that $1 \leq \mathrm{codom. dim}(\Csc).$
\end{proof}

\subsection{0-Auslander extriangulated categories}
\label{ssection:0-Auslander}

\begin{definition}
 An
 extriangulated category is a \defn{0-Auslander extriangulated category} if it satisfies the equivalent conditions of \cref{Thm:Auslander:equiv}.
 Explicitly, it has enough projectives, enough injectives, global dimension at most one, dominant dimension at least one, and codominant dimension at least one.
 
A 0-Auslander extriangulated category is \defn{reduced} if (up to isomorphism) its only projective-injective object is 0.
\end{definition}

\begin{remark} \label{rk: ss or hered}
 If $\Csc$ is a 0-Auslander extriangulated category, then it is either split, i.e. $\Ebb (-, -) = 0$ (hence we have $\gldim\Csc = 0$ and $\domdim\Csc = \codomdim\Csc = \infty$), or it has $\gldim\Csc=1=\domdim\Csc = \codomdim\Csc$.
\end{remark}

\begin{remark}
 Unravelling the definition, reduced 0-Auslander extriangulated categories are precisely the categories arising in \cite[Section 4]{PadrolPaluPilaudPlamondon}:
\begin{enumerate}
 \item For every object $X\in\Csc$, there is a conflation $P_1 \infl P_0 \defl X$ with $P_0,P_1$ projective in $\Csc$.
 \item For any projective $P\in\Csc$, the morphism $P\infl 0$ is an inflation.
\end{enumerate}
\end{remark}

\begin{remark}
\label{rk:reduced}
 The ideal quotient of an arbitrary 0-Auslander category by the morphisms factoring through projective-injective objects is a reduced 0-Auslander extriangulated category. This will be used many times in the proofs.
\end{remark}

\begin{proposition}
\label{prop:module as quotient by injectives}
Let $\Csc$ be a 0-Auslander extriangulated category with full subcategory of projectives $\Pcal$, and full subcategory of injectives $\Ical$.
Write $\overline{\Pcal}$ for the image of $\Pcal$ in the quotient of $\Csc$ by the ideal generated by the projective-injective objects.
Then the functor sending $X\in\Csc$ to $\Csc/_{\Pcal\cap\Ical} (-,X)|_{\overline{\Pcal}}$ induces an equivalence of categories
\[
\Csc/\Ical \cong \modd \overline{\Pcal}.
\]
\end{proposition}

\begin{proof}
This was proved for reduced 0-Auslander extriangulated categories in \cite[Section 4]{PadrolPaluPilaudPlamondon}.
The non-reduced case immediately follows by reduction.
\end{proof}

By similarly applying \cite[Remark 4.32 and Corollary 4.35]{PadrolPaluPilaudPlamondon}, we obtain:

\begin{lemma}
\label{lemma:equiv proj inj}
There are quasi-inverse equivalences of additive categories $\Sigma : \ovl{\Pcal} \overset{\simeq}{\leftrightarrows} \und{\Ical} : \Omega$ given on objects by exchanging $P$ and $I$ in a fixed conflation $P\infl Q\defl I$, with $Q$ projective-injective.
\end{lemma}

\subsection{Examples of 0-Auslander extriangulated categories}
\label{ssection:examples0-Auslander}

\subsubsection{The minimal example}
The most elementary example of a reduced 0-Auslander category over a field $k$ has two non-zero indecomposable objects (up to isomorpism) $P$ and $I$ with endomorphism rings isomorphic to $k$, no non-zero morphisms between them, and a one-dimensional space $\mathbb{E}(I, P) \cong k$ realized by $\sfr$-triangles of the form

\[
P \infl 0 \defl I \overset{\delta}\dashrightarrow.
\]

While this small example may already look rather peculiar, it admits several interpretations in terms of well-known categories. Namely, it is equivalent, as an extriangulated category, to each of the following categories:

\begin{itemize}
    \item The homotopy category $K^{[-1,0]}(\proj  k A_1)$ of complexes with projective components and concentrated in degrees -1 and 0 over the path algebra of the quiver $A_1$, with the extriangulated structure induced by the embedding into the bounded homotopy category;
    \item The cluster category of type $A_1$, taken with the maximal relative structure making one of the two cluster tilting objects projective;
    \item The quotient of the module category of the path algebra of the quiver $A_2$ by the ideal of morphisms factoring through projective-injective objects, with the extriangulated structure induced from the abelian structure on the module category;
    \item The stable category of the module category of the preprojective algebra of type $A_2$, taken with the maximal relative structure making one of the two simple objects projective.
\end{itemize}

\subsubsection{Module categories}
\label{Ex:RepQ}

Let $\Lambda$ be a basic connected finite-dimensional algebra over an algebraically closed field $k$. The category of finite-dimensional $\Lambda$-modules, considered with the standard, i.e. abelian, extriangulated structure, has dominant dimension (at most) one if and only if $\Lambda$ has dominant dimension (at most) one in the classical sense \cite{Tachikawa}. By \cref{rk: ss or hered}, $\modd \Lambda$ is 0-Auslander if and only if $\Lambda$ is semisimple or $\gldim\Lambda = 1 = \domdim\Lambda.$ The latter happens if and only if $\Lambda$ is Morita equivalent to the $m \times m$ upper triangular matrix algebra over $k$, for $m \in \mathbb{N}$; see \cite[Proposition 1.17]{Iyama} for a detailed proof of a slightly more general result. Equivalently, $\gldim\modd \Lambda = 1 = \domdim\modd\Lambda$ if and only if $\Lambda$ is (Morita equivalent to) the path algebra of a linearly oriented quiver of Dynkin type $A$.

This might suggest that there are very few examples of 0-Auslander extriangulated categories. However, for any $\Lambda$, there is an associated reduced 0-Auslander extriangulated category, of which $\modd\Lambda$ is a quotient: its homotopy category of 2-term complexes with projective components.

\subsubsection{Homotopy category of 2-term complexes}
\label{Ex:2-termHomotopy}

Let $K^\text{b}(\proj \Lambda)$ be the homotopy category of bounded complexes of projective modules over an Artin algebra $\Lambda$.
Consider the full subcategory $\Csc = K^{[-1,0]}(\proj \Lambda)$ of $K^\text{b}(\proj \Lambda)$ whose objects are those complexes concentrated in (cohomological) degrees -1 and 0.
Objects in $\Csc$ are called two-term complexes and have been studied in~\cite{DerksenFei}, and in~\cite{AdachiIyamaReiten,Kimura} in connection with $\tau$-tilting theory and silting theory.
Then $\Csc$ is a reduced 0-Auslander extriangulated category, with projective objects being the stalk complexes concentrated in degree 0 and injective objects those concentrated in degree -1. In this setting, \cref{prop:module as quotient by injectives} means that $\modd\Lambda$ is the costable category of $\Csc$.
\cref{Th:mutation} thus applies and recovers 2-term silting mutation~\cite{Kimura}.

\subsubsection{Category of 2-term complexes}
\label{Ex:2-termComplexes}

There is an ``unreduced'' version of the previous example:
Let $\Acal$ be an additive category and let $C^{[-1,0]}(\Acal)$ be the full subcategory of the category of complexes over $\Acal$, whose objects are concentrated in degrees -1 and 0.
Endow $C^{[-1,0]}(\Acal)$ with the exact structure given by component-wise split sequences.
By~\cite[Proposition 3.6, Corollary 3.9]{BautistaSalorioZuazua}, $C^{[-1,0]}(\Acal)$ is a 0-Auslander exact category with projective objects having no injective summands the complexes concentrated in degree 0, injective objects having no projective summands the complexes concentrated in degree -1, and projective-injective indecomposable objects the complexes of the form $P\xrightarrow{1}P$, for $P\in\Acal$ indecomposable.

Applying \cref{Th:mutation} to this setting, with $\Acal=\proj A$, yields a version of 2-term silting mutation ``with coefficients''.

\subsubsection{Maximal almost rigid modules}
\label{Ex:AlmostRigid}

Fix an algebraically closed field $k$, and let $Q$ be an orientation of a Dynkin diagram of type $A_{n+2}$.
Reading \cite{reading2006cambrian} considered a polygon $P(Q)$ whose geometric realization depends on the orientation of $Q$ and defined a partial order on the set of triangulations of $P(Q)$. He proved that the resulting poset is a lattice which he called a \defn{Cambrian lattice} of $Q$. This construction generalizes Tamari lattices which are shown to correspond to $Q$ being linearly oriented. Barnard--Gunawan--Meehan--Schiffler \cite{BarnardGunawanMeehanSchiffler} found a representation-theoretic realization of such Cambrian lattices for quivers of type $A$ by defining maximal almost rigid modules and finding a natural partial order on them.

\begin{definition}[\cite{BarnardGunawanMeehanSchiffler}]
Let $M$ be a basic representation of $Q$ over $k$, and let $M=M_1\oplus\cdots\oplus M_r$ be a decomposition of $M$ into (pairwise non-isomorphic) indecomposables. The representation $M$ is called \defn{almost rigid} if, for any $i,j=1,\ldots,r$ and any non-split short exact sequence $0\to M_i\to E\to M_j\to 0$, the middle term $E$ is indecomposable, and it is called \defn{maximal almost rigid} if moreover $M\oplus N$ almost rigid implies $N\in\add M$.
\end{definition}

Let $\Esc$ be the category $\modd kQ$, endowed with the collection of all short exact sequences whose middle term is not indecomposable.
The following result is due to Thomas Br\"ustle, Eric Hanson, Sunny Roy and Ralf Schiffler~\cite{BrustleSchiffler-Oberwolfach}:

\begin{theorem}
The category $\Esc$ is a 0-Auslander exact category whose silting objects precisely coincide with the maximal almost rigid representations.
\end{theorem}

Applying \ref{Th:mutation} to the category $\Esc$ thus recovers the mutation of maximal almost rigid representations see~\cite[Theorem 6.8]{BarnardGunawanMeehanSchiffler}.

\subsubsection{Cluster categories}
\label{Ex:ClusterCategories}

Let $K$ be a field and let $\Csc$ be a $K$-linear triangulated category with shift functor $\Sigma$.
Assume that $\Csc$
\begin{itemize}
 \item is Hom-finite: for any $X,Y\in\Csc$, $\Csc(X,Y)$ is a finite-dimensional $K$-vector space;
 \item is 2-Calabi--Yau: $\Sigma^2$ is a Serre functor, or in other words there are binatural isomorphisms $\Csc(X,\Sigma Y) \cong D\Csc(Y,\Sigma X)$, where $D=\Hom_K(-,K)$;
 \item has a cluster-tilting object $T$:
 \begin{itemize}
  \item T is rigid: $\Csc(T,\Sigma T)=0$ and 
  \item $\Csc(T,\Sigma X) = 0$ (or equivalently $\Csc(X,\Sigma T)=0$) implies $X\in\add T$.
 \end{itemize}
\end{itemize}

Examples of such categories are given by cluster categories of acyclic quivers~\cite{BuanMarshReinekeReitenTodorov} (see also~\cite{CalderoChapotonSchiffler}) and of algebras of global dimension 2 or Jacobi-finite quivers with potentials~\cite{Amiot-ClusterCats}.

Replace the triangulated structure of $\Csc$ by the maximal relative~\cite{HerschendLiuNakaoka-I} extriangulated structure making $T$ projective (see~\cref{definition:relative structures}): endow $\Csc$ with those triangles
$X\to Y\to Z\xrightarrow{\eta}\Sigma X$ for which $\eta$ factors through $\add\Sigma T$.
With this extriangulated structure, $\Csc$ is a reduced 0-Auslander extriangulated category, with projective objects $\add T$ and injective objects $\add\Sigma T$.
This extriangulated structure was introduced in~\cite{PadrolPaluPilaudPlamondon} where it was used for studying polytopal realizations of $g$-vector fans associated with cluster algebras of finite type.

Applied to this setting, \cref{Th:mutation,Cor:mutation} recover the mutation theory of cluster-tilting objects (with a natural orientation): Even though our results only see one of the two exchange triangles associated with mutations (the one for which $g$-vectors, or indices, are additive), the 2-Calabi--Yau property gives the other exchange triangle.

\begin{remark}
 We note that integral cluster categories (or acyclic cluster categories over a principal ideal domain)  of~\cite{KellerScherotzke-Integral} are also 0-Asulander. \cref{Th:mutation} applies to them whenever all rigid objects admit a unique decomposition into indecomposables with local endomorphism rings. The unique decomposition property holds for rigid objects in such categories over arbitrary principal ideal domains, as explained in \cite[Section 5]{KellerScherotzke-Integral} based on results of Crawley-Boevey \cite{CrawleyBoevey}. The locality amounts to the locality of the ground ring. Thus, \cref{Th:mutation} applies to acyclic cluster categories over local principal ideal domains - in other words, over fields (as discussed above) or discrete valuation rings.
 \end{remark}


\subsubsection{Rigid subcategories of triangulated categories}
\label{Ex:RigidTriangulated}

Let $\Dsc$ be a triangulated category with some rigid full subcategory $\Rcal$.
Endow $\Dsc$ with the maximal relative extriangulated structure making all objects in $\Rcal$ projective.
For this relative structure, the full subcategory $\Rcal\ast\Sigma\Rcal$ is extension-closed, hence extriangulated.
One easily checks that $\Rcal\ast\Sigma\Rcal$ is reduced 0-Auslander.

\cref{Th:mutation} thus applies to
\begin{itemize}
 \item Extended cohearts~\cite{IyamaJorgensenYang,PauksztelloZvonareva} of co-$t$-structures~\cite{Pauksztello,Bondarko}, inducing mutation for intermediate co-$t$-structures (this example is detailed in \cref{section:ExtendedCohearts});
 \item Triangulated categories with a fixed cluster tilting object (or subcategory). Our mutation theory thus recovers relative tilting theory~\cite[Theorem 4.10]{YangZhu} (and generalizes it to the setting of~\cite{YangZhouZhu-Ghost}), without making use of $\tau$-tilting theory;
 \item Settings generalizing the previous one by only assuming rigidity instead of cluster tilting, see for example~\cite{FuGengLiu,ZhouZhu-2termRelative}. This implies in particular a nice mutation theory for $d$-rigid objects that are finitely presented with respect to some fixed $d$-rigid subcategory;
\end{itemize}

\subsubsection{2-term subcategories in perfect derived categories}

Let $A$ be a non-positive differential graded algebra. Its perfect derived category $\mbox{per} A$ has a bounded co-$t$-structure with coheart $\add_{\mbox{per} A} A$. The category $\add_{\mbox{per} A} A \ast \add_{\mbox{per} A} A[1]$ gives thus a special case of the previous example, as it is the extended coheart of this co-$t$-structure. Since for a finite-dimensional algebra $\Lambda$ considered as a dg algebra with $0$ differential, its perfect derived category is $K^\text{b}(\proj \Lambda)$, this also generalizes the example from \cref{Ex:2-termHomotopy}.
Such categories appeared in a version of the Brenner-Butler theorem for 2-term subcategories in $K^\text{b}(\proj \Lambda)$ in recent work \cite{BuanZhou}.

The fundamental domain of Amiot's equivalence \cite{Amiot-ClusterCats} is a special case of such categories which plays an important role in the context of additive categorifications of cluster algebras.

\subsubsection{Rigid subcategories of exact categories}
\label{Ex:RigidExact}

Let $\Esc$ be an exact category with enough injectives, and let $\Rcal$ be a rigid full subcategory of $\Esc$, containing all injectives.
Write $\Csc=\operatorname{Coker}(\Rcal,\Rcal)$ for the full subcategory of $\Esc$ whose objects are those $X$ for which there exists a conflation $R_1\infl R_0\defl X$ with $R_0,R_1\in\Rcal$.

Endow $\Esc$ with the maximal relative exact structure, denoted $\Esc_\Rcal$, making all objects in $\Rcal$ projective.
Because $\Rcal$ is rigid, $\Csc$ is extension-closed in $\Esc_\Rcal$ hence exact, for this relative structure.
Moreover, all (non-relative) conflations $R_1\infl R_0\defl X$ of $\Esc$ expressing $X$ as an object of $\operatorname{Coker}(\Rcal,\Rcal)$ are conflations in $\Csc$ (for the relative exact structure).

In particular, a conflation $R\infl Q\defl \Sigma R$ in $\Esc$ with $R\in\Rcal$ and $Q$ injective in $\Esc$ gives a conflation in $\Csc$.
Moreover, $Q$ is projective-injective in $\Csc$ (because it is injective in $\Esc$, hence in $\Csc$, and also belongs to $\Rcal$, hence is projective in $\Csc$).
The conflations $R_1\infl R_0\defl X$ witness the fact that $\Csc$ has enough projectives and has vanishing $\Ext^2$.
It follows that $\Sigma R$ is injective in $\Csc$.
Hence, $\Csc$ is a 0-Auslander extriangulated category, which is generally non-reduced.

This class of examples contains the categories of 2-term complexes of \cref{Ex:2-termComplexes}, but also the Frobenius exact cluster categories appearing in~\cite{GeissLeclerclSchroer-survey,FuKeller,JensenKingSu,Pressland,Pressland-corrigendum} thus showing relative versions (in the sense of \cref{Ex:RelativeTilting}) of cluster-tilting mutation in that setting.

\subsubsection{Rigid subcategories of extriangulated categories}
\label{Ex:RigidExtriangulated}

One easily generalizes \cref{Ex:RigidExact} to extriangulated categories, up to replacing cokernels of inflations by cones over inflations.
This gives a common generalization of \cref{Ex:RigidExact,Ex:RigidTriangulated}.

As a consequence we recover relative tilting theory also for extriangulated categories~\cite{LiuZhou-relative}, as detailed in the following section.

\subsubsection{Relative tilting theory}
\label{Ex:RelativeTilting}

In this section, we give a (possibly new) point-of-view on relative tilting theory~\cite{YangZhu}, thus including it into our approach to mutation via 0-Auslander extriangulated categories.

Let $\Csc$ be an extriangulated category with enough projectives and enough injectives.
Assume that $\Csc$ has a \defn{left-cluster tilting} subcategory $\Rcal$: $\Rcal$ is a rigid, additive full subcategory such that, for every object $X\in \Csc$, there is an $\sfr$-triangle $X\infl R_X^0 \defl R_X^1 \dashrightarrow$ with $R_X^0,R_X^1\in\Rcal$.

\begin{notation}
For each object $X\in\Csc$, we fix two $\sfr$-triangles
\[
X\overset{i_X}{\infl} I_X \overset{p_X}{\defl} \Sigma X \overset{\eta_X}{\dashrightarrow}
\text{ and } \Omega X\overset{\iota_X}{\infl} P_X \overset{\pi_X}{\defl} X \overset{\eps_X}{\dashrightarrow},
\]
where $P_X$ is projective and $I_X$ is injective.
If $\Xcal$ is any full additive subcategory of $\Csc$, we write $\Omega\Xcal$ (resp. $\Sigma\Xcal$) for the full subcategory whose objects are all those $Y$ that appear in an $\sfr$-triangle of the form $Y\infl P\defl X \dashrightarrow$ with $P$ projective and $X$ in $\Xcal$ (resp. of the form $X\infl I\defl Y\dashrightarrow$ with $I$ injective and $Y$ in $\Xcal$).
We note that subcategories of the form $\Omega\Xcal$ contain all projectives.
Dually subcategories $\Sigma\Xcal$ contain all injectives.
\end{notation}

By~\cite[Proposition 3.19]{HerschendLiuNakaoka-I}, there is a relative extriangulated structure on $\Csc$ given by:
\[\Ebb^\Rcal(Z,X) = \{\delta\in\Ebb(Z,X) \;|\; X\xrightarrow{\forall f}R \text{ with } R\in\Rcal, f_\ast\delta=0\}.\]
This is the biggest closed additive sub-bifunctor of $\Ebb$ for which all objects in $\Rcal$ are injective.

As already mentioned in \cref{Ex:RigidTriangulated}, there is a dual version, denoted by $\Ebb_\Rcal$ that makes all objects in $\Rcal$ projectives.

\begin{definition}
An object $X\in\Csc$ (or a full additive subcategory $\Xcal$) is \defn{$\Rcal$-rigid}, or alternatively \defn{relative rigid}, if it is rigid in $(\Csc,\Ebb^\Rcal)$.
\end{definition}

Although this is not the definition that appears in the literature, the lemma below shows that it is equivalent to it.

\begin{lemma}
 For any $X\in\Csc$, the following are equivalent:
 \begin{enumerate}[(a)]
     \item $X$ is $\Rcal$-rigid,
     \item Any morphism in $[\Rcal](X,\Sigma X)$ factors through $p_X$.
 \end{enumerate}
 \end{lemma}

The lemma follows from the more detailed:

\begin{proposition}
Let $\Rcal$ be a left-cluster tilting subcategory of $\Csc$.
Then we have:
\begin{enumerate}
    \item $\Ebb^\Rcal = \Ebb_{\Omega\Rcal}$;
    \item for any $X,Y\in\Csc$, $\Ebb_{\Omega\Rcal}(X,Y)$ is the quotient of $[\Rcal](X,\Sigma Y)$ by the image of $(p_Y)_\ast: \Csc(X,I_Y) \to \Csc(X,\Sigma Y)$;
    \item $(\Csc,\Ebb^\Rcal)$ has enough injectives and is hereditary;
    \item If moreover, $\Rcal$ contains the projective objects of $(\Csc,\Ebb)$, then the extriangulated category $(\Csc,\Ebb_{\Omega\Rcal})=(\Csc,\Ebb^\Rcal)$ is 0-Auslander with subcategory of injectives $\Rcal$ and subcategory of projectives $\Omega\Rcal$.
\end{enumerate}

\end{proposition}

\begin{proof}
We consider the exact sequence
\[
(\ast) \hspace{15pt} \Csc(X,I_Y)\to\Csc(X,\Sigma Y) \xrightarrow{(\eta_Y)_\sharp} \Ebb(X,Y) \to 0
\]
associated with the $\sfr$-triangle $Y\overset{i_Y}{\infl} I_Y\overset{p_Y}{\defl}\Sigma Y\overset{\eta_Y}{\dashrightarrow}$,
and we prove the statement in six easy steps:
\begin{enumerate}[(a)]
    \item For any $u\in [\Rcal](X,\Sigma Y)$, we have $u^\ast\eta_Y \in \Ebb_{\Omega\Rcal}(X,Y)$.
    \item For any $u\in \Csc(X,\Sigma Y)$ such that $u^\ast\eta_Y\in\Ebb_{\Omega\Rcal}(X,Y)$, we have $u\in [\Rcal](X,\Sigma Y)$.
    \item $\Ebb_{\Omega\Rcal}\subseteq\Ebb^\Rcal$.
    \item $\Ebb^\Rcal\subseteq\Ebb_{\Omega\Rcal}$.
    \item $(\Csc,\Ebb^\Rcal)$ has injective dimension at most one.
    \item If $\Rcal$ contains all projectives, then $(\Csc,\Ebb^\Rcal)$ has codominant dimension at least one.
\end{enumerate}
We now begin the proof.

(a) Let $u\in[\Rcal](X,\Sigma Y)$, let $R\in\Rcal$, let $\Omega R\overset{\iota_R}{\infl} P_R \defl R \dashrightarrow$ and let $P\oplus \Omega R\xrightarrow{[\pi\, f]}X$ be any morphism with $P$ projective.
We want to show that $[\pi\,f]^\ast u^\ast \eta_Y=0$, or equivalently (by the exact sequence ($\ast$) and since $P$ is projective) that $uf$ factors through $p_Y$.
Because $\Rcal$ is rigid, any morphism from $\Omega R$ to some $R'\in\Rcal$ factors through $\iota_R$.
Since $u$ factors through $\Rcal$ by assumption, this implies that the composition $uf$ factors through $\iota_R:\Omega R\to P_R$.
The latter object being projective, $uf$ factors through $p_Y$.
\[\begin{tikzcd}
	&& {\Omega R} & {P_R} & R & {} \\
	&& X & {R'} \\
	Y & {I_Y} & {\Sigma Y} & {}
	\arrow["{\iota_R}", tail, from=1-3, to=1-4]
	\arrow[two heads, from=1-4, to=1-5]
	\arrow[dashed, from=1-5, to=1-6]
	\arrow["f"', from=1-3, to=2-3]
	\arrow["u"'{pos=0.2}, from=2-3, to=3-3]
	\arrow[from=2-3, to=2-4]
	\arrow[from=2-4, to=3-3]
	\arrow[tail, from=3-1, to=3-2]
	\arrow[two heads, from=3-2, to=3-3]
	\arrow[dashed, from=3-3, to=3-4]
	\arrow[dotted, from=1-4, to=2-4]
	\arrow[curve={height=-18pt}, dotted, from=1-4, to=3-2]
\end{tikzcd}\]

(b) Let $u\in \Csc(X,\Sigma Y)$ be such that $u^\ast\eta_Y\in\Ebb_{\Omega\Rcal}(X,Y)$.
There are $\sfr$-triangles $X\infl R^0_X \defl R_X^1 \dashrightarrow$ and $\Omega R_X^0\infl P\defl R_X^0\dashrightarrow$ with $R_X^0,R_X^1\in\Rcal$ and where $P$ is projective.
By (ET4$^\text{op}$) and \cite[Proposition 1.20]{LiuNakaoka}, there is a commutative diagram
\[\begin{tikzcd}
	{\Omega R_X^0} & {\Omega R_X^0} \\
	{\Omega R_X^1} & P & {R_X^1} \\
	X & {R_X^0} & {R_X^1}
	\arrow[tail, from=1-1, to=2-1]
	\arrow["p"', two heads, from=2-1, to=3-1]
	\arrow[tail, from=1-2, to=2-2]
	\arrow["q", two heads, from=2-2, to=3-2]
	\arrow["i", tail, from=2-1, to=2-2]
	\arrow[two heads, from=2-2, to=2-3]
	\arrow["j"', tail, from=3-1, to=3-2]
	\arrow[two heads, from=3-2, to=3-3]
	\arrow[Rightarrow, no head, from=2-3, to=3-3]
	\arrow[Rightarrow, no head, from=1-1, to=1-2]
	\arrow["{(wPO)}"{description}, draw=none, from=2-1, to=3-2]
\end{tikzcd}\]
where the square indicated by $(wPO)$ is a weak push-out.
By assumption the composition $up$ factors through $p_Y$.
Since $I_Y$ is injective and $i$ is an inflation, we obtain a commutative diagram as below and applying the property of a weak push-out gives a factorisation of $u$ through $j$.
Hence $u$ belongs to $[\Rcal](X,\Sigma Y)$.
\[\begin{tikzcd}
	{\Omega R_X^1} & X \\
	P & {R_X^0} \\
	{I_Y} && {\Sigma Y}
	\arrow["i", tail, from=1-1, to=2-1]
	\arrow["q", two heads, from=2-1, to=2-2]
	\arrow["j", tail, from=1-2, to=2-2]
	\arrow["{p_Y}", two heads, from=3-1, to=3-3]
	\arrow["u", curve={height=-12pt}, from=1-2, to=3-3]
	\arrow["{(1)}"{description}, curve={height=18pt}, dotted, from=1-1, to=3-1]
	\arrow["{(2)}", dotted, from=2-1, to=3-1]
	\arrow["{(3)}", dotted, from=2-2, to=3-3]
	\arrow["p", two heads, from=1-1, to=1-2]
\end{tikzcd}\]

(c) Let $\delta\in\Ebb_{\Omega\Rcal}(X,Y)$.
Combining (b) with the exact sequence $(\ast)$, we obtain the existence of a morphism $u\in[\Rcal](X,\Sigma Y)$ such that $\delta = u^\ast\eta_Y$.
Fix a factorisation $u=u_2u_1$ where the domain of $u_2$ belongs to $\Rcal$.
We then have, for any morphism $Y\xrightarrow{g}R$ with $R\in\Rcal$:
$g_\ast\delta= g_\ast u^\ast\eta_Y = g_\ast u_1^\ast u_2^\ast\eta_Y = u_1^\ast (u_2^\ast g_\ast\eta_Y) =0$
where the extension inside the brackets vanishes since $\Rcal$ is rigid.
Hence, $\delta\in\Ebb^\Rcal(X,Y)$.
\[\begin{tikzcd}
	Y & E & X & {} \\
	Y & {E'} & {R'} & {} \\
	Y & {I_Y} & {\Sigma Y} & {} \\
	R
	\arrow[tail, from=1-1, to=1-2]
	\arrow[two heads, from=1-2, to=1-3]
	\arrow[tail, from=3-1, to=3-2]
	\arrow[two heads, from=3-2, to=3-3]
	\arrow["{u_1}", from=1-3, to=2-3]
	\arrow["{u_2}", from=2-3, to=3-3]
	\arrow["\delta", dashed, from=1-3, to=1-4]
	\arrow["{\eta_Y}", dashed, from=3-3, to=3-4]
	\arrow["g"', from=3-1, to=4-1]
	\arrow[tail, from=2-1, to=2-2]
	\arrow[two heads, from=2-2, to=2-3]
	\arrow["{(u_2)^\ast\eta_Y}", dashed, from=2-3, to=2-4]
	\arrow[from=1-2, to=2-2]
	\arrow[from=2-2, to=3-2]
	\arrow[Rightarrow, no head, from=1-1, to=2-1]
	\arrow[Rightarrow, no head, from=2-1, to=3-1]
	\arrow["u"'{pos=0.2}, curve={height=6pt}, from=1-3, to=3-3]
\end{tikzcd}\]

(d) Let $\delta\in\Ebb^\Rcal(X,Y)$, let $R\in\Rcal$ and let $\Omega R\xrightarrow{f}X$ be any morphism.
There is some $u\in\Csc(X,\Sigma Y)$ such that $\delta = u^\ast\eta_Y$.
Fix an $\sfr$-triangle $Y\infl R_Y^0\defl R_Y^1 \overset{\eps}{\dashrightarrow}$ with $R_Y^0,R_Y^1\in\Rcal$, and apply \cref{lemma:shiftedOctahedron}, together with the injectivity of $I_Y$, to obtain the solid diagram below.
Since $\delta$ belongs to $\Ebb^\Rcal(X,Y)$, the inflation $i$ factors via a morphism (1).
Equivalently, $u$ factors via (2).
The projection  onto $R_Y^1$ of (2) precomposed with $f$ factors through $R_Y^0$ because $\eps\in\Ebb(R_Y^1,Y)=\Ebb_{\Omega \Rcal}(R_Y^1,Y)$ (the latter equality is proven as in (a), using that $R_Y^1$ belongs to $\Rcal$ and that $\Rcal$ is rigid).
Hence (2) precomposed with $f$ factors as the sum of two compositions, one using (3) and the other using (4), which implies that $uf$ factors as (3) followed by $p_Y$.
Because the top-right square is a weak pull-back, we obtain a factorisation of $f$ via (5) which shows that $\delta$ belongs to $\Ebb_{\Omega\Rcal}(X,Y)$.

\[\begin{tikzcd}
	&& {\Omega R} \\
	Y & E & X & {} \\
	Y & {I_Y} & {\Sigma Y} & {} \\
	{R_Y^0} & {\small{I_Y\oplus R_Y^1}} & {\Sigma Y} & {} \\
	{R_Y^1} & {R_Y^1} \\
	{} & {}
	\arrow["f", from=1-3, to=2-3]
	\arrow[tail, from=2-1, to=2-2]
	\arrow[two heads, from=2-2, to=2-3]
	\arrow["\delta", dashed, from=2-3, to=2-4]
	\arrow[Rightarrow, no head, from=2-1, to=3-1]
	\arrow["i"', tail, from=3-1, to=4-1]
	\arrow[two heads, from=4-1, to=5-1]
	\arrow["\varepsilon"', dashed, from=5-1, to=6-1]
	\arrow[from=2-2, to=3-2]
	\arrow[tail, from=3-2, to=4-2]
	\arrow[two heads, from=4-2, to=5-2]
	\arrow[dashed, from=5-2, to=6-2]
	\arrow["u", from=2-3, to=3-3]
	\arrow[Rightarrow, no head, from=3-3, to=4-3]
	\arrow[tail, from=3-1, to=3-2]
	\arrow["{_{p_Y}}", two heads, from=3-2, to=3-3]
	\arrow["{\eta_Y}", dashed, from=3-3, to=3-4]
	\arrow[tail, from=4-1, to=4-2]
	\arrow[two heads, from=4-2, to=4-3]
	\arrow[dashed, from=4-3, to=4-4]
	\arrow[Rightarrow, no head, from=5-1, to=5-2]
	\arrow["{(1)}"{description, pos=0.3}, curve={height=6pt}, dotted, from=2-2, to=4-1]
	\arrow["{(2)}"{description, pos=0.7}, curve={height=-6pt}, dotted, from=2-3, to=4-2]
	\arrow["{(3)}"{description, pos=0.7}, curve={height=-12pt}, dotted, from=1-3, to=3-2]
	\arrow["{(4)}"{description, pos=0.8}, dotted, from=1-3, to=4-1]
	\arrow["{(5)}"{description}, curve={height=6pt}, dotted, from=1-3, to=2-2]
\end{tikzcd}\]

(e) For any $X\in \Csc$, there is an $\sfr$-triangle $X\infl R_X^0\defl R_X^1\overset{\eps}{\dashrightarrow}$ with $R_X^0, R_X^1\in\Rcal$.
Because $\Rcal$ is rigid, $\eps$ belongs to $\Ebb^\Rcal(R_X^1,X)$.
By definition of $\Ebb^\Rcal$, objects in $\Rcal$ are $\Ebb^\Rcal$-injective so that this $\sfr$-triangle witnesses the fact that $(\Csc,\Ebb^\Rcal)$ has injective dimension at most one.

(f) For each $R\in\Rcal$, fix an $\sfr$-triangle $\Omega R\infl P_R\defl R\overset{\eps_R}{\dashrightarrow}$, with $P_R$ $\Ebb$-projective, and assume that all $P_R$ belong to $\Rcal$.
The same remark as in (e) shows that $\eps_R$ belongs to $\Ebb^\Rcal(R,\Omega R)$.
Moreover, $R$ is $\Ebb^\Rcal$-injective and $P_R$, which is $\Ebb$-projective and in $\Rcal$, is $\Ebb^\Rcal$-projective-injective.
Since $\Ebb^\Rcal=\Ebb_{\Omega\Rcal}$, all $\Ebb^\Rcal$-projectives are, up to isomorphism, of the form $P\oplus \Omega R$ for some $R\in\Rcal$ and some $\Ebb$-projective $P$.
Therefore $(\Csc,\Ebb^\Rcal)$ has codominant dimension at least one.

This implies statement (4) of the proposition by \cref{Thm:Auslander:equiv}.
\end{proof}

As a consequence, \cref{theorem:mutationIntro} gives a proof of the existence of mutation for relative-tilting objects, and of exchange $\sfr$-triangles, that does not refer to $\tau$-tilting theory.

\subsubsection{Gentle algebras}
\label{Ex:Gentle}

Let $(Q,I)$ be a gentle bound quiver.
There is an associated \defn{blossoming} gentle bound quiver $(Q\bls,I\bls)$ first introduced in~\cite{Asashiba-Sugaku} and re-introduced independently in~\cite{BrustleDouvilleMousavandThomasYildirim,PaluPilaudPlamondon-nonkissing}.
All vertices of the blossoming bound quiver $(Q\bls,I\bls)$ that are not leaves (sinks or sources) are 4-valent, and deleting its leaves recovers the bound quiver $(Q,I)$.
Alternatively, one can view $(Q\bls,I\bls)$ as obtained from the same surface with dissection~\cite{AssemBrustleCharbonneau-JodoinPlamondon,BaurCoelhoSimoes,OpperPlamondonSchroll,PaluPilaudPlamondon-surfaces} as $(Q,I)$, but with more vertices corresponding to boundary arcs.

Using the blossoming bound quiver, the $\tau$-tilting theory of gentle algebras was related in~\cite{BrustleDouvilleMousavandThomasYildirim,PaluPilaudPlamondon-nonkissing} to the combinatorics of non-kissing walks~\cite{McConville}: walks on $(Q,I)$ can be defined as maximal strings for $(Q\bls,I\bls)$.

In~\cite{IyamaNakaokaPalu}, a strategy was proposed on a specific example for constructing an exact category whose objects are walks and where extensions categorify kissings of walks.
We carry out this strategy in full generality in~\cref{section:gentle}.
The category is obtained as an extension-closed full subcategory of $\modd (Q\bls,I\bls)$, and turns out to be (non-reduced) 0-Auslander.
Thus, \cref{Th:mutation} gives a new proof that there are flips for non-kissing facets~\cite{McConville,BrustleDouvilleMousavandThomasYildirim,PaluPilaudPlamondon-nonkissing} and also recovers the mutation theory of support $\tau$-tilting modules~\cite{AdachiIyamaReiten} over gentle algebras.




\section{Mutation of silting subcategories}
\label{section:tilting}

In this section, we compare several possible definitions of \emph{maximality} for rigid objects in hereditary extriangulated categories.
Under rather mild hypotheses, they are all shown to be equivalent.
We also make use of versions of Bongartz completion and of indices ($g$-vectors) and coindices in order to establish a mutation theory for silting subcategories in 0-Auslander extriangulated categories.
Our main point is that this setting allows for quite elementary proofs, while being general enough to recover various known mutation theories used in representation theory. 

\subsection{Comparison of definitions}
\label{ssection:equivDef}

Assume that $\Csc$ is a 0-Auslander extriangulated category, and let $\Pcal$ be the full subcategory of $\Csc$ formed by the projective objects. The full subcategory of injective objects is denoted $\Ical=\Sigma\Pcal$ (see~\cref{lemma:equiv proj inj}).

\begin{definition}
 An object $R\in\Csc$ is \defn{rigid} if it satisfies $\Ebb(R,R)=0$. A full additive subcategory $\rc$ of $\Csc$ is rigid, if all its objects are rigid or, equivalently, if for any $R,R'\in\rc$, we have $\Ebb(R,R')=0$.
 Recall from \cref{subsubsec:localization} that a full subcategory of $\Csc$ is \defn{thick} if it is closed under taking summands, extensions, cones of inflations and fibers of deflations. 
 For any subcategory $\Xcal$ of $\Csc$, we let $\operatorname{thick}\Xcal$ be the smallest thick subcategory of $\Csc$ containing $\Xcal$. 
\end{definition}

\begin{notation}
 We write $\pr\Rcal$ for $\Cone(\rc,\rc)$: the full subcategory whose objects are those $X$ for which there is an $\sfr$-triangle $R_1\infl R_0\defl X\dashrightarrow$ with $R_0,R_1\in\Rcal$. Dually, we write $\copr\Rcal$ for $\Fib(\Rcal,\Rcal)$, the full subcategory whose objects $Y$ admit an $\sfr$-triangle $Y\infl R_0\defl R_1\dashrightarrow$ with $R_0,R_1\in\Rcal$.
\end{notation}


\begin{theorem}\label{thm: equivalent versions of tilting}
For a given rigid, strictly full, additive subcategory $\rc$ of $\Csc$ that is stable under taking summands, 
consider the following conditions:
\begin{enumerate}[(i)]
 \item[(0)] $\rc$ is \defn{silting}: $\operatorname{thick}\rc=\Csc$.
 \item\label{def: maximal rigid} $\rc$ is \defn{maximal rigid}: For any $X\in\Csc$ such that $\rc\oplus \add X$ is rigid, we have $X\in\rc$.
 \item\label{def: E-tilting} $\rc$ is \defn{$\Ebb$-tilting}: For any $X\in\Csc$ such that $\Ebb(X,\rc)=0$ and $\Ebb(\rc,X)=0$, we have $X\in\rc$.
 \item\label{def: tilting} $\rc$ is \defn{tilting}: For any $P\in\Pcal$, there is a conflation $P\infl R^0\defl R^1$ with $R^0,R^1\in\rc$.
 \item\label{def: cotilting} $\rc$ is \defn{cotilting}: For any $I\in \Ical$, there is a conflation $R_1\infl R_0\defl I$ with $R_0,R_1\in\rc$.
 \item\label{def: torsion projective} $\rc$ is \defn{torsion-projective}: $\rc^{\perp_\E} = \operatorname{Cone}(\rc^{\perp_\Ebb},\rc)$.
 \item[(\ref{def: torsion projective}')] $\Rcal$ is \defn{cotorsion-injective}: $^{\perp_\Ebb}\Rcal=\Fib(\Rcal,\,\!^{\perp_\Ebb}\Rcal)$.
 \item\label{def: strongly torsion projective} $\rc$ is \defn{strongly torsion-projective}: $\rc^{\perp_\E} = \pr\rc$.
 \item[(\ref{def: strongly torsion projective}')] $\Rcal$ is \defn{strongly cotorsion-injective}: $^{\perp_\Ebb}\Rcal=\copr\Rcal$.
 \item\label{def: complete rigid} If $\Csc$ is Krull--Schmidt and if we have $\Pcal=\add P$ and $\rc=\add R$, we say that $R$ is \defn{complete rigid} if $R$ and $P$ have the same number of isomorphism classes of indecomposable summands.
\end{enumerate}
Then \defn{(0)} $\Leftrightarrow$ (\ref{def: strongly torsion projective}) $\Leftrightarrow$ (\ref{def: strongly torsion projective}') $\Leftrightarrow$ (\ref{def: torsion projective}) $\Leftrightarrow$ (\ref{def: torsion projective}') $\Leftrightarrow$ (\ref{def: cotilting}) $\Leftrightarrow$ (\ref{def: tilting}) $\Rightarrow$ (\ref{def: E-tilting}) $\Rightarrow$ (\ref{def: maximal rigid}).
If moreover any $P\in\Pcal$ has a left $\rc$-approximation, or equivalently if any $I\in\Ical$ has a right $\rc$-approximation, then (\ref{def: maximal rigid}) $\Rightarrow$ (\ref{def: tilting}) and, in case $\Pcal = \add P$ and $\Csc$ is Krull--Schmidt, (\ref{def: complete rigid}) $\Leftrightarrow$ (\ref{def: tilting}).
\end{theorem}

\begin{remark}
\label{rk:def tilting}
\begin{enumerate}[(a)]
\item By definition, we have the implication (\ref{def: tilting}) $\Rightarrow$ \defn{\rm{(0)}}. The implication \defn{\rm{(0)}} $\Rightarrow$ (\ref{def: strongly torsion projective}) is \cite[Proposition 4.7 and Lemma 4.11]{AdachiTsukamoto}. Also \cite[Proposition 5.5]{AdachiTsukamoto} shows that (0) $\Leftrightarrow$ (iii) by noting that $\Rcal^\vee=\operatorname{copr}\Rcal$ since $\Csc$ is hereditary and $\Rcal$ is rigid hence extension-closed.
 \item  The proof that (\ref{def: complete rigid}) $\Leftrightarrow$ (\ref{def: tilting}) will require some form of Bongartz lemma, and is thus postponed to \cref{subsection:complete rigid}.
 \item The choice of the term \emph{complete} in (\ref{def: complete rigid}) is explained by \cref{Prop: complete tilting}.
 \item We note that definitions (\ref{def: maximal rigid}) to (\ref{def: strongly torsion projective}') readily imply that $\Rcal$ contains all projective-injective objects, allowing to consider the reduced case only in the proofs.
 \item We also note that being $\Ebb$-tilting is a weaker condition than being cluster-tilting, but both definitions are equivalent when $\Csc$ is 2-Calabi--Yau.
\end{enumerate}
\end{remark}

\begin{proof}
We only prove the statement for reduced 0-Auslander extriangulated categories, since the general case follows by reduction. The equivalence (0) $\Leftrightarrow$ (\ref{def: tilting}) 
follows from \cite{AdachiTsukamoto}, as explained in \cref{rk:def tilting} (a).

We prove the following implications:
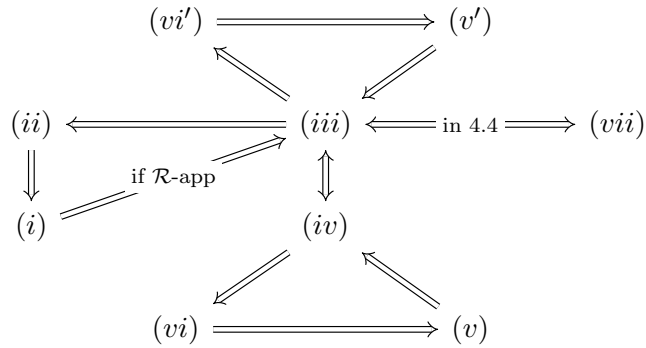
\[\begin{tikzcd}
	& {(vi')} && {(v')} \\
	{(ii)} && {(iii)} && {(vii)} \\
	{(i)} && {(iv)} \\
	& {(vi)} && {(v)}
	\arrow[Rightarrow, from=2-3, to=1-2]
	\arrow[Rightarrow, from=1-2, to=1-4]
	\arrow[Rightarrow, from=1-4, to=2-3]
	\arrow[Rightarrow, 2tail reversed, from=2-3, to=3-3]
	\arrow[Rightarrow, from=3-3, to=4-2]
	\arrow[Rightarrow, from=4-2, to=4-4]
	\arrow[Rightarrow, from=4-4, to=3-3]
	\arrow["{\text{in }4.4}"{description}, Rightarrow, 2tail reversed, from=2-3, to=2-5]
	\arrow[Rightarrow, from=2-3, to=2-1]
	\arrow[Rightarrow, from=2-1, to=3-1]
	\arrow["{\text{if } \mathcal{R}\text{-app} }"{description}, Rightarrow, from=3-1, to=2-3]
\end{tikzcd}\]

(\ref{def: cotilting}) $\Leftrightarrow$ (\ref{def: tilting}) We only prove (\ref{def: cotilting}) $\Rightarrow$ (\ref{def: tilting}) since the converse is dual.
Let $P\in\Pcal$ and assume that there is a conflation $R_1\infl R_0\defl \susp P$, with $R_0,R_1\in\rc$.
Since $P\infl 0 \defl \susp P$ is a conflation, we can apply~\cite[Proposition 3.15]{NakaokaPalu} in order to obtain a commutative diagram made of conflations
\[
\xy
(-7,21)*++{P}="-12";
(7,21)*++{P}="-14";
(-21,7)*++{R_1}="0";
(-7,7)*++{R_1}="2";
(7,7)*++{0}="4";
(-21,-7)*++{R_1}="10";
(-7,-7)*++{R_0}="12";
(7,-7)*++{\susp P}="14";
{\ar@{=} "-12";"-14"};
{\ar@{>->} "-12";"2"};
{\ar@{>->} "-14";"4"};
{\ar@{>->} "0";"2"};
{\ar@{->>} "2";"4"};
{\ar@{=} "0";"10"};
{\ar@{->>} "2";"12"};
{\ar@{->>} "4";"14"};
{\ar@{>->} "10";"12"};
{\ar@{->>} "12";"14"};
{\ar@{}|\circlearrowright "-12";"4"};
{\ar@{}|\circlearrowright "0";"12"};
{\ar@{}|\circlearrowright "2";"14"};
\endxy
\]
hence a conflation $P\infl R_1 \defl R_0$.

(\ref{def: strongly torsion projective}) $\Rightarrow$ (\ref{def: torsion projective}) because $\rc$ is rigid.

(\ref{def: torsion projective}) $\Rightarrow$ (\ref{def: cotilting}) Assume that $\rc^{\perp_\Ebb}$ is included in $\operatorname{Cone}(\rc^{\perp_\Ebb},\rc)$.
Since $\Ebb(\rc,\susp P) \cong \Ebb^2(\rc,P) = 0$, there is a conflation $X\infl R_0 \defl \susp P$ (hence a conflation $P\infl X\defl R_0$), with $R_0\in\rc$ and $X\in\rc^{\perp_\Ebb}$. By assumption, there is also a conflation $Y\infl R_1\defl X$, with $R_1\in\rc$ and $Y\in\rc^{\perp_\Ebb}$.
By composing the deflations $R_1\defl X\defl R_0$, we obtain a commutative diagram of conflations
\[
\xy
(-21,7)*++{Y}="0";
(-7,7)*++{Z}="2";
(7,7)*++{P}="4";
(-21,-7)*++{Y}="10";
(-7,-7)*++{R_1}="12";
(7,-7)*++{X}="14";
(-7,-21)*++{R_0}="22";
(7,-21)*++{R_0}="24";
{\ar@{>->} "0";"2"};
{\ar@{->>} "2";"4"};
{\ar@{=} "0";"10"};
{\ar@{>->} "2";"12"};
{\ar@{>->} "4";"14"};
{\ar@{>->} "10";"12"};
{\ar@{->>} "12";"14"};
{\ar@{->>} "12";"22"};
{\ar@{->>} "14";"24"};
{\ar@{=} "22";"24"};
{\ar@{}|\circlearrowright "0";"12"};
{\ar@{}|\circlearrowright "2";"14"};
{\ar@{}|\circlearrowright "12";"24"};
\endxy
\]
where the first row splits because $P$ is projective in $\Csc$.
This gives a conflation $Y\oplus P \infl R_1\defl R_0$, that we push-out along the canonical projection $Y\oplus P\to Y$:
\[
\xy
(-9,7)*++{Y\oplus P}="2";
(9,7)*++{R_1}="4";
(25,7)*++{R_0}="6";
(-9,-7)*++{Y}="12";
(9,-7)*++{Y\oplus R_0}="14";
(25,-7)*++{R_0}="16";
(-9,-21)*++{\susp P}="22";
(9,-21)*++{\susp P}="24";
(-9,-32)*+{}="32";
(9,-32)*+{}="34";
{\ar@{>->} "2";"4"};
{\ar@{->>} "4";"6"};
{\ar@{>->} "2";"12"};
{\ar@{>->} "4";"14"};
{\ar@{=} "6";"16"};
{\ar@{>->} "12";"14"};
{\ar@{->>} "14";"16"};
{\ar@{->>} "12";"22"};
{\ar@{->>} "14";"24"};
{\ar@{=} "22";"24"};
{\ar@{}|\circlearrowright "2";"14"};
{\ar@{}|\circlearrowright "4";"16"};
{\ar@{}|\circlearrowright "12";"24"};
\endxy
\]
The second row splits because $Y\in\rc^{\perp_\Ebb}$, and we obtain the commutative diagram of conflations:
\[
\xy
(-10,21)*++{X}="-12";
(10,21)*++{X}="-14";
(-30,7)*++{R_1}="0";
(-10,7)*++{R_1\oplus R_0}="2";
(10,7)*++{R_0}="4";
(-30,-7)*++{R_1}="10";
(-10,-7)*++{Y\oplus R_0}="12";
(10,-7)*++{\susp P}="14";
{\ar@{=} "-12";"-14"};
{\ar@{>->} "-12";"2"};
{\ar@{>->} "-14";"4"};
{\ar@{>->} "0";"2"};
{\ar@{->>} "2";"4"};
{\ar@{=} "0";"10"};
{\ar@{->>} "2";"12"};
{\ar@{->>} "4";"14"};
{\ar@{>->} "10";"12"};
{\ar@{->>} "12";"14"};
{\ar@{}|\circlearrowright "-12";"4"};
{\ar@{}|\circlearrowright "0";"12"};
{\ar@{}|\circlearrowright "2";"14"};
\endxy
\]
where we have used that $\rc$ is rigid.
Moreover, the middle column splits because the conflation $Y\infl R_1\defl X$ induces an exact sequence
\[
 0 = \Ebb(R_1,X) \to \Ebb(Y,X) \to \Ebb^2(X,X)=0,
\]
and $X\in\rc^{\perp_\Ebb}$.
We conclude that $X\in\rc$ since $\rc$ is stable under taking sums and summands, and hence, that $\susp P \in \pr\rc$, thus showing that $\rc$ is cotilting.

(\ref{def: cotilting}) $\Rightarrow$ (\ref{def: strongly torsion projective}) Assume that $\rc$ is cotilting.
Since $\Csc$ is hereditary, for any deflation $X\defl Y$ with $X\in\rc^{\perp_\Ebb}$, we have $Y\in\rc^{\perp_\Ebb}$.
This implies the inclusion of $\pr\rc$ in $\rc^{\perp_\Ebb}$.
Let $X\in\rc^{\perp_\Ebb}$.
Fix conflations $P_0\infl X\defl \susp P_1$, $R_1\infl R_0\defl \susp P_1$ and $S_1\infl S_0\defl \Sigma P_0$, where $P_0,P_1\in\Pcal$ and $R_0,R_1,S_0,S_1\in\rc$.
There is a commutative diagram of conflations
\[
\xy
(-7,21)*++{R_1}="-12";
(7,21)*++{R_1}="-14";
(-21,7)*++{P_0}="0";
(-7,7)*++{Y}="2";
(7,7)*++{R_0}="4";
(-21,-7)*++{P_0}="10";
(-7,-7)*++{X}="12";
(7,-7)*++{\susp P_1}="14";
{\ar@{=} "-12";"-14"};
{\ar@{>->} "-12";"2"};
{\ar@{>->} "-14";"4"};
{\ar@{>->} "0";"2"};
{\ar@{->>} "2";"4"};
{\ar@{=} "0";"10"};
{\ar@{->>} "2";"12"};
{\ar@{->>} "4";"14"};
{\ar@{>->} "10";"12"};
{\ar@{->>} "12";"14"};
{\ar@{}|\circlearrowright "-12";"4"};
{\ar@{}|\circlearrowright "0";"12"};
{\ar@{}|\circlearrowright "2";"14"};
\endxy
\]
where $R_1$ and $X$, hence $Y$, belong to $\rc^{\perp_\Ebb}$.
Therefore, the conflation in the middle row of
\[
\xy
(-9,21)*++{S_1}="-12";
(9,21)*++{S_1}="-14";
(-27,7)*++{Y}="0";
(-9,7)*++{S_1\oplus R_0}="2";
(9,7)*++{S_0}="4";
(-27,-7)*++{Y}="10";
(-9,-7)*++{R_0}="12";
(9,-7)*++{\susp P_0}="14";
{\ar@{=} "-12";"-14"};
{\ar@{>->} "-12";"2"};
{\ar@{>->} "-14";"4"};
{\ar@{>->} "0";"2"};
{\ar@{->>} "2";"4"};
{\ar@{=} "0";"10"};
{\ar@{->>} "2";"12"};
{\ar@{->>} "4";"14"};
{\ar@{>->} "10";"12"};
{\ar@{->>} "12";"14"};
{\ar@{}|\circlearrowright "-12";"4"};
{\ar@{}|\circlearrowright "0";"12"};
{\ar@{}|\circlearrowright "2";"14"};
\endxy
\]
splits, which implies $Y\in\rc$.
We thus have $X\in\pr\rc$.

(\ref{def: tilting}) $\Rightarrow$ (\ref{def: E-tilting}) Assume $\rc$ tilting, and let $X\in\Csc$ be such that $\Ebb(X,\rc)=0=\Ebb(\rc,X)$.
By assumption, there are conflations $P_1\infl P_0\defl X$ and $P_1\infl R^0\defl R^1$ with $P_0,P_1\in\Pcal$ and $R^0,R^1\in\rc$.
There is an induced commutative diagram made of conflations:
\[
\xy
(-9,7)*++{P_1}="2";
(9,7)*++{P_0}="4";
(25,7)*++{X}="6";
(-9,-7)*++{R^0}="12";
(9,-7)*++{R^0\oplus X}="14";
(25,-7)*++{X}="16";
(-9,-21)*++{R^1}="22";
(9,-21)*++{R^1}="24";
(-9,-32)*+{}="32";
(9,-32)*+{}="34";
{\ar@{>->} "2";"4"};
{\ar@{->>} "4";"6"};
{\ar@{>->} "2";"12"};
{\ar@{>->} "4";"14"};
{\ar@{=} "6";"16"};
{\ar@{>->}"12";"14"};
{\ar@{->>} "14";"16"};
{\ar@{->>} "12";"22"};
{\ar@{->>} "14";"24"};
{\ar@{=} "22";"24"};
{\ar@{}|\circlearrowright "2";"14"};
{\ar@{}|\circlearrowright "4";"16"};
{\ar@{}|\circlearrowright "12";"24"};
\endxy
\]
where the conflation in the second row splits because $\Ebb(X,\rc)=0$.
Consider a conflation $P_0\infl S^0\defl S^1$ with $S^0,S^1\in\rc$.
Together with $P_0\infl R^0\oplus X\defl R^1$, it induces a commutative diagram
\[
\xy
(-9,7)*++{P_0}="2";
(9,7)*++{R^0\oplus X}="4";
(25,7)*++{R^1}="6";
(-9,-7)*++{S^0}="12";
(9,-7)*++{S^0\oplus R^1}="14";
(25,-7)*++{R^1}="16";
(-9,-21)*++{S^1}="22";
(9,-21)*++{S^1}="24";
(-9,-32)*++{}="32";
(9,-32)*++{}="34";
{\ar@{>->} "2";"4"};
{\ar@{->>} "4";"6"};
{\ar@{>->} "2";"12"};
{\ar@{>->} "4";"14"};
{\ar@{=} "6";"16"};
{\ar@{>->} "12";"14"};
{\ar@{->>} "14";"16"};
{\ar@{->>} "12";"22"};
{\ar@{->>} "14";"24"};
{\ar@{=} "22";"24"};
{\ar@{}|\circlearrowright "2";"14"};
{\ar@{}|\circlearrowright "4";"16"};
{\ar@{}|\circlearrowright "12";"24"};
\endxy
\]
where the conflation in the second row splits because $\rc$ is rigid.
Since $\Ebb(\rc,X)=0$ and $\rc$ is rigid, the conflation $R^0\oplus X\infl S^0\oplus R^1\defl S^1$ splits and we have $X\in\rc$.

(\ref{def: E-tilting}) $\Rightarrow$ (\ref{def: maximal rigid})
immediately follows from the definitions.

Assume that $P\in\Pcal$ has a left $\rc$-approximation. We show the implication (\ref{def: maximal rigid}) $\Rightarrow$ (\ref{def: tilting}) and that $\susp P$ has a right $\rc$-approximation.
Let $P\to R^0$ be an $\rc$-approximation.
Because $P\infl 0$ is an inflation, any morphism with domain $P$ is also an inflation.
Let thus $P\infl R^0 \defl X$ be a conflation.
For any $R\in\rc$, it induces an exact sequence
\[
\Ebb(R,R^0) \to \Ebb(R,X) \to \Ebb^2(R,P).
\]
Since $R$ is rigid, the first term vanishes, and by heredity the last term also vanishes.
We thus have $\Ebb(R,X)=0$.
Consider the exact sequence:
\[
\Csc(R^0,R)\to\Csc(P,R)\to \Ebb(X,R) \to \Ebb(R^0,R).
\]
By rigidity of $\rc$, the last term vanishes.
Moreover, by definition of an approximation, the first morphism of the exact sequence is a surjection.
This shows that we also have $\Ebb(X,R)=0$.
Moreover, the exactness of the sequence
\[
\Ebb(X,R^0) \to \Ebb(X,X) \to \Ebb^2(X,P)
\]
implies that $X$ is rigid, and hence that $X\in\rc$.
Thus, $\rc$ is tilting.
Moreover, by pushing-out any conflation of the form $P\infl R^0\defl R^1$ along the conflation $P\infl 0\defl \susp P$, we obtain a conflation $R^0\infl R^1\defl \susp P$.
Since $\rc$ is rigid, the deflation $R^1\defl \susp P$ is a right $\rc$-approximation.
This shows the claim since every injective object is of the form $\susp P$ for some projective $P$.

The proof that the existence, for any injective $I$, of a right $\rc$-approximation of $I$ implies that
(\ref{def: maximal rigid}) $\Rightarrow$ (\ref{def: tilting}) and that any projective object has a left $\rc$-approximation is dual and left to the reader.

The implications (\ref{def: strongly torsion projective}') $\Rightarrow$ (\ref{def: torsion projective}') $\Rightarrow$ (\ref{def: tilting}) $\Rightarrow$ (\ref{def: strongly torsion projective}') follow from the implications (\ref{def: strongly torsion projective}) $\Rightarrow$ (\ref{def: torsion projective}) $\Rightarrow$ (\ref{def: cotilting}) $\Rightarrow$ (\ref{def: strongly torsion projective}) by duality.
\end{proof}

\begin{remark}
Assuming the extriangulated category $\Csc$ to be 0-Auslander is unnecessary for most of the implications proven above.
Recall that $\Csc$ has projective dimension at most one if any object $X$ appears in an $\sfr$-triangle $P_1\infl P_0\defl X\dashrightarrow$ with $P_0,P_1$ projective, and injective dimension at most one if  any object $X$ appears in an $\sfr$-triangle $X\infl I^0\defl I^1\dashrightarrow$ with $I^0,I^1$ injective.
We also consider the following assumptions on $\Csc$:
\begin{itemize}
    \item[(H)] The extriangulated category $\Csc$ is hereditary.
    \item[(P)] For any projective object $P\in\Csc$, there is an $\sfr$-triangle $P\infl Q\defl I\dashrightarrow$ with $I$ injective and $Q$ projective-injective.
    \item[(I)] For any injective object $I\in\Csc$, there is an $\sfr$-triangle $P\infl Q\defl I\dashrightarrow$ with $P$ projective and $Q$ projective-injective.
\end{itemize}
We note that, by \cref{Thm:Auslander:equiv}, $\Csc$ is 0-Auslander if and only if it satisfies (H)+(P) and has enough projectives if and only if it satisfies (H)+(I) and has enough injectives.
For future reference, we give here an explicit list of those assumptions that are used in the proof:

\begin{itemize}
    \item No specific assumptions on $\Csc$: (ii) $\Rightarrow$ (i), (vi) $\Rightarrow$ (v) and (vi') $\Rightarrow$ (v'). 
    \item $\mathrm{pd}\Csc\leq 1$: (iii) $\Rightarrow$ (ii).
    \item $\mathrm{id}\Csc\leq 1$: (iv) $\Rightarrow$ (ii). \footnote{More precisely, the proof we presented went (iv)  $\Rightarrow$ (iii)  $\Rightarrow$ (ii), but the arguments dual to those in our proof of the implication (iii) $\Rightarrow$ (ii) give a direct proof of the implication (iv) $\Rightarrow$ (ii) and require only $\mathrm{id}\Csc\leq 1$.}
    \item (H)+(P): (iv) $\Rightarrow$ (iii), (v') $\Rightarrow$ (iii); and (i) $\Rightarrow$ (iii) under the assumption that all projectives have a left $\Rcal$-approximation.
    \item (H)+(I): (iii) $\Rightarrow$ (iv), (v) $\Rightarrow$ (iv); and (i) $\Rightarrow$ (iii) under the assumption that all injectives have a right $\Rcal$-approximation.
    \item $\Csc$ is 0-Auslander: (iv) $\Rightarrow$ (vi), (iii) $\Rightarrow$ (vi') and (iii) $\Leftrightarrow$ (vii).
\end{itemize}
\end{remark}

\subsection{Bongartz completions}
\label{ssection:Bongartz}

We again assume that $\Csc$ is a 0-Auslander extriangulated category (we note however that conditions (H)+(I) are sufficient) with full subcategory of projectives $\Pcal$ and full subcategory of injectives $\Ical$.

The following proposition is an extriangulated version of (co-)Bongartz completion, and is inspired from \cite[Definition-Proposition 4.21]{Jasso-Reduction}.
\begin{proposition}
\label{prop:Bongartz}
Let $U\in\Csc$ be rigid, and let $I\in\Csc$ be injective. Assume that $U'\xrightarrow{f} I$ is a right $\add (U, \Pcal\cap\Ical)$-approximation. 
Then, up to adding a projective-injective summand to $U'$, we may assume that $f$ is a deflation and, in the conflation $P_U\infl U'\defl I$, we have that $R_U=P_U\oplus U$ is rigid.
If moreover $\Ical=\add I$, then $\add (R_U,\Pcal\cap\Ical)$ is cotilting.
\end{proposition}

\begin{proof}
The statement follows from the case  when $\Csc$ is reduced 0-Auslander, which we assume now on.
First note that, since $0\defl I$ is a deflation, any morphism with codomain $I$ is also a deflation.
 Rigidity of $R_U$ follows from the long exact sequences associated with the conflation $P_U\infl U'\defl I$.
 More precisely, in the exact sequence
 \[
  \Csc(U,U') \to \Csc(U,I) \to \Ebb(U,P_U) \to \Ebb(U,U'),
 \]
the first morphism is surjective by construction, and because $U$ is rigid, the last term is zero.
This shows that $\Ebb(U,P_U)=0$.
Next, consider the exact sequence
\[
\Ebb(U',U)\to\Ebb(P_U,U) \to \Ebb^2(I,U).
\]
The first term is zero since $U$ is rigid.
Moreover, we have $\Ebb^2(I,U)=0$ by heredity.
We thus have $\Ebb(P_U,U)=0$.
To see that $P_U$ is rigid, notice that both end terms of the exact sequence
\[
\Ebb(U',P_U) \to \Ebb(P_U,P_U) \to \Ebb^2(I,P_U)
\]
vanish.
We have proved that $R_U$ is rigid.
If $\Ical=\add I$ then $\add (R_U,\Pcal\cap\Ical)$ is cotilting thanks to the conflation $P_U\infl U'\defl I$.
\end{proof}

\begin{remark}
 Since $\Csc$ is hereditary, we also have $P_U^{\perp_\E} = (P_U\oplus U)^{\perp_\E}$.
\end{remark}

\begin{corollary}
\label{corollary: Bongartz for subcategories}
Let $\uc\subseteq\Csc$ be a rigid full subcategory stable under taking summands and containing all projective-injectives.
Assume that any injective object $I\in\Csc$ has a right $\uc$-approximation $f^I$ (which can be assumed a deflation) and write  $P^I_\uc$ for the fiber (or cocone) of $f^I$. 
Then $\add(\uc\bigoplus_{I\in\Inj(\Csc)}P^I_\uc)$ is a cotilting subcategory of $\Csc$.
\end{corollary}

\begin{proof}
This follows immediately from \cref{prop:Bongartz}.
\end{proof}

\begin{lemma}
\label{lemma:Bongartz projective}
The object $P^I_\uc$ in the corollary above is projective in $\uc^{\perp_\Ebb}$, hence also in the reduction of $\Csc$ by $\uc$ (as defined in \cref{Rem: extriangulated reduction}).
\end{lemma}

\begin{proof}
Applying (ET4$^\text{op}$) to the conflations $P_\uc^I \infl U^I \overset{f^I}{\defl} I$ and $P\infl Q\defl I$, and using projectivity of $Q$ gives a conflation
$P\infl P^I_\uc\oplus Q\defl U^I$.
We thus have an exact sequence $\Ebb(U^I,-)\to\Ebb(P^I_\uc\oplus Q,-)\to\Ebb(P,-)$ from which the claim follows.
\end{proof}


\subsection{Indices and coindices}
\label{ssection:indices}

Assume that $\Csc$ is a 0-Auslander extriangulated category, with full subcategory of projective objects $\Pcal$.
Fix a tilting subcategory $\rc$ of $\Csc$.

In this section, we remark that it is possible to define the coindex of $P\in\Pcal$ with respect to any tilting subcategory $\rc$ of $\Csc$.
This allows to notice that some results of~\cite{JorgensenYakimov} hold in our context.
\cref{prop:JY} will turn out to be useful when considering complete rigid objects in \cref{subsection:complete rigid}. We also refer to the recent~\cite{JorgensenShah-Index}, where similar ideas, partly inspired from~\cite{PadrolPaluPilaudPlamondon}, are considered in relation to modified Caldero--Chapoton maps~\cite{HolmJorgensen-modifiedCC,HolmJorgensen-modifiedCC2,JorgensenShah-modifiedCC}.

Recall that, if $X\in\Csc$, then its index is the (well-defined) element
\[
\ind_\Pcal X =  [P_0]-[P_1]\in\kzerosp{\Pcal},
\]
where $P_1\infl P_0\defl X$ is a conflation with $P_0,P_1\in\Pcal$.
Note that our assumptions imply $\kzero{\Csc}\cong\kzerosp{\Pcal}$.

Since $\rc$ is tilting, any $P\in\Pcal$ admits a conflation $P\infl R^0\defl R^1$ with $R^0,R^1\in\rc$. Define
\[
 \coind_\rc P = [R^0]-[R^1]\in\kzerosp{\rc}.
\]
One easily checks that $\coind_\rc P$ is well-defined (see \cite[Lemma 4.36]{PadrolPaluPilaudPlamondon}).

Our next proposition is inspired from~\cite[Theorem 1.2]{JorgensenYakimov} and~\cite[Corollary 6.20]{DemonetIyamaJasso} (for 2-Calabi--Yau triangulated categories).
\begin{proposition}\label{prop:JY} Let $\Csc$ be a 0-Auslander extriangulated category, with full subcategory of projective objects $\Pcal$.
Assume that $\rc$ is tilting.
Then the morphisms
\[
 \coind_\rc : \kzerosp{\Pcal} \rightleftarrows \kzerosp{\rc} : \ind_\Pcal
\]
are inverse isomorphisms.
\end{proposition}

\begin{proof}
For any $P\in\Pcal$, fix a conflation $P\infl R^0\defl R^1$ with $R^0,R^1\in\rc$. By~\cite[Lemma 4.38]{PadrolPaluPilaudPlamondon} we have $\ind_\Pcal [R^0] = \ind_\Pcal [P] + \ind_\Pcal [R^1]$. Hence
\begin{eqnarray*}
[P] & = & \ind_\Pcal [P] \\
& = & \ind_\Pcal [R^0] - \ind_\Pcal [R^1] \\
& = & \ind_\Pcal \left([R^0] - [R^1]\right) \\
& = & \ind_\Pcal\coind_\rc [P].
\end{eqnarray*}
Now, for any $R\in\rc$, there is a conflation $P_1\infl P_0\defl R$ with $P_0,P_1\in\Pcal$.
Consider a conflation $P_0\infl U^0\defl U^1$ with $U^0,U^1\in\rc$.
Then there is a commutative diagram of conflations:
\[
\xy
(-9,7)*++{P_1}="2";
(6,7)*++{P_0}="4";
(25,7)*++{R}="6";
(-9,-7)*++{P_1}="12";
(6,-7)*++{U^0}="14";
(25,-7)*++{R\oplus U^1}="16";
(6,-21)*++{U^1}="24";
(25,-21)*++{U^1}="26";
{\ar@{>->} "2";"4"};
{\ar@{->>} "4";"6"};
{\ar@{=} "2";"12"};
{\ar@{>->} "4";"14"};
{\ar@{>->} "6";"16"};
{\ar@{>->} "12";"14"};
{\ar@{->>} "14";"16"};
{\ar@{->>} "14";"24"};
{\ar@{->>} "16";"26"};
{\ar@{=} "24";"26"};
{\ar@{}|\circlearrowright "2";"14"};
{\ar@{}|\circlearrowright "4";"16"};
{\ar@{}|\circlearrowright "14";"26"};
\endxy
\]
where the third column splits because $\rc$ is rigid.
This implies the following equalities:
\begin{eqnarray*}
 \coind_\rc [P_0] & = & [U^0]-[U^1] \\
 \coind_\rc [P_1] & = & [U^0]-[U^1] - [R] \hspace{1cm}\\
 \text{hence } \hspace{1.5cm} [R] & = &  \coind_\rc [P_0] - \coind_\rc [P_1] \\
  & = & \coind_\rc\ind_\Pcal [R].
\end{eqnarray*}
\end{proof}

\begin{definition}
Assume that $\Csc$ is Krull--Schmidt and let $R\in\rc$ be indecomposable. Following~\cite{JorgensenYakimov}, the (categorical or homological) c-vector of $(R,\rc)$ with respect to $\Pcal$ is the element
 \[
  c_\Pcal(R,\rc) = [R]^\ast \circ \coind_\rc \text{ of } \kzerosp{\Pcal}^\ast,
 \]
where $(-)^\ast = \Hom_\zb(-,\zb)$, and $[R]^\ast\in\kzerosp{\rc}^\ast$ is defined by $\langle [R]^\ast, [S]\rangle = \del_{[R],[S]}$, for any $S\in\Pcal$ indecomposable.
\end{definition}

\begin{corollary}
 For any two indecomposable objects $R,S\in\rc$, we have
 \[
  \langle c_\Pcal(R,\rc) , \ind_\Pcal S \rangle = \del_{[R],[S]}.
 \]
\end{corollary}

\subsection{Complete rigid objects}
\label{subsection:complete rigid}

Assume that $\Csc$ is a 0-Auslander extriangulated category, with full subcategory of injective objects $\Ical$.
Assume that all rigid subcategories ($\Ical, \Rcal, \add R, \ldots$) of $\Csc$ apprearing below are Krull--Schmidt. 
Also assume, in this section, that $\Ical=\add I$, where $I=I_1\oplus\cdots\oplus I_n$ is basic and each $I_j$ indecomposable. 
An immediate consequence of \cref{prop:JY} is:

\begin{corollary}
\label{corollary:m=n}
 Asumme that $R=R_1\oplus\cdots\oplus R_m$ is basic tilting. Then $m=n$.
\end{corollary}

In fact, \cref{prop:Bongartz} shows that some converse holds.

\begin{proposition}\label{Prop: complete tilting}
 Let $R=R_1\oplus\cdots\oplus R_m$ be basic rigid in $\Csc$. Assume that $I$ has a right $\add(R)$-approximation. Then $m\leq n$, with equality if and only if $R$ is tilting.
\end{proposition}

\begin{proof}
 Let $X_R$ be the Bongartz completion of $R$, given by \cref{prop:Bongartz}.
 Then $R\oplus X_R$ is tilting, and \cref{corollary:m=n} shows that $m\leq n$.
 Assume that $m=n$. Then $R$ and $R\oplus X_R$ have the same number of isomorphism classes of indecomposable summands. Thus $X_R$ belongs to $\add R$, which implies that $R$ is tilting.
\end{proof}

\begin{corollary}
Let $\Rcal$ be a tilting subcategory such that $I$ has a right $\Rcal$-approximation.
Then there is a tilting object $R$ such that $\Rcal=\add R$.
\end{corollary}

\begin{proof}
Let $R_0\to I$ be a right $\Rcal$-approximation, and let $R$ be a maximal basic summand of $R_0$.
Up to adding some indecomposables in $\Rcal$ to $R$, we may assume that $R$ is a basic rigid object with $n$ summands.
By \cref{Prop: complete tilting}, $R$ is tilting.
\cref{thm: equivalent versions of tilting} shows that $R$ is maximal rigid, hence that $\Rcal=\add R$.
\end{proof}

\subsection{$\Ebb$-tilting objects and reductions of hereditary extriangulated categories}
\label{ssection:tilting and reduction}

Our proof of \cref{Th:mutation} makes use of the reduction procedure for hereditary extriangulated categories, as introduced in \cref{ssection:reduction}.

Let $\Csc$ be an extriangulated category (which is not assumed 0-Auslander).
Let $\Rcal$ be a rigid subcategory of $\Csc$, and let $\CR=\frac{\Rint}{[\Rcal]}$ be the associated reduction.

\begin{lemma}
\label{lemma:E tilting implies Ebar tilting}
For any $X\in\Rint$, the following are equivalent.
\begin{itemize}
\item[{\rm (i)}] $\add(\Rcal, X)\se\C$ is $\E$-tilting.
\item[{\rm (ii)}] $\add X\se\CR$ is $\ovl{\E}$-tilting.
\end{itemize}
\end{lemma}

\begin{proof}
This is an immediate consequence of the fact that for any $X,X\ppr\in\Rint$, we have $\ovl{\E}(X,X\ppr)=\E(X,X\ppr)$ by definition.
\end{proof}

\begin{proposition}\label{Prop: projective E-tilting}
If there is an $\E$-tilting \textbf{projective} object $X\in\C$, then $\Proj_{(\C,\E)}=\add X$ holds. Dually for injectives.
\end{proposition}
\begin{proof}
Since any $P\in\Proj_{(\C,\E)}$ satisfies $\E(P,X)=\E(X,P)=0$, it follows $P\in\add X$ since $\add X$ is $\E$-tilting by assumption.
\end{proof}

\begin{remark}\label{Prop: X to 0 to Y consequences}
Suppose there is an $\sfr$-triangle $X\to 0\to Y\ov{\thh}{\dra}$. The following holds.
\begin{enumerate}
\item $X$ is uniquely determined up to isomorphism by $Y$, and vice-versa.
\item $X=0$ if and only if $Y=0$.
\item If $X\ne0$, we have $\thh\ne0$. In particular, we have $\E(Y,X)\ne0$.
\item There are natural isomorphisms
\[ \thh\ush\co\C(X,-)\ov{\cong}{\lra}\E(Y,-)\quad\text{and}\quad\thh\ssh\co\C(-,Y)\ov{\cong}{\lra}\E(-,X). \]
In particular, we have $\End_{\C}(X) \cong \E(Y, X) \cong \End_{\C}(Y)$.
\item $X$ is indecomposable if and only if $Y$ is indecomposable.
\end{enumerate}
\end{remark}

\begin{proposition}\label{Prop: confX hereditary}
Assume $\C$ is hereditary. If there is a conflation $\confX$, then the following holds.
\begin{enumerate}
\item $X$ is projective, and $Y$ is injective. In particular, both $X$ and $Y$ are rigid.
\item If moreover $X$ is $\E$-tilting, then $\Proj_{(\C,\E)}=\add X$. Dually for $Y$.
\end{enumerate}
\end{proposition}
\begin{proof}
{\rm (1)} is immediate from the exactness of $0\to\E(X,-)\to0$ and $0\to\E(-,Y)\to0$. {\rm (2)} follows from \cref{Prop: projective E-tilting}.
\end{proof}

\begin{corollary}\label{Cor: hereditary X infl 0 vs 0 defl X}
Assume $\C$ is hereditary. If $X$ admits an inflation $X\infl0$ and a deflation $0\defl X$ simultaneously, then $X=0$. In particular, if there is a conflation $X\infl 0\defl Y$, then $X\cong Y$ holds if and only if $X=0$.
\end{corollary}
\begin{proof}
By \cref{Prop: confX hereditary}, the existence of $X\infl0$ implies that $X$ is projective. Thus a conflation $W\infl0\defl X$ should split, which forces $X=0$.
\end{proof}

\begin{corollary}\label{Cor: hereditary X to 0 to Y}
Assume $\C$ is hereditary. Suppose that there exist both of
\begin{itemize}
\item $\confX$ in which $X$ is $\E$-tilting,
\item $X\ppr\infl0\defl Y\ppr$ in which $Y\ppr$ is $\E$-tilting.
\end{itemize}
Then we have $\add X=\Proj_{(\C,E)}\ni X\ppr$ and $\add Y\ppr=\Inj_{(\C,\E)}\ni Y$.

Thus, if moreover $X$ is indecomposable, it should be the unique indecomposable projective in $\C$. Similarly for $Y\ppr$.

We also have an equivalence of categories $\Sig\co\Proj_{(\C,E)}\ov{\simeq}{\longleftrightarrow}\Inj_{(\C,E)}\co\Om$.
\end{corollary}


\begin{remark}
\Cref{Cor: hereditary X to 0 to Y} does not exclude possible existence of rigid objects which are non-projective and non-injective, and this is the main obstacle in proving existence of mutation for silting subcategories.
\end{remark}

\subsection{Irreducible mutations}
\label{ssection:mutation}


Recall that a morphism $X\xrightarrow{f}R$ is left minimal if any endomorphism $g$ satisfying $gf=f$ is an isomorphism.
It is radical if for any $R\xrightarrow{g}X$, $1-gf$ is an isomorphism. 
Note that it follows from \cref{Appendix: l.e.s.} that, in an $\sfr$-triangle $X\overset{f}{\infl} Y\overset{g}{\defl}Z\dashrightarrow$, the morphism $f$ is left minimal if and only if the morphism $g$ is radical.

\begin{lemma}\label{lemma:no common factors}
Let $\Csc$ be an extriangulated category, let $\Rcal$ be a Krull--Schmidt, rigid subcategory, let $P\in\Csc$ be projective, and let
$P\infl R^0\defl R^1 \dashrightarrow$ be an $\sfr$-triangle with $R^0,R^1\in\Rcal$ and where the inflation is left minimal. 
Then $\add R^0 \cap \add R^1 = 0$.
\end{lemma}

\begin{proof}
The proof of \cite[Lemma 2.25]{AiharaIyama} carries over with only minor modifications.
Fix an $\sfr$-triangle $P\overset{f}{\infl} R^0\overset{g}{\defl} R^1 \dashrightarrow$ as in the statement.
Let $\al:R^0\to R^1$ be any morphism. We want to prove that $\al$ is radical. 

Because $P$ is projective, there is a (plain) commutative diagram:
\[\begin{tikzcd}
	& P & {R^0} & {R^1} & {} \\
	P & {R^0} & {R^1} & {} & {}
	\arrow["f", tail, from=1-2, to=1-3]
	\arrow["g", two heads, from=1-3, to=1-4]
	\arrow["g", two heads, from=2-2, to=2-3]
	\arrow["\alpha", from=1-3, to=2-3]
	\arrow["a"', from=1-2, to=2-2]
	\arrow["u"{description}, curve={height=6pt}, dashed, from=1-3, to=2-2]
	\arrow["v"{description}, dotted, from=1-4, to=2-3]
	\arrow[dashed, from=1-4, to=1-5]
	\arrow["f", tail, from=2-1, to=2-2]
	\arrow[dashed, from=2-3, to=2-4]
\end{tikzcd}\]
By rigidity of $\Rcal$, the morphism $f$ is a left $\Rcal$-approximation, hence the existence of a morphism $u$ making the leftmost upper triangle commute.
Since $g$ is a weak cokernel of $f$, there is a morphism $v$ such that $\al = vg+gu$.
By assumption, $f$ is left minimal and thus $g$ is radical.
This shows that $\al$ is radical, and hence that any morphism from $R^0$ to $R^1$ is radical.
If $R^0$ and $R^1$ had some common (up to isomorphism) summand $R$, then the composition $R^0\defl R\infl R^1$ would not be radical.
\end{proof}

\begin{remark}
 If $\Rcal$ is Krull--Schmidt, any $\sfr$-triangle of the form $X\infl R^0\defl R^1\dashrightarrow$, with $R^0,R^1\in\Rcal$, is isomorphic to the direct sum of an $\sfr$-triangle as in the statement of \cref{lemma:no common factors} and of an $\sfr$-triangle of the form $0\infl R \overset{1}{\defl} R\dashrightarrow$.
\end{remark}



\begin{definition}
 A rigid subcategory $\Rcal'$ is called \defn{almost complete} if it contains all projective-injectives and if there exists an indecomposable $R\in\Csc$ such that $\add (\Rcal'\cup\{R\})$ is tilting.
 Such an indecomposable $R$ is then called a \defn{complement} for $\Rcal'$.
\end{definition}


\begin{theorem}
\label{Th:mutation}
Assume that $\Csc$ is a 0-Auslander extriangulated category.
Let $\Rcal'$ be an almost complete rigid subcategory of $\Csc$ (containing all projective-injectives).
We make the following assumptions:
\begin{itemize}
    \item Any projective object in $\Csc$ has a left $\Rcal'$-approximation and any injective object has a right $\Rcal'$-approximation.
    \item Any projective object in the reduction $\Csc_{\Rcal'}$ of $\Csc$ by $\Rcal'$ has a left $\add(X)$-approximation if $X$ is any complement for $\Rcal'$.
    \item The category $\overline{\Csc}_{\Rcal'}$ is Krull--Schmidt.
\end{itemize}
Then there exist, up to isomorphism, precisely two complements for $\Rcal'$.
\end{theorem}

\begin{remark}
The second assumption is satisfied for example if $\Csc/[\Pcal\cap\Ical]$ is Hom-finite, and the first one is if moreover $\Rcal'$ only has finitely many indecomposables up to isomorphism.
We also note that the second assumption is equivalent to the dual assumption on injective objects.
\end{remark}

\begin{proof}
Replacing $\Csc$ by $\Csc/[\Pcal\cap\Ical]$, we may assume that $\Csc$ is reduced 0-Auslander.
Let $(\overline{\Csc}_{\Rcal'},\overline{\Ebb})$ be the reduction of $\Csc$ at $\Rcal'$.
Let $I$ be an indecomposable injective in $\Csc$ that does not belong to $\Rcal'$ (such an injective exists otherwise $\Rcal'$ would be cotilting).
Fix an $\sfr$-triangle $P\infl 0\defl I\dashrightarrow$, where $P$ is thus projective.
By assumption, $P$ has a left $\Rcal'$-approximation and $I$ has a right $\Rcal'$-approximation.
We can thus apply \cref{corollary: Bongartz for subcategories}, \cref{lemma:Bongartz projective} and their duals, to obtain $\sfr$-triangles $Q\infl R \defl I \overset{\eps}{\dashrightarrow}$ and $P\infl R'\defl J \dashrightarrow$ in $\Csc$ where $R,R'\in\Rcal'$, $\add(\Rcal'\cup\{Q\})$ and $\add(\Rcal'\cup\{J\})$ are rigid and where $Q$ is projective in $\overline{\Csc}_{\Rcal'}$ and $J$ is injective in  $\overline{\Csc}_{\Rcal'}$.
Pulling-back $\eps$ along the deflation $0\defl I$ gives an $\sfr$-triangle $P\infl Q\defl R\dashrightarrow$ and axiom (ET4) gives a diagram:
\[\begin{tikzcd}
	P & {Q} & R & {} \\
	{R'} & {R'\oplus R} & R & {} \\
	{J} & {J} \\
	{} & {}
	\arrow[tail, from=1-1, to=1-2]
	\arrow[two heads, from=1-2, to=1-3]
	\arrow[Rightarrow, no head, from=1-3, to=2-3]
	\arrow[tail, from=1-2, to=2-2]
	\arrow[tail, from=1-1, to=2-1]
	\arrow[two heads, from=2-1, to=3-1]
	\arrow[tail, from=2-1, to=2-2]
	\arrow[two heads, from=2-2, to=2-3]
	\arrow[Rightarrow, no head, from=3-1, to=3-2]
	\arrow[two heads, from=2-2, to=3-2]
	\arrow[dashed, from=1-3, to=1-4]
	\arrow[dashed, from=2-3, to=2-4]
	\arrow[dashed, from=3-2, to=4-2]
	\arrow[dashed, from=3-1, to=4-1]
\end{tikzcd}\]
in $\Csc$ (where we have used rigidity of $\Rcal'$), hence an $\sfr$-triangle $Q\infl 0\defl J\dashrightarrow$ in $\overline{\Csc}_{\Rcal'}$.
Moreover, $I\notin\Rcal'$ implies $J\neq 0$ and thus $Q\neq 0$ by \cref{Prop: X to 0 to Y consequences} (2).

We first show that $\Rcal'$ has at most two complements.
Let $X$ be a complement for $\Rcal'$.
In particular, $X$ is rigid in $\Csc$, hence also in $\overline{\Csc}_{\Rcal'}$.

By the beginning of the proof, $Q\to 0$ is an inflation in $\overline{\Csc}_{\Rcal'}$.
Thus any left $\add(X)$-approximation $Q\xrightarrow{i}X_0$ in $\overline{\Csc}_{\Rcal'}$ is an inflation and there is an $\sfr$-triangle
$Q\overset{i}{\infl} X_0\defl Y\dashrightarrow$ in $\overline{\Csc}_{\Rcal'}$, where $i$ can be assumed left minimal since $\overline{\Csc}_{\Rcal'}$ is Krull--Schmidt.
It induces an exact sequence $\overline{\Ebb}(X,X_0)\to\overline{\Ebb}(X,Y)\to\overline{\Ebb}^2(X,Q)$, whose first term vanishes since $X$ is rigid and whose third term also vanishes since $\overline{\Csc}_{\Rcal'}$ is hereditary by \cref{Prop: reductions of hereditary extricats}.
There is also an exact sequence $\overline{\Csc}_{\Rcal'}(X_0,X)\to\overline{\Csc}_{\Rcal'}(Q,X)\to\overline{\Ebb}(Y,X)\to\overline{\Ebb}(X_0,X)$ where the first map is surjective since $i$ is a left approximation and where the fourth term is zero since $X$ is rigid.
This shows that $\overline{\Ebb}(X,Y)=0=\overline{\Ebb}(Y,X)$.
Because $\add(\Rcal'\cup\{X\})$ is $\Ebb$-tilting in $\Csc$, $X$ is $\overline{\Ebb}$-tilting by \cref{lemma:E tilting implies Ebar tilting}, hence $Y$ belongs to $\add X$.
By \cref{lemma:no common factors} we have $X_0 = 0$ or $Y=0$, but not both (because $Q$ is non-zero).
If $Y=0$, then $X_0$ is isomorphic to $Q$.
Since $X$ is indecomposable, we have $X\in\add X_0$ and thus $X$ is projective.
By \cref{Prop: projective E-tilting}, $X$ is the unique indecomposable projective object in $\overline{\Csc}_{\Rcal'}$.
If $X_0=0$, then \cref{Prop: confX hereditary} shows that $Y$, hence $X$, is injective.
The dual of \cref{Prop: projective E-tilting} implies that $X$ is the unique indecomposable injective object in $\overline{\Csc}_{\Rcal'}$.
Thus $\Rcal'$ has at most two complements.

We next show that $\Rcal'$ has at least two complements.
Note that \cref{Cor: hereditary X infl 0 vs 0 defl X} shows that $Q$ and $J$ cannot be isomorphic (otherwise both would be zero and $\Rcal'$ would already be tilting).
The previous part of the proof shows that $J$ is of the form $I_0^{\oplus r}$ where $I_0$ is the unique indecomposable injective object in $\overline{\Csc}_{\Rcal'}$.
The proof applies to any indecomposable injective $I\in\Csc$ that does not belong to $\Rcal'$.
Hence \cref{corollary: Bongartz for subcategories} implies that $\add(\Rcal'\cup \{Q\})$ is cotilting so that $Q$ is a complement for $\Rcal$.
Similarly, $J$ is also a complement.
Therefore, $\Rcal'$ has at least two complements.
\end{proof}

Under the same assumptions on $\Csc$ as in \cref{Th:mutation} we obtain the following.
\begin{corollary}
\label{Cor:mutation}
 Let $\Rcal$ be a silting subcategory of $\Csc$, $R\in\Rcal$ be any non projective-injective indecomposable object, and $\Rcal' =\add\{S\in\Rcal, S \text{ \emph{indecomposable and} } S\ncong R \}$.
 Then there exists an indecomposable object $R^\ast$, unique up to isomorphism, such that $R^\ast$ is not isomorphic to $R$ and 
 $\add\left(\Rcal'\cup \{R^\ast\}\right)$
 is silting.
 Moreover, there is an exchange $\sfr$-triangle $R\infl R'\defl R^\ast \dashrightarrow$ with $R'\in\Rcal'$
 or an exchange $\sfr$-triangle $R^\ast\infl R''\defl R \dashrightarrow$ with $R''\in\Rcal'$, but not both at the same time.
\end{corollary}

Applying~\cref{Cor:mutation} to the specific examples listed in \cref{ssection:examples0-Auslander}, we recover:
\begin{itemize}
 \item Cluster-tilting theory~\cite{BuanMarshReinekeReitenTodorov,IyamaYoshino}.
 \item 2-term silting mutation~\cite{Kimura}.
 \item Relative tilting theory~\cite{YangZhu}.
 \item Mutation of intermediate co-$t$-structures~\cite{AiharaIyama,BrustleYang, KoenigYang} (see \cref{section:ExtendedCohearts}).
 \item Via bijections in \cite{PauksztelloZvonareva} and in \cite{liu2020hereditary},  we recover mutations of functorially finite torsion classes and support $\tau$-tilting subcategories in $\tau$-tilting theory (see  \cref{rem: PZ_torsion_tau-tilting}).
 \item Combinatorics of walks and non-kissing~\cite{McConville,BrustleDouvilleMousavandThomasYildirim,PaluPilaudPlamondon-nonkissing} (see \cref{section:gentle}).
 \item Mutation of maximal almost rigid modules in type A \cite{BarnardGunawanMeehanSchiffler}.
\end{itemize}



\section{Tilting subcategories, cotorsion pairs and intermediate co-$t$-structures}
\label{section:ExtendedCohearts}

In this section, we apply \cref{Th:mutation} to extended cohearts of co-$t$-structures.
As a consequence, there is a well-behaved notion of mutation for intermediate co-$t$-structures extending the previously known one from the bounded case.

\subsection{Cotorsion pairs in 0-Auslander extriangulated categories}

Let $\CEs$ be a 0-Auslander extriangulated category. Recall the following
\begin{definition}
\label{def:ccp}
	A complete cotorsion pair in an extriangulated category $\Csc$ is a pair $(\Xcal,\Ycal)$ of (strictly) full subcategories closed under direct summands and satisfying:
	\begin{enumerate}[(a)]
		\item $\Xcal\perp_{\Ebb}\Ycal$: for any $X\in\Xcal$ and any $Y\in\Ycal$, we have $\Ebb(X,Y)=0$.
		\item $\Csc=\Cone(\Ycal,\Xcal)$: for any $C\in\Csc$, there is an $\sfr$-triangle $Y\infl X\defl C\dashrightarrow$ with $X\in\Xcal$ and $Y\in\Ycal$.
		\item $\Csc=\Fib(\Ycal,\Xcal)$: for any $C\in\Csc$, there is an $\sfr$-triangle $C\infl Y\defl X\dashrightarrow$ with $X\in\Xcal$ and $Y\in\Ycal$.
	\end{enumerate}

\end{definition}

We note that it follows from the definition that, if $(\Xcal,\Ycal)$ is a complete cotorsion pair, then $\Ycal = \Xcal^{\perp_\Ebb}$ and $\Xcal=\,^{\perp_\Ebb}\Ycal$.
It follows that:
\begin{itemize}
	\item $\Xcal$ and $\Ycal$ uniquely determine each other,
	\item both $\Xcal$ and $\Ycal$ are extension-closed,
	\item $\Proj_{(\C,\E)} \subseteq \Xcal$ and $\Inj_{(\C, \E)} \subseteq \Ycal$,
	\item if $\Csc$ is hereditary, in any $\sfr$-triangle
	\[Y\infl X\defl C\dashrightarrow\]
	with $X\in\Xcal$ and $Y\in\Ycal$, we have $Y\in\Xcal\cap\Ycal$; and
	in any $\sfr$-triangle
	\[C\infl Y\defl X\dashrightarrow\]
	with $X\in\Xcal$ and $Y\in\Ycal$, we have $X\in\Xcal\cap\Ycal$.
\end{itemize}

In order to apply \cref{Th:mutation} to intermediate co-$t$-structures, we will need the following result (which is a specific case of \cite[Theorem 5.7]{AdachiTsukamoto}):

\begin{theorem}[Adachi--Tsukamoto]
\label{thm:cCP}
	Let $\Csc$ be a 0-Auslander extriangulated category.
	Write $\cCP(\Csc)$ for the class of all complete cotorsion pairs in $\Csc$, and $\Tilt(\Csc)$ for the class of all tilting subcategories of $\Csc$. 
	There are mutually inverse bijections
	\begin{eqnarray*}
		\vph_1: \cCP(\Csc) & \leftrightarrow & \Tilt(\Csc) :\psi_1,
	\end{eqnarray*}
where 
\[\vph_1(\Xcal,\Ycal)=\Xcal\cap\Ycal \text{ and}\]
\[\psi_1(\Rcal)=(^{\perp_{\Ebb}}\Rcal,\Rcal^{\perp_\Ebb}) = (\copr\Rcal,\pr\Rcal).\]
\end{theorem}

\begin{proof}
This is a specific case of \cite[Theorem 5.7]{AdachiTsukamoto}.
For sake of completeness, we include a proof, which becomes quite easy in our specific setting.

	$\bullet$ The map $\vph_1$ is well-defined: Let $(\Xcal,\Ycal)$ be a complete cotorsion pair in $\Csc$, and let $\Rcal = \Xcal\cap\Ycal$.
	We claim that $\Rcal$ is a tilting subcategory of $\Csc$.
By (a) in the definition of $\cCP(\Csc)$, $\Rcal$ is rigid.
For any projective object $P\in\Csc$, consider an $\sfr$-triangle $P\infl Y\defl X\dashrightarrow$ with $X\in\Xcal$ and $Y\in\Ycal$.
In the exact sequence
\[\Ebb(X,-)|_{\Ycal}\to\Ebb(Y,-)|_{\Ycal}\to\Ebb(P,-)|_{\Ycal},\]
the left-most term vanishes since $\Xcal\perp_\Ebb\Ycal$ and the right-most term vanishes since $P$ is projective.
Thus $\Ebb(Y,-)|_{\Ycal}=0$ showing that $Y\in\Xcal$, hence $Y\in\Rcal$.
Moreover, in the exact sequence
\[\Ebb(-,Y)|_\Xcal\to\Ebb(-,X)|_\Xcal\to\Ebb^2(-,P)|_\Xcal\]
the left-most term vanishes since $\Xcal\perp_\Ebb\Ycal$ and the right-most term vanishes since $\Csc$ has global dimension at most one.
Hence $\Ebb(-,X)|_\Xcal = 0$ and $X\in\Rcal$.
This shows that $\Rcal$ is a tilting subcategory of $\Csc$.

	$\bullet$ The map $\psi_1$ is well-defined: Let $\Rcal$ be a tilting subcategory of $\Csc$ and let $\Xcal = \copr\Rcal$, $\Ycal=\pr\Rcal$.
	By \cref{thm: equivalent versions of tilting}, $\Rcal$ is both strongly torsion projective and strongly cotorsion injective.
	We thus have $^{\perp_\Ebb}\Rcal = \copr\Rcal$ and $\Rcal^{\perp_\Ebb}=\pr\Rcal$.
	We now check that $(\Xcal,\Ycal)$ satisfies (a) (b) and (c) in \cref{def:ccp}.
	
	(a) For any $X\in{^{\perp_\Ebb}\Rcal}$, and any $Y\in\pr\Rcal$, there is an $\sfr$-triangle $R_1\to R_0\to Y\dashrightarrow$, with $R_0,R_1\in\Rcal$.
	In the exact sequence
	\[\Ebb(X,R_0)\to\Ebb(X,Y)\to\Ebb^2(X,R_1),\]
	both end-terms vanish, hence $\Xcal\perp_\Ebb\Ycal$.
	
	(b) We prove that $\Csc=\Cone(\Rcal,\copr\Rcal)$: For any $C\in\Csc$, there is an $\sfr$-triangle $P_1\infl P_0\defl C\dashrightarrow$ with $P_0,P_1$ projective in $\Csc$.
	Since $\Rcal$ is tilting, there are $\sfr$-triangles $P_1\infl R_0\defl R_1\dashrightarrow$ and $P_0\infl R'_0\defl R'_1\dashrightarrow$ with $R_0,R_1,R'_0,R'_1\in\Rcal$.
	Applying \cref{lemma:shiftedOctahedron} twice, we obtain the following diagrams of morphisms of $\sfr$-triangles

	\[\begin{tikzcd}
	{P_1} & {P_0} & C & {} && {P_0} & E & {R_1} & {} \\
	{R_0} & E & C & {} && {R'_0} & {R'_0\oplus R_1} & {R_1} & {} \\
	{R_1} & {R_1} &&&& {R'_1} & {R'_1} \\
	{} & {} &&&& {} & {}
	\arrow[tail, from=1-1, to=1-2]
	\arrow[two heads, from=1-2, to=1-3]
	\arrow[dashed, from=1-3, to=1-4]
	\arrow[tail, from=2-1, to=2-2]
	\arrow[two heads, from=2-2, to=2-3]
	\arrow[dashed, from=2-3, to=2-4]
	\arrow[tail, from=1-1, to=2-1]
	\arrow[tail, from=1-2, to=2-2]
	\arrow[Rightarrow, no head, from=1-3, to=2-3]
	\arrow[two heads, from=2-1, to=3-1]
	\arrow[two heads, from=2-2, to=3-2]
	\arrow[Rightarrow, no head, from=3-1, to=3-2]
	\arrow[dashed, from=3-1, to=4-1]
	\arrow[dashed, from=3-2, to=4-2]
	\arrow[tail, from=1-6, to=1-7]
	\arrow[two heads, from=1-7, to=1-8]
	\arrow[tail, from=1-6, to=2-6]
	\arrow[tail, from=1-7, to=2-7]
	\arrow[two heads, from=2-6, to=3-6]
	\arrow[two heads, from=2-7, to=3-7]
	\arrow[dashed, from=3-6, to=4-6]
	\arrow[dashed, from=3-7, to=4-7]
	\arrow[Rightarrow, no head, from=3-6, to=3-7]
	\arrow[tail, from=2-6, to=2-7]
	\arrow[two heads, from=2-7, to=2-8]
	\arrow[dashed, from=1-8, to=1-9]
	\arrow[dashed, from=2-8, to=2-9]
	\arrow[Rightarrow, no head, from=1-8, to=2-8]
	\end{tikzcd}\]
	showing that $C$ belongs to $\Cone(\Rcal,\copr\Rcal)$.
	
	(c) We have $\Csc=\Fib(\pr\Rcal,\Rcal)$: This is dual to (b) and left to the reader.
	
	$\bullet$ We have $\vph_1\circ\psi_1 = 1$: By \cref{thm: equivalent versions of tilting}, any tilting subcategory $\Rcal$ is $\Ebb$-tilting, hence $\Rcal^{\perp_\Ebb} \cap\,\!^{\perp_\Ebb}\Rcal=\Rcal$.
	
	$\bullet$ We have $\psi_1\circ\vph_1=1$: Let $(\Xcal,\Ycal)$ be a complete cotorsion pair. 
	We have to show that $^{\perp_\Ebb}(\Xcal\cap\Ycal) =\,\! ^{\perp_\Ebb}\Ycal$ and that $(\Xcal\cap\Ycal)^{\perp_\Ebb}=\Xcal^{\perp_\Ebb}$.
	We only prove the first equality, the second one being dual.
	Let $U\in\,\!^{\perp_\Ebb}(\Xcal\cap\Ycal)$ and let $Y\in\Ycal$.
	There is an $\sfr$-triangle $Y'\infl X'\defl Y\dashrightarrow$ with $X'\in\Xcal$ and $Y'\in\Ycal$. 
	Because $\Ycal$ is extension-closed in $\Csc$, we have $X'\in\Xcal\cap\Ycal$, which implies that, in the exact sequence
	\[\Ebb(U,X')\to\Ebb(U,Y)\to\Ebb^2(U,Y')\]
	both end terms vanish.
	Thus $\Ebb(U,Y)=0$.
\end{proof}





\subsection{Extended cohearts and mutation of intermediate co-$t$-structures}

In this section, we let $\Dsc$ be a triangulated category. It is not assumed to be essentially small or $R$-linear except when stated otherwise. 
We first recall the necessary definitions and one of the main results in~\cite{PauksztelloZvonareva}.

\begin{definition}[\cite{Pauksztello,Bondarko}]
	A \defn{co-$t$-structure}, or a \defn{weight structure}, on $\Dsc$ is a pair $(\Acal,\Bcal)$ of full subcategories, closed under summands, such that:
	\begin{enumerate}[(i)]
 	 \item $\Acal\perp\Bcal$: for any $A\in\Acal$ and any $B\in\Bcal$, we have $\Dsc(A,B)=0$,
 	 \item $\Dsc=\Acal\ast\Bcal$: for any $D\in\Dsc$ there is a triangle $A\to D\to B\to\Sigma A$, with $A\in\Acal$ and $B\in\Bcal$.
 	 \item $\Acal\subseteq\Sigma\Acal$ (or equivalently, $\Sigma\Bcal\subseteq\Bcal$).
	\end{enumerate}
 A co-t-structure $(\Acal,\Bcal)$ is \defn{bounded} if
\[
\bigcup_{i\in\mathbb{Z}} \susp^i \Acal = \bigcup_{i\in\mathbb{Z}} \susp^i \Bcal = 
\Dsc.
\]
If $(\Acal,\Bcal)$ is a co-$t$-structure, its \defn{coheart} is the full subcategory $\Scal=(\Sigma\Acal) \cap \Bcal$, and its \defn{extended coheart} is the full subcategory $\Csc=\Scal\ast\Sigma\Scal = \Sigma^2\Acal\cap\Bcal$.
\end{definition}

As for cotorsion pairs, the subcategories $\Acal,\Bcal$ forming a co-$t$-structure determine each other and are closed under extensions (see the comments after \cref{def:ccp}), hence $\Csc$ is extension-closed in $\Dsc$ and inherits an extriangulated structure.

Write $\Cot(\Dsc)$ for the collection of all co-$t$-structures on $\Dsc$.
If $(\Acal,\Bcal)$ is a fixed co-$t$-structure on $\Dsc$, write 
$\Cot_{(\Acal,\Bcal)}(\Dsc)$ for the collection of \defn{intermediate} co-$t$-structures:
\[\Cot_{(\Acal,\Bcal)}(\Dsc) = \{(\Acal',\Bcal')\in\Cot(\Dsc)\;|\; \Acal\subseteq\Acal'\subseteq\Sigma\Acal\}.\]
We note that, if $(\Acal,\Bcal)$ and $(\Acal',\Bcal')$ are co-$t$-structures, the condition $\Acal\subseteq\Acal'\subseteq\Sigma\Acal$ is equivalent to the condition $\Sigma\Bcal\subseteq\Bcal'\subseteq\Bcal$. Also, we have both $\Acal\subseteq\Acal'$ and $\Bcal\subseteq\Bcal'$ if and only if the co-$t$-structures coincide.

\begin{remark} \label{rem:BoundedCoT}
A coheart of a co-$t$-structure on $\Dsc$ is a weakly idempotent complete additive category, which is not necessarily idempotent complete or abelian (see \cite{bondarko2020hearts} for the details).
It is always a presilting subcategory of $\Dsc$. It is silting in $\Dsc$ if and only if the co-$t$-structure is bounded; moreover, there is a bijection between silting subcategories and bounded co-$t$-structures, see \cite[Corollary 5.9]{hernandez2013auslander}. Each co-$t$-structure restricts to a bounded co-$t$-structure on $\thick(\Scal)$, and this is the largest thick subcategory of $\Dsc$ which admits a bounded co-$t$-structure restricted from a given one. A co-$t$-structure is bounded if and only if each its intermediate co-$t$-structure is bounded, since we have $\thick(\Scal') = \thick(\Scal)$ for each intermediate co-$t$-structure with coheart $\Scal'$, or, equivalently, the coheart of each intermediate co-$t$-structure is silting in $\thick(\Scal)$.
\end{remark}

\begin{theorem}[Pauksztello--Zvonareva \cite{PauksztelloZvonareva}]
\label{thm:PZ}
Let $\Dsc$ be a triangulated category and let $(\Acal,\Bcal)$ be a co-$t$-structure on $\Dcal$.
There are mutually inverse bijective correspondences
	\begin{eqnarray*}
		\vph_2: \Cot_{(\Acal,\Bcal)}(\Dsc) & \leftrightarrow & \cCP(\Csc) :\psi_2,
	\end{eqnarray*}
where
\[\vph_2(\Acal',\Bcal')=(\Bcal\cap\Sigma\Acal',\Bcal'\cap\Sigma^2\Acal) \text{ and}\]
\[\psi_2(\Xcal,\Ycal)=\left(\add(\Sigma^{-1}\Acal\ast\Sigma^{-1}\Xcal),\add(\Ycal\ast\Sigma^2\Bcal)\right).\]
\end{theorem}

\begin{remark}
	The bijective correspondence $\vph_2$ moreover sends cohearts to cores: with the notations of the theorem, let $(\Acal',\Bcal')\in\Cot_{(\Acal,\Bcal)}(\Dsc)$ and let $(\Xcal,\Ycal)=\vph_2(\Acal',\Bcal')$. Then $\Xcal\cap\Ycal = \Scal' = \Sigma\Acal'\cap\Bcal'$.
\end{remark}

Note that this result has the following nice corollary. It might be known to experts, but we were not able to find this statement in the literature.

\begin{corollary} \label{cor:RestrictionExtensionCoT}
Let $\Dsc$ be a triangulated category and let $(\Acal,\Bcal)$ be a co-$t$-structure on $\Dcal$ with 
coheart $\Scal$. Then for each co-$t$-structure $(\Acal'',\Bcal'')$ on $\Tsc := \thick(\Scal)$ intermediate with respect to the restriction $(\Acal \cap \Tsc,\Bcal \cap \Tsc)$ of $(\Acal,\Bcal)$, there exists a unique co-$t$-structure $(\Acal',\Bcal') \in \Cot_{(\Acal,\Bcal)}(\Dsc)$ extending $(\Acal'',\Bcal'')$.
\end{corollary}

\begin{proof}
Note that $\Csc \subset \Tsc$.
Therefore, both collections $\Cot_{(\Acal,\Bcal)}(\Dsc)$ and $\Cot_{(\Acal \cap \Tsc,\Bcal \cap \Tsc)}(\Tsc)$ are in bijective correspondences with $\cCP(\Csc)$, hence they are themselves in a bijective correspondence. It remains to show that the bijection $\Cot_{(\Acal,\Bcal)}(\Dsc) \to \Cot_{(\Acal \cap \Tsc,\Bcal \cap \Tsc)}(\Tsc)$ is given by the restriction. Let $(\Xcal, \Ycal)$ be the cotorsion pair corresponding to $(\Acal'',\Bcal'')$. Then
\[
(\Acal',\Bcal') :=  \left(\add(\Sigma^{-1}\Acal\ast\Sigma^{-1}\Xcal),\add(\Ycal\ast\Sigma^2\Bcal)\right)
\]
is the corresponding co-$t$-structure on $\Dcal$, while we have 
\[
(\Acal'',\Bcal'') = \left(\add(\Sigma^{-1}(\Acal \cap \Tsc) \ast\Sigma^{-1}\Xcal),\add(\Ycal\ast\Sigma^2 (\Bcal \cap \Tsc))\right).
\]
Observe that we have $\Acal'' \subset \Acal' \cap \Tsc$ and $\Bcal'' \subset \Bcal' \cap \Tsc$. Thus, $(\Acal'',\Bcal'') = (\Acal' \cap \Tsc, \Bcal' \cap \Tsc)$.
\end{proof}

\begin{remark} \label{rem: caution_about_extending_co_t}
Note that \cref{cor:RestrictionExtensionCoT} ensures the existence and uniqueness of extensions of co-t-structures intermediate with respect to the restriction $(\Acal \cap \Tsc,\Bcal \cap \Tsc)$, to co-t-structures intermediate with respect to $(\Acal,\Bcal)$.  A similar result is true for co-t-structures ``intermediate'' in a broader sense, i.e. such that $\susp^a\mathcal{A} \subseteq \mathcal{A}' \subseteq \susp^b \mathcal{A}$ for finite $a < b$. This follows from \cite[Corollary 3.9]{AdachiTsukamoto} applied twice as in the proof of \cref{cor:RestrictionExtensionCoT}. \cite[Corollary 3.9]{AdachiTsukamoto} 
generalizes \cref{thm:PZ} to suitable 
 $(b-a)$-Auslander extriangulated categories.

The ``intermediate'' assumptions, however, cannot be dropped completely. Here is an elementary explanation: each non-zero triangulated category $\Dsc$ admits at least two distinct co-t-structures $(0, \Dsc)$ and $(\Dsc, 0)$. For either of them, the coheart is the zero subcategory. They are both unbounded and the largest thick subcategory which admits a bounded co-t-structure restricted from either of them is again the zero subcategory $0 = \thick(0)$. The restricted co-t-structures, of course, both equal $(0, 0)$. In other words, the co-t-structure $(0, 0)$ on the zero subcategory admits two distinct extensions to $\Dsc$. We do not get a contradiction with \cref{cor:RestrictionExtensionCoT} precisely because  $(0, \Dsc)$ is not intermediate with respect to $(\Dsc, 0)$ and vice versa.

Note also that there exists a similar result for t-structures with finite distance from a given one, see \cite{chen2022extensions} and references to Keller \cite{keller2005triangulated} and Marks-Zvonareva \cite{marks2022lifting} therein.
\end{remark}


\begin{lemma} \label{lem:ExtendedCoheart0Ausl}
Let $\Dsc$ be an essentially small $R$-linear  triangulated category and let $(\Acal,\Bcal)$ be a co-$t$-structure with associated coheart $\Scal$ and extended coheart $\Csc$.
Then the extriangulated category $\Csc$ is reduced 0-Auslander, with projective objects $\Scal$ and injective objects $\Sigma\Scal$. 
\end{lemma}

\begin{proof}
Since $\Csc=\Scal\ast\Sigma\Scal$, it is enough to show that $\Scal$ is projective in $\Csc$.
Let $S\in\Scal$ and $C\in\Csc$.
Then $S$ belongs to $\Sigma\Acal\cap\Bcal$ while $\Sigma C$ lies in $\Sigma^3\Acal\cap\Sigma\Bcal$.
Because $\Sigma\Acal$ is left Hom-orthogonal to $\Sigma\Bcal$, we have $\Dsc(S,\Sigma C)=0$.
\end{proof}

\begin{corollary} \label{cor:IntermCoTvsSilting}
Let $\Dsc$ be an essentially small $R$-linear  triangulated category and let $(\Acal,\Bcal)$ be a co-$t$-structure with associated extended coheart $\Csc$. There are mutually inverse bijective correspondences 
\begin{eqnarray*}
		\vph_3: \Cot_{(\Acal,\Bcal)}(\Dsc) & \leftrightarrow & \Tilt(\Csc) :\psi_3,
	\end{eqnarray*}
given by 
\[\vph_3(\Acal',\Bcal')=\susp\Acal'\cap\Bcal' \text{ and}\]
\[\psi_3(\Rcal)=(\left(\add(\Sigma^{-1}\Acal\ast \susp^{-1}(\copr\Rcal), \add((\pr\Rcal)\ast\Sigma^2\Bcal)\right),\]
where $\pr\Rcal, \copr\Rcal$ are considered in $\Csc$.
\end{corollary}

\begin{proof}
Combine \cref{lem:ExtendedCoheart0Ausl} with  \cref{thm:cCP,thm:PZ}.
\end{proof}

\begin{corollary} 
\label{cor:MutationIntermCoT}
\emph{[Mutation of intermediate co-$t$-structures]}
 Let $\Dsc$ be an essentially small $R$-linear   triangulated category and let $(\Acal,\Bcal)$ be a co-$t$-structure with extended coheart $\Csc$.
 Let $(\Acal',\Bcal')\in\Cot_{(\Acal,\Bcal)}(\Dsc)$ with coheart $\Scal'$.
 Assume that the category $\Csc$ satisfies the assumptions of \cref{Th:mutation} and let $X\in\Scal'$ be indecomposable.
 Then there is a unique intermediate co-$t$-structure $(\Ccal,\Dcal)\in\Cot_{(\Acal,\Bcal)}(\Dsc)$ whose coheart $\Tcal$ satisfies $\Tcal\cap\Scal' =\add\{X'\in\Scal', X' \text{ \emph{indecomposable and} } X' \ncong X\}$.
 Moreover, there is an indecomposable object $Y\in\Tcal$ such that 
 \[\add\{Y'\in\Tcal, Y' \text{ \emph{indecomposable and} } Y' \ncong Y\} = \add\{X'\in\Scal', X' \text{ \emph{indecomposable and} } X' \ncong X\}.\]
\end{corollary}

\begin{proof}
	Combine \cref{Th:mutation}, \cref{lem:ExtendedCoheart0Ausl} and \cref{cor:IntermCoTvsSilting}. 
 \end{proof}

\begin{remark}
We note that the assumptions of \cref{Th:mutation} are satisfied for example if $\Csc$ is Krull--Schmidt  and any almost complete rigid subcategory is functorially finite in $\Csc$, as in Condition (F) in \cite[Secion 2.4]{AiharaIyama}. Note though that our assumptions concern the subcategory $\Csc$ and not the entire category $\Dsc$ which might have no silting subcategories.
\end{remark}

The theory of mutations of bounded co-$t$-structures is implicit in \cite{AiharaIyama} and stated explicitly in certain generality in \cite{KoenigYang}, 
see also \cite[Section 6]{jorgensen2018co}. It goes through the bijection of bounded co-$t$-structures on $\Dsc$ and silting subcategories in $\Dsc$ mentioned in \cref{rem:BoundedCoT} and silting mutations from \cite{AiharaIyama}. 
Mutations of co-$t$-structures, intermediate with respect to a certain bounded co-$t$-structure (and so also bounded), are  considered in \cite[Corollary 4.3]{BrustleYang}. It follows from the results in this section that mutations of intermediate co-$t$-structures can be canonically defined via mutations of silting subcategories in the extended coheart of the original co-$t$-structure, which do not have to be silting in the entire $\Dsc$.

\begin{proposition} \label{prop: mut_co_t_compatibility}
\begin{itemize}
\item[(i)] Mutations in \cref{cor:MutationIntermCoT} recover known mutations of co-$t$-structures, intermediate with respect to a bounded one.
\item[(ii)] Mutations in \cref{cor:MutationIntermCoT} commute with the restriction to and extension from $\thick(\Scal).$
\end{itemize}
\end{proposition}

\begin{proof}
Part (i) follows from part (ii), which in turn is direct consequence of \cref{cor:RestrictionExtensionCoT} and its proof. 
\end{proof}

\begin{remark} \label{rem:caution_about_mut_co_t}
The upshot of \cref{cor:MutationIntermCoT} and \cref{prop: mut_co_t_compatibility} is that 
 the theory of mutation for bounded co-t-structures implies that for non-bounded ones by restricting to and extending from thick subcategories. The fact that the extension of a mutated restriction exists and is unique follows from \cref{cor:RestrictionExtensionCoT} and seems to be novel. See also the discussion in \cref{rem: caution_about_extending_co_t}.
\end{remark}

\begin{remark} \label{rem: PZ_torsion_tau-tilting}
Another result of Pauksztello--Zvonareva \cite[Theorem 3.6]{PauksztelloZvonareva} provides a bijection between complete cotorsion  pairs in extended cohearts $\Scal\ast\Sigma\Scal$ of co-$t$-structures and functorially finite torsion classes in $\modd\Scal$, under certain natural conditions. By adapting the proofs in \cite[Section 3]{PauksztelloZvonareva} to the setting of $\Hom$-finite Krull-Schmidt 0-Auslander categories with noetherian $\modd\overline{\Pcal}$ and using our \cref{prop:module as quotient by injectives}, we can obtain a bijection between complete cotorsion pairs in such 0-Auslander categories and functorially finite torsion classes in $\modd\overline{\Pcal}$.  A similar bijection was given in \cite{liu2020hereditary} for a certain class of 0-Auslander subcategories (defined in terms of 2-rigid subcategories $\Wcal$ with $\Ical \subsetneq \Wcal$) in extriangulated categories; there, support $\tau$-tilting subcategories were considered instead of functorially finite torsion classes. Via these bijections, \cref{Th:mutation} recovers mutations of functorially finite torsion classes and support $\tau$-tilting subcategories in $\tau$-tilting theory \cite{AdachiIyamaReiten, iyama2014introduction, treffinger2021tau, IyamaJorgensenYang}.

Generalizing \cite[Theorem 3.6]{PauksztelloZvonareva}, the functor inducing the equivalence in \cref{prop:module as quotient by injectives} sends cotorsionfree classes to functorially finite torsion classes. A \defn{maximal green sequence} for a finite-dimensional algebra $\Lambda$ is a maximal chain of functorially finite torsion classes in the lattice of torsion classes in $\modd\Lambda$. In our setup, this motivates a (dual, in some sense) definition of \defn{maximal green sequences for 0-Auslander categories} as a maximal chain in the poset of cotorsionfree classes in complete cotorsion pairs. We do not investigate such sequences further in this paper.
\end{remark}

\section{0-Auslander 
categories from gentle algebras}
\label{section:gentle}

Inspired by~\cite[Figure 30]{PaluPilaudPlamondon-nonkissing}, it was suggested in~\cite{IyamaNakaokaPalu} that the $\tau$-tilting theory for gentle algebras might be viewed as a mutation theory for maximal rigid objects in a certain exact category constructed from the module category over the blossoming gentle algebra. In [loc. cit.], an explicit small example was considered.  In this section, we carry out this strategy for all gentle algebras.
We define two exact subcategories of the module category over the blossoming algebra and show that they are equivalent.
The first subcategory, studied in \cref{ssection: extricat for gentle}, is algebraically more convenient, while the second (see \cref{ssection:ExtricatWalk}) is more obviously related to the combinatorics of non-kissing walks~\cite{McConville,BrustleDouvilleMousavandThomasYildirim,PaluPilaudPlamondon-nonkissing}.
We prove that those categories are 0-Auslander and their maximal rigid objects are in bijection with support $\tau$-tilting modules~\cite{AdachiIyamaReiten} over the gentle algebra, respectively with non-kissing facets. Under this identification, straight walks bijectively correspond to indecomposable projective-injective objects. The ideal quotients of these categories by the projective-injective objects are no longer exact, but are reduced 0-Auslander extriangulated categories.
The mutation theory developed in~\cref{section:tilting} thus recovers $\tau$-tilting mutation for gentle algebras, and flips of non-kissing facets. 

\subsection{Reminder on gentle algebras and non-kissing}

Fix an algebraically closed field $K$.
If $\alpha$ and $\beta$ are two arrows in a quiver with $t(\alpha)=s(\beta)$, we write $\alpha\beta$ for their composition.

\subsubsection{Gentle bound quivers}

\begin{definition}\cite{ButlerRingel}
A gentle algebra is a basic, finite-dimensional algebra $A=KQ/I$ where:
\begin{itemize}
 \item Each vertex in $Q$ is the source of at most two arrows and the target of at most two arrows.
 \item The ideal $I$ is generated by paths of length two.
 \item For each arrow $\alpha$ in $Q$ with target vertex $i$, there is at most one arrow $\beta$ and at most one arrow $\gamma$ with source $i$, such that $\alpha\beta\notin I$ and $\alpha\gamma\in I$.
 \item For each arrow $\alpha$ in $Q$ with source vertex $j$, there is at most one arrow $\beta$ and at most one arrow $\gamma$ with target $j$, such that $\beta\alpha\notin I$ and $\gamma\alpha\in I$.
\end{itemize}
\end{definition}

\subsubsection{String and band modules}

Indecomposable representations of a gentle bound quiver $(Q,I)$ are parameterized by strings and bands.
A string for $(Q,I)$ is a non-oriented path in $Q$, written as a word $\alpha_1^{\pm 1}\cdots\alpha_r^{\pm 1},$ where:
\begin{itemize}
 \item for every subpath $\alpha_i\alpha_{i+1},$ we have $t(\alpha_i)=s(\alpha_{i+1})$ and $\alpha_i\alpha_{i+1}\notin I$;
 \item for every subpath $\alpha_i^{-1}\alpha_{i+1}^{-1},$ we have $s(\alpha_i)=t(\alpha_{i+1})$ and $\alpha_{i+1}\alpha_{i}\notin I$;
 \item for every subpath $\alpha_i\alpha_{i+1}^{-1},$ or $\alpha_i^{-1}\alpha_{i+1}$ we have $\alpha_i\neq\alpha_{i+1}$.
\end{itemize}
Note that strings of length 0 are allowed.
For each vertex $p\in Q_0$, we write $\eps_p$ for the string of length 0 at $p$, and also for the associated idempotent in $KQ/I$.
Strings are usually considered unoriented (e.g. the strings $\alpha\beta^{-1}$ and $\beta\alpha^{-1}$ are implicitly identified).

With each string $\sigma$, one can associate an indecomposable representation $M_\sigma$ of $(Q,I)$: The vertices appearing in the string encode basis vectors, while the arrows show which basis vector is sent to which other basis vector under the corresponding linear map.
The representation associated with the string $\eps_p$ is the simple object concentrated at vertex $p$.

\begin{example}
\label{example:gentleQuiver}
 The quiver below, with the ideal of relations generated by $\beta\epsilon$ and $\delta\gamma$, is gentle.
 On the right, the string $\epsilon^{-1}\delta^{-1}\alpha\beta\gamma$ is drawn.:
\[\begin{tikzcd}[sep=tiny]
	&&&&&&&&& 1 \\
	& 2 &&&&&&& 3 && 2 \\
	1 && 3 & 5 &&& {\epsilon^{-1}\delta^{-1}\alpha\beta\gamma :} & 5 &&&& 3 \\
	&&& 4 &&&&&&&&& 4
	\arrow["\alpha", from=3-1, to=2-2]
	\arrow[""{name=0, anchor=center, inner sep=0}, "\beta", from=2-2, to=3-3]
	\arrow[""{name=1, anchor=center, inner sep=0}, "\gamma"', from=3-3, to=4-4]
	\arrow[""{name=2, anchor=center, inner sep=0}, "\delta", from=3-1, to=3-3]
	\arrow[""{name=3, anchor=center, inner sep=0}, "\epsilon", from=3-3, to=3-4]
	\arrow["\delta"', from=1-10, to=2-9]
	\arrow["\epsilon"', from=2-9, to=3-8]
	\arrow["\alpha", from=1-10, to=2-11]
	\arrow["\beta", from=2-11, to=3-12]
	\arrow["\gamma", from=3-12, to=4-13]
	\arrow[shift right=1, color={rgb,255:red,214;green,92;blue,92}, curve={height=-12pt}, shorten <=8pt, shorten >=8pt, no head, from=0, to=3]
	\arrow[shift left=2, color={rgb,255:red,214;green,92;blue,92}, curve={height=12pt}, shorten <=11pt, shorten >=11pt, no head, from=2, to=1]
\end{tikzcd}\]
The associated indecomposable representation is:
\[\begin{tikzcd}[sep=small]
	& K \\
	K && {K^2} & K \\
	&&& K.
	\arrow["1", from=2-1, to=1-2]
	\arrow["{\left[^1_0\right]}", from=1-2, to=2-3]
	\arrow["{^{[1\,0]}}"', from=2-3, to=3-4]
	\arrow["{\left[^0_1\right]}", from=2-1, to=2-3]
	\arrow["{_{[0\, 1]}}", from=2-3, to=2-4]
\end{tikzcd}\]
This is the indecomposable projective cover of the simple module concentrated at vertex 1.
\end{example}

A band is a string of length at least one, whose source and target coincide, and such that
\begin{itemize}
 \item its square is a string;
 \item it is not the power of a smaller string.
\end{itemize}
Given a band, an element $\lambda\in K^\times$ and a positive integer $d$, there is an associated indecomposable representation.
The vector space at each vertex appearing in the band is $K^d$ (note that the source and target vertices of the band are identified), all linear maps are identities except for one which is given by a Jordan block of size $d$ with eigenvalue $\lambda$.

All indecomposable representations of a gentle bound quiver are given by string and band modules. We refer to~\cite[p.161]{ButlerRingel} for a more detailed statement.

\subsubsection{Blossoming}

Let $(Q,I)$ be a gentle bound quiver.
In order to simplify computations (of the Avella-Alaminos--Gei{\ss} invariant~\cite{Avella-AlaminosGeiss} in~\cite{Asashiba-Sugaku} or of the Auslander--Reiten translation in~\cite{BrustleDouvilleMousavandThomasYildirim,PaluPilaudPlamondon-nonkissing}), it is convenient to expand $(Q,I)$ by making all vertices in $Q_0$ 4-valent by adding as few arrows as possible.
The resulting gentle bound quiver is called the blossoming bound quiver of $(Q,I)$ and is denoted $(Q\bls,I\bls)$. 
It is characterised (up to isomorphism) by the fact that:
\begin{itemize}
 \item it is gentle;
 \item its vertices are either leaves (i.e. sinks or sources) or 4-valent;
 \item by deleting its leaves, one recovers $(Q,I)$.
\end{itemize}
The leaves of $Q\bls$ are called its blossoming vertices, or simply its blossoms.
\begin{example}
 The blossoming bound quiver of the gentle bound quiver appearing in \cref{example:gentleQuiver} is:

\[\begin{tikzcd}
	& {} && {} \\
	&& 2 && {} \\
	{} & 1 && 3 & 5 & {} \\
	{} && {} & 4 & {} & {} \\
	&&& {}
	\arrow[""{name=0, anchor=center, inner sep=0}, "\alpha", from=3-2, to=2-3]
	\arrow[""{name=1, anchor=center, inner sep=0}, "\beta", from=2-3, to=3-4]
	\arrow[""{name=2, anchor=center, inner sep=0}, "\gamma", from=3-4, to=4-4]
	\arrow[""{name=3, anchor=center, inner sep=0}, "\delta", from=3-2, to=3-4]
	\arrow[""{name=4, anchor=center, inner sep=0}, "\epsilon", from=3-4, to=3-5]
	\arrow[""{name=5, anchor=center, inner sep=0}, shorten <=5pt, from=3-1, to=3-2]
	\arrow[""{name=6, anchor=center, inner sep=0}, shorten <=12pt, from=4-1, to=3-2]
	\arrow[""{name=7, anchor=center, inner sep=0}, shorten <=12pt, from=1-4, to=2-3]
	\arrow[""{name=8, anchor=center, inner sep=0}, shorten >=12pt, from=2-3, to=1-2]
	\arrow[""{name=9, anchor=center, inner sep=0}, shorten >=6pt, from=3-5, to=3-6]
	\arrow[""{name=10, anchor=center, inner sep=0}, shorten <=4pt, from=2-5, to=3-5]
	\arrow[""{name=11, anchor=center, inner sep=0}, shorten <=6pt, from=4-3, to=4-4]
	\arrow[""{name=12, anchor=center, inner sep=0}, shorten >=4pt, from=4-4, to=5-4]
	\arrow[""{name=13, anchor=center, inner sep=0}, shorten >=6pt, from=4-4, to=4-5]
	\arrow[""{name=14, anchor=center, inner sep=0}, shorten >=4pt, from=3-5, to=4-5]
	\arrow[shift right=1, draw={rgb,255:red,214;green,92;blue,92}, curve={height=-12pt}, shorten <=8pt, shorten >=8pt, no head, from=1, to=4]
	\arrow[shift left=2, draw={rgb,255:red,214;green,92;blue,92}, curve={height=12pt}, shorten <=8pt, shorten >=8pt, no head, from=3, to=2]
	\arrow[shift right=1, color={rgb,255:red,214;green,92;blue,92}, curve={height=-6pt}, shorten <=8pt, shorten >=8pt, no head, from=5, to=0]
	\arrow[shift left=1, color={rgb,255:red,214;green,92;blue,92}, curve={height=6pt}, shorten <=10pt, shorten >=10pt, no head, from=6, to=3]
	\arrow[shift right=2, color={rgb,255:red,214;green,92;blue,92}, curve={height=-6pt}, shorten <=7pt, shorten >=7pt, no head, from=0, to=8]
	\arrow[shift right=2, color={rgb,255:red,214;green,92;blue,92}, curve={height=-6pt}, shorten <=7pt, shorten >=7pt, no head, from=7, to=1]
	\arrow[shift left=1, color={rgb,255:red,214;green,92;blue,92}, curve={height=6pt}, shorten <=4pt, shorten >=4pt, no head, from=11, to=12]
	\arrow[shift right=1, color={rgb,255:red,214;green,92;blue,92}, curve={height=-6pt}, shorten <=4pt, shorten >=4pt, no head, from=13, to=12]
	\arrow[shift right=1, color={rgb,255:red,214;green,92;blue,92}, curve={height=-6pt}, shorten <=4pt, shorten >=4pt, no head, from=10, to=9]
	\arrow[shift left=1, color={rgb,255:red,214;green,92;blue,92}, curve={height=6pt}, shorten <=4pt, shorten >=4pt, no head, from=4, to=14]
\end{tikzcd}\]
\end{example}

We denote by $A\bls$ the algebra $KQ\bls/I\bls.$
A walk on $(Q,I)$ is defined as a maximal string for $(Q\bls,I\bls)$.
A cohook is a string of the form $\al_1\cdots\al_r\beta^{-1}$.
By adding left and a right cohooks to a string for $(Q,I)$, one obtains a walk on $(Q,I)$:
If $\sigma$ is a string for $(Q,I)$ and if $\omega$ is the associated walk, then we have $M_\omega = \tau_{A\bls}M_\sigma$.
This yields a bijection between strings for $(Q,I)$ and walks on $(Q,I)$ that do not correspond to injective string modules over $(Q\bls,I\bls)$.
It is shown in~\cite{BrustleDouvilleMousavandThomasYildirim,PaluPilaudPlamondon-nonkissing} that the $\tau$-rigidity of representations of $(Q,I)$ can be described purely combinatorially by making use of this bijection.
The compatibility relation between walks corresponding to the $\tau$-rigidity is called non-kissing and was introduced (for grids) in~\cite{McConville}.
Maximal sets of pairwise non-kissing walks on $(Q,I)$, called non-kissing facets, are thus in bijection with isomorphism classes of basic $\tau$-tilting pairs (in the sense of~\cite{AdachiIyamaReiten}) over $A$. By~\cite{McConville,BrustleDouvilleMousavandThomasYildirim,PaluPilaudPlamondon-nonkissing}, there is a notion of flip for non-kissing facets, similar to the notion of flip for triangulations.
Under the bijection explained above, flips correspond to $\tau$-tilting mutations.

\subsection{Exact blossoming of the module category of a gentle algebra}\label{ssection: extricat for gentle}

\begin{definition}\label{DefGentle_E}
Let $\Esc\se\modd A\bls$ be the full subcategory consisting of those $E\in \modd A\bls$ that satisfy the following conditions.
\begin{itemize}
\item[{\rm (e1)}] $\Hom_{A\bls}(\soc P,E)=0$ for any $P\in\proj A\bls$.
\item[{\rm (e2)}] $\mathrm{pd}_{A\bls}E\le 1$.
\end{itemize}
It is straightforward to check that  $\Esc\se\modd A\bls$ is closed under isomorphisms, extensions, and direct summands.
\end{definition}

\begin{remark}\label{RemGentle_E}
For any $E\in\modd A\bls$, the following holds.
\begin{enumerate}
\item $E$ satisfies {\rm (e1)} if and only if it satisfies $\Hom_{A\bls}(S,E)=0$ for any simple projective $S\in\modd A\bls$, if and only if it does not contain any non-zero simple projective in its composition series.
\item If $E$ satisfies {\rm (e2)}, then $\Ext^1_{A\bls}(E,\Fac J)=0$ holds for any $J\in\inj A\bls$. This is true in any abelian category.
\item $E$ belongs to $\Esc$ if and only if there exists a short exact sequence
\begin{equation}\label{ExPPE}
0\to P_1\ov{\iota}{\lra}P_0\to E\to 0
\end{equation}
in $\modd A\bls$ satisfying $P_0,P_1\in\proj A\bls$ and $\iota(\soc P_1)=\soc P_0$.
\end{enumerate}
\end{remark}

\begin{proposition}\label{ProGentle_ProjE}
The following holds.
\begin{enumerate}
\item For any $P\in\proj A\bls$, we have $P/\soc P\in\proj\Esc$.
\item $\Esc$ has enough projectives. Moreover, any $E\in\Esc$ has a projective resolution of the form
\[ 0\to P_1/\soc P_1\to P_0/\soc P_0\to E\to 0. \]
In particular, we have $\mathrm{pd}_{\Esc}E\le 1$.
\item Any object in $\proj\Esc$ is a direct summand of an object of the form $P/\soc P$ for some $P\in\proj A\bls$.
\end{enumerate}
\end{proposition}

\begin{proof}
{\rm (1)} Remark that we have $\soc P\in\proj A\bls$. Thus, $P/\soc P\in\Esc$ is a direct consequence of \cref{RemGentle_E} {\rm (3)}.
From the short exact sequence $0\to\soc P\to P\to P/\soc P\to 0$ and the assumption of $\Hom_{A\bls}(\soc P,E)=0$, we also have $\Ext^1_{A\bls}(P/\soc P,E)=0$ for any $E\in\Esc$. This shows $P/\soc P\in\proj\Esc$.

{\rm (2)} This follows from {\rm (1)} and \cref{RemGentle_E} {\rm (3)}.

{\rm (3)} This follows immediately from {\rm (1)} and {\rm (2)}.
\end{proof}

\begin{lemma}\label{LemGentle_InjE}
The following holds.
\begin{enumerate}
\item If $E\in\Esc$ belongs to $\Fac J$ for some $J\in\inj A\bls$, then $E\in\inj\Esc$.
\item If a short exact sequence in $\modd A\bls$
\[ 0\to P\ov{\iota}{\lra}H\to I\to0 \]
satisfies $P\in\proj A\bls$, $H\in(\proj A\bls)\cap(\inj A\bls)$ and $\iota(\soc P)=\soc H$, then we have $I\in\inj\Esc$.
\item For any $P\in\proj A\bls$, there exists a short exact sequence $0\to P\ov{\iota}{\lra}H\to I\to0$ in $\modd A\bls$ satisfying the conditions in {\rm (2)}. Remark that this also induces the following short exact sequence in $\Esc$,
\[ 0\to P/\soc P\ov{\iota}{\lra}H/\soc H\to I\to0 \]
in which we have $P/\soc P\in\proj\Esc$, $H/\soc H\in(\proj\Esc)\cap(\inj\Esc)$, and $I\in\inj\Esc$.
\end{enumerate}
\end{lemma}

\begin{proof}
{\rm (1)} This follows from \cref{RemGentle_E} {\rm (2)}.

{\rm (2)} This follows from {\rm (1)} and \cref{RemGentle_E} {\rm (3)}.

{\rm (3)} Since $\soc P\in\proj A\bls$, we just have to take an injective hull $P\ov{\iota}{\lra}H$ and put $I=\Cok\iota$.
\end{proof}

\begin{corollary}
\label{cor:E is 0-Auslander}
The category $\Esc$ is $0-$Auslander.
\end{corollary}

\begin{proof}
This is a combination of \cref{ProGentle_ProjE}.(2) and \cref{LemGentle_InjE}.(3): these state that $\mathrm{pd}(\Esc) \le 1$ and that $\mathrm{dom.dim}(\Esc) \ge 1,$ respectively.
\end{proof}

By \cref{Thm:Auslander:equiv}, we find that $\Esc$ also has enough injectives and that $\mathrm{id}(\Esc, \Ebb) \le 1 \le \mathrm{codom.dim}(\Esc)$. For the sake of completeness, we describe the injectives and write this more explicitly.

\begin{proposition}\label{PropGentle_InjE}
The following holds.
\begin{enumerate}
\item For any $E\in\Esc$, there exists $H\in(\proj A\bls)\cap(\inj A\bls)$ and a short exact sequence
\begin{equation}\label{InjResolE}
0\to E\to I^0\to I^1\to0
\end{equation}
satisfying $I^0,I^1\in\Esc\cap\Fac H$.
\item $\Esc$ has enough injectives. More precisely, for any $E\in\Esc$, the sequence $(\ref{InjResolE})$ gives an injective resolution of $E$ in $\Esc$. In particular, we have $\mathrm{id}_{\Esc}E\le 1$.
\item Any object in $\inj\Esc$ belongs to $\Esc\cap\Fac H$ for some $H\in(\proj A\bls)\cap(\inj A\bls)$. Moreover for any $I\in\inj \Esc$,  there is a short exact sequence in $\modd A\bls$
\[ 0\to P\ov{\iota}{\lra}H\to I\to0 \]
satisfying conditions in \cref{LemGentle_InjE} {\rm (2)}. Remark that this also induces the following short exact sequence in $\Esc$,
\[ 0\to P/\soc P\ov{\iota}{\lra}H/\soc H\to I\to0 \]
in which we have $P/\soc P\in\proj\Esc$, $H/\soc H\in(\proj\Esc)\cap(\inj\Esc)$, and $I\in\inj\Esc$.
\end{enumerate}
\end{proposition}

\begin{proof}
{\rm (1)} By \cref{RemGentle_E} {\rm (3)} and \cref{LemGentle_InjE} {\rm (3)}, we have short exact sequences in $\modd A\bls$
\[ 0\to P_1\ov{\iota_1}{\lra}P_0\to E\to0,\quad 0\to P_0\ov{\iota_2}{\lra}H\to I^1\to0 \]
satisfying $P_0,P_1\in\proj A\bls$, $H\in(\proj A\bls)\cap(\inj A\bls)$, and
\[ \iota_1(\soc P_1)=\soc P_0,\quad \iota_2(\soc P_0)=\soc H. \]
Thus we obtain the following commutative diagram in $\modd A\bls$ for some $I^0\in\modd A\bls$,
\[
\xy
(-6,28)*+{0}="-22";
(6,28)*+{0}="-24";
(-6,18)*+{P_1}="-12";
(6,18)*+{P_1}="-14";
(-18,6)*+{0}="0";
(-6,6)*+{P_0}="2";
(6,6)*+{H}="4";
(18,6)*+{I^1}="6";
(30,6)*+{0}="8";
(-18,-6)*+{0}="10";
(-6,-6)*+{E}="12";
(6,-6)*+{I^0}="14";
(18,-6)*+{I^1}="16";
(30,-6)*+{0}="18";
(-6,-16)*+{0}="22";
(6,-16)*+{0}="24";
{\ar "-22";"-12"};
{\ar "-24";"-14"};
{\ar@{=} "-12";"-14"};
{\ar_{\iota_1} "-12";"2"};
{\ar "-14";"4"};
{\ar "0";"2"};
{\ar_{\iota_2} "2";"4"};
{\ar "4";"6"};
{\ar "6";"8"};
{\ar "2";"12"};
{\ar "4";"14"};
{\ar@{=} "6";"16"};
{\ar "10";"12"};
{\ar "12";"14"};
{\ar "14";"16"};
{\ar "16";"18"};
{\ar "12";"22"};
{\ar "14";"24"};
{\ar@{}|\circlearrowright "-12";"4"};
{\ar@{}|\circlearrowright "2";"14"};
{\ar@{}|\circlearrowright "4";"16"};
\endxy
\]
in which rows and columns are exact. \Cref{RemGentle_E} {\rm (3)} shows $I^0,I^1\in\Esc$.

{\rm (2)} This follows from {\rm (1)}, since \cref{LemGentle_InjE} {\rm (2)} shows $I^0,I^1\in\inj\Esc$.

{\rm (3)} By {\rm (2)}, any $I\in\inj\Esc$ is a direct summand of an object in $\Esc\cap\Fac H$ for some $H\in(\proj A\bls)\cap(\inj A\bls)$. Since $\Esc\cap\Fac H\se\modd A\bls$ is closed under isomorphisms and direct summands, it follows $I\in\Esc\cap\Fac H$.

Take an epimorphism $H\to I$ in $\modd A\bls$. Since $\mathrm{pd}_{A\bls}I\le1$, it gives a short exact sequence $0\to P\ov{\iota}{\lra} H\to I\to0$ with $P\in\proj A\bls$, which necessarily satisfies $\iota(\soc P)=\soc H$ since $\Hom_{A\bls}(\soc H,I)=0$.
\end{proof}

\vspace{1cm}

Let $\ep_p\in A\bls$ denote the idempotent element corresponding to each $p\in Q_0\bls$. Define $\ep\in A\bls$ to be
\[ \ep=\sum_{p\in Q_0\bls\setminus Q_0}\ep_p, \]
namely, the sum of idempotents corresponding to blossom vertices. Denote the two-sided ideal $A\bls\ep A\bls\se A\bls$ by $\afr$.

\begin{definition}\label{DefGentle_FunctorG}
Denote the functor tensoring with $A=A\bls/\afr$ by
\[ G\co\modd A\bls\to\modd A\ ;\ M\mapsto M/\afr M. \]
\end{definition}

The following proposition can be seen as a consequence of \cref{prop:module as quotient by injectives} and \cref{cor:E is 0-Auslander}; we nonetheless give a representation-theoretic proof. 

\begin{proposition}\label{PropGentle_Equiv}
The functor $G$ induces an equivalence of categories $\ovl{G}\co\Esc/\inj\Esc\ov{\simeq}{\lra}\modd A$.
\end{proposition}

\begin{proof}
Since $G$ obviously satisfies $G(H)=0$ for any $H\in(\proj A\bls)\cap(\inj A\bls)$, \cref{PropGentle_InjE} {\rm (3)} and the right-exactness of $G$ shows $G(\inj\Esc)=0$. Thus $G$ naturally induces an additive functor $\ovl{G}\co\Esc/\inj\Esc\to\modd A$. Let us show $\ovl{G}$ is an equivalence of categories.

Let $E,F\in\Esc$ be any pair of objects. By \cref{PropGentle_InjE}, we have a short exact sequence
\[ 0\to E\to I^0\to I^1\to0 \]
in $\modd A\bls$, satisfying $I^0,I^1\in\inj\Esc$. Since $F$ does not contain any non-zero simple projective in its composition series, so does $\afr F$. This implies that $\topp(\afr F)\in\modd A\bls$ is a direct sum of simples corresponding to source blossom vertices, and thus $\afr F$ has a projective cover $H\to \afr F$ for some $H\in(\proj A\bls)\cap(\inj A\bls)$. Thus by \cref{RemGentle_E} {\rm (2)}, we have $\Ext^1_{A\bls}(I^1,\afr F)=\Ext^1_{A\bls}(E,\afr F)=0$, and obtain the following exact sequences.
\[ 0\to\Hom_{A\bls}(E,\afr F)\to\Hom_{A\bls}(E,F)\to\Hom_{A\bls}(E,F/\afr F)\to0, \]
\[ \Hom_{A\bls}(I^0,\afr F)\to\Hom_{A\bls}(E,\afr F)\to0. \]
They naturally induce the isomorphism
\[ (\Esc/\inj\Esc)(E,F)\cong\Hom_{A\bls}(E,F/\afr F)\cong\Hom_A(E/\afr E,F/\afr F), \]
which shows that $\ovl{G}$ is fully faithful.

It remains to show that $\ovl{G}$ is essentially surjective.
Remark that for any $v\in Q_0$, we have $G(\ep_v A\bls)=\ep_vA\bls/\soc(\ep_v A\bls)\cong\ep_vA$. Let $M\in\modd A$ be any object, and take its projective presentation
\[ Q_1\ov{q}{\lra}Q_0\to M\to0\quad(Q_0,Q_1\in\proj A) \]
in $\modd A$. By the remark above, we may find $P_0,P_1\in\proj A\bls$ and epimorphisms in $\modd A\bls$
\[ \pi_i\co P_i\to Q_i\quad(i=1,2) \]
which give $G(P_i)\cong Q_i$. Here, $Q_0,Q_1\in\modd A$ are regarded as objects in $\modd A\bls$ through the fully faithful embedding $\modd A\hookrightarrow\modd A\bls$ induced by the restriction of the action along the surjective homomorphism of $K$-algebras $A\bls\to A\bls/\afr\cong A$. Since $P_1\in\proj A\bls$, we may find $p\in\Hom_{A\bls}(P_1,P_0)$ which makes
\[
\xy
(-6,6)*+{P_1}="0";
(6,6)*+{P_0}="2";
(-6,-6)*+{Q_1}="4";
(6,-6)*+{Q_0}="6";
{\ar^{p} "0";"2"};
{\ar_{\pi_1} "0";"4"};
{\ar^{\pi_0} "2";"6"};
{\ar_{q} "4";"6"};
{\ar@{}|\circlearrowright "0";"6"};
\endxy
\]
commutative in $\modd A\bls$. Since $G$ is right exact, this gives $G(\Cok p)\cong M$ in $\modd A$.

We will modify $p\co P_1\to P_0$ to make the cokernel belong to $\Esc$.
Take an injective hull $P_1\ov{i}{\lra}H$, with $H\in(\proj A\bls)\cap(\inj A\bls)$. Then
\[ p\ppr=\begin{bmatrix}p\\ i\end{bmatrix}\co P_1\to P_0\oplus H \]
is a monomorphism in $\modd A\bls$. By adding the semisimple projective $S=\soc(P_0\oplus H)/p\ppr(\soc P_1)$ to $P_1$, we obtain a monomorphism
\[ \iota=\left[\begin{array}{cc}p&\ast\\ i&\ast\end{array}\right]\co P_1\oplus S\to P_0\oplus H \]
which satisfies $\iota(\soc(P_1\oplus S))=\soc(P_0\oplus H)$. Since we have $G(S)=G(H)=0$, we have $G(\Cok\iota)\cong G(\Cok p)\cong M$. By \cref{RemGentle_E} {\rm (3)}, we have $\Cok\iota\in\Esc$.
\end{proof}

\subsection{The category of walks}
\label{ssection:ExtricatWalk}

The exact category $\Esc$ of \cref{ssection: extricat for gentle} is convenient to define and study from a purely algebraic point of view, but is not well-adapted to the combinatorics of walks (i.e. maximal strings in the blossoming bound quiver).
In this section, we exchange this subcategory for an equivalent one, which will be more adapted to walks.

\begin{definition}
 The exact category of walks is defined as the full subcategory $\Fsc$ of $\modd A\bls$ whose objects $M$ satisfy $\Hom_{A\bls}(\soc P,\tau M)=0$ (equivalently $\Ext^1_{A\bls}(M,\soc P)=0$) for any $P\in\proj A\bls$, $M$ has projective dimension at most one, and has no simple projective summands.
\end{definition}

\begin{remark}
The subcategory $\Fsc$ is extension-closed in $\modd A\bls$, and is thus an exact category.
\end{remark}

\begin{remark}
The subcategory $\Fsc$ contains all band modules, since they are all of projective dimension one, and not supported at sinks or sources.
\end{remark}

\begin{example}
In \cref{figure: ARquiverModEx3}, taken from \cite{IyamaNakaokaPalu}, we highlight the indecomposable $A\bls$-modules that belong to the subcategories $\Fsc$ and $\Esc$.
Those in $\Fsc$ precisely coincide with the modules whose support contains two different blossoming vertices.
Moreover, we see that $\Fsc$ is stable under the action of $\tau_{A\bls}^2$.
Those two facts are instances of \cref{prop: in F iff walk} and \cref{cor: tau2-stable} respectively.
One can also notice that the full subquivers of this Aulsander--Reiten quiver given by $\Esc$ and $\Fsc$ are isomorphic. This is explained by the equivalence of exact categories proven in \cref{prop: equivalence F to E}.
\end{example}

\begin{landscape}
\begin{figure}[H]
\begin{tikzpicture}[scale=0.6, fl/.style={->,>=latex}]
\foreach \x in {0,1,...,4} { 
  \draw[fl] (2*\x-14,-2*\x+2) -- (2*\x-13,-2*\x+1) ;
  \draw[fl] (2*\x-8,-2*\x+4) -- (2*\x-7,-2*\x+3) ;
};
\foreach \x in {0,1,2,3} {
  \draw[fl] (2*\x-10,-2*\x+2) -- (2*\x-9,-2*\x+1) ;
  \draw[fl] (2*\x-10,2*\x-5) -- (2*\x-9,2*\x-4) ;
};
\foreach \x in {0,1,2} {
  \draw[fl] (2*\x-14,2*\x-1) -- (2*\x-13,2*\x) ;
  \draw[fl] (2*\x-6,2*\x-5) -- (2*\x-5,2*\x-4) ;
  \draw[fl] (2*\x-4,2*\x-7) -- (2*\x-3,2*\x-6) ;
};
\foreach \x in {0,1} {
  \draw[fl] (2*\x-10,2*\x-1) -- (2*\x-9,2*\x) ;
  \draw[fl] (2*\x-2,-2*\x+2) -- (2*\x-1,-2*\x+1) ;
};
\begin{scope}[yshift=1cm, xshift=-1cm, rotate=180, fl/.style={<-,>=latex}]
 \foreach \x in {0,1,...,4} { 
  \draw[fl] (2*\x-14,-2*\x+2) -- (2*\x-13,-2*\x+1) ;
  \draw[fl] (2*\x-8,-2*\x+4) -- (2*\x-7,-2*\x+3) ;
};
\foreach \x in {0,1,2,3} {
  \draw[fl] (2*\x-10,-2*\x+2) -- (2*\x-9,-2*\x+1) ;
  \draw[fl] (2*\x-10,2*\x-5) -- (2*\x-9,2*\x-4) ;
};
\foreach \x in {0,1,2} {
  \draw[fl] (2*\x-14,2*\x-1) -- (2*\x-13,2*\x) ;
  \draw[fl] (2*\x-6,2*\x-5) -- (2*\x-5,2*\x-4) ;
  \draw[fl] (2*\x-4,2*\x-7) -- (2*\x-3,2*\x-6) ;
};
\foreach \x in {0,1} {
  \draw[fl] (2*\x-10,2*\x-1) -- (2*\x-9,2*\x) ;
};
\end{scope}
\begin{scope}[xshift=-0.5cm, yshift=0.5cm]
\draw (0,0) node {$2$} ;
\draw (-14,-2) node {$h$} ;
\draw (14,2) node {$b$} ;
\draw (-14,2) node {$f$} ;
\draw (14,-2) node {$a$} ;
\draw (-12,0) node {$\bsm3\\f\;h\esm$} ;
\draw[green] (-12,0) +(-30pt,-30pt) rectangle +(30pt,30pt) ;
\draw (12,0) node {$\bsm a\;b\\1\esm$} ;
\draw[blue] (12,0) circle (.5cm) ;
\draw[green] (12,0) +(-30pt,-30pt) rectangle +(30pt,30pt) ;
\draw (-10,-6) node {$g$} ;
\draw (10,6) node {$e$} ;
\draw (-10,-2) node {$\bsm3\\f\esm$} ;
\draw (10,2) node {$\bsm a\\1\esm$} ;
\draw (-10,2) node {$\bsm3\\h\esm$} ;
\draw (10,-2) node {$\bsm b\\1\esm$} ;
\draw[blue] (10,-2) ellipse (.4cm and .6cm) ;
\draw (-8,-4) node {$\bsm 2\\3\;g\\f\;\phantom{\;g}\esm$} ;
\draw[green] (-8,-4) +(-30pt,-30pt) rectangle +(30pt,30pt) ;
\draw (8,4) node {$\bsm a\;\phantom{\;e}\\1\;e\\2\esm$} ;
\draw[blue] (8,4.05) ellipse (.55cm and .7cm) ;
\draw[green] (8.1,4) +(-30pt,-30pt) rectangle +(30pt,30pt) ;
\draw (-8,0) node {$3$} ;
\draw[blue] (-8,0) circle (.5cm) ;
\draw (8,0) node {$1$} ;
\draw (-8,4) node {$\bsm d\\3\\h\esm$} ;
\draw[green] (-8,4) +(-30pt,-30pt) rectangle +(30pt,30pt) ;
\draw (8,-4) node {$\bsm b\\1\\c\esm$} ;
\draw[green] (8,-4) +(-30pt,-30pt) rectangle +(30pt,30pt) ;
\draw (-6,-6) node {$\bsm2\\3\\f\esm$} ;
\draw (6,6) node {$\bsm a\\1\\2\esm$} ;
\draw[blue] (6,6) ellipse (.3cm and .7cm) ;
\draw (-6,-2) node {$\bsm2\\3\;g\esm$} ;
\draw (6,2) node {$\bsm1\;e\\2\esm$} ;
\draw[blue] (6,2) circle (.55cm) ;
\draw (-6,2) node {$\bsm d\\3\esm$} ;
\draw[blue] (-6,2) ellipse (.4cm and .6cm) ;
\draw (6,-2) node {$\bsm1\\c\esm$} ;
\draw (-4,-8) node {$\bsm e\\2\\3\\f\esm$} ;
\draw[green] (-4,-8) +(-30pt,-30pt) rectangle +(30pt,30pt) ;
\draw (4,8) node {$\bsm a\\1\\2\\g\esm$} ;
\draw[green] (4,8) +(-30pt,-30pt) rectangle +(30pt,30pt) ;
\draw (-4,-4) node {$\bsm 2\\3\esm$} ;
\draw[blue] (-4,-4) ellipse (.4cm and .6cm) ;
\draw (4,4) node {$\bsm 1\\2\esm$} ;
\draw[blue] (4,4) ellipse (.4cm and .6cm) ;
\draw (-4,0) node {$\bsm d\phantom{3}2\\\phantom{d\;}\,3\phantom{2}g\esm$} ;
\draw[green] (-4,0) +(-30pt,-30pt) rectangle +(30pt,30pt) ;
\draw (4,0) node {$\bsm \phantom{c\;}1\phantom{2}e\\c\phantom{1}2\esm$} ;
\draw[green] (4,0) +(-30pt,-30pt) rectangle +(30pt,30pt) ;
\draw (-2,-6) node {$\bsm e\\2\\3\esm$} ;
\draw[blue] (-2,-6) ellipse (.3cm and .7cm) ;
\draw (2,6) node {$\bsm 1\\2\\g\esm$} ;
\draw (-2,-2) node {$\bsm d\phantom{3}2\\3\esm$} ;
\draw[blue] (-2,-2) circle (.55cm) ;
\draw (2,2) node {$\bsm1\\c\;\,2\esm$} ;
\draw (-2,2) node {$\bsm 2\\g\esm$} ;
\draw (2,-2) node {$\bsm e\\2\esm$} ;
\draw (-2,6) node {$c$} ;
\draw (2,-6) node {$d$} ;
\draw (0,-4) node {$\bsm \phantom{d3}e\\d\phantom{3}2\\3\esm$} ;
\draw[blue] (0,-4) ellipse (.55cm and .7cm) ;
\draw[green] (0,-4) +(-30pt,-30pt) rectangle +(30pt,30pt) ;
\draw (0,4) node {$\bsm 1\\c\phantom{3}2\\\phantom{c3}g\esm$} ;
\draw[green] (0,4) +(-30pt,-30pt) rectangle +(30pt,30pt) ;
\draw[red] (14.3,2.5) -- (14.3,1.5) ;
\draw[red] (14.3,-1.5) -- (14.3,-2.5) ;
\draw[red] (12.7,0.5) -- (12.7,-0.5) ;
\draw[red] (8.8,4.5) -- (8.8,3.5) ;
\draw[red] (10.3,6.5) -- (10.3,5.5) ;
\draw[red] (8.3,-3.4) -- (8.3,-4.6) ;
\draw[red] (7.7,-3.4) -- (7.7,-4.6) ;
\draw[red] (3.7,8.7) -- (3.7,7.3) ;
\draw[red] (4.3,8.7) -- (4.3,7.3) ;
\draw[red] (2.3,-5.5) -- (2.3,-6.5) ;
\draw[red] (-0.4,4.7) -- (-0.4,3.3) ;
\begin{scope}[rotate=180]
\draw[red] (14.3,2.5) -- (14.3,1.5) ;
\draw[red] (14.3,-1.5) -- (14.3,-2.5) ;
\draw[red] (12.37,0.6) -- (12.37,-0.6) ;
\draw[red] (8.4,4.7) -- (8.4,3.3) ;
\draw[red] (10.3,6.5) -- (10.3,5.5) ;
\draw[red] (8.3,-3.4) -- (8.3,-4.6) ;
\draw[red] (7.7,-3.4) -- (7.7,-4.6) ;
\draw[red] (3.7,8.7) -- (3.7,7.3) ;
\draw[red] (4.3,8.7) -- (4.3,7.3) ;
\draw[red] (2.3,-5.5) -- (2.3,-6.5) ;
\draw[red] (-0.7,4.5) -- (-0.7,3.5) ;
\end{scope}
\draw[dashed, thick, blue] (3.4,4.45) -- (3.4,3.45) ;
\draw[dashed, thick, blue] (5.6,6.6) -- (5.6,5.3) ;
\draw[dashed, thick, blue] (6.4,6.6) -- (6.4,5.3) ;
\draw[dashed, thick, blue] (9.45,-1.6) -- (9.45,-2.7) ;
\draw[dashed, thick, blue] (9,4.5) -- (9,3.4) ;
\draw[dashed, thick, blue] (10.55,-1.6) -- (10.55,-2.7) ;
\draw[dashed, thick, blue] (12.9,0.5) -- (12.9,-0.6) ;
\begin{scope}[xshift=4cm, rotate=180]
\draw[dashed, thick, blue] (3.1,4.5) -- (3.1,3.4) ;
\draw[dashed, thick, blue] (5.6,6.6) -- (5.6,5.3) ;
\draw[dashed, thick, blue] (6.4,6.6) -- (6.4,5.3) ;
\draw[dashed, thick, blue] (9.45,-1.6) -- (9.45,-2.7) ;
\draw[dashed, thick, blue] (8.6,4.5) -- (8.6,3.4) ;
\draw[dashed, thick, blue] (10.55,-1.6) -- (10.55,-2.7) ;
\draw[dashed, thick, blue] (12.7,0.5) -- (12.7,-0.5) ;
\end{scope}
\end{scope}
\end{tikzpicture}

\caption{The Auslander--Reiten quiver of $\modd A\bls$: The indecomposable objects in $\mathscr{E}$ are circled in blue; the projective (resp. injective) modules are highlighted by a red vertical line to their left (resp. to their right) and similarly for the relative projectives (resp. injectives) in $\mathscr{E}$ by using dashed blue vertical lines. The indecomposable objects in $\mathscr{F}$ are squared in green. Projectives (resp. injectives) in $\mathscr{F}$ are the projective (resp. injective) $A\bls$-modules that belong to $\mathscr{F}$.}
\label{figure: ARquiverModEx3}
\end{figure}
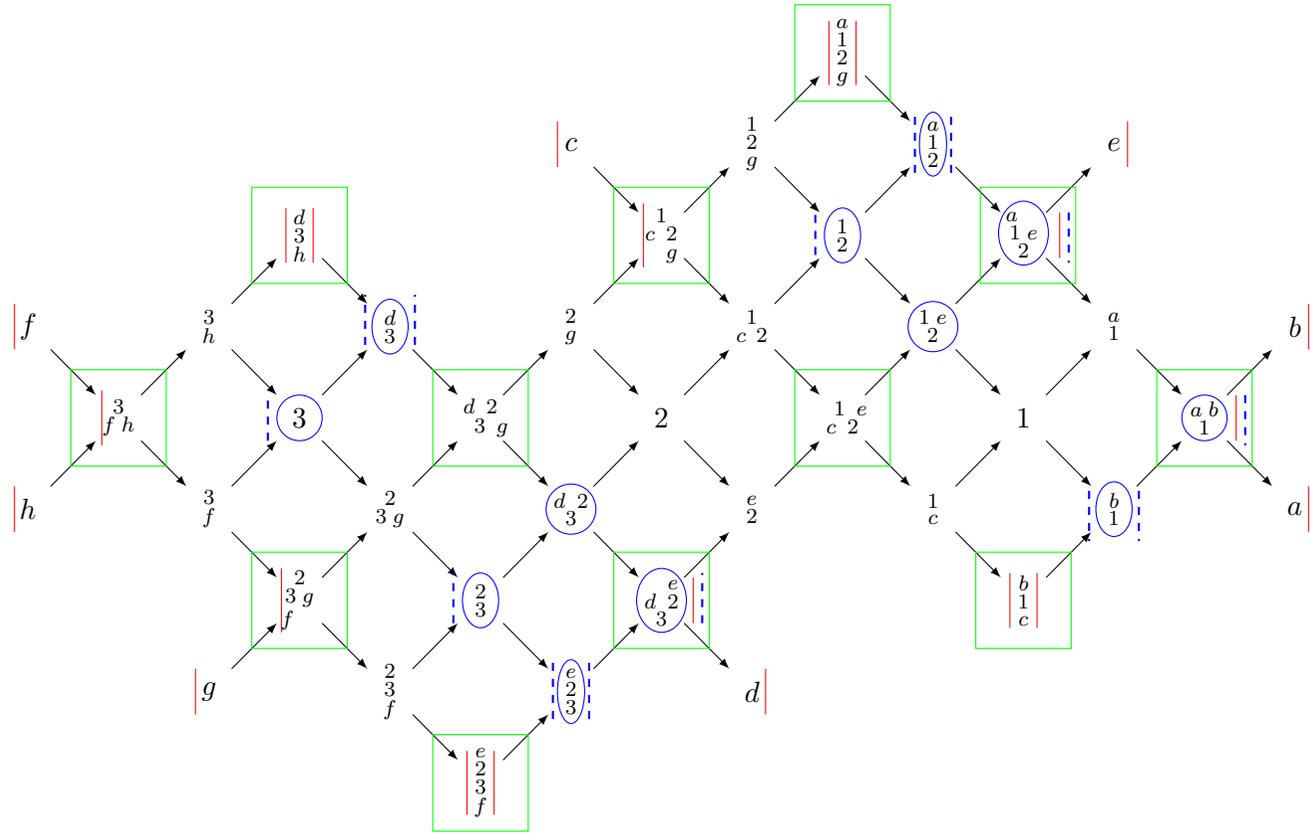

\restoregeometry

\end{landscape}


\begin{proposition}\label{prop: equivalence F to E}
 There is an equivalence of exact categories $F: \Fsc \xrightarrow{\simeq} \Esc$, sending an $A\bls$-module $M\in \Fsc$ to the cokernel $FM$ of $\eps.\soc M \to M$, where $\eps$ is the sum of the idempotents $e_t$, over all sinks $t\in Q_0\bls$.
 A quasi-inverse is constructed as follows: For an $A\bls$-module $N\in\Esc$, consider its minimal projective resolution $0\to S_1\oplus P_1\to P_0\to N$, where $S_1$ is semisimple and $P_0,P_1$ have no simple summands.
 Then the cokernel $G'N$ of $P_1\to P_0$ belongs to $\Fsc$.
\end{proposition}

\begin{proof}
The assignment $M\mapsto FM$ is functorial by construction.

$\bullet$ For any $M\in\Fsc$, $FM$ belongs to $\Esc$: The short exact sequence $\eps.\soc M \infl M\defl FM$ induces, for any simple projective $S$, an exact sequence
\[
0\to (S,\eps.\soc M) \xrightarrow{f} (S,M)\xrightarrow{g} (S,FM)\to \Ext^1_{A\bls}(S,\eps.\soc M)
\]
where we have written $(L,N)$ for $\Hom_{A\bls}(L,N)$.
By definition of $\eps$, the map $f$ is bijective, hence $g$ is zero.
Moreover, $S$ being projective, we have $\Ext^1_{A\bls}(S\eps.\soc M)=0$.
Thus $\Hom_{A\bls}(S,M)=0$.

Because $M$ has projective dimension at most one, there is a short exact sequence $P_1\infl P_0\defl M$, with $P_0,P_1$ projective $A\bls$-modules.
There is a diagram
\[\begin{tikzcd}
	{P_1} & {P_1} \\
	{P_1\oplus K} & {P_0} & FM \\
	K & M & FM
	\arrow[no head, from=1-1, to=1-2]
	\arrow[shift right=1, no head, from=1-1, to=1-2]
	\arrow[no head, from=2-3, to=3-3]
	\arrow[shift right=1, no head, from=2-3, to=3-3]
	\arrow[tail, from=1-1, to=2-1]
	\arrow[tail, from=1-2, to=2-2]
	\arrow[tail, from=2-1, to=2-2]
	\arrow[tail, from=3-1, to=3-2]
	\arrow[two heads, from=2-1, to=3-1]
	\arrow[two heads, from=2-2, to=3-2]
	\arrow[two heads, from=3-2, to=3-3]
	\arrow[two heads, from=2-2, to=2-3]
\end{tikzcd}\]
(where $K=\eps.\soc M$) showing that $FM$ also has projective dimension at most one.

$\bullet$ For any $N\in\Esc$, $G'N$ belongs to $\Fsc$: By definition, $G'N$ has projective dimension at most one and there is a diagram
\[\begin{tikzcd}
	{P_1} & {P_1} \\
	{S_1\oplus P_1} & {P_0} & N \\
	{S_1} & G'N & N
	\arrow[no head, from=1-1, to=1-2]
	\arrow[shift right=1, no head, from=1-1, to=1-2]
	\arrow[no head, from=2-3, to=3-3]
	\arrow[shift right=1, no head, from=2-3, to=3-3]
	\arrow[tail, from=1-1, to=2-1]
	\arrow[tail, from=1-2, to=2-2]
	\arrow[tail, from=2-1, to=2-2]
	\arrow[tail, from=3-1, to=3-2]
	\arrow[two heads, from=2-1, to=3-1]
	\arrow[two heads, from=2-2, to=3-2]
	\arrow[two heads, from=3-2, to=3-3]
	\arrow[two heads, from=2-2, to=2-3]
\end{tikzcd}\]
Because $P_0$ has no simple summands, $G'N$ has no simple projective summands.
Moreover, for any simple projective $A\bls$-module $S$, there is an induced exact sequence
\[
 (S_1,S)\xrightarrow{h} \Ext^1_{A\bls}(N,S) \to \Ext^1_{A\bls}(G'N,S) \to \Ext^1_{A\bls}(S_1,S)
\]
where the last term vanishes since $S_1$ is projective, and where the map $h$ is an isomorphism (because $\Hom_{A\bls}(P_1,S) = 0 = \Hom_{A\bls}(P_0,S)$).
This shows that $\Ext^1_{A\bls}(G'N,S)=0$ 

$\bullet$ The funtor $F$ is dense: Let $N\in\Esc$.
The short exact sequence $S_1\infl G'N\defl N$ induces an isomorphism $\soc G'N \cong \soc N\oplus S_1$, where $\soc N$ has no projective summands.
This implies $FG'N\cong N$.

$\bullet$ The functor $F$ is full: Let $M,N\in\Fsc$ and let  $FM \xrightarrow{f} FN$ be any morphism.
The composition $M\defl FM\to FN$ factors through $N\defl FN$, as in the diagram below, because $\Ext^1_{A\bls}(M,\eps.\soc N)=0$:
\[\begin{tikzcd}
	& M & FM \\
	{\varepsilon.\operatorname{soc} N} & N & FN
	\arrow[two heads, from=1-2, to=1-3]
	\arrow[two heads, from=2-2, to=2-3]
	\arrow["f", from=1-3, to=2-3]
	\arrow[tail, from=2-1, to=2-2]
	\arrow["g"', dashed, from=1-2, to=2-2]
\end{tikzcd}\]
We thus have $Fg=f$.

$\bullet$ The functor $F$ is faithful:
Given a commutative diagram
\[\begin{tikzcd}
	& M & FM \\
	{\varepsilon.\operatorname{soc} N} & N & FN
	\arrow[two heads, from=1-2, to=1-3]
	\arrow[two heads, from=2-2, to=2-3]
	\arrow["0", from=1-3, to=2-3]
	\arrow["{i_N}"', tail, from=2-1, to=2-2]
	\arrow["g", from=1-2, to=2-2]
	\arrow["h"', curve={height=-6pt}, dashed, from=1-2, to=2-1]
\end{tikzcd}\]
there is a morphism $h$ such that $i_Nh=g$.
However, $M$ has no simple projective summands.
This implies that $h$, hence $g$, vanishes.

$\bullet$ The functor $F$ is exact: Let $L\infl M\defl N$ be a short exact sequence with all three terms in $\Fsc$.
There is an induced split exact sequence $\eps.\soc L\to \eps.\soc M\to\eps.\soc N$.
By the snake lemma, the induced sequence between the cokernels $FL\to FM\to FN$ is short exact.
\end{proof}

\begin{corollary}
 An $A\bls$-module belongs to $\Fsc$ if and only if it is the cokernel of an injection between projective $A\bls$-modules having no simple summands.
\end{corollary}

\begin{proof}
 By \cref{prop: equivalence F to E}, any module in $\Fsc$ has projective dimension one and has a projective presentation by projective modules having no simple summands.
 Conversely, let $P_1\infl P_0\defl M$ be short exact with $P_0,P_1$ projective without any simple summand.
 Then $M$ is not simple projective and has projective dimension at most one. Moreover, for any simple projective $A\bls$-module $S$, $\Ext^1_{A\bls}(M,S)$ is given by the cokernel of the induced map $\Hom_{A\bls}(P_1,S)\to\Hom{A\bls}(P_0,S)$ whose domain and codomain both vanish.
 Hence $\Ext^1_{A\bls}(M,S) =0$ and $M$ belongs to $\Fsc$.
\end{proof}

\begin{lemma}
\label{lemma:Fdomdim}
For any $v\in Q_0$, there is a short exact sequence in $\modd A\bls$
\[
0 \to P_v \to I_a\oplus I_b \to I_v \to 0,
\]
where $\soc P_v = S_a\oplus S_b$.
\end{lemma}

\begin{proof}
This is immediate from string combinatorics:
\[\begin{tikzcd}[sep=small]
	&& v && {} & {} & {} \\
	0 & {} && {} & {\hspace{30pt}\oplus} & {} & {} && {} & 0 \\
	& a & {} && a & {b\hspace{25pt}} && v \\
	&&& b
	\arrow[no head, from=1-3, to=3-2]
	\arrow[no head, from=1-3, to=4-4]
	\arrow[from=2-1, to=2-2]
	\arrow[from=2-4, to=2-5]
	\arrow["v"{description}, shorten <=6pt, no head, from=1-5, to=3-5]
	\arrow["v"{description, pos=0.4}, shift right=5, no head, from=1-6, to=3-6]
	\arrow[from=2-6, to=2-7]
	\arrow[no head, from=3-8, to=1-7]
	\arrow[no head, from=3-8, to=2-9]
	\arrow[from=2-9, to=2-10]
\end{tikzcd}\] 
\end{proof}

\begin{corollary}
\label{cor:projFinjF}
The projective objects of $\Fsc$ are precisely the projective $A\bls$-modules having no simple summands, and the injective objects of $\Fsc$ are precisely the injective $A\bls$-modules having no simple summands.
\end{corollary}

\begin{proof}
 Any non-simple projective $A\bls$-module belongs to $\Fsc$ and is projective in $\Fsc$.
 Conversely, let $M\in\Fsc$ and let $P_1\infl P_0\defl M$ be short exact with $P_0,P_1$ projective $A\bls$-modules having no simple summands.
 If $M$ is projective in $\Fsc$, then this sequence splits and $M$ is a summand of $P_0$, hence a projective in $\modd A\bls$ having no simple summands.
 
 Let $I$ be indecomposable injective in $\modd A\bls$.
 If $I$ is also projective, then it belongs to $\Fsc$ and is thus injective in $\Fsc$.
 Otherwise, $I$ has simple socle not supported at a sink, hence $I$ belongs to $\Esc$ and $\eps.\soc I = 0$, showing that $I=GI$ belongs to $\Fsc$.
 
Conversely, let $M$ be injective in $\Fsc$.
 Then, by \cref{prop: equivalence F to E}, $FM$ is injective in $\Esc$, which implies that $FM$ belongs to $\Fac H$ by \cref{LemGentle_InjE,PropGentle_InjE}.
 From the short exact sequence $\eps.\soc M\infl M\defl FM$, one easily deduce that $FM$ belongs to $\Fac H$ if and only if $M$ belongs to $\Fac H$ (using that $M$ has no simple projective summands).
It follows that $M$ admits a projective resolution of the form $P_1\infl H_0\defl M$, where $P_1$ is projective having no simple summands and $H_0$ is projective-injective.
By \cref{lemma:Fdomdim}, there is a short exact sequence $P_1\infl H_1\defl I_1$ with $H_1$ projective-injective and $I_1$ injective having no simple summands.
Therefore, $M$ is a summand of $H_0\oplus I_1$, hence an injective with no simple summands.
\end{proof}

\begin{lemma}\label{lemma:pdim1}
 Let $\sigma$ be a string for $(Q\bls,I\bls)$.
 Then $M_\sigma$ is of projective dimension at most one if and only if one of the following three conditions is satisfied:
 \begin{enumerate}
  \item $M_\sigma$ is simple projective over $A\bls$.
  \item $M_\sigma$ is simple, supported at a non-blossoming vertex, and is projective as an $A$-module.
  \item $\sigma$ contains at least one arrow or inverse arrow and
  \begin{enumerate}
   \item starts at a blossom or is such that the target of the unique arrow $\al$, such that $\al^{-1}\sigma$ is a string, is a sink and
   \item ends in a blossom or is such that the target of the unique arrow $\beta$, such that $\sigma\beta$ is a string, is a sink.
  \end{enumerate}
 \end{enumerate}
\end{lemma}

\begin{proof}
 Let $\sigma$ be a string for $(Q\bls,I\bls)$.
 
$\bullet$ Assume first that $M_\sigma=S_i$ is simple.
If $i$ is a sink, then $S_i$ is projective, hence of projective dimension at most one.
Otherwise, its projective presentation is of the form $P_j\oplus P_k \to P_i \defl S_i$, where the quiver $Q\bls$ contains the subquiver $j\leftarrow i\to k$.
Because each non-blossoming vertex has two successors, the morphism $P_j\oplus P_k \to P_i$ is injective if and only if both $j$ and $k$ are sinks, if and only if $i$ is a sink in $Q$.

$\bullet$ Assume that $M_\sigma$ is not simple, and let $P_1\to P_0$ be a projective presentation of $M_\sigma$.

If the starting vertex of $\sigma$ is not a blossom, then $P_1$ contains $P_{t(\al)}$ as a summand, where $\al^{-1}\sigma$ is a string.
The restriction $P_{t(\al)}\to P_0$ of $P_1\to P_0$ is injective if and only if $t(\al)$ is not the source of two distinct arrows, if and only if $t(\al)$ is a blossoming vertex.
A similar argument applies at the endpoint of $\sigma$.
By string combinatorics (see for example \cite[Proposition 1.43]{PaluPilaudPlamondon-nonkissing}), we have that $P_1\to P_0$ is injective if and only if $P_{t(\al)}\oplus P_{t(\beta)}\to P_0$ is injective (with the convention that $P_{t(\al)}$ is zero if $\sigma$ starts at a blossoming vertex, and similarly for $P_{t(\beta)}$).
\end{proof}

The following proposition justifies the name of the exact category $\Fsc$.

\begin{proposition}\label{prop: in F iff walk}
 Let $\sigma$ be a string for the gentle bound quiver $(Q\bls,I\bls)$. Then the string module $M_\sigma$ belongs to $\Fsc$ if and only if $\sigma$ is a maximal string for $(Q\bls,I\bls)$, hence a walk for $(Q,I)$.
\end{proposition}

\begin{proof}
Let $\sigma$ be a string for $(Q\bls,I\bls)$.

$\bullet$ Assume that $M_\sigma$ belongs to $\Fsc$, i.e. that $M_\sigma$ is not simple projective over $A\bls$, has projective dimension at most one, and $\Ext_{A\bls}^1(M_\sigma,S)=0$ (equivalently $\Hom_{A\bls}(S,\tau M_\sigma)=0$) for each simple projective $A\bls$-module $S$.
Notice first that \cref{lemma:pdim1} implies that $M_\sigma$ cannot be simple.
Indeed, case (1) of \cref{lemma:pdim1} is excluded by definition of $\Fsc$ and case (2) is excluded because every vertex that is a sink in $Q$ is the source, in $Q\bls$, of an arrow with target a blossoming vertex $p$, thus has a non-split extension with the simple projective $S_p$.
Hence, $\sigma$ is of the form given in condition (3) of \cref{lemma:pdim1}.
The requirement on vanishing of extensions says that there cannot exist an arrow $\alpha$ as in (3)(a) nor an arrow $\beta$ as in (3)(b).
Thus $\sigma$, which has length at least one, starts and ends at a blossom, meaning that it is a walk on $(Q,I)$.

$\bullet$ Assume that $\sigma$ is a walk on $(Q,I)$. Then $\sigma$ starts at and ends in a blossoming vertex, and $M_\sigma$ is not a simple module.
Hence, by \cref{lemma:pdim1}, $M_\sigma$ has projective dimension at most one.
Moreover, because $\sigma$ is maximal, $\tau M_\sigma$ is the indecomposable module associated with the string $\sigma'$ obtained from $\sigma$ by removing a left and a right cohook.
In particular, $\sigma'$ is not supported on any blossom, showing that $\Hom_{A\bls}(S,\tau M_\sigma)=0$ for each simple projective $A\bls$-module $S$.
\end{proof}

Recall from \cite{PaluPilaudPlamondon-nonkissing} that:
\begin{itemize}
 \item A projective walk is a walk of the form $\alpha_1^{-1}\cdots\alpha_r^{-1}\beta_1\cdots\beta_s$, hence the string of an indecomposable projective $A\bls$-module at a non-blossoming vertex.
 \item A shifted projective walk is a walk of the form $\alpha_1\cdots\alpha_r\beta_1^{-1}\cdots\beta_s^{-1}$, hence the string of an indecomposable injective $A\bls$-module at a non-blossoming vertex.
 \item A straight walk is a walk that is a path in $Q\bls$, hence the string of a projective-injective $A\bls$-module.
\end{itemize}

\begin{proposition}
Let $\omega$ be a maximal string for $(Q\bls,I\bls)$.
The following properties holds:
\begin{enumerate} 
 \item The string module $M_\omega$ is projective and non-injective in $\Fsc$ if and only if $\omega$ is a projective walk.
 \item The string module $M_\omega$ is injective and non-projective in $\Fsc$ if and only if $\omega$ is a shifted projective walk.
 \item The string module $M_\omega$ is projective and injective in $\Fsc$ if and only if $\omega$ is a straight walk.
\end{enumerate}
\end{proposition}

\begin{proof}
This is a reformulation of \cref{cor:projFinjF}.
\end{proof}

Although the following corollary is also immediate from \cref{prop: equivalence F to E,cor:E is 0-Auslander}, we make it explicit for future references.

\begin{corollary}
\label{cor:F is 0-Auslander}
The exact category $\Fsc$ is 0-Auslander. 
\end{corollary}

\begin{lemma}
\label{lemma:costableHom}
 Let $\sigma,\sigma'$ be two strings for $(Q,I)$ and let $\omega,\omega'$ be the associated walks.
 Then, we have
 \[
 \Hom_A(M_\sigma,M_{\sigma'}) \cong \overline{\Fsc}(M_\omega,M_{\omega'}),
 \]
where $\overline{\Fsc}$ denotes the quotient by the ideal of morphisms factoring through an injective object.
\end{lemma}

\begin{proof}
Noting that $M_\sigma = GFM_\omega$, the statement follows from combining~\cref{prop: equivalence F to E,PropGentle_Equiv}
\end{proof}

\begin{lemma}
\label{lemma:comparisonOfTau}
 Let $\sigma,\sigma'$ be two strings for $(Q,I)$, and let $\omega,\omega'$ be the associated walks.
 Assume that $M_{\sigma'}=\tau_A M_\sigma$.
 Then $M_{\omega'}=\tau_{A\bls}^2M_\omega$.
\end{lemma}

\begin{proof}
 This is~\cite[Proposition 2.49]{PaluPilaudPlamondon-nonkissing}.
\end{proof}

\begin{corollary}
\label{cor: tau2-stable}
The subcategory $\Fsc$ of $\modd A\bls$ is $\tau^2$-stable.
\end{corollary}

\begin{proof}
This follows from \cref{prop: in F iff walk,lemma:comparisonOfTau}.
\end{proof}

\begin{proposition}
\label{prop:Fff}
The subcategory $\Fsc$ of $\modd A\bls$ is functorially finite.
\end{proposition}

\begin{proof}
Recall that middle terms of almost-split sequences in $\modd A\bls$ have precisely one or two indecomposable summands.
Meshes in the Auslander--Reiten quiver of $A\bls$ are thus either squares or triangles.
Moreover, arrows in this Auslander--Reiten quiver are given by adding or deleting one hook or cohook.
All this will be used implicitly in the proof.

Let $M\in\modd A\bls$ be indecomposable not in $\Fsc$.
We construct a right $\Fsc$-approximation of $M$, the construction of a left approximation being similar.
Note that $M$ cannot be a band module.
Thus it is a string module, associated with some string $\sigma$, at most one of whose endpoints is a blossom.
Let us consider three cases.

$\bullet$ If $\sigma$ is reduced to a sink.
Then $M$ is simple projective and its right approximation is 0.

$\bullet$ If $\sigma$ has length at most one, and has one endpoint a blossom.
Then at the other endpoint of $\sigma$ (which we may assume is its starting point, and which might be the blossom itself if $\sigma$ is reduced to a source) it is possible to add a cohook $c$.
Since every cohook in $(Q\bls,I\bls)$ ends in a blossom, $c\sigma$ is a walk, hence $M_{c\sigma}\in\Fsc$ by \cref{prop: in F iff walk}.
Moreover, in the ray of the Auslander--Reiten quiver of $A\bls$ ending in $M_\sigma$ and not containing $M_{c\sigma}$, all indecomposables are associated with strings starting at the same endpoint as $\sigma$, hence strings that are not walks.
The string module $M_{c\sigma}$ is thus the right approximation of $M_\sigma$.

$\bullet$ If $\sigma$ has length at most one, and does not contain any blossom.
Then if $N$ is an indecomposable $A\bls$-module sitting in one of the two rays of the Auslander--Reiten quiver of $A\bls$ ending in $M_\sigma$, then its string is obtained by a succession of addings of cohooks or removings of hooks, always on the same side of the string $\sigma$.
Hence one of its endpoints is an endpoint of $\sigma$, thus in $Q_0$.
Using \cref{prop: in F iff walk}, this implies that $N$ does not belong to $\Fsc$.
Therefore, any morphism from a module in $\Fsc$ to $M_\sigma$ factors through $\tau_{A\bls}M_\sigma$.
The string $\rho$ associated with $\tau_{A\bls}M_\sigma$ is obtained from $\sigma$ by adding a cohook on each side of $\sigma$.
It follows that $\tau_{A\bls}M_\sigma$ belongs to $\Fsc$ (using \cref{prop: in F iff walk} again) and that it is the right approximation of $M_\sigma$.
\end{proof}

\begin{corollary}
\label{cor:tau_F}
The category $\Fsc$ has almost-split extensions and, for any $M\in\Fsc$, we have $\tau_\Fsc M = \tau_{A\bls}^2M$.
\end{corollary}

\begin{proof}
First note that $\Fsc$ is $\tau_{A\bls}^2$-stable by \cref{cor: tau2-stable}, so that the statement makes sense.
Since $\Fsc$ is functorially finite (\cref{prop:Fff}), it has almost-split extensions.
Moreover, for any indecomposable module $M\in\Fsc$, its Auslander--Reiten translation $\tau_\Fsc M$ is a summand of the right $\Fsc$-approximation of $\tau_{A\bls}M$.
If $M$ is a band module, then $\tau_{A\bls}M = M$ and we have $\tau_\Fsc M = \tau_{A\bls}M = \tau_{A\bls}^2M$.
If $M$ is projective in $\modd A\bls$, then it is projective in $\Fsc$ by \cref{cor:projFinjF} and we have $\tau_\Fsc M = 0 = \tau_{A\bls}^2M$.
Otherwise, $M$ is a string module associated with a non-projective walk $\omega$, and the string associated with $\tau_{A\bls}M$ is obtained from that walk by removing a hook at each side of $\omega$.
We are thus in the third case of the proof of \cref{prop:Fff}, and the right $\Fsc$-approximation of $\tau_{A\bls}M$ is $\tau_{A\bls}^2M$, which is indecomposable.
Therefore $\tau_\Fsc M = \tau_{A\bls}^2M$.
\end{proof}

\begin{notation}
\label{notation:X_M}
If $M$ belongs to $\modd A$, there is a unique (up to isomorphism) object $X_M$ in $\Fsc$ such that
\begin{itemize}
 \item $GFX_M = M$, where $F$ and $G$ are the functors from \cref{PropGentle_Equiv,prop: equivalence F to E}, and
 \item $X_M$ has no non-zero injective summands.
\end{itemize}
In other words we have $X_{M_\sigma}=M_\omega\in \modd A\bls$, where $\omega$ is the walk associated with $\sigma$ if $\sigma$ is a string and where $\omega=\sigma$ if $\sigma$ is a band.
If $v\in Q_0$ and $P=P_v$ is the associated indecomposable projective $A$-module, let $X_{(0,P)}$ be the indecomposable injective $A\bls$-module $I_v\bls$ associated with vertex $v$.
Let $X_{(M,P_{v_1}\oplus\cdots\oplus P_{v_r})}=X_M\oplus X_{(0,P_{v_1})}\oplus\cdots\oplus X_{(0,P_{v_r})}$.

Following \cite{AdachiIyamaReiten}, for two pairs $(M,P),(N,Q)\in\modd A\times\proj A$ we write $(M,P)\perp_\tau(N,Q)$ if and only if $\Hom_A(M,\tau N)=0$ and $\Hom_A(Q,M)=0$.
\end{notation}

\begin{theorem}
\label{thm:kissings as extensions}
Assume that $(Q,I)$ is gentle, and let $A=KQ/I$.
\begin{enumerate}
 \item Let $M,N\in\modd A\times\proj A$.
Then $M \perp_\tau N$ if and only if $\Ebb(X_N,X_M)=0$ in $\Fsc$.
 \item Let $\omega, \omega'$ be two walks for $(Q,I)$.
Then $\E(M_{\omega'},M_\omega)\neq 0$ in $\Fsc$ if and only if the walk $\omega$ kisses the walk $\omega'$.
\end{enumerate}
\end{theorem}

\begin{proof}
The second part of the statement follows from (1) and~\cite[Theorem 2.46]{PaluPilaudPlamondon-nonkissing}.
Let us prove (1):
Recall from \cref{prop:Fff} that the category $\Fsc$ is functorially finite in $\modd A\bls$ and hence has Auslander--Reiten--Serre duality. Moreover, for any walk or band $\omega$, we have $\tau_{\Fsc}M_\omega = \tau_{A\bls}^2M_\omega$ by \cref{cor:tau_F}.

$\bullet$ Assume first that $\omega$ and $\omega'$ are the walks associated with some strings $\sigma,\sigma'$.
Let $\sigma''$ be such that $M_{\sigma''}=\tau_A M_{\sigma'}$, and let $\omega''$ be the walk associated with $\sigma''$.
By~\cref{lemma:costableHom,lemma:comparisonOfTau,cor:tau_F} and Auslander--Reiten--Serre duality, we have
\begin{eqnarray*}
 \Hom_A(M_\sigma,\tau_AM_{\sigma'}) & = & \Hom_A(M_\sigma,M_{\sigma''}) \\
 & \cong & \overline{\Fsc}(M_\omega,\tau_{A\bls}^2M_{\omega'}) \\
 & \cong & \overline{\Fsc}(M_\omega,\tau_{\Fsc}M_{\omega'}) \\
 & \cong & D\Ebb(M_{\omega'},M_\omega).
\end{eqnarray*}

$\bullet$ We note that the sequence of isomorphisms in the first point above is still valid if $\sigma$ or $\sigma'$ are bands, with the following convention: the walk associated with a band for $(A,I)$ is the same band, viewed as a band for $(Q\bls,I\bls)$. We also note that $\sigma''=\sigma'$ when $\sigma'$ is a band.

$\bullet$ Let $v\in Q_0$, let $\sigma$ be a string or a band for $(Q,I)$ and let $\omega$ be the associated walk (with the convention that $\omega=\sigma$ is $\sigma$ is a band).
Let $c,d$ be the blossoming vertices appearing in the top of the injective $A\bls$-module $I_v\bls$ (with $d=\emptyset$ if the top of $I_v\bls$ is simple), and consider the short exact sequence
\[
0 \to P_v\bls \to P_c\bls\oplus P_d\bls \to I_v\bls \to 0,
\]
of \cref{lemma:Fdomdim}, with the convention that $P_d=0$ if $d=\emptyset$.
There is an induced exact sequence
\[
\Hom_{A\bls}(P_c\bls\oplus P_d\bls,M_\omega) \xrightarrow{\vph} \Hom_{A\bls}(P_v\bls,M_\omega) \to \Ebb(I_v\bls,M_\omega) \to 0.
\]
If $\sigma$ is a band, then $\sigma=\omega$ does not contain any blossoming vertex and the first term of the exact sequence vanishes, showing that $\Hom_A(P_v,M\sigma)\cong\Hom_{A\bls}(P_v\bls,M_\omega) \cong \Ebb(I_v\bls,M_\omega)$.
If $\sigma$ is a string, then a basis for $\Hom_{A\bls}(P_v\bls,M_\omega)$ is parametrized by the occurrences of $v$ in $\omega$ and a basis for the image of $\vph$ is given by those morphisms that are parametrized by the occurrences of $v$ in $\omega$ that do not belong to $\sigma$ (indeed, only the endpoints of $\omega$ are blossoms).
Hence $\Ebb(I_v\bls,M_\omega) \cong \Hom_{A\bls}(P_v\bls,M_\omega) / \operatorname{Im}\vph$ is non-zero if and only if $v$ occurs in $\sigma$, if and only if $\Hom_A(P_v,M_\sigma) \neq 0$.

The remaining cases being trivially satisfied, this finishes the proof.
\end{proof}

By combining \cref{Th:mutation,cor:F is 0-Auslander,thm:kissings as extensions}, we obtain the following:

\begin{corollary}
For any finite-dimensional gentle algebra, the flip of non-kissing facets is an instance of mutations of maximal rigid objects in 0-Auslander extriangulated categories.
\end{corollary}


\section{Negative extensions 
in hereditary extriangulated categories}
\label{section:negative}

In \cite{GNP1}, we defined covariant and contravariant left derived functors of the functor $\Hom$ in essentially small extriangulated categories having enough projective, respectively enough injective morphsims. We proved that they agree under certain conditions; in this case, they form a bivariant $\delta-$functor satisfying natural universal properties, whose components we called \emph{balanced negative extensions}. We conjectured that one can define balanced negative extensions in each extriangulated category. In this section, we give a partial confirmation in the hereditary case. Namely, we define balanced negative extensions $\EbbPB^{\bullet}$ in hereditary categories satisfying certain restrictive conditions $(\rm{N+})$. Also, we construct another bivariant $\delta-$functor $\EbbE^{\bullet}$ for each $0-$Auslander category embedded into a triangulated category in a special way. These two functors do not agree in general, but we have nontrivial examples when they do agree. The latter $\delta-$functor a priori depends on the embedding, but for categories linear over fields and having finite-dimensional $\Hom-$ and $\Ebb-$spaces,
 dimensions of its components are independent of this choice. Moreover, we expect that whenever a $0-$Auslander category admits universal balanced negative extensions, these are given by $\EbbE^{\bullet}$, and so $\EbbE^{\bullet}$ is the correct replacement of $\EbbPB$ when the condition $(\rm{N+})$ is not satisfied. We will work in the following generality, which, in particular, allows for a relatively easy presentation of relevant results from \cite{GNP1}.

\begin{assumption} \label{assumption: hered_enough}
In this section, we always assume that $\CEs$ is an essentially small hereditary $R$-linear extriangulated category which has enough projective objects and enough injective objects.
\end{assumption}

 
By \cref{assumption: hered_enough}, to each object $X\in\Csc$ we can assign a pair of $\sfr$-triangles
\begin{equation}\label{sTri_Chosen_proj}
Q_X\ov{q_X}{\lra}P_X\ov{p_X}{\lra}X\ov{\om_X}{\dra},
\end{equation}
\begin{equation}\label{sTri_Chosen_inj}
X\ov{i_X}{\lra}I_X\ov{j_X}{\lra}J_X\ov{\iota_X}{\dra}
\end{equation}
in which  $P_X\in\Csc$ is projective and $I_X\in\Csc$ is injective. These $\sfr$-triangles are automatically {\it dominant} (resp. {\it codominant}) in the sense of \cite{GNP1}, meaning that the morhpisms
\[
\om_X\ush: \Csc(Q_X,-) \to \Ebb(X, -) \quad \mbox{and} \quad {\iota_X}\ssh: \Csc(-,J_X) \to \Ebb(-,X)
\]
are epimorphic.

Remark that since $\CEs$ is hereditary, $Q_X$ is projective and $J_X$ is injective, so these $\sfr$-triangles can be thought of as a projective resolution (resp. an injective coresolution) of $X$. For each $X$, we fix a pair of $\sfr$-triangles (\ref{sTri_Chosen_proj}) and (\ref{sTri_Chosen_inj}) and use these in our constructions. Note though that the connected sequences of functors we define will not depend on such a choice.



\subsection{Reminder on negative extensions}\label{Subsection_CD}

Recall that we defined in \cite{GNP1} a \defn{ covariant connected sequence of functors} as a pair $(\Fm=\{F^n\colon\cat\op\ti\cat\to\Mod R\}_{n\in [a, b]}, \del\ssh)$, where 

\begin{itemize}
\item $a \in \left\{- \infty\right\} \cup \Zbb$ and $b \in \Zbb \cup \left\{+ \infty \right\}$ such that $a \leq b$;
\item $\Fm=\{F^n\colon\cat\op\ti\cat\to\Mod R\}_{n\in [a, b]}$ a sequence of $R$-bilinear functors;
\item $\del\ssh=\del_{F,\sharp}^n\in\cat(F^n(-,C),F^{n+1}(-,A))$ is a collection of morphisms in  $\Mod\cat$, for any $\Ebb$-extension $\del\in\Ebb(C,A)$ and $n \in [a, b],$ natural with respect to morphisms of $\Ebb$-extensions.
\end{itemize}

\defn{Contravariant} connected sequences were defined dually. A \defn{bivariant connected sequence of functors} is a triplet $(\Fm,\del\ssh,\del\ush)$ where $(\Fm,\del\ssh)$ and  $(\Fm,\del\ush)$ form a covariant, respectively a contravariant connected sequence.
As in \cite{GNP1}, we use the shorthand notation $f\sas=F^n(-,f)$ and  $f\uas=F^n(f,-)$ for any morphism $f$ in $\cat$ and any $F^n$ in any connected sequence.

A  \defn{covariant partial $\delta$-functor} is a connected sequence $(\Fm=\{F^n\colon\cat\op\ti\cat\to\Mod R\}_{n\in [a, b]}, \del\ssh)$ such that
for any $\sfr$-triangle $A\ov{x}{\lra}B\ov{y}{\lra}C\ov{\del}{\dra}$, the complex
\begin{equation} \label{eqn:complex}
F^a(-,A)\ov{F^a(-,x)}{\lra}
F^a(-,B)\ov{F^a(-,y)}{\lra}
F^a(-,C)\ov{\del^a_{F,\sharp}}{\lra}F^{a+1}(-,A)\ov{F^{a+1}(-,x)}{\lra}\cdots
\ov{F^b(-,y)}{\lra}
F^b(-,C) 
\end{equation}
is exact except at the end-terms. It is a \defn{$\delta$-functor} if $a = -\infty, b = +\infty$. For a connected sequence $(\Fm, \del\ssh)$, we say that an $\sfr$-triangle is $(\Fm, \del\ssh)$\defn{-acyclic} if the associated unbounded complex (\ref{eqn:complex}) is exact. In other words, a sequence $(\Fm, \del\ssh)$ is a $\delta$-functor if and only if each $\sfr$-triangle in $\CEs$ is $(\Fm, \del\ssh)$-acyclic.

\defn{Contravariant} and \defn{bivariant} (also called \defn{balanced}) \defn{(partial) $\delta$-functors} are defined similarly.

 The bifunctors $\Hom$ and $\Ebb$ are half-exact in either argument and, therefore, are bivariant partial $\delta$-functors. Moreover, the sequence $(\Ebb^{\bullet}=\{\Ebb^n\}_{n\ge0}, \del\ssh, \del\ush)$ is also a bivariant partial $\delta$-functor, where we denote $\Ebb^0=\Hom_{\cat}$, see \cite[Example 5.8]{GNP1} (under \cref{assumption: hered_enough}, $\Ebb^n = 0$ for $n \geq 2$).

 We are interested in connected sequences $(\Fm, \del\ssh)$ such that $F^n=\Ebb^n$, and $\del\ssh\colon F^n(-,C)\to F^{n+1}(-,A)$ agrees with $\del\ssh\colon \Ebb^n(-,C)\to \Ebb^{n+1}(-,A)$ for each $n\ge 0, \, A,C\in\cat$. We then say that $(\Fm, \del\ssh)$ \defn{has} $F^n=\Ebb^n$ \defn{with the canonical connecting morphisms} for $n\ge0$. Similarly for contravariant and bivariant sequences. To prove that the connected sequences we construct later are $\delta$-functors, we will apply the following useful result.

\begin{lemma} \label{lem:enough_to_check_on_dominant}
Assume that $\CEs$ satisfies \cref{assumption: hered_enough}. Then a covariant $(\Fm, \del\ssh)$ having $F^n=\Ebb^n$ with the canonical connecting morphisms for $n\ge0$ is a $\delta$-functor if and only if for each $X$, some (equivalently, each) dominant $\sfr$-triangle (\ref{sTri_Chosen_proj}) is  $(\Fm, \del\ssh)$-acyclic. Dually, a contravariant $(\Fm, \del\ush)$ having $F^n=\Ebb^n$ with the canonical connecting morphisms for $n\ge0$ is a $\delta$-functor if and only if for each $X$, some (equivalently, each) dominant $\sfr$-triangle (\ref{sTri_Chosen_inj}) is  $(\Fm, \del\ush)$-acyclic. 
\end{lemma}

\begin{proof}
This is a direct application of \cite[Proposition 4.22]{GNP1} and its dual version for contravariant functors to our setting and generality, since the dominance of $\sfr$-triangle (\ref{sTri_Chosen_proj}) implies that it can be taken as $\theta$ in the condition \cite[Proposition 4.22.(2)]{GNP1} for any $\delta \in \Ebb(X,Y)$, for any $Y \in \Csc$.
\end{proof}

\begin{remark}
In fact, the same result holds (with the same proof) for any category $\CEs$ of finite (positive) global dimension and having {\it enough  projective morphisms} for $(\Fm, \del\ssh)$ (resp.  {\it enough  injective morphisms} for $(\Fm, \del\ush)$), in the sense of \cite{GNP1}.
\end{remark}
 
 A balanced partial $\delta$-functor $(F^n)_{n \ge -1}, \del\ssh, \del\ush)$ having $F^n=\Ebb^n$ with the canonical connecting morphisms for $n\ge0$ (more precisely, its part with $n \in [-1, 0]$) is the same as a \emph{negative first extension structure} on $\CEs$ in the sense of \cite{AdachiEnomotoTsukamoto}. We are interested in \emph{universal} negative extensions, where the natural universal property was introduced in \cite{GNP1}:

 A contravariant $\delta$-functor  $(F^n)_{n \in \mathbb{Z}}, \del\ssh)$ is \defn{universal} among $\delta$-functors having  $F^n=\Ebb^n$ with the canonical connecting morphisms for $n\ge0$ if for each $((G^n), n \in \mathbb{Z}; \del\ssh)$ having  $G^n=\Ebb^n$ with the canonical connecting morphisms for $n\ge0$, there exists a sequence of natural transformations
$\{\vp^n\colon G^n\to F^n\}_{n \in \mathbb{Z}}$ which satisfies the following conditions:
\begin{itemize}
\item[{\rm (i)}] $\vp^n=\id$ for any $n \ge 0$;
\item[{\rm (ii)}] For any $n \in \mathbb{Z}$, the following diagram
\[
\xy
(-12,7)*+{G^{n}(A,-)}="0";
(12,7)*+{G^{n+1}(C,-)}="2";
(-12,-7)*+{E^{n}(A,-)}="4";
(12,-7)*+{E^{n+1}(C,-)}="6";
{\ar^{\del^{\sharp, n}} "0";"2"};
{\ar_{\vp^{n}_{A,-}} "0";"4"};
{\ar^{\vp^{n+1}_{C,-}} "2";"6"};
{\ar_{\del^{\sharp, n}} "4";"6"};
{\ar@{}|\circlearrowright "0";"6"};
\endxy
\]
is commutative in $\cat\Mod$ for any $\del\in\Ebb(C,A)$.
\end{itemize}

If such universal $\delta$-functor  $(F^n)_{n \in \mathbb{Z}}, \del\ssh)$ exists, $F^{-n}$ is called the \defn{$n$-th left contravariant derived functor} of the bifunctor $\Hom$, or the \defn{universal contravariant negative $n$-th extension}.

Similar definitions apply for covariant and for bivariant (= balanced) $\delta$-functors. Note that \defn{balanced negative extensions} are considered \defn{universal} if they are universal among the balanced $\delta$-functors (having  $F^n~=~\Ebb^n$ with the canonical connecting morphisms for $n\ge0$), which is a weaker property than being a bivariant $\delta$-functor being universal both among covariant and contravariant $\delta$-functors.

In \cite[Section 5]{GNP1}, for categories having enough projective, respectively enough injective morphisms, we defined universal covariant $(\EbbI^{\bullet},\del\ssh)$ and contravariant $(\EbbII^{\bullet},\del\ush)$ $\delta$-functors having  $\EbbI^n~=~\EbbII^n~=~\Ebb^n$ with the canonical connecting morphisms for $n\ge0$. Their definition involves the dualities 
\[(-)^{\vee}\colon \cat\Mod\overset{\leftarrow}{\rightarrow}\Mod\cat \colon (-)^{\vee},\]
see \cite[Section 5.1]{GNP1}. In particular, we have 
\[
\EbbI^{-n}(-,X) = (\Ebb^n(X,-))^{\vee}, \quad 
\EbbII^{-n}(X,-) = (\Ebb^n(-,X))^{\vee}.
\]
In our special case - that is, under \cref{assumption: hered_enough}, - their components $\EbbI^{-n}, \EbbII^{-n}$ admit alternative easier recursive descriptions given in \cite[Remark 5.23]{GNP1}. Namely, we have 
\[ \EbbI^{-n}(-,X)\cong\Ker\big(\EbbI^{-(n-1)}(-,Q_X)\ov{(q_C)\sas}{\lra}\EbbI^{-(n-1)}(-,P_X)\big), \]
for any $\sfr$-triangle of the form (\ref{sTri_Chosen_proj}), for each $X\in\cat$. This is well-defined up to unique natural isomorphisms, independently of the choice of $\sfr$-triangles. Dually for $\EbbII^{\bullet}$, where one can use any $\sfr$-triangle (\ref{sTri_Chosen_inj}) for the first argument. In what follows we will not use the precise definition of connecting morphisms, only their existence; the reader is referred to \cite[Definition 5.3]{GNP1}.

\begin{remark}
\begin{itemize}
\item If $\CEs$ is exact, then $\EbbI^{-n}(X,Y)=\EbbII^{-n}(X, Y)=0$.
\item It is clear from the definition via dualities 
that for a category of (positive) global dimension $n$, we have $\EbbI^{-m} = \EbbII^{-m} = 0,$ for each $m > n$. In particular, under \cref{assumption: hered_enough}, we have $\EbbI^{-m} = \EbbII^{-m} = 0,$ for each $m \ge 2$ (this can also be seen directly from the alternatve description). 
\end{itemize}
\end{remark}

The universality of $(\EbbI^{\bullet},\del\ssh)$ and $(\EbbII^{\bullet},\del\ush)$ is proved in \cite[Proposition 5.19, Corollary 5.21, Corollary 5.22]{GNP1}. It means that $(\EbbI^{-n}), n > 0$ and $(\EbbII^{-n}), n > 0$ should be seen as negative extensions in the extriangulated category $\CEs$. In general, thus defined negative extensions are not balanced: given an arbitrary extriangulated category, we often do not have $\EbbI^{-n} \cong \EbbII^{-n}$ for all $n > 0$. 

A special case \cite[Corollary 5.39]{GNP1} says that if $\CEs$ satisfies \cref{assumption: hered_enough} and conditions {\rm (NI$+$)} and {\rm (NII+)} defined below, then $(\EbbI^n\cong\EbbII^n)_{n\in\Zbb}$ gives 
a universal bivariant $\delta$-functor having $\Ebb^n$ with the canonical connecting morphisms for $n\ge0$. 

\begin{condition}
\begin{itemize}
\item[{\rm (NI$+$)}] $((\om_Y)\ssh)^{-1}((\Ical)(X,Q_Y))=0$
holds in $\EbbI^{-1}(X,Y)$ for any $X,Y\in\cat$ and $\om_Y$ in any $\sfr$-triangle of the form (\ref{sTri_Chosen_proj}) for $Y$.
\item[{\rm (NII$+$)}] $((\iota_X)\ush)^{-1}((\Pcal)(J_X,Y))=0$ holds in $\EbbII^{-1}(X,Y)$ for any $X,Y\in\cat$ and $\iota_X$ in any $\sfr$-triangle of the form (\ref{sTri_Chosen_inj}) for $X$.
\end{itemize}
\end{condition}

\begin{example}
Conditions {\rm (NI$+$), (NII$+$)} are satisfied if $\CEs$ is an exact category or the category $K^{[-1, 0]}(\proj \Lambda)$, where $\Lambda$ is a finite-dimensional algebra, see \cite[Remark 5.32, Example 5.34]{GNP1}.
\end{example}

\begin{remark} \label{rem:NI_and_NII}
Certain weaker conditions (NI) and (NII) for categories which are $\Hom$-finite over a field $k$ were shown in \cite[Proposition 5.30]{GNP1} to be sufficient to ensure that \[\mbox{dim}_k \EbbI^{-n} = \mbox{dim}_k\EbbII^{-n}, n > 0.\] 
Under \cref{assumption: hered_enough}, we can reformulate these conditions as follows by using \cite[Corollary 5.9, Corollary 5.11]{GNP1}:
\begin{itemize}
\item[{\rm (NI)}] $\EbbI^{-1}(I, -) = 0$ holds for any injective object $I \in \Csc$.
\item[{\rm (NII)}]  $\EbbII^{-1}(-, P) = 0$ holds for any projective object $P \in \Csc$.
\end{itemize}
\end{remark}

\vspace{0.4cm}

 Consider the following properties a bivariant $\delta$-functor $(E^{\bullet}, \del\ssh,\del\ush)$ might have:

\begin{itemize}
\item[(U)] (Universality) $(E^{\bullet}, \del\ssh,\del\ush)$ is universal among $\delta$-functors having $F^n=\Ebb^n$ with the canonical connecting morphisms for $n\ge0$;
\item[(C)] (Compatibility) $(E^{\bullet}, \del\ssh,\del\ush)$ has $F^n=\Ebb^n$ with the canonical connecting morphisms for $n\ge0$ and
\[E^{-n}(I, -) \cong \EbbI^{-n}(I, -); \quad E^{-n}(-, P) \cong \EbbII^{-n}(-, P),\]
for all $I$ injective, $P$ projective, $n > 0$.
\end{itemize}

Note that under conditions (NI+) and (NII+) on $\CEs$, the $\delta$-functor 
$(\EbbI^n\cong\EbbII^n)_{n\in\Zbb}$ satisfies
both properties (U) and (C). This motivates the following questions (cf. \cite[Section 5.5]{GNP1}).


\begin{problem} \label{ProblemBivariant}
Assume that $\CEs$ satisfies \cref{assumption: hered_enough}. 
\begin{itemize}
\item[(a)] Does $\CEs$ admits a bivariant $\delta$-functor $(\Ebb^{\bullet}, \del\ssh,\del\ush)$ 
having property (C)?
\item[(b)] Does $\CEs$ admit $(\Ebb^{\bullet}, \del\ssh,\del\ush)$ satisfying property (U)?
\end{itemize}
\end{problem}

In fact, we expect that if 
a bivariant $\delta$-functor has property (U), it has property (C), 
so a positive answer \cref{ProblemBivariant}.(b) would also give the positive answer to \cref{ProblemBivariant}.(a). Moreover, we have the following easy observation.

\begin{lemma}
\label{prop:comparison:univ}
Under \cref{assumption: hered_enough}, if $\CEs$ admits a pair of bivariant $\delta$-functors $(\Ebb_{un}^{\bullet}, \del\ssh, \del\ush)$ having both properties  $\rm{(U)}$ and $\rm{(C)}$ and $(E^{\bullet}, \del\ssh, \del\ush)$ having property $\rm{(C)}$.
Then we have $(\Ebb_{un}^{\bullet}, \del\ssh, \del\ush) \overset\sim\to (\Ebb^{\bullet}, \del\ssh, \del\ush)$.
\end{lemma} 

\begin{proof}
The universality of $(\Ebb_{un}^{\bullet}, \del\ssh, \del\ush)$ implies the existence of a morphism of $\delta$-functors  $(\Ebb_{un}^{\bullet}, \del\ssh, \del\ush) \to (\Ebb^{\bullet}, \del\ssh, \del\ush)$. Since both $\delta$-functors have property (C), they agree on pairs $(I, -)$ and $(-,P)$. That the morphisms between components are in fact isomorphisms for each pair of objects then follows by the five lemma from the existence of (co)resolutions (\ref{sTri_Chosen_proj}) and (\ref{sTri_Chosen_inj}) for each object in $\Csc$.
\end{proof}

\begin{itemize}
\item As explained above, for $\CEs$ satisfying conditions (NI+) and (NII+), it follows from \cite{GNP1} that the answer to both \cref{ProblemBivariant}.(a) and (b) is positive and the $\delta$-functor in question is given by
$(\EbbI^n\cong\EbbII^n)_{n\in\Zbb}$ with suitable connecting morphisms.
\item In \cref{ssection:universal}, for $\CEs$ satisfying a certain condition (N+), we construct a bivariant $\delta$-functor $\EbbPB^{\bullet}$ giving the positive answer both to \cref{ProblemBivariant}.(a) and (b);
\item In \cref{ssection:computation}, for a class of $0$-Auslander categories embedded in triangulated categories, we construct bivariant $\delta$-functors $\EbbE^{\bullet}$ satisfying property (C).
\end{itemize}

In fact, under \cref{assumption: hered_enough}, property (C) defines the bifunctor $\Ebb^{-2}$ uniquely up to natural isomorphism and implies that 
\[E^{-m} = 0, \quad \forall m \ge 3,\] as will be explained in the next subsection. The nontrivial problems are then to find $\Ebb^{-1}$,  $\del\ssh$, and $\del\ush$ completing it to a $\delta$-functor, and to check if the result is universal.

\subsection{Definition of $\Ebb^{-2}$}

In \cite{GNP1}, for each $X, Y \in \Csc$ and each choice of $\sfr$-triangles (\ref{sTri_Chosen_proj}) and (\ref{sTri_Chosen_inj}), we studied the following two complexes:
\[ C_{\Irm}^{\bullet}\colon\quad \EbbI^{-1}(J_X,Y)\ov{j\uas}{\lra}\EbbI^{-1}(I_X,Y)\ov{i\uas}{\lra}\EbbI^{-1}(X,Y)\ov{\iota\ush\circ\om\ssh}{\lra}\Ebb(J_X,Q_Y) \]
and
\[ C_{\IIrm}^{\bullet}\colon\quad \EbbII^{-1}(X,Q_Y)\ov{q\sas}{\lra}\EbbII^{-1}(X,P_Y)\ov{p\sas}{\lra}\EbbII^{-1}(X,Y)\ov{\om\ssh\circ\iota\ush}{\lra}\Ebb(J_X,Q_Y). \]

We proved the following.

\begin{fact}\label{PropComparisonIandII}
\begin{itemize}
\item[(i)] \cite[Lemma 5.26]{GNP1} The rightmost morphisms in $C_{\Irm}^{\bullet}$ and $C_{\IIrm}^{\bullet}$, i.e. $\iota_X\ush\circ \om_{Y\sharp}$ and $\om_{Y\sharp}\circ\iota_X\ush$, have the same image.
\item[(ii)] \cite[Proposition 5.28]{GNP1} For any $X, Y \in \Csc$, the cohomologies of the complexes $C_{\Irm}^{\bullet}$ and $C_{\IIrm}^{\bullet}$ are isomorphic. 
\end{itemize}
\end{fact}

Note that \cref{PropComparisonIandII}.(ii) implies, in particular, that the morphisms $\EbbI^{-1}(J_X,Y)\ov{j_X\uas}{\lra}\EbbI^{-1}(I_X,Y)$ and $\EbbII^{-1}(X,Q_Y)\ov{q_{Y\ast}}{\lra}\EbbII^{-1}(X,P_Y)$ have isomorphic kernels. These kernels define the bifunctor $\Ebb^{-2}$. More precisely, by taking kernels in the commutative square
\[
\xy
(-14,6)*+{\Csc(J_X,Q_Y)}="0";
(14,6)*+{\Csc(J_X,P_Y)}="2";
(-14,-6)*+{\Csc(I_X,Q_Y)}="4";
(14,-6)*+{\Csc(I_X,P_Y)}="6";
{\ar^{q_{Y\ast}} "0";"2"};
{\ar_{j_X\uas} "0";"4"};
{\ar^{j_X\uas} "2";"6"};
{\ar_{q_{Y\ast}} "4";"6"};
{\ar@{}|\circlearrowright "0";"6"};
\endxy,
\]
we can define $\Ebb^{-2}(X,Y)\in\Mod R$, $\eta^{\Irm}_{X,Y}\in(\Mod R)(\Ebb^{-2}(X,Y),\EbbI^{-1}(J_X,Y))$ and 
\newline $\eta^{\IIrm}_{X,Y}\in(\Mod R)(\Ebb^{-2}(X,Y),\EbbII^{-1}(X,Q_Y))$ so that they fit in the following commutative diagram 
\begin{equation}\label{Comm_Added}
\xy
(-24,16)*+{0}="2";
(4,16)*+{0}="4";
(32,16)*+{0}="6";
(-42,6)*+{0}="10";
(-24,6)*+{\Ebb^{-2}(X,Y)}="12";
(4,6)*+{\EbbII^{-1}(X,Q_Y)}="14";
(32,6)*+{\EbbII^{-1}(X,P_Y)}="16";
(-42,-6)*+{0}="20";
(-24,-6)*+{\EbbI^{-1}(J_X,Y)}="22";
(4,-6)*+{\Csc(J_X,Q_Y)}="24";
(32,-6)*+{\Csc(J_X,P_Y)}="26";
(-42,-18)*+{0}="30";
(-24,-18)*+{\EbbI^{-1}(I_X,Y)}="32";
(4,-18)*+{\Csc(I_X,Q_Y)}="34";
(32,-18)*+{\Csc(I_X,P_Y)}="36";
{\ar_{} "2";"12"};
{\ar_{} "4";"14"};
{\ar_{} "6";"16"};
{\ar_{} "10";"12"};
{\ar^{\eta^{\IIrm}_{X,Y}} "12";"14"};
{\ar^{q_{Y\ast}} "14";"16"};
{\ar_{\eta^{\Irm}_{X,Y}} "12";"22"};
{\ar_{\iota_X\ush} "14";"24"};
{\ar^{\iota_X\ush} "16";"26"};
{\ar_{} "20";"22"};
{\ar_{\om_{Y\sharp}} "22";"24"};
{\ar_{q_{Y\ast}} "24";"26"};
{\ar_{j_X\uas} "22";"32"};
{\ar_{j_X\uas} "24";"34"};
{\ar^{j_X\uas} "26";"36"};
{\ar_{} "30";"32"};
{\ar_{\om_{Y\sharp}} "32";"34"};
{\ar_{q_{Y\ast}} "34";"36"};
{\ar@{}|\circlearrowright "12";"24"};
{\ar@{}|\circlearrowright "14";"26"};
{\ar@{}|\circlearrowright "22";"34"};
{\ar@{}|\circlearrowright "24";"36"};
\endxy
\end{equation}
whose rows and columns are exact.

\begin{lemma}  \label{LemUniv_heredII}
Take any $s\in\Csc(J_{X\ppr},J_X)$ and $t\in\Csc(Q_Y,Q_{Y\ppr})$ satisfying $x\sas\iota_{X\ppr}=s\uas\iota_X$ and $y\uas\om_{Y\ppr}=t\sas\om_Y$, which exist since $I_X\in\Csc$ is injective and $P_Y\in\Csc$ is projective. Then there exists a unique morphism $\Ebb^{-2}(x,y)\colon\Ebb^{-2}(X,Y)\to\Ebb^{-2}(X\ppr,Y\ppr)$ which makes
\begin{equation}\label{CommI_II}
\xy
(-16,6)*+{\Ebb^{-2}(X,Y)}="0";
(16,6)*+{\EbbI^{-1}(J_X,Y)}="2";
(-16,-6)*+{\Ebb^{-2}(X\ppr,Y\ppr)}="4";
(16,-6)*+{\EbbI^{-1}(J_{X\ppr},Y\ppr)}="6";
{\ar^{\eta^{\Irm}_{X,Y}} "0";"2"};
{\ar_{\Ebb^{-2}(x,y)} "0";"4"};
{\ar^{s\uas y\sas} "2";"6"};
{\ar_{\eta^{\Irm}_{X\ppr,Y\ppr}} "4";"6"};
{\ar@{}|\circlearrowright "0";"6"};
\endxy
\quad\text{and}\quad
\xy
(-16,6)*+{\Ebb^{-2}(X,Y)}="0";
(16,6)*+{\EbbII^{-1}(X,Q_Y)}="2";
(-16,-6)*+{\Ebb^{-2}(X\ppr,Y\ppr)}="4";
(16,-6)*+{\EbbII^{-1}(X\ppr,Q_{Y\ppr})}="6";
{\ar^{\eta^{\IIrm}_{X,Y}} "0";"2"};
{\ar_{\Ebb^{-2}(x,y)} "0";"4"};
{\ar^{x\uas t\sas} "2";"6"};
{\ar_{\eta^{\IIrm}_{X\ppr,Y\ppr}} "4";"6"};
{\ar@{}|\circlearrowright "0";"6"};
\endxy
\end{equation}
commutative simultaneously. 
Remark that either of these commutativity conditions determines $\Ebb^{-2}(x,y)$ uniquely. Since the left square does not involve $t$ while the right one does not involve $s$, this in particular implies that $\Ebb^{-2}(x,y)$ does not depend on the choices of $s$ and $t$.
\end{lemma}

\begin{proof}
 By the universality of kernel, there exists a unique morphism $\Ebb^{-2}(x,y)$ which makes the left square of (\ref{CommI_II}) commutative. It is not difficult to show that this morphism makes the right diagram also commutative, by using the commutativity of (\ref{Comm_Added}) for $X\ppr,Y\ppr$ and the monomorphicity of morphisms appearing in the left-upper square in it.
\end{proof}

\begin{proposition}
$\Ebb^{-2}\colon\cat\op\ti\Csc\to\Mod R$ is an $R$-bilinear functor. Moreover, $\Ebb^{-2}$ is uniquely determined up to natural isomorphism, independently of the choices $(\ref{sTri_Chosen_proj})$ and $(\ref{sTri_Chosen_inj})$.
\end{proposition}

\begin{proof}
This follows from \cref{LemUniv_heredII} by the usual, straightforward argument.
\end{proof}

\begin{remark}  \label{rem:E^-2_as_fiber_product}
Note that the top left commutative square in the diagram (\ref{Comm_Added}) 
means that we have
\begin{equation} \label{eq:E^-2_as_fiber_product}
\Ebb^{-2}(X, Y) = \EbbI^{-1}(J_X, Y) \times_{\Ccal(J_X, Q_Y)} \EbbII^{-1}(X, Q_Y).
\end{equation}
\end{remark}

\begin{proposition}
Assume that $\CEs$ satisfies \cref{assumption: hered_enough}. Let $(E^{\bullet}, \del\ssh,\del\ush)$ be a $\delta$-functor having $E^n=\Ebb^n$ with the canonical connecting morphisms for $n\ge0$. If it satisfies property (C), then  $E^{-2}$ is naturally isomorphic to $\Ebb^{-2}$ and 
\[E^{-m} = 0, \quad \forall m \ge 3.\]
\end{proposition}

\begin{proof}
This follows immediately from the long exact sequences for $E^n$ applied to $\sfr$-triangles of the form (\ref{sTri_Chosen_proj}) for the first argument or, equivalently, to $\sfr$-triangles of the form (\ref{sTri_Chosen_inj}) for the second argument. Indeed, they show that $E^{-2}$ agrees with the kernel of $\EbbI^{-1}(J_X,Y)\ov{j_X\uas}{\lra}\EbbI^{-1}(I_X,Y)$ and with the kernel $\EbbII^{-1}(X,Q_Y)\ov{q_{Y\ast}}{\lra}\EbbII^{-1}(X,P_Y)$, and that $E^{-n} = 0$ for $n \ge 3$.
\end{proof}

\subsection{Negative extensions via pullbacks}
\label{ssection:universal}

In this subsection, we complete $\Ebb^{-2}$ to a bivariant connected sequence of functors in an extriangulated category $\CEs$ satisfying \cref{assumption: hered_enough}.
Under certain condition, this sequence is a $\delta$-functor having property (U), so its components give the (universal) balanced negative extensions in $\CEs$ (see \cref{PropEquivAcycBiProl} below).



\begin{definition}\label{DefUniv_hered}
Let $X,Y\in\Csc$ be any pair of objects. Using the complexes $C_{\Irm}^{\bullet},C_{\mathrm{I\! I}}^{\bullet}$ in \cref{PropComparisonIandII}, we define 
$\EbbPB^{-1}(X,Y)\in\Mod R$ and $\wp^{\suf}_{X,Y}\in(\Mod R)(\EbbPB^{-1}(X,Y),\Ebb^{-1}_{\suf}(X,Y))$ $(\sufe)$ by a pullback diagram taken as below.
\[
\xy
(-14,6)*+{\EbbPB^{-1}(X,Y)}="0";
(14,6)*+{\EbbI^{-1}(X,Y)}="2";
(-14,-6)*+{\EbbII^{-1}(X,Y)}="4";
(14,-6)*+{\Ebb(J_X,Q_Y)}="6";
{\ar^{\wp^{\Irm}_{X,Y}} "0";"2"};
{\ar_{\wp^{\IIrm}_{X,Y}} "0";"4"};
{\ar^{\iota_X\ush\circ \om_{Y\sharp}} "2";"6"};
{\ar_{\om_{Y\sharp}\circ\iota_X\ush} "4";"6"};
{\ar@{}|\circlearrowright "0";"6"};
\endxy
\]
Remark that by \cref{PropComparisonIandII}.(i),  $\iota_X\ush\circ \om_{Y\sharp}$ and $\om_{Y\sharp}\circ\iota_X\ush$ appearing in the diagram have the same image, hence $\wp^{\suf}_{X,Y}$ is surjective for $\sufe$.
\end{definition}

\begin{lemma}\label{LemUniv_heredI}
Let $x\in\Csc(X\ppr,X)$ and $y\in\Csc(Y,Y\ppr)$ be any pair of morphisms. 
Then there exists a unique morphism $\EbbPB^{-1}(x,y)\colon\EbbPB^{-1}(X,Y)\to\EbbPB^{-1}(X\ppr,Y\ppr)$ which makes
\[
\xy
(-14,6)*+{\EbbPB^{-1}(X,Y)}="0";
(14,6)*+{\Ebb_{\suf}^{-1}(X,Y)}="2";
(-14,-6)*+{\EbbPB^{-1}(X\ppr,Y\ppr)}="4";
(14,-6)*+{\Ebb_{\suf}^{-1}(X\ppr,Y\ppr)}="6";
{\ar^{\wp^{\suf}_{X,Y}} "0";"2"};
{\ar_{\EbbPB^{-1}(x,y)} "0";"4"};
{\ar^{x\uas y\sas} "2";"6"};
{\ar_{\wp^{\suf}_{X\ppr,Y\ppr}} "4";"6"};
{\ar@{}|\circlearrowright "0";"6"};
\endxy
\]
commutative for $\sufe$.
%
%
%
%
\end{lemma}
\begin{proof}
This follows from the commutativity of
\[
\xy
(-32,6)*+{\EbbI^{-1}(X,Y)}="0";
(0,6)*+{\Ebb(J_X,Q_Y)}="2";
(32,6)*+{\EbbII^{-1}(X,Y)}="4";
(-32,-6)*+{\EbbI^{-1}(X\ppr,Y\ppr)}="10";
(0,-6)*+{\Ebb(J_{X\ppr},Q_{Y\ppr})}="12";
(32,-6)*+{\EbbII^{-1}(X\ppr,Y\ppr)}="14";
{\ar^{\iota_X\ush\circ\om_{Y\sharp}} "0";"2"};
{\ar_{\om_{Y\sharp}\circ\iota_X\ush} "4";"2"};
{\ar_{x\uas y\sas} "0";"10"};
{\ar^{s\uas t\sas} "2";"12"};
{\ar^{x\uas y\sas} "4";"14"};
{\ar_{\iota_{X\ppr}\ush\circ\om_{Y\ppr\sharp}} "10";"12"};
{\ar^{\om_{Y\ppr\sharp}\circ\iota_{X\ppr}\ush} "14";"12"};
{\ar@{}|\circlearrowright "0";"12"};
{\ar@{}|\circlearrowright "4";"12"};
\endxy
\]
and the universality of the pullback.

\end{proof}

We set $\EbbPB^{-2}: = \Ebb^{-2}$. In what follows, we abbreviate $\EbbPB^{-1}(x,y)$ and $\EbbPB^{-2}(x,y)$ simply to $x\uas y\sas$, as before.

\begin{proposition}\label{PropUniv_hered}
With the above definition, 
$\EbbPB^{-1}\colon\cat\op\ti\Csc\to\Mod R$ is an $R$-bilinear functor, and $\wp^{\suf}\colon\EbbPB^{-1}\to\Ebb_{\suf}^{-1}$ is a natural transformation for $\sufe$. Moreover, $\EbbPB^{-1}$ is uniquely determined up to natural isomorphism compatible with $\wp^{\suf}$ for $\sufe$, independently of the choices $(\ref{sTri_Chosen_proj})$ and $(\ref{sTri_Chosen_inj})$.
\end{proposition}
\begin{proof}
This follows from \cref{LemUniv_heredI} by the usual, straightforward argument.
\end{proof}

\begin{corollary}\label{CorUniv_hered}
The following holds.
\begin{enumerate}
\item If $I$ is injective, then $\wp^{\Irm}_{I,-}\colon\EbbPB^{-1}(I,-)\ov{\cong}{\lra}\EbbI^{-1}(I,-)$ is an isomorphism, $\EbbPB^{-2}(I,-)=0$.
\item If $P$ is projective, then $\wp^{\IIrm}_{-,P}\colon\EbbPB^{-1}(-,P)\ov{\cong}{\lra}\EbbII^{-1}(-,P)$ is an isomorphism, $\EbbPB^{-2}(-,P)=0$.
\end{enumerate}
\end{corollary}
\begin{proof}
{\rm (1)} follows from the independency shown in \cref{PropUniv_hered}, since we may choose an $\sfr$-triangle satisfying $J_X=0$ if $X$ is injective. Similarly for {\rm (2)}.
\end{proof}

\begin{lemma}\label{LemWellDef_delta-2}
Let $\del\in\Ebb(W,X)$ be any element. Remark that there exists $w\in\Csc(W,J_X)$ such that $w\uas\iota_X=\del$ since $I_X\in\Csc$ is injective. For the composition $\xi_{\del}$ of
\[ \EbbPB^{-2}(X,-)\ov{\eta^{\Irm}_{X,-}}{\lra}\EbbI^{-1}(J_X,-)\ov{w\uas}{\lra}\EbbI^{-1}(W,-), \]
the following holds.
\begin{enumerate}
\item $\xi_{\del}$ is independent of the choice of $w$. 
\item 
There exists a unique $\del\ush\colon\EbbPB^{-2}(X,-)\to\EbbPB^{-1}(W,-)$ which makes
\[
\xy
(0,10)*+{\EbbPB^{-2}(X,-)}="0";
(-28,-6)*+{\EbbII^{-1}(W,-)}="2";
(-18,4)*+{}="3";
(0,-6)*+{\EbbPB^{-1}(W,-)}="4";
(18,4)*+{}="5";
(28,-6)*+{\EbbI^{-1}(W,-)}="6";
{\ar_{0} "0";"2"};
{\ar_{\del\ush} "0";"4"};
{\ar^{\xi_{\del}} "0";"6"};
{\ar^{\wp^{\IIrm}_{W,-}} "4";"2"};
{\ar_{\wp^{\Irm}_{W,-}} "4";"6"};
{\ar@{}|\circlearrowright "3";"4"};
{\ar@{}|\circlearrowright "4";"5"};
\endxy
\]
commutative.
\end{enumerate}
\end{lemma}
\begin{proof}
{\rm (1)} If $w\ppr\in\Csc(W,J_X)$ also satisfies $w^{\prime\ast}\iota_X=\del$, then there exists $l\in\Csc(W,I_X)$ such that $w-w\ppr=j_X\circ l$ by the exactness of $\Csc(W,I_X)\ov{j_X\circ-}{\lra}\Csc(W,J_X)\ov{(\iota_X)\ssh}{\lra}\Ebb(W,X)$.
Then it follows
\[ w\uas\circ\eta^{\Irm}_{X,Y}=w^{\prime\ast}\circ\eta^{\Irm}_{X,Y}+l\uas\circ j_X\uas\circ\eta^{\Irm}_{X,Y}=w^{\prime\ast}\circ\eta^{\Irm}_{X,Y} \]
since the composition of $\EbbPB^{-2}(X,-)\ov{\eta^{\Irm}_{X,-}}{\lra}\EbbI^{-1}(J_X,-)\ov{j_X\uas}{\lra}\EbbI^{-1}(I_X,-)$ is zero by the definition of $\eta^{\Irm}_{X,-}$.

{\rm (2)} Let $Y\in\Csc$ be any object. It is enough to show that
\[
\xy
(-16,6)*+{\EbbPB^{-2}(X,Y)}="0";
(16,6)*+{\EbbI^{-1}(W,Y)}="2";
(-16,-6)*+{\EbbII^{-1}(W,Y)}="4";
(16,-6)*+{\Ebb(J_W,Q_Y)}="6";
{\ar^{(\xi_{\del})_Y} "0";"2"};
{\ar_{0} "0";"4"};
{\ar^{\iota_W\ush\circ\om_{Y\sharp}} "2";"6"};
{\ar_{\om_{Y\sharp}\circ \iota_W\ush} "4";"6"};
{\ar@{}|\circlearrowright "0";"6"};
\endxy
\]
is commutative, namely that the equation
\begin{equation}\label{EquationToBeHeld}
\iota_W\ush\circ\om_{Y\sharp}\circ(\xi_{\del})_Y=0
\end{equation}
holds. Indeed if this is shown, we obtain a unique morphism $(\del\ush)_Y$ which makes
\[
\xy
(0,10)*+{\EbbPB^{-2}(X,Y)}="0";
(-28,-6)*+{\EbbII^{-1}(W,Y)}="2";
(-18,4)*+{}="3";
(0,-6)*+{\EbbPB^{-1}(W,Y)}="4";
(18,4)*+{}="5";
(28,-6)*+{\EbbI^{-1}(W,Y)}="6";
{\ar_{0} "0";"2"};
{\ar_{(\del\ush)_Y} "0";"4"};
{\ar^{(\xi_{\del})_Y} "0";"6"};
{\ar^{\wp^{\IIrm}_{W,Y}} "4";"2"};
{\ar_{\wp^{\Irm}_{W,Y}} "4";"6"};
{\ar@{}|\circlearrowright "3";"4"};
{\ar@{}|\circlearrowright "4";"5"};
\endxy
\]
commutative, which is easily shown to be natural in $Y$, by using the universality of pullbacks.

It remains to show (\ref{EquationToBeHeld}). Since $\CEs$ is hereditary, $J_X$ is injective and we have $\Ebb(J_W,J_X)=0$. In particular $w\sas\iota_W=0$ holds for $w\in\Csc(W,J_X)$.
Thus (\ref{EquationToBeHeld}) follows from the commutativity of the following diagram.
\[
\xy
(-46,6)*+{\EbbPB^{-2}(X,Y)}="-2";
(-32,-3)*+{}="-1";
(-14,6)*+{\EbbI^{-1}(J_X,Y)}="0";
(14,6)*+{\Csc(J_X,Q_Y)}="2";
(-14,-6)*+{\EbbI^{-1}(W,Y)}="4";
(14,-6)*+{\Csc(W,Q_Y)}="6";
(32,2)*+{}="7";
(44,-6)*+{\Ebb(J_W,Q_Y)}="8";
{\ar^{\eta^{\Irm}_{X,Y}} "-2";"0"};
{\ar_{(\xi_{\del})_Y} "-2";"4"};
{\ar^{\om_{Y\sharp}} "0";"2"};
{\ar_{w\uas} "0";"4"};
{\ar^{w\uas} "2";"6"};
{\ar_{\om_{Y\sharp}} "4";"6"};
{\ar^{(w\sas\iota_W)\ush=0} "2";"8"};
{\ar_{\iota_W\ush} "6";"8"};
{\ar@{}|\circlearrowright "-1";"0"};
{\ar@{}|\circlearrowright "0";"6"};
{\ar@{}|\circlearrowright "6";"7"};
\endxy
\]
\end{proof}

\begin{remark}\label{RemWD-2}
In \cref{LemWellDef_delta-2}, in particular if $\del=\iota_X$, then we have $\xi_{\del}=\eta^{\Irm}_{X,-}$ since we can choose $w=\id$, and
\[
\xy
(0,10)*+{\EbbPB^{-2}(X,-)}="0";
(0,-6)*+{\EbbPB^{-1}(J_X,-)}="4";
(18,4)*+{}="5";
(28,-6)*+{\EbbI^{-1}(J_X,-)}="6";
%
{\ar_{\iota_X\ush} "0";"4"};
{\ar^{\eta^{\Irm}_{X,-}} "0";"6"};
%
{\ar_{\wp^{\Irm}_{J_X,-}} "4";"6"};
%
{\ar@{}|\circlearrowright "4";"5"};
\endxy
\]
becomes commutative.
\end{remark}

\begin{definition}\label{DefConnect_Univ_hered}
For any extension $\del\in\Ebb(C,A)$, we define as follows.
\begin{enumerate}
\item Define $\del\ush\colon\EbbPB^{-1}(A,-)\to\Csc(C,-)$ to be the composition of $\EbbPB^{-1}(A,-)\ov{\wp^{\IIrm}_{A,-}}{\lra}\EbbII^{-1}(A,-)\ov{\del\ush}{\lra}\Csc(C,-)$. 
\item Define $\del\ssh\colon\EbbPB^{-1}(-,C)\to\Csc(-,A)$ to be the composition of
$\EbbPB^{-1}(-,C)\ov{\wp^{\Irm}_{-,C}}{\lra}\EbbI^{-1}(-,C)\ov{\del\ssh}{\lra}\Csc(-,A)$. 
\item 
Define $\del\ush\colon\EbbPB^{-2}(A,-)\to\EbbPB^{-1}(C,-)$ to be the unique morphism which makes
\[
\xy
(0,10)*+{\EbbPB^{-2}(A,-)}="0";
(-28,-6)*+{\EbbII^{-1}(C,-)}="2";
(-18,4)*+{}="3";
(0,-6)*+{\EbbPB^{-1}(C,-)}="4";
(28,10)*+{\EbbI^{-1}(J_A,-)}="5";
(28,-6)*+{\EbbI^{-1}(C,-)}="6";
{\ar_{0} "0";"2"};
{\ar^{\eta^{\Irm}_{A,-}} "0";"5"};
{\ar_{\del\ush} "0";"4"};
{\ar^{w\uas} "5";"6"};
{\ar^{\wp^{\IIrm}_{C,-}} "4";"2"};
{\ar_{\wp^{\Irm}_{C,-}} "4";"6"};
{\ar@{}|\circlearrowright "3";"4"};
{\ar@{}|\circlearrowright "0";"6"};
\endxy
\]
commutative for any $w\in\Csc(C,J_A)$ such that $w\uas\iota_A=\del$, by using \cref{LemWellDef_delta-2}.
\item 
Define $\del\ssh\colon\EbbPB^{-2}(-,C)\to\EbbPB^{-1}(-,A)$ to be the unique morphism which makes
\[
\xy
(0,10)*+{\EbbPB^{-2}(-,C)}="0";
(-28,-6)*+{\EbbI^{-1}(-,A)}="2";
(-18,4)*+{}="3";
(0,-6)*+{\EbbPB^{-1}(-,A)}="4";
(28,10)*+{\EbbII^{-1}(-,Q_C)}="5";
(28,-6)*+{\EbbII^{-1}(-,A)}="6";
{\ar_{0} "0";"2"};
{\ar^{\eta^{\IIrm}_{-,C}} "0";"5"};
{\ar_{\del\ssh} "0";"4"};
{\ar^{t\sas} "5";"6"};
{\ar^{\wp^{\Irm}_{-,A}} "4";"2"};
{\ar_{\wp^{\IIrm}_{-,A}} "4";"6"};
{\ar@{}|\circlearrowright "3";"4"};
{\ar@{}|\circlearrowright "0";"6"};
\endxy
\]
commutative for any $t\in\Csc(Q_C,A)$ such that $t\sas\om_C=\del$, by using the dual of \cref{LemWellDef_delta-2}.
\end{enumerate}
\end{definition}

\begin{proposition}\label{PropExistBiProl}
The following forms a bivariant connected sequence of functors $(\EbbPB^{\bullet},\del\ssh,\del\ush)$. 
\begin{itemize}
\item $\EbbPB^n~=~\Ebb^n$ with the canonical connecting morphisms for $n\ge0$. 
\item $\EbbPB^{-1}$ is the one obtained in \cref{PropUniv_hered}. $\EbbPB^{-2} = \Ebb^{-2}$.  Also, we put $\EbbPB^{-n}=0$ for $n>2$.
\item For any $\del\in\Ebb(C,A)$, those $\del\ush,\del\ssh$ on $\EbbPB^{-1}$ and $\EbbPB^{-2}$ are those in \cref{DefConnect_Univ_hered}.
\end{itemize}
\end{proposition}
\begin{proof}
It remains to show the naturality of each of
\begin{itemize}
\item[{\rm (i)}] $\del\ush\colon\EbbPB^{-1}(A,-)\to\Csc(C,-)$, $\del\ush\colon\EbbPB^{-2}(A,-)\to\EbbPB^{-1}(C,-)$, and
\item[{\rm (ii)}]  $\del\ssh\colon\EbbPB^{-1}(-,C)\to\Csc(-,A)$, $\del\ssh\colon\EbbPB^{-2}(-,C)\to\EbbPB^{-1}(-,A)$
\end{itemize}
with respect to morphisms of $\Ebb$-extensions.
By duality, it is enough to show {\rm (i)}.
Naturality of $\del\ush\colon\EbbPB^{-1}(A,-)\to\Csc(C,-)$ follows from the fact that $(\EbbII^{\bullet},\del\ush)$ is a contravariant $\delta$-functor and the naturality of $\wp^{\IIrm}\colon\EbbPB^{-1}\to\EbbII^{-1}$.

Let us show the naturality of $\del\ush\colon\EbbPB^{-2}(A,-)\to\EbbPB^{-1}(C,-)$. Let $\del\in\Ebb(C,A),\del\ppr\in\Ebb(C\ppr,A\ppr)$ be any pair of extensions, and let $(a,c)\colon \del\to \del\ppr$ be any morphism. Take $s\in\Csc(C,J_A),s\ppr\in\Csc(C\ppr,J_{A\ppr}),u\in\Csc(J_A,J_{A\ppr})$ such that
\[ s\uas\iota_A=\del,\ \ s^{\prime\ast}\iota_{A\ppr}=\del\ppr,\ \ a\sas\iota_A=u\uas\iota_{A\ppr} \]
arbitrarily.
Then for $a\sas\del=c\uas\del\ppr\in\Ebb(C,A\ppr)$, we have $(s\ppr\circ c)\uas\iota_{A\ppr}=(u\circ s)\uas\iota_{A\ppr}=a\sas\del$. Thus by \cref{LemWellDef_delta-2} {\rm (1)}, compositions of
\begin{eqnarray*}
&\EbbPB^{-2}(A\ppr,-)\ov{\eta^{\Irm}_{A\ppr,-}}{\lra}\EbbI^{-1}(J_{A\ppr},-)\ov{(s\ppr\circ c)\uas}{\lra}\EbbI^{-1}(C,-),&\\
&\EbbPB^{-2}(A\ppr,-)\ov{\eta^{\Irm}_{A\ppr,-}}{\lra}\EbbI^{-1}(J_{A\ppr},-)\ov{(u\circ s)\uas}{\lra}\EbbI^{-1}(C,-)&
\end{eqnarray*}
are equal, which we denoted by $\xi_{a\sas\del}$. Then it follows that either of $\del\ush\circ a\uas,c\uas\circ\del^{\prime\sharp}\colon \EbbPB^{-2}(A\ppr,-)\to\EbbPB^{-1}(C,-)$ makes the following diagram commutative,
\[
\xy
(0,10)*+{\EbbPB^{-2}(A\ppr,-)}="0";
(-28,-6)*+{\EbbII^{-1}(C,-)}="2";
(-18,4)*+{}="3";
(0,-6)*+{\EbbPB^{-1}(C,-)}="4";
(18,4)*+{}="5";
(28,-6)*+{\EbbI^{-1}(C,-)}="6";
{\ar_{0} "0";"2"};
{\ar_{} "0";"4"};
{\ar^{\xi_{a\sas\del}} "0";"6"};
{\ar^{\wp^{\IIrm}_{C,-}} "4";"2"};
{\ar_{\wp^{\Irm}_{C,-}} "4";"6"};
{\ar@{}|\circlearrowright "3";"4"};
{\ar@{}|\circlearrowright "4";"5"};
\endxy
\]
hence $\del\ush\circ a\uas=(a\sas\del)\ush=c\uas\circ\del^{\prime\sharp}$ follows from (the uniqueness in) \cref{LemWellDef_delta-2} {\rm (2)}.
In particular the diagram
\[
\xy
(-14,6)*+{\EbbPB^{-2}(A\ppr,-)}="0";
(14,6)*+{\EbbPB^{-1}(C\ppr,-)}="2";
(-14,-6)*+{\EbbPB^{-2}(A,-)}="4";
(14,-6)*+{\EbbPB^{-1}(C,-)}="6";
{\ar^{\del^{\prime\sharp}} "0";"2"};
{\ar_{a\uas} "0";"4"};
{\ar^{c\uas} "2";"6"};
{\ar_{\del\ush} "4";"6"};
{\ar@{}|\circlearrowright "0";"6"};
\endxy
\]
is commutative, which shows the naturality as desired.
\end{proof}

Consider the following condition: 
\begin{itemize}
\item[{\rm (N$+$)}] For any $X,Y\in\Csc$, complex $C_{\Irm}^{\bullet}$ in \cref{PropComparisonIandII} associated with (\ref{sTri_Chosen_proj})  and (\ref{sTri_Chosen_inj})
\[ \EbbI^{-1}(J_X,Y)\ov{j_X\uas}{\lra}\EbbI^{-1}(I_X,Y)\ov{i_X\uas}{\lra}\EbbI^{-1}(X,Y)\ov{\iota_X\ush\circ\om_{Y\sharp}}{\lra}\Ebb(J_X,Q_Y) \]
is exact at the both middle terms.
\end{itemize}

\begin{remark}\label{RemN+SelfDual}
By \cref{PropComparisonIandII}.(ii), condition {\rm (N$+$)} is equivalent to asking the complex $C_{\IIrm}^{\bullet}$

\[ \EbbII^{-1}(X,Q_Y)\ov{q_{Y\ast}}{\lra}\EbbII^{-1}(X,P_Y)\ov{p_{Y\ast}}{\lra}\EbbII^{-1}(X,Y)\ov{\om_{Y\sharp}\circ\iota_X\ush}{\lra}\Ebb(J_X,Q_Y) \]
to be exact at the both middle terms
for any $X,Y\in\Csc$. In this sense, condition {\rm (N$+$)} is self-dual.
\end{remark}

\begin{theorem}\label{PropEquivAcycBiProl}
The following are equivalent.
\begin{enumerate}
\item $\CEs$ satisfies {\rm (N$+$)}.
\item $\om_X$ is $(\EbbPB^{\bullet},\del\ssh)$-acyclic for any $X\in\Csc$, where $\om_X$ is the extension chosen in (\ref{sTri_Chosen_proj}). 
\item $\iota_X$ is $(\EbbPB^{\bullet},\del\ush)$-acyclic for any $X\in\Csc$, where $\iota_X$ is the extension chosen in (\ref{sTri_Chosen_inj}).
\item $(\EbbPB^{\bullet},\del\ssh)$ is a covariant $\delta$-functor.
\item $(\EbbPB^{\bullet},\del\ush)$ is a contravariant $\delta$-functors.
\item $(\EbbPB^{\bullet},\del\ssh,\del\ush)$ is a bivariant $\delta$-functor.
\end{enumerate}
Moreover, if one of these equivalent conditions is satisfied, then $(\EbbPB^{\bullet},\del\ssh,\del\ush)$ satisfies properties (U) and (C). 
\end{theorem}
\begin{proof}
By definition, {\rm (6)} means that both of {\rm (4)} and {\rm (5)} hold.
Equivalence {\rm (2)} $\EQ$ {\rm (4)} follows from \cite[Proposition 4.22]{GNP1}. Dually for {\rm (3)} $\EQ$ {\rm (5)}.

Since {\rm (1)} $\EQ$ {\rm (3)} can be shown dually, it suffices to show {\rm (1)} $\EQ$ {\rm (2)}. 
Let $X,Y\in\Csc$ be any pair of objects. By the definitions so far,
\begin{equation}\label{DiagUnivBivProl1}
\xy
(-56,8)*+{\EbbPB^{-2}(X,Y)}="0";
(-28,8)*+{\EbbPB^{-1}(J_X,Y)}="2";
(0,8)*+{\EbbPB^{-1}(I_X,Y)}="4";
(28,8)*+{\EbbPB^{-1}(X,Y)}="6";
(74,8)*+{\Csc(J_X,Y)}="8";
(51,2)*+{\EbbII^{-1}(X,Y)}="a";
(51,12)*+{}="b";
(51,-10)*+{}="c";
(-56,-8)*+{\EbbPB^{-2}(X,Y)}="10";
(-28,-8)*+{\EbbI^{-1}(J_X,Y)}="12";
(0,-8)*+{\EbbI^{-1}(I_X,Y)}="14";
(28,-8)*+{\EbbI^{-1}(X,Y)}="16";
(74,-8)*+{\Ebb(J_X,Q_Y)}="18";
{\ar^{\iota_X\ush} "0";"2"};
{\ar^{j_X\uas} "2";"4"};
{\ar^{i_X\uas} "4";"6"};
{\ar@/^0.80pc/^{\iota_X\ush} "6";"8"};
{\ar_(0.44){\wp^{\IIrm}_{X,Y}} "6";"a"};
{\ar_(0.56){\iota_X\ush} "a";"8"};
{\ar@{=} "0";"10"};
{\ar^{\wp^{\Irm}_{J_X,Y}}_{\cong} "2";"12"};
{\ar^{\wp^{\Irm}_{I_X,Y}}_{\cong} "4";"14"};
{\ar_{\wp^{\Irm}_{X,Y}} "6";"16"};
{\ar^{\om_{Y\sharp}} "8";"18"};
{\ar_{\eta^{\Irm}_{X,Y}} "10";"12"};
{\ar_{j_X\uas} "12";"14"};
{\ar_{i_X\uas} "14";"16"};
{\ar_{\iota_X\ush\circ\om_{Y\sharp}} "16";"18"};
{\ar@{}|\circlearrowright "0";"12"};
{\ar@{}|\circlearrowright "2";"14"};
{\ar@{}|\circlearrowright "4";"16"};
{\ar@{}|\circlearrowright "a";"b"};
{\ar@{}|\circlearrowright "a";"c"};
\endxy
\end{equation}
is commutative. Remark that the sequence
\[ 0\to\EbbPB^{-2}(X,Y)\ov{\eta^{\Irm}_{X,Y}}{\lra}\EbbI^{-1}(J_X,Y)\ov{j_X\uas}{\lra}\EbbI^{-1}(I_X,Y) \]
is always exact by the definition of $\EbbPB^{-2}(X,Y)$. 

In the rightmost square, since $\iota\ush_X\colon\EbbII^{-1}(X,Y)\to\Csc(J_X,Y)$ is monomorphic and since
\[
\xy
(-14,6)*+{\EbbPB^{-1}(X,Y)}="0";
(14,6)*+{\EbbII^{-1}(X,Y)}="2";
(-14,-6)*+{\EbbI^{-1}(X,Y)}="4";
(14,-6)*+{\Csc(J_X,Q_Y)}="6";
{\ar^{\wp^{\IIrm}_{X,Y}} "0";"2"};
{\ar_{\wp^{\Irm}_{X,Y}} "0";"4"};
{\ar^{\om_{Y\sharp}\circ\iota_X\ush} "2";"6"};
{\ar_{\iota_X\ush\circ\om_{Y\sharp}} "4";"6"};
{\ar@{}|\circlearrowright "0";"6"};
\endxy
\]
is a pullback by the definition of $\EbbPB^{-1}(X,Y)$, we obtain an isomorphism $\kap$ for kernels which makes
\begin{equation}\label{DiagUnivBivProl2}
\xy
(-20,6)*+{0}="2";
(0,6)*+{\Ker(\iota_X\ush)}="4";
(28,6)*+{\EbbPB^{-1}(X,Y)}="6";
(56,6)*+{\Csc(J_X,Y)}="8";
(-20,-6)*+{0}="12";
(0,-6)*+{\Ker(\iota_X\ush\circ\om_{Y\sharp})}="14";
(28,-6)*+{\EbbI^{-1}(X,Y)}="16";
(56,-6)*+{\Ebb(J_X,Q_Y)}="18";
{\ar^{} "2";"4"};
{\ar^{} "4";"6"};
{\ar^{\iota_X\ush} "6";"8"};
{\ar^{\cong}_{\kap} "4";"14"};
{\ar^{\wp^{\Irm}_{X,Y}} "6";"16"};
{\ar^{\om_{Y\sharp}} "8";"18"};
{\ar_{} "12";"14"};
{\ar_{} "14";"16"};
{\ar_{\iota_X\ush\circ\om_{Y\sharp}} "16";"18"};
{\ar@{}|\circlearrowright "4";"16"};
{\ar@{}|\circlearrowright "6";"18"};
\endxy
\end{equation}
commutative. Thus $(\ref{DiagUnivBivProl1})$ and $(\ref{DiagUnivBivProl2})$ give an isomorphism of sequences
\[
\xy
(-63,6)*+{0}="-2";
(-44,6)*+{\EbbPB^{-2}(X,Y)}="0";
(-16,6)*+{\EbbPB^{-1}(J_X,Y)}="2";
(12,6)*+{\EbbPB^{-1}(I_X,Y)}="4";
(40,6)*+{\Ker(\iota_X\ush)}="6";
(60,6)*+{0}="8";
(-63,-6)*+{0}="-12";
(-44,-6)*+{\EbbPB^{-2}(X,Y)}="10";
(-16,-6)*+{\EbbI^{-1}(J_X,Y)}="12";
(12,-6)*+{\EbbI^{-1}(I_X,Y)}="14";
(40,-6)*+{\Ker(\iota_X\ush\circ\om_{Y\sharp})}="16";
(60,-6)*+{0}="18";
{\ar^{} "-2";"0"};
{\ar^{\iota_X\ush} "0";"2"};
{\ar^{j_X\uas} "2";"4"};
{\ar^{i_X\uas} "4";"6"};
{\ar_{} "6";"8"};
{\ar@{=} "0";"10"};
{\ar^{\wp^{\Irm}_{J_X,Y}}_{\cong} "2";"12"};
{\ar^{\wp^{\Irm}_{I_X,Y}}_{\cong} "4";"14"};
{\ar^{\kap}_{\cong} "6";"16"};
{\ar^{} "-12";"10"};
{\ar_{\eta^{\Irm}_{X,Y}} "10";"12"};
{\ar_{j_X\uas} "12";"14"};
{\ar_{i_X\uas} "14";"16"};
{\ar_{} "16";"18"};
{\ar@{}|\circlearrowright "0";"12"};
{\ar@{}|\circlearrowright "2";"14"};
{\ar@{}|\circlearrowright "4";"16"};
{\ar@{}|\circlearrowright "6";"18"};
\endxy
\]
in which we denoted the morphisms induced by $i_X\uas$ to kernels by the same symbols.
Consequently, its lower row is exact if and only if the upper row is exact. This is nothing but {\rm (1)} $\EQ$ {\rm (2)}.

Property (C) holds by \cref{CorUniv_hered}.

It remains to show the universality. Let $(F^{\bullet},\del\ssh,\del\ush)$ be any bivariant $\delta$-functor having  $F^n~=~\Ebb^n$ with the canonical connecting morphisms for $n\ge0$. By \cite[Corollary 5.21, Corollary 5.22]{GNP1}, there exists a unique morphism
$\vp^{\bullet}\colon (F^{\bullet},\del\ssh)\to(\EbbI^{\bullet},\del\ssh)$ of covariant $\delta$-functors and a unique morphism $\psi^{\bullet}\colon (F^{\bullet},\del\ush)\to(\EbbII^{\bullet},\del\ush)$ of contravariant $\delta$-functors.

\medskip
\noindent\und{Construction of $\lam^{(-1)}\colon F^{-1}\to \EbbPB^{-1}$.}

For any $X,Y\in\Csc$, the commutativity of
\[
\xy
(-14,6)*+{F^{-1}(X,Y)}="0";
(14,6)*+{\Csc(X,Q_Y)}="2";
(-14,-6)*+{\Csc(J_X,Y)}="4";
(14,-6)*+{\Ebb(J_X,Q_Y)}="6";
{\ar^{\om_{Y\sharp}} "0";"2"};
{\ar_{\iota_X\ush} "0";"4"};
{\ar^{\iota_X\ush} "2";"6"};
{\ar_{\om_{Y\sharp}} "4";"6"};
{\ar@{}|\circlearrowright "0";"6"};
\endxy
\]
implies that of 
\[
\xy
(-14,6)*+{F^{-1}(X,Y)}="0";
(14,6)*+{\EbbI^{-1}(X,Y)}="2";
(-14,-6)*+{\EbbII^{-1}(X,Y)}="4";
(14,-6)*+{\Ebb(J_X,Q_Y)}="6";
{\ar^{\vp_{X,Y}^{(-1)}} "0";"2"};
{\ar_{\psi_{X,Y}^{(-1)}} "0";"4"};
{\ar^{\iota_X\ush\circ\om_{Y\sharp}} "2";"6"};
{\ar_{\om_{Y\sharp}\circ\iota_X\ush} "4";"6"};
{\ar@{}|\circlearrowright "0";"6"};
\endxy.
\]
Thus we obtain a unique homomorphism $\lam^{(-1)}_{X,Y}\colon F^{-1}(X,Y)\to\EbbPB^{-1}(X,Y)$ which makes
\[
\xy
(0,10)*+{F^{-1}(X,Y)}="0";
(-28,-6)*+{\EbbI^{-1}(X,Y)}="2";
(-18,4)*+{}="3";
(0,-6)*+{\EbbPB^{-1}(X,Y)}="4";
(18,4)*+{}="5";
(28,-6)*+{\EbbII^{-1}(X,Y)}="6";
{\ar_{\vp^{(-1)}_{X,Y}} "0";"2"};
{\ar_{\lam^{(-1)}_{X,Y}} "0";"4"};
{\ar^{\psi^{(-1)}_{X,Y}} "0";"6"};
{\ar^{\wp^{\Irm}_{X,Y}} "4";"2"};
{\ar_{\wp^{\IIrm}_{X,Y}} "4";"6"};
{\ar@{}|\circlearrowright "3";"4"};
{\ar@{}|\circlearrowright "4";"5"};
\endxy
\]
commutative, by the universality of pullback.
It is straightforward to check that $\lam^{(-1)}=\{\lam^{(-1)}_{X,Y}\}_{X,Y\in\Csc}$ forms a natural transformation $\lam^{(-1)}\colon F^{-1}\to \EbbPB^{-1}$, again by using the universality of pullbacks.

\medskip
\noindent\und{Construction of $\lam^{(-2)}\colon F^{-2}\to \EbbPB^{-2}$.}

Remark that the left-upper square in $(\ref{Comm_Added})$ becomes a pullback square. Let us see that $\lam^{(-2)}$ can be obtained in a similar way as for $\lam^{(-1)}$ using this square. Indeed, for any $X,Y\in\Csc$, the commutativity of
\[
\xy
(-14,6)*+{F^{-2}(X,Y)}="0";
(14,6)*+{F^{-1}(X,Q_Y)}="2";
(-14,-6)*+{F^{-1}(J_X,Y)}="4";
(14,-6)*+{\Csc(J_X,Q_Y)}="6";
{\ar^{\om_{Y\sharp}} "0";"2"};
{\ar_{\iota_X\ush} "0";"4"};
{\ar^{\iota_X\ush} "2";"6"};
{\ar_{\om_{Y\sharp}} "4";"6"};
{\ar@{}|\circlearrowright "0";"6"};
\endxy
\]
implies that of 
\[
\xy
(-16,6)*+{F^{-2}(X,Y)}="0";
(16,6)*+{\EbbII^{-1}(X,Q_Y)}="2";
(-16,-6)*+{\EbbI^{-1}(J_X,Y)}="4";
(16,-6)*+{\Csc(J_X,Q_Y)}="6";
{\ar^{\psi_{X,Q_Y}^{(-1)}\circ\om_{Y\sharp}} "0";"2"};
{\ar_{\vp_{J_X,Y}^{(-1)}\circ\iota_X\ush} "0";"4"};
{\ar^{\iota_X\ush} "2";"6"};
{\ar_{\om_{Y\sharp}} "4";"6"};
{\ar@{}|\circlearrowright "0";"6"};
\endxy,
\]
hence we obtain a unique homomorphism $\lam^{(-2)}_{X,Y}\colon F^{-2}(X,Y)\to\EbbPB^{-2}(X,Y)$ which makes
\[
\xy
(0,10)*+{F^{-2}(X,Y)}="0";
(-28,-6)*+{\EbbI^{-1}(J_X,Y)}="2";
(-18,4)*+{}="3";
(0,-6)*+{\EbbPB^{-2}(X,Y)}="4";
(18,4)*+{}="5";
(28,-6)*+{\EbbII^{-1}(X,Q_Y)}="6";
{\ar_{\vp^{(-1)}_{J_X,Y}\circ\iota_X\ush} "0";"2"};
{\ar_{\lam^{(-2)}_{X,Y}} "0";"4"};
{\ar^{\psi^{(-1)}_{X,Q_Y}\circ\om_{Y\sharp}} "0";"6"};
{\ar^{\eta^{\Irm}_{X,Y}} "4";"2"};
{\ar_{\eta^{\IIrm}_{X,Y}} "4";"6"};
{\ar@{}|\circlearrowright "3";"4"};
{\ar@{}|\circlearrowright "4";"5"};
\endxy
\]
commutative, by the universality of pullback.
We can check that $\lam^{(-2)}=\{\lam^{(-2)}_{X,Y}\}_{X,Y\in\Csc}$ forms a natural transformation $\lam^{(-2)}\colon F^{-2}\to \EbbPB^{-2}$, in a straightforward way.

So far we have obtained $\lam^{\bullet}\colon F^{\bullet}\to\EbbPB^{\bullet}$. It remains to show its compatibility with $\del\ssh$ and $\del\ush$. By duality, it is enough to show the compatibility with $\del\ssh$.

Let $\del\in\Ebb(C,A)$ be any element. It is enough to show the commutativity of
\begin{equation}\label{CommD}
\xy
(-13,6)*+{F^{-2}(X,C)}="0";
(13,6)*+{F^{-1}(X,A)}="2";
(-13,-6)*+{\EbbPB^{-2}(X,C)}="4";
(13,-6)*+{\EbbPB^{-1}(X,A)}="6";
{\ar^{\del\ssh} "0";"2"};
{\ar_{\lam^{(-2)}_{X,C}} "0";"4"};
{\ar^{\lam^{(-1)}_{X,A}} "2";"6"};
{\ar_{\del\ssh} "4";"6"};
{\ar@{}|\circlearrowright "0";"6"};
\endxy\quad\text{and}\quad
\xy
(-13,6)*+{F^{-1}(X,C)}="0";
(13,6)*+{\Csc(X,A)}="2";
(-13,-6)*+{\EbbPB^{-1}(X,C)}="4";
(13,-6)*+{\Csc(X,A)}="6";
{\ar^{\del\ssh} "0";"2"};
{\ar_{\lam^{(-1)}_{X,C}} "0";"4"};
{\ar@{=} "2";"6"};
{\ar_{\del\ssh} "4";"6"};
{\ar@{}|\circlearrowright "0";"6"};
\endxy
\end{equation}
for any $X\in\Csc$. By definition, commutativity of the right square of $(\ref{CommD})$ follows immediately from that of
\[
\xy
(-30,8)*+{F^{-1}(X,C)}="0";
(30,8)*+{\Csc(X,A)}="2";
(-30,-8)*+{\EbbPB^{-1}(X,C)}="4";
(4,-8)*+{\EbbI^{-1}(X,C)}="6";
(30,-8)*+{\Csc(X,A)}="8";
{\ar^{\del\ssh} "0";"2"};
{\ar_{\lam^{(-1)}_{X,C}} "0";"4"};
{\ar@{=} "2";"8"};
{\ar^{\vp^{(-1)}_{X,C}} "0";"6"};
{\ar_{\del\ssh} "6";"8"};
{\ar_{\wp^{\Irm}_{X,C}} "4";"6"};
{\ar@/_1.60pc/_{\del\ssh} "4";"8"};
{\ar@{}|\circlearrowright "6"; (-44,2) };
{\ar@{}|\circlearrowright "6"; (14,8) };
{\ar@{}|\circlearrowright "6"; (-4,-16) };
\endxy.
\]
Commutativity of the left square of $(\ref{CommD})$ follows from the commutativity of
\[
\xy
(-30,10)*+{F^{-2}(X,C)}="0";
(30,10)*+{F^{-1}(X,A)}="2";
(-30,-10)*+{\EbbPB^{-2}(X,C)}="4";
(30,-10)*+{\EbbPB^{-1}(X,A)}="6";
(-4,0)*+{\EbbI^{-1}(X,A)}="8";
{\ar^{\del\ssh} "0";"2"};
{\ar_{\lam^{(-2)}_{X,C}} "0";"4"};
{\ar^{\lam^{(-1)}_{X,A}} "2";"6"};
{\ar_{\del\ssh} "4";"6"};
{\ar_{0} "0";"8"};
{\ar_(0.56){\vp^{(-1)}_{X,A}} "2";"8"};
{\ar^(0.56){\wp^{\Irm}_{X,A}} "6";"8"};
{\ar@{}|\circlearrowright "8";(-2,12)};
{\ar@{}|\circlearrowright "8";(40,0)};
{\ar@{}|\circlearrowright "8";(-32,-8)};
\endxy
\]
in which $\vp^{(-1)}_{X,A}\circ\del\ssh=0$ holds since $\EbbI^{-2}(X,C)=0$, and of
\[
\xy
(-36,18)*+{F^{-2}(X,C)}="0";
(36,18)*+{F^{-1}(X,A)}="2";
(-36,-18)*+{\EbbPB^{-2}(X,C)}="4";
(36,-18)*+{\EbbPB^{-1}(X,A)}="6";
(-12,6)*+{F^{-1}(X,Q_C)}="8";
(-12,-6)*+{\EbbII^{-1}(X,Q_C)}="10";
(16,-6)*+{\EbbII^{-1}(X,A)}="12";
{\ar^{\del\ssh} "0";"2"};
{\ar_{\lam^{(-2)}_{X,C}} "0";"4"};
{\ar^{\lam^{(-1)}_{X,A}} "2";"6"};
{\ar_{\del\ssh} "4";"6"};
{\ar_{\om_{C\sharp}} "0";"8"};
{\ar_{\psi^{(-1)}_{X,Q_C}} "8";"10"};
{\ar_{t\sas} "8";"2"};
{\ar_{t\sas} "10";"12"};
{\ar_{\psi^{(-1)}_{X,A}} "2";"12"};
{\ar_{\wp^{\IIrm}_{X,A}} "6";"12"};
{\ar_{\eta^{\IIrm}_{X,C}} "4";"10"};
{\ar@{}|\circlearrowright "10";(-44,4)};
{\ar@{}|\circlearrowright "10";(24,10)};
{\ar@{}|\circlearrowright "10";(-4,32)};
{\ar@{}|\circlearrowright "10";(14,-20)};
{\ar@{}|\circlearrowright "10";(68,0)};
\endxy
\]
for any $t\in\Csc(Q_C,A)$ such that $t\sas\om_C=\del$, and the universality of pullback.

The above argument shows that $\lam^{\bullet}\colon (F^{\bullet},\del\ssh,\del\ush)\to (\EbbPB^{\bullet},\del\ssh,\del\ush)$ is a morphism of bivariant $\delta$-functors.
Its uniqueness is deduced from the uniqueness of $\vp^{\bullet}$ and $\psi^{\bullet}$, in the usual way.
\end{proof}

\begin{remark} 
Under \cref{assumption: hered_enough}, we have the following (see  \cite[Condition 5.12]{GNP1} for the undefined notation, and also \cref{rem:NI_and_NII}).
\begin{enumerate}
\item $\CEs$ satisfies {\rm (NI$+$)} if and only if it satisfies {\rm (NI)} and {\rm (N$+$)}.
\item Dually, $\CEs$ satisfies {\rm (NII$+$)} if and only if it satisfies {\rm (NII)} and {\rm (N$+$)}.
\end{enumerate}

In particular, if $\CEs$ satisfies both conditions {\rm (NI$+$)} and {\rm (NII$+$)}, it satisfies condition {\rm (N$+$)}.
Indeed, in this case, we have $\EbbPB^{\bullet} \cong \EbbI^{\bullet} \cong \EbbII^{\bullet}$ and \cref{PropEquivAcycBiProl}.(6) together with the universality of $(\EbbPB^{\bullet}, \del\ssh, \del\ush)$ follows from \cite[Corollary 5.39]{GNP1} recalled in Section 7.1.

The converse is not true. A basic example of a category satisfying condition {\rm (N$+$)}, but not having $\EbbI^{\bullet} \cong \EbbII^{\bullet}$ is given by the cluster category of type $A_2$ with the largest relative structure making a given cluster-tilting object projective. See \cref{ex: A_2}, \cref{ex: acyclic_seed} for details.
\end{remark}

\subsection{Balanced $\delta$-functor in embedded reduced $0$-Auslander categories} 
\label{ssection:computation}

Let $\Dsc$ be a triangulated category with a rigid, full subcategory $\Tsc$.
Define $\Csc$ to be the full subcategory $\Tsc\ast\susp\Tsc$ of $\Dsc$.
Endow $\Dsc$ with the largest relative extriangulated structure making objects in $\Tsc$ projective.
Then $\Csc$ is extension closed in $\Dsc$ for this extriangulated structure, and thus inherits an extriangulated structure $(\Csc,\Ebb,\sfr)$. The latter is $0$-Auslander, as  discussed briefly in \cref{Ex:RigidTriangulated} and in detail in \cref{Ex:RelativeTilting}, and has $\Tsc$ and $\susp\Tsc$ as the full subcategories of projectives and of injectives, respectively. 

In what follows, the term ``triangle'' will always refer to triangles in the triangulated category $\Dsc$, while ``conflations'' will refer to $\sfr$-triangles in $\Csc$.

\begin{remark} \label{rem:Ebb_in_embedded}
 By definition, for any $A,C\in\Csc$, we have $\Ebb(C,A) = [\susp\Tsc](C,\susp A) \subseteq \Dsc(C, \susp A)$.
\end{remark}

For any $X\in\Csc$, there is a triangle $T_1^X \to T_0^X \to X \to \susp T_1^X$, with $T_0^X,T_1^X\in\Tsc$. We fix such a triangle for each $X$.
It defines a projective resolution of the form (\ref{sTri_Chosen_proj}) of $X$ in $\Tsc$. By shifting this triangle in $\Dsc$, we also have a triangle $X \to \susp T_1^X \to \susp T_0^X \to \susp X$, defining an injective coresolution of the form (\ref{sTri_Chosen_inj}) of $X$ in $\Tsc$. In other words, in the notation of (\ref{sTri_Chosen_proj}), (\ref{sTri_Chosen_inj}), we have
\[P_X = T_0^X, Q_X = T_1^X, I_X = \susp T_1^X, J_X = \susp T_0^X.\]

\begin{lemma} \label{lem:Ebb-1_0_Auslander_one-sided}
 For any $A,C\in\Csc$, we have
\begin{eqnarray*}
  \EbbI^{-1}(C,A) & \cong & \Dsc/_{[\susp^{-1}\Tsc]}(C,\susp^{-1}A) \\
 \EbbII^{-1}(C,A) & \cong &  \Dsc/_{[\susp\Tsc]}(C,\susp^{-1}A).
\end{eqnarray*}

\end{lemma}

\begin{proof}
$\EbbI^{-1}(C,A)$ identifies with the kernel of $\Csc(C,T_1^A)\to\Csc(C,T_0^A)$.
Applying $\Dsc(C,-)$ to the triangle $\susp^{-1}T_0^A\to \susp^{-1}A\to T_1^A \to T_0^A$ gives the exact sequence
\[
 \Dsc(C,\susp^{-1}T_0^A)\to\Dsc(C,\susp^{-1}A)\to\Csc(C,T_1^A)\to\Csc(C,T_0^A)
\]
from which we deduce that $\EbbI^{-1}(C,A)$ identifies with
the cokernel of the first map.
Since $\Tsc$ is rigid, this cokernel is precisely $\Dsc/_{[\susp^{-1}\Tsc]}(C,\susp^{-1}A)$.
Similar computations based on a triangle of the form $\susp T_1^A\to\susp T_0^A\to \susp C\to \susp^2 T_1^A$ show that $\EbbII^{-1}(C,A)$ identifies with $\Dsc/_{[\susp^2\Tsc]}(\susp C,A)\cong\Dsc/_{[\susp\Tsc]}(C,\susp^{-1}A)$.
\end{proof}

\begin{remark}
From \cref{lem:Ebb-1_0_Auslander_one-sided} and its proof, it is straightforward to check that for any $A,C\in\Csc$, the morphisms $\iota_C\ush\circ \om_{A\sharp}: \EbbI^{-1}(C, A) \to \Ebb(J_C, Q_A)$ and  $\om_{A\sharp}\circ\iota_C\ush: \EbbII^{-1}(C, A) \to \Ebb(J_C, Q_A)$ factor through the ideal quotients $\Ebb_{\mathbf{i}}^{-1}(C, \susp^{-1} A) \to \Dsc/_{\left\langle[\susp\Tsc],[\susp^{-1}\Tsc]\right\rangle}(C,\susp^{-1}A)$  for $\sufe$.
Indeed, 
$\om_{A\sharp}: \EbbE^{-1}(C, A) \to \Csc(C, Q_A)$ is isomorphic to $\Dsc/_{[\susp^{-1}\Tsc]}(C,\susp^{-1}A) \hookrightarrow \Dsc/_{[\susp^{-1}\Tsc]}(C,T_1^A)$ (recall that $\Tsc$ is rigid, hence so is $\susp^{-1}\Tsc$), and the image of $\iota_C\ush:  \Dsc(C,T_1^A) \to \Dsc(T_0^C, T_1^A)$ is the cokernel of $\Dsc(\susp T_1^C, T_1^A) \ov{i_C\uas}{\lra} \Dsc(C,T_1^A),$ hence
is isomorphic to 
\[\Dsc/_{[\susp\Tsc]}(C,T_1^A) = \Dsc/_{\left\langle[\susp\Tsc],[\susp^{-1}\Tsc]\right\rangle}(C,T_1^A).\]
Then $\iota_X\ush\circ \om_{Y\sharp}$ factors as
\[\Dsc/_{[\susp^{-1}\Tsc]}(C,\susp^{-1}A) \twoheadrightarrow  \Dsc/_{\left\langle[\susp\Tsc],[\susp^{-1}\Tsc]\right\rangle}(C,\susp^{-1}A)
\to
\Dsc/_{\left\langle[\susp\Tsc],[\susp^{-1}\Tsc]\right\rangle}(C,T_1^A)
\hookrightarrow
\Dsc(T_0^C, T_1^A),
\]
where the first morphism is the cokernel of 
$\Dsc/_{[\susp^{-1}\Tsc]}(\susp T_1^C, \susp^{-1} A) \ov{i\uas}{\lra} \Dsc/_{[\susp^{-1}\Tsc]}(C, \susp^{-1} A),$
and the second morphism is the induced morphism of cokernels. A similar factorisation of  $\om_{A\sharp}\circ\iota_C\ush$ can be written via dual arguments.
The pullback diagram in \cref{DefUniv_hered} then gives
\begin{eqnarray} \label{EBP_explicit}
  \EbbPB^{-1}(C,A) & \cong & \Dsc/_{([\susp\Tsc] \cap [\susp^{-1}\Tsc])}(C,\susp^{-1}A).
\end{eqnarray}

Note that the ideal $([\susp\Tsc] \cap [\susp^{-1}\Tsc])(-, -)$ we quotient out is  formed by the morphisms that admit a factorisation through an object in $\susp\Tsc$ and a possibly different factorisation through an object in $\susp^{-1}\Tsc$. This ideal can be strictly larger than the ideal generated by the morphisms which admit a factorisation through an object in $\susp\Tsc$ and an object in $\susp^{-1}\Tsc$, in some order. It can also be strictly larger than the ideal of the morphisms factoring through an object in the intersection $\susp\Tsc \cap \susp^{-1}\Tsc$ of subcategories in $\Dsc$.
\end{remark}

\begin{lemma} \label{lem: E-2_explicit}
\[\Ebb^{-2}(C, A) =  \Dsc/_{[\susp^{-2}\Tsc]}(T_0^C,\susp^{-2}A)\times_{\Dsc(T_0^C,\susp^{-1}T_1^A)}\Dsc/_{[\susp\Tsc]}(C,\susp^{-1}T_1^A).\]
\end{lemma}

\begin{proof}
This follows directly from \cref{rem:E^-2_as_fiber_product}. Indeed, in the notation of that remark, we have
\begin{eqnarray*}
\EbbI^{-1}(J_C, A) = \Dsc/_{[\susp^{-1}\Tsc]}(\susp T_0^C,\susp^{-1}A) = \Dsc/_{[\susp^{-2}\Tsc]}(T_0^C,\susp^{-2}A);\\
\Csc(J_C, Q_A) = \Csc(\susp T_0^C, T_1^A) = \Dsc(T_0^C, \susp^{-1} T_1^A);\\
\EbbII^{-1}(C, Q_A) =  \Dsc/_{[\susp\Tsc]}(C,\susp^{-1}T_1^A).
\end{eqnarray*}
\end{proof}

\begin{remark} 
An element in $\Dsc(T_0^C,\susp^{-2}A)\times_{\Dsc(T_0^C,\susp^{-1}T_1^A)}\Dsc(C,\susp^{-1}T_1^A)$ is a commutative square
\[
 \xy
(-54,0)*+{(\ast)};
(-18,7)*+{T_0^C}="4";
(0,7)*+{C}="6";
(18,7)*+{\susp T_1^C}="8";
(-36,-7)*+{\susp^{-2}T_0^A}="12";
(-18,-7)*+{\susp^{-2}A}="14";
(0,-7)*+{\susp^{-1}T_1^A.}="16";
{\ar "4";"6"};
{\ar "6";"8"};
{\ar_{f} "4";"14"};
{\ar^{g} "6";"16"};
{\ar "12";"14"};
{\ar "14";"16"};
{\ar@{}|\circlearrowright "4";"16"};
\endxy
\]
Because $\Tsc$ is rigid, the morphism $f$ belongs to the ideal $[\susp^{-2}\Tsc]$ if and only if it factors through the morphism $\susp^{-2}T_0^A\to\susp^{-2}A$.
Similarly, $g$ belongs to $[\susp\Tsc]$ if and only if it factors through $C\to\susp T_1^C$.
This shows that commutativity of the square only depends on the classes of $f$ modulo $[\susp^{-2}\Tsc]$ and of $g$ modulo $[\susp\Tsc]$.
We also have: \begin{eqnarray*}
   g\in[\susp\Tsc] & \iff &T_0^C\to C\xrightarrow{g}\susp^{-1}T_1^A \text{ vanishes} \\
   & \iff & T_A\xrightarrow{f} \susp^{-2}A\to \susp^{-1}T_1^A \text{ vanishes} \\
   & \iff & f\in[\susp^{-2}\Tsc].
 \end{eqnarray*}
Hence $\Ebb^{-2}$ is the kernel of the canonical map
\[
\Dsc(T_0^C,\susp^{-2}A)\times_{\Dsc(T_0^C,\susp^{-2}\susp T_1^A)}\Dsc(C,\susp^{-2}\susp T_1^A) \longrightarrow 
\Dsc(T_0^C,\susp^{-2}\susp T_1^A).
\]
\end{remark}

\begin{definition}
\label{def:negativeComputation}
We define
\begin{eqnarray*}
 \EbbE^{-1}(C,A) & := & \Dsc/_{\left[\susp\Tsc\!\to\!\susp^{-1}\Tsc\right]}(C,\susp^{-1}A). 
\end{eqnarray*}
\end{definition}

\begin{lemma} \label{lem:PropCForEmbE}
\begin{enumerate}[(a)]
 \item For any $A,C\in\Csc$ and any $T\in\Tsc$, we have $\EbbE^{-1}(C,T)=\EbbII^{-1}(C,T)$ and $\EbbE^{-1}(\susp T, A)=\EbbI^{-1}(\susp T,A)$.
 \item $\EbbE^{-1}(C,A)$ 
 is bifunctorial.
\end{enumerate}
\end{lemma}

\begin{proof}
(a) For any $T\in\Tsc$ and any $C\in\Csc$ we have
\begin{eqnarray*}
 \EbbE^{-1}(C,T) & = & \Dsc/_{\left[\Sigma\Tsc\to\Sigma^{-1}\Tsc\right]}(C,\Sigma^{-1}T) \\
 & = & \Dsc/_{[\Sigma\Tsc]}(C,\Sigma^{-1}T) \\
 & = & \EbbII^{-1}(C,T).
\end{eqnarray*}
The second equality is obtained similarly.

(b) The bifunctoriality of $\EbbE^{-1}$ is inherited from that of $\Dsc(-,-)$.
\end{proof}


\begin{definition} \label{def: conn_EbbE}
Set 
\[\EbbE^n = \Ebb^n, n \ge 0; \qquad
\EbbE^{-2} := \Ebb^{-2}; \qquad
\EbbE^{-m} := 0, m \ge 3.
\] 
We complete $(\EbbE^{\bullet})$ to a bivariant connected sequence of functors as follows.

 Consider a conflation $X\overset{i}{\to} Y\overset{p}{\to} Z \ov{\del}{\dra}$  in $\Csc$. By definition of $\CEs$, there is a triangle in $\Dsc$ of the form: $X\xrightarrow{i} Y\xrightarrow{p} Z\xrightarrow{\del} \susp X$, where $\del\in[\susp\Tsc]$ (see also \cref{rem:Ebb_in_embedded}).
Moreover, the projective resolutions of $X$ and $Z$ of the form (\ref{sTri_Chosen_proj}) induce a projective resolution 
$T_1^X\oplus T_1^Z \infl T_0^X\oplus T_0^Z \defl Y \dra$ of $Y$.

We define the connecting morphisms:

\begin{itemize}
\item The map $\del\ush: \EbbE^{-1}(X,U)=\Dsc/_{\left[\susp^2\Tsc\to\Tsc\right]}(\susp X,U)\to\Csc(Z,U)$ is induced by precomposing with $\del$. It is well-defined because $\del\in[\susp\Tsc]$ and $\Tsc$ is rigid.

\item The map $\del\ush: \EbbE^{-2}(X,U)\to\EbbE^{-1}(Z,U)$ is defined as follows.
Recall that 
\begin{eqnarray*}
\EbbE^{-2}(X,U) & = & \Dsc/_{[\Tsc]}(\susp^2T_0^X,U)\times_{\Dsc(\susp^2 T_0^X,\susp T_1^U)}\Dsc/_{[\susp^3\Tsc]}(\susp^2 X,\susp T_1^U) \text{ and } \\
\EbbE^{-1}(Z,U) & = & \Dsc/_{\left[\susp^2\Tsc\to\Tsc\right]}(\susp Z,U).
\end{eqnarray*}
Any element $(f,g)$ in $\EbbE^{-2}(X,U)$ can be completed to a morphism of triangles $(h,f,g)$ below.
Since the morphism $\susp\del$ belongs to the ideal $[\susp^2\Tsc]$, it factors through the morphism $\susp^2 T_0^X\to\susp^2 X$ as in the diagram below.
Define $\del\ush((f,g))$ to be the composition $f\circ u$.
\[
\xy
(9,18)*+{\susp Z}="top"; 
(-18,7)*+{\susp^2 T_1^X}="2";
(0,7)*+{\susp^2 T_0^X}="4";
(18,7)*+{\susp^2 X}="6";
(-18,-7)*+{T_0^U}="12";
(0,-7)*+{U}="14";
(18,-7)*+{\susp T_1^U}="16";
{\ar "2";"4"};
{\ar "4";"6"};
{\ar^{h} "2";"12"};
{\ar^{f} "4";"14"};
{\ar^{g} "6";"16"};
{\ar "12";"14"};
{\ar "14";"16"};
{\ar^{\susp\del} "top";"6"};
{\ar@{-->}^u "top";"4"};
{\ar@/_/@{..>}_{\text{lifts } u-v} "top";"2"};
{\ar@{}|\circlearrowright "2";"14"};
{\ar@{}|\circlearrowright "4";"16"};
\endxy
\]
To see that the map is well-defined, note that if $v$ is any other choice of a lift of $\susp\del$ to $\susp^2 T_0^X$, then the difference $u-v$ factors through
$\susp^2 T_1^X\to\susp^2 T_0^X$.
This implies that the difference between $f\circ u$ and $f\circ v$ factors through $h$, and hence belongs to the ideal $[\susp^2\Tsc\to\Tsc]$.

\item Connecting morphisms  $\del\ssh$ in the second argument are defined in the dual way.
\end{itemize}
\end{definition}

\begin{proposition}
\label{remark:natural}
The connecting morphisms $(\del\ssh, \del\ush)$ defined above are covariant, resp. contravariant natural transformations in each degree, natural with respect to morphisms of $\Ebb$-extensions, and so $(\EbbE^{\bullet}, \del\ssh, \del\ush)$ is a bivariant connected sequence of functors.
\end{proposition}

\begin{proof}
Each morphism of $\Ebb$-extensions is a morphism of triangles in $\Dsc$. We keep the notation from \cref{def: conn_EbbE}.

\begin{itemize}
\item The map $\del\ush: \EbbE^{-1}(X,U)\to\Csc(Z,U)$ is natural in $U$ because it is induced by composition with $\del$ (considered as a morphism in $\Dsc(Z, \susp Z)$). 

\item To see that the map $\del\ush: \EbbE^{-2}(X,U)\to\EbbE^{-1}(Z,U)$ is also natural in $U$, simply note that if $U\xrightarrow{w} U'$ is any morphism, then both compositions $\EbbE^{-2}(X,U)\to\EbbE^{-1}(Z,U)\to\EbbE^{-1}(Z,U')$ and $\EbbE^{-2}(X,U)\to\EbbE^{-2}(Z,U')\to\EbbE^{-1}(Z,U')$ send the class of $(f,g)$ to the class of the composition $wfu$.

\item A morphism from the triangle in \cref{def: conn_EbbE}  realizing $\del \in \Ebb(Z, X)$ to another triangle
\newline $X' \xrightarrow{i'} Y' \xrightarrow{j'} Z' \xrightarrow{\del'} \susp X'$ realizing $\del' \in \Ebb(Z', X')$ is defined by a pair of morphisms $(x, z) \in \Dsc(X, X') \times \Dsc(Z, Z')$ such that $\del \circ z = \susp x \circ \del'$. This commutativity directly induces the commutativity of the diagram
\[
\xy
(-14,6)*+{\EbbE^{-1}(X',-)}="0";
(14,6)*+{\Csc(Z',-)}="2";
(-14,-6)*+{\EbbE^{-1}(X,-)}="4";
(14,-6)*+{\Csc(Z,-)}="6";
{\ar^{\del^{\prime\sharp}} "0";"2"};
{\ar_{x\uas} "0";"4"};
{\ar^{z\uas} "2";"6"};
{\ar_{\del\ush} "4";"6"};
{\ar@{}|\circlearrowright "0";"6"};
\endxy
\]
This shows the naturality of $\del\ush: \EbbE^{-1}(X,U)\to\Csc(Z,U)$ with respect to morphisms of $\Ebb$-extensions.

\item  The naturality of $\del\ush: \EbbE^{-2}(X,U)\to\EbbE^{-1}(Z,U)$ with respect to morphisms of $\Ebb$-extensions amounts to the commutativity of the diagram 

\begin{equation} \label{eqn: natur_EbbE-2}
\xy
(-14,6)*+{\EbbE^{-2}(X\ppr,-)}="0";
(14,6)*+{\EbbE^{-1}(Z\ppr,-)}="2";
(-14,-6)*+{\EbbE^{-2}(X,-)}="4";
(14,-6)*+{\EbbE^{-1}(Z,-)}="6";
{\ar^{\del^{\prime\sharp}} "0";"2"};
{\ar_{x\uas} "0";"4"};
{\ar^{z\uas} "2";"6"};
{\ar_{\del\ush} "4";"6"};
{\ar@{}|\circlearrowright "0";"6"};
\endxy
\end{equation}

 The functoriality of $\EbbE^{-2}(-, U) = \EbbE^{-2}(-, U)$ amounts to the fact that $x^*$ sends $(f', g') \in \EbbE^{-2}(X', U)$ to $(f' \circ \mu, g' \circ \susp^2 x)$, for any completion of $\susp^2 x \in \Dsc(\susp^2 X, \susp^2 X')$ to the diagram of the form

\[
\xy
%
(-18,7)*+{\susp^2 T_1^X}="2";
(0,7)*+{\susp^2 T_0^X}="4";
(18,7)*+{\susp^2 X}="6";
(-18,-7)*+{\susp^2 T_1^{X'}}="12";
(0,-7)*+{\susp^2 T_0^{X'}}="14";
(18,-7)*+{\susp^2 X',}="16";
{\ar "2";"4"};
{\ar^{p} "4";"6"};
{\ar@{-->} "2";"12"};
{\ar@{-->}^{t} "4";"14"};
{\ar^{\susp^2 x} "6";"16"};
{\ar "12";"14"};
{\ar^{p'} "14";"16"};
%
%
{\ar@{}|\circlearrowright "2";"14"};
{\ar@{}|\circlearrowright "4";"16"};
\endxy
\]
and result gives an element of $\EbbE^{-2}(X, U)$ which does not depend on the choice of such a diagram.

In the diagram (\ref{eqn: natur_EbbE-2}) we then have:

\begin{itemize}
\item $(f', g') \in \EbbE^{-2}(X', U)$ is mapped by $\del^{\prime\sharp}$ to $f' \circ u'$, for some $u' \in \Dsc(\susp Z', \susp^2 T_0^{X'})$ such that $\susp\del' = p' \circ u'$. 

\item By functoriality of $\EbbE^{-1}(-, U)$, $z^*$ maps $f' \circ u'$ to $f' \circ u' \circ \susp z$.

\item By functoriality of $\EbbE^{-2}(-, U)$, $x^*$ maps $(f', g')$ to $(f' \circ t, g' \circ \susp^2 x)$. 

\item Finally, $\del\ush$ maps $(f' \circ t, g' \circ \susp^2 x)$ to $f' \circ t \circ u$, for some $u \in \Dsc(\susp Z', \susp^2 T_0^{X'})$ such that $\susp \del = p \circ u$.
 \end{itemize}

The diagram  (\ref{eqn: natur_EbbE-2}) is then commutative if and only if for each $(f', g') \in \EbbE^{-2}(X', U)$, we have 
\[
f' \circ (t \circ u - u' \circ \susp z) \in [\susp^2\Tsc\to\Tsc](\susp Z, U) \subseteq \Dsc(\susp Z, U).
\] 
Note that we have 
\begin{eqnarray*}
p' \circ (t \circ u - u' \circ \susp z) = 
\susp^2 x \circ p \circ u - p' \circ u' \circ \susp z = 
\susp^ x \circ \del - \susp \del' \circ \susp z = 
0
\end{eqnarray*}
As in the argument in \cref{def: conn_EbbE}, this implies that $(\mu \circ u - u' \circ \susp z)$ factors through $\susp^2 T_1^{X'} \to \susp^2 T_0^{X'}$. Therefore, the difference $f' \circ (t \circ u - u' \circ \susp z)$ factors through $h' \in \Dsc(\susp^2 T_1^{X'}, T_0^U)$ completing $(f', g')$ to a morphism of triangles, and hence belongs to the ideal $[\susp^2\Tsc\to\Tsc]$. 

\end{itemize}
\end{proof}


\begin{proposition} \label{prop:EmbEisDelta}
$(\EbbE, \delta\ssh, \delta\ush)$ is a bivariant $\delta$-functor satisfying property $\rm{(C)}$. 
\end{proposition}

\begin{proof}
Let $C,A\in\Csc$. 
By \cref{lem:enough_to_check_on_dominant}, to show that it is a bivariant $\delta$-functor it is enough to prove that for each $X$ some $\sfr$-triangle of the form (\ref{sTri_Chosen_proj})  is $(\EbbE, \delta\ssh)$-acyclic and some $\sfr$-triangle of the form (\ref{sTri_Chosen_inj}) is $(\EbbE, \delta\ush)$-acyclic.  

Let $C,A\in\Csc$. 
As usual, fix two triangles $T_1^A\to T_0^A\to A \to \susp T_1^A$, and $T_1^C\to T_0^C\to C \to \susp T_1^C$ with $T_0^A,T_1^A,T_0^C,T_1^C\in\Tsc$ giving  $\sfr$-triangles of the form (\ref{sTri_Chosen_proj}) and, up to shift, of the form (\ref{sTri_Chosen_inj}) for $A$ and $C$.
The  $(\EbbE, \delta\ssh)$-acyclicity of the dominant $\sfr$-triangle  $T_1^A\to T_0^A\to A \dra$ is equivalent to the exactness of the following sequence:

\begin{tabular}{rccclc}
  & & 0&$\to$ & $\Ebb^{-2}(C, A)$ & $\to$ \\
  $\Dsc/_{[\susp\Tsc]}(C,\susp^{-1}T_1^A)$&$\to$ &$\Dsc/_{[\susp\Tsc]}(C,\susp^{-1}T_0^A)$&$\to$& $\Dsc/_{\left[\susp\Tsc  \to  \susp^{-1}\Tsc\right]}(C,\susp^{-1}A)$ & $\to$ \\
$\Csc(C,T_1^A)$ &$\to$& $\Csc(C,T_0^A)$.& & & 
\end{tabular}

 

%
%
%
%
%
%


Exactness at $\Csc(C,T_1^A)$ and $\Dsc/_{\left[\susp\Tsc\!\to\!\susp^{-1}\Tsc\right]}(C,\susp^{-1}A)$ are immediate from the triangle $T_1^A\to T_0^A\to A\to\susp T_1^A$.

Exactness at $\Dsc/_{[\susp\Tsc]}(C,\susp^{-1}T_0^A)$:
Let $f\in\Dsc(C,\susp^{-1}T_0^A)$ be such that the composition $C\xrightarrow{f}\susp^{-1}T_0^A\to\susp^{-1}A$ vanishes in $\Dsc/_{\left[\susp\Tsc\!\to\!\susp^{-1}\Tsc\right]}(C,\susp^{-1}A)$.
By definition, there is a commutative diagram, in plain below:
\[
 \xy
(-18,7)*+{T_0^C}="2";
(0,7)*+{C}="4";
(18,7)*+{\susp T_1^C}="6";
(30,0)*+{\susp^{-1}T}="8";
(-18,-7)*+{\susp^{-1}T_1^A}="12";
(0,-7)*+{\susp^{-1}T_0^A}="14";
(18,-7)*+{\susp^{-1}A}="16";
(38,-7)*+{T_1^A}="18";
{\ar "2";"4"};
{\ar "4";"6"};
{\ar^{f} "4";"14"};
{\ar "12";"14"};
{\ar "14";"16"};
{\ar "16";"18"};
{\ar "6";"8"};
{\ar "8";"16"};
{\ar@{-->} "8";"14"};
{\ar@{-->} "4";"12"};
%
{\ar@{}|\circlearrowright "4";"16"};
\endxy
\]
Since $\Tsc$ is rigid, the morphism $\susp^{-1}T\to\susp^{-1}A$ factors through $\susp^{-1}T_0^A$ (dashed, on the right).
The difference between $f$ and the composition $C\to\susp T_1^C\to \susp^{-1}T\to\susp^{-1}T_0^A$ vanishes once post-composed with the morphism $\susp^{-1}T_0^A\to\susp^{-1}A$.
Thus, there is a morphism $C\to \susp^{-1}T_1^A$ (dashed, on the left) such that $f$ is the sum of the two compositions involving the dashed morphisms.
This means precisely that the class of $f$ in $\Dsc/_{[\susp\Tsc]}(C,\susp^{-1}T_0^A)$ is in the image of 
$$\Dsc/_{[\susp\Tsc]}(C,\susp^{-1}T_1^A)\to \Dsc/_{[\susp\Tsc]}(C,\susp^{-1}T_0^A).$$

Exactness at $\Dsc/_{[\susp\Tsc]}(C,\susp^{-1}T_1^A)$ and at $\Ebb^{-2}(C, A)$ are immediate consequences of the definition of $\Ebb^{-2}(C, A)$ as the kernel of the morphism $\Dsc/_{[\susp\Tsc]}(C,\susp^{-1}T_1^A) \to \Dsc/_{[\susp\Tsc]}(C,\susp^{-1}T_0^A)$.

The $(\EbbE, \delta\ush)$-acyclicity of the codominant $\sfr$-triangle  $C \to \susp T_1^C \to \susp T_0^C \dra$ is proved dually.

Property (C) is nothing but \cref{lem:PropCForEmbE} combined with the definition of $\EbbE^{-n}$ for $n \ge 2$.



\end{proof}

\subsection{Comparison and universality} 

Let $(\Csc', \Ebb, \mathfrak{s})$ be an extriangulated category. Assume that there exists a triangulated category $\Dsc$ with a rigid full subcategory $\Tsc$ and an equivalence of extriangulated categories $F: \Csc' \overset\sim\to \Csc = \Tsc \ast \Sigma \Tsc$, where the extriangulated structure on the category on the right hand side is, as before, induced by the largest relative extriangulated structure on $\Dsc$ making objects in $\Tsc$ projective. In this situation, we can apply  constructions from both of the previous two subsections and define connected sequences of functors  $\EbbPB^{\bullet}$ and $(F\op\ti F) \circ \EbbE^{\bullet}$ on $\Csc'$, where $\EbbE^{\bullet}$ is defined on $\Tsc \ast \Sigma \Tsc \hookrightarrow \Dsc$ as in \cref{ssection:computation}. The second sequence a priori depends on the embedding $\Tsc \hookrightarrow \Dsc$ and the equivalence $F$. It is natural to compare these sequences and see if they form (universal balanced) $\delta$-functors. Results in \cref{ssection:computation} imply that the second sequence is always a $\delta$-functor with property (C), while the first sequence is a  $\delta$-functor with properties (C) and (U) under condition (N+) on the category $\Csc.$ 

\begin{proposition}
Assume that $R = K$ is a field, $\Csc$ is $R$-linear and has finite-dimensional $\Hom$- and $\Ebb$-spaces. Then $\dim \EbbE^{-i}(-, -)$ are independent of the choices of the categories $\Tsc \hookrightarrow \Dsc$ and of the equivalence $\Csc \overset\sim\to \Tsc \ast \Sigma \Tsc.$
\end{proposition}


\begin{proof}
First note that for any projective $P \in \Csc',$ we have $\EbbE^{-i}(-, P) \cong \EbbII^{-i}(-, P)$ for any projective and so does not depend on the embedding or the equivalence. The result then follows by using the long exact sequence in the second argument and the fact that $\Csc$ has projective dimension 1.
\end{proof}

\begin{lemma} \label{lem:comparison:univ}
Assume that there exists a balanced $\delta-$functor $\Ebb_{un}^{\bullet}$ having  $\Ebb_{un}^n~=~\Ebb^n$ with the canonical connecting morphisms for $n\ge0$ and satisfying conditions $\rm{(U)}$ and $\rm{(C)}$.
Then we have $\Ebb_{un}^{\bullet} \overset\sim\to \EbbE^{\bullet}.$
\end{lemma} 
 
\begin{proof}
This follows from \cref{prop:EmbEisDelta} by \cref{prop:comparison:univ}.
\end{proof}

We also give a more explicit reformulation of the condition (N+) in terms of restrictions of certain sub-bifunctors of $\Dsc$ on $\Csc\op\ti\susp^{-1}\Csc$ and compare the constructions from \cref{ssection:universal} and \cref{ssection:computation}.

\begin{corollary} \label{cor: N+_comparison}
The following are equivalent:
\begin{itemize}
    \item[(i)] Condition $\rm{(N+)}$ is satisfied;
    \item[(ii)] We have \[([\susp \Tsc] \cap [\susp^{-1} \Tsc]) (\Csc, \susp^{-1} \Csc) = [\susp \Tsc \to \susp^{-1} \Tsc] (\Csc, \susp^{-1} \Csc);\]
    \item[(iii)] $\EbbPB^{\bullet}$ is a $\delta-$functor;
    \item[(iv)] $\EbbPB^{\bullet} \overset\sim\to \EbbE^{\bullet}$.
\end{itemize}
\end{corollary}

\begin{proof}
Conditions (i) and (iii) are equivalent by \cref{PropEquivAcycBiProl}. Moreover, by the same theorem, (i) implies that $\EbbPB^{\bullet} = \Ebb_{un}^{\bullet}$ is universal and satisfies the conditions of \cref{lem:comparison:univ}. Thus, (i) implies (iv). If (iv) holds, then (iii) holds authomatically, since $\EbbE^{\bullet}$ is a $\delta-$functor by \cref{prop:EmbEisDelta}.

Finally, the equivalence (ii) $\Leftrightarrow$ (iv) follows from the comparison of the formula (\ref{EBP_explicit}) for $\EbbPB^{-1}$ with that for $\EbbE^{-1}$ in \cref{def:negativeComputation}.
Alternatively, using the above paragraph it amounts to the equivalence (i) $\Leftrightarrow$ (ii). To show the latter, note that in our settting, the complex $C_{\Irm}^{\bullet}$ is always exact in the second term, so the condition (N+) holds if and only if it is exact in the third term (for each $C, A \in \Csc$). The latter happens if and only if 
\[[\susp \Tsc]/_{[\susp^{-1} \Tsc]} (\susp \Csc, \susp^{-1} \Csc) = [\susp \Tsc]/_{[\susp \Tsc \to \susp^{-1} \Tsc]} (\susp \Csc, \susp^{-1} \Csc),\] 
which is clearly equivalent to the reformulation in (ii). Dually, the exactness of $C_{\IIrm}^{\bullet}$ in the third term for each $C, A \in \Csc$ amounts to asking that
\[[\susp^{-1} \Tsc]/_{[\susp \Tsc]} (\susp \Csc, \susp^{-1} \Csc) = [\susp^{-1} \Tsc]/_{[\susp \Tsc \to \susp^{-1} \Tsc]} (\susp \Csc, \susp^{-1} \Csc),\]
which is also equivalent to the reformulation in (ii).
\end{proof}

\subsection{Examples}

\begin{example}
$\Csc$ is exact if and only if $\EbbI^{-1} = \EbbII^{-1} = 0.$ This can be immediately deduced from \cite[Proposition 2.6]{AdachiEnomotoTsukamoto} by using the universal properties of $\EbbI$ and $\EbbII$. Also, the ``only if'' part is \cite[Proposition 5.4]{GNP1} and its dual, and the ``if'' part follows directly from \cite[Corollary 3.18]{NakaokaPalu}. 
\end{example}

\begin{example}
In $K^{[-1, 0]}(\proj \Lambda),$ we have $\EbbI^{-1} = \EbbII^{-1} = \EbbPB^{-1} = \EbbE^{-1}$ and $\Ebb^{-2} = 0.$ More precisely, $\Fbb^{\bullet}$ with $\Fbb^{-1} = \EbbI^{-1} = 
\EbbII^{-1}$ and $\Fbb^{-i} = 0,$ for $i \geq 2,$ is the universal bivariant $\delta-$functor among those having $\Fbb^i = \Ebb^i, $ for $i \geq 0.$ This happens because conditions (NI+) and (NII+) are satisfied in this case (see \cite[Example 5.34, Corollary 5.39, Example 5.42]{GNP1}).
\end{example}

\begin{example} \label{ex:frob_quotient}
Let $\Csc$ be a triangulated category with a cluster-tilting subcategory $\ct$. By \cite[Theorem 3.3]{KoenigZhu}, the quotient category $\cA = \Csc/\ct$ is abelian.\footnote{Note that in \cite{KoenigZhu} the authors call \emph{tilting} those subcategories of triangulated categories which are nowadays known as \emph{cluster-tilting} subcategories.} Assume that $\cA$ is Frobenius. By \cite[Proposition 4.6]{KoenigZhu}, we then have $\Sigma\ct = \Sigma^{-1}\ct.$ Thus, in this case for the relative structure on $\Csc$ making $\mathcal{T}$ the full subcategory of projectives, we  have $\EbbI^{-1} = \EbbII^{-1}$. Then conditions (NI+) and (NII+) are satisfied, and, in particular, (N+) is satisfied. Thus, we have $\EbbI^{-1} = \EbbII^{-1} = \EbbPB^{-1} = \EbbE^{-1}$ and $\Ebb^{-2} = 0.$
\end{example}

\begin{example} \label{ex:self-inj}
We have the following important special case of Example \ref{ex:frob_quotient}. Let $(Q, W)$ be a Jacobi-finite quiver with potential. Let $\mathcal{T} = \add T$ be the standard initial cluster-tilting subcategory in the associated cluster category $\Csc$. Assume that the Jacobian algebra $J(Q, W) = (\End_{\Csc} T)^{op}$ is self-injective. By \cite{KoenigZhu} as in the previous example, and also by \cite[Proposition 3.6]{IyamaOppermann}, we have $\Sigma \mathcal{T} = \Sigma^{-1} \mathcal{T}.$ Thus, in this case, for the relative structure on $\Csc$ making $\mathcal{T}$ the full subcategory of projectives, we have $\EbbI^{-1} = \EbbII^{-1} = \EbbPB^{-1} = \EbbE^{-1}$ and $\EbbPB^{-2} = \EbbE^{-2} = 0.$
\end{example}

\begin{example}

Assume that in the notation of the previous subsection, we have $C = \Sigma^{-1} T_a = \Sigma^{-2} A,$ for some $T_a \in \mathcal{T}.$ 

Then $\Dsc(C, \Sigma^{-1} A) = \Dsc(\Sigma^{-1} T_a, T_a) = 0.$ Therefore, $\EbbE^{-1}(C, A) = \EbbI^{-1}(C, A) = \EbbII^{-1}(C, A) = 0.$

We also have $T_0^A = 0$ and $T_1^A = T_a$. By \cref{lem: E-2_explicit}, we get 
\begin{eqnarray*}
\Ebb^{-2}(C, A) & = & \Dsc/_{(\susp^{-2}\Tsc)}(T_0^C,\susp^{-2}A)\times_{\Dsc(T_0^C,\susp^{-1}T_1^A)}\Dsc/_{(\susp\Tsc)}(C,\susp^{-1}T_1^A) \\
& = &  \Dsc/_{(\susp^{-2}\Tsc)}(T_0^C, \Sigma^{-1} T_a)\times_{\Dsc(T_0^C,\susp^{-1}T_a)}\Dsc/_{(\susp\Tsc)}(C,C) \\
& = & \Dsc(T_0^C, \Sigma^{-1} T_a)\times_{\Dsc(T_0^C,\susp^{-1}T_a)}\Dsc/_{(\susp\Tsc)}(C,C) \\
& = & \Dsc/_{(\susp\Tsc)}(C,C).
\end{eqnarray*}


If $C \notin \Sigma \mathcal{T},$ this contains the identity automorphism of $C$ and so is non-zero.

To summarize, in this scenario we have
$\EbbI^{-1}(C, A) =\EbbII^{-1}(C, A) = 0,$ but $\Ebb^{-2}(C, A) \neq 0$. The latter happens because $\EbbII^{-1} (C, \mathcal{T}) \neq 0.$
\end{example}

For an explicit example of a pair $C, A$ with $\EbbI^{-1}(C, A) =\EbbII^{-1}(C, A) = 0,$ but $\Ebb^{-2}(C, A) \neq 0,$ we can consider the cluster category of the quiver $A_2.$ By the above argument, it is sufficient to find a cluster--tilting subcategory $\mathcal{T}$ with an object $T_a$ such that $\Sigma^{-1} T_a \notin \Sigma \mathcal{T}$.

\begin{example} \label{ex: A_2}
Let $\Csc$ be the cluster category of type $A_2$, whose Auslander--Reiten quiver is shown below.
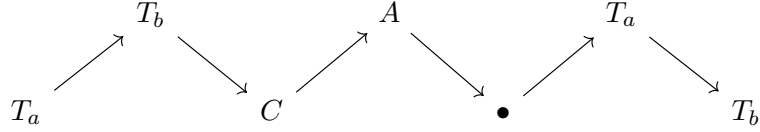
\[\begin{tikzcd}
	& {T_b} && A && {T_a} \\
	{T_a} && C && \bullet && {T_b}
	\arrow[from=2-1, to=1-2]
	\arrow[from=1-2, to=2-3]
	\arrow[from=2-3, to=1-4]
	\arrow[from=1-4, to=2-5]
	\arrow[from=2-5, to=1-6]
	\arrow[from=1-6, to=2-7]
\end{tikzcd}\]
The object $T=T_a\oplus T_b$ is cluster--tilting, and we have
$A=\Sigma T_a$, $C=\Sigma^{-1}T_b$.
Endow $\Csc$ with the relative extriangulated structure given by this choice of a cluster tilting object: the $\sfr$-triangles are those triangles for which the third morphism factors through $\Sigma \Tsc$, where $\Tsc=\add T$.
Because $\Csc(C,\Sigma^{-1}A)=0$, we have $\EbbI^{-1}(C,A)=0=\EbbII^{-1}(C,A)$, which shows that $\EbbPB^{-1}(C,A)=0$.
We also have $\EbbE^{-1}(C,A)=0$.
However, $\Ebb^{-2}(C,A)\neq0$, as shown by the following computation.
There are $\sfr$-triangles
\[
T_a\infl T_b\defl C \text{ and } T_a\infl 0 \defl A.
\]
Thus,
\begin{eqnarray*}
\Ebb^{-2}(C,A) & = &
\Csc/_{(\susp^{-2}\Tsc)}(T_b,\susp^{-2}A)\times_{\Csc(T_,\susp^{-1}T_a)}\Csc/_{(\susp\Tsc)}(C,\susp^{-1}T_a) \\
& = & \Csc(T_b,C)\times_{\Csc(T_b,C)}\Csc(C,C) \\
& = & \Csc(C,C) \neq  0.
\end{eqnarray*}
\end{example}

\begin{example} \label{ex: acyclic_seed}
For acyclic cluster-tilting objects in (acyclic) cluster categories, if a map factors both through $\Sigma^{-1} \mathcal{T}$ and through $\Sigma \mathcal{T},$ then it has to factor through $\Sigma \mathcal{T} \to \Sigma^{-1} \ct,$ since the Auslander-Reiten quiver is directed. Thus, in this case, $\EbbPB^{-1} = \EbbE^{-1}$, and the condition (N+) holds. Usually (and as \cref{ex: A_2} shows), we have $\EbbI^{-1} \neq \EbbII^{-1}$, and so $\Ebb^{-2} \neq 0.$
\end{example}

We see that the condition (N+) holds both for acyclic and self-injective seeds. These are, in a sense, two extreme cases, since a cluster-tilted algebra is always 1-Gorenstein \cite{KellerReiten}. It is thus natural to ask if the condition holds for any cluster seed. The answer turns out to be negative, and the counterexample is given by the smallest case which is neither acyclic, nor self-injective:

\begin{example}[Non-acyclic seed in type $A_4$]
Let $\Csc$ be the cluster category of type $A_4$, whose Auslander--Reiten quiver is shown below.
Consider the cluster tilting object $T=T_a\oplus T_b\oplus T_c\oplus T_d$ whose endomorphism algebra has a Gabriel quiver with a 3-cycle, but is not self-injective.
Endow $\Csc$ with the relative extriangulated structure $\Ebb_{\Tsc}$ associated with $\Tsc = \add T$. The irreducible morphism $f: \susp^{-1} T_b \to \susp T_c$ factors both through $[\susp\Tsc]$ and $[\susp^{-1}\Tsc]$, but does not factor through $[\susp\Tsc \to \susp^{-1}\Tsc]$. By \cref{cor: N+_comparison}, this means that the condition (N+) is not satisfied for this extriangulated category, and so $\EbbPB^{\bullet}$ is not a $\delta$-functor. Let us give a more explicit proof of this by providing an example of an $\sfr$-triangle which is not $(\EbbPB^{\bullet},\del\ssh,\del\ush)$-acyclic.

For the sake of contradiction, assume that condition (N+) holds.
Then from the $\sfr$-triangle $\Sigma^{-1}T_b \infl \Sigma T_c\oplus \Sigma T_a \defl \Sigma T_b \dashrightarrow$, we obtain a long exact sequence

\begin{eqnarray*}
\EbbPB^{-1}(\Sigma T_b,\Sigma^{-1}T_a) \to \EbbPB^{-1}(\Sigma T_c\oplus\Sigma T_a,\Sigma^{-1}T_a) \to \EbbPB^{-1}(\Sigma^{-1}T_b,\Sigma^{-1}T_a) \to\\ \Csc(\Sigma T_b,\Sigma^{-1}T_a)\to \Csc(\Sigma T_c\oplus\Sigma T_a,\Sigma^{-1}T_a).
\end{eqnarray*}
We have
$$\EbbPB^{-1}(\Sigma T_b,\Sigma^{-1}T_a) = \EbbI^{-1}(\Sigma T_b,\Sigma^{-1}T_a) = \Csc/_{[\Sigma^{-1}\Tsc]}(\Sigma T_b,\Sigma^{-2}T_a) =  \Csc/_{[\Sigma^{-1}\Tsc]}(\Sigma T_b,\Sigma T_c) =0,$$
and 
$$\EbbPB^{-1}(\Sigma T_c\oplus\Sigma T_a,\Sigma^{-1}T_a) = \EbbI^{-1}(\Sigma T_c\oplus\Sigma T_a,\Sigma^{-1}T_a) = \Csc/_{[\Sigma^{-1}\Tsc]}(\Sigma T_c\oplus \Sigma T_a,\Sigma T_c) \cong \Csc(\Sigma T_c,\Sigma T_c)\neq 0.$$
Therefore, we must have $\EbbPB^{-1}(\Sigma^{-1}T_b,\Sigma^{-1}T_a)\neq 0$.
But on the other hand, we have
$$\EbbI^{-1}(\Sigma^{-1}T_b,\Sigma^{-1}T_a)=\Csc/_{[\Sigma \Tsc]}(\Sigma^{-1} T_b,\Sigma T_c) = 0$$
 and
$$\EbbII^{-1}(\Sigma^{-1}T_b,\Sigma^{-1}T_a)=\Csc/_{[\Sigma^{-1} \Tsc]}(\Sigma^{-1} T_b,\Sigma T_c) = 0,$$
so that the pullback diagram in \cref{DefUniv_hered} gives $\EbbPB^{-1}(\Sigma^{-1}T_b,\Sigma^{-1}T_a)=0$, a contradiction.


\end{example}
\[\begin{tikzcd} 
\begin{tikzpicture}

\node (v1) at (-5,-3.5) {$T_a$};
\node (v2) at (-4,-2.5) {$T_b$};
\node (v3) at (-3,-1.5) {$\bullet$};
\node (v13) at (-2, -0.5) {$T_d$};
\node (v5) at (-2,-2.5) {$\Sigma^{-1}T_b$};
\node (v7) at (-1,-3.5) {$\Sigma T_c$};
\node (v4) at (-3,-3.5) {$\Sigma^{-1}T_a$};
\node (v6) at (-1,-1.5) {$\bullet$};
\node (v8) at (0,-2.5) {$\bullet$};
\node (v14) at (0, -0.5) {$\Sigma T_a$};
\node (v10) at (1,-3.5) {$T_c$};
\node (v9) at (1,-1.5) {$\Sigma T_b$};
\node (v11) at (2,-2.5) {$\bullet$};
\node (v15) at (2, -0.5) {$T_a$};
\node (v12) at (3,-3.5) {$\Sigma T_d$};
\node (v16) at (3, -1.5) {$T_b$};
\node (v17) at (4, -2.5) {$\bullet$};
\node (v18) at (5, - 3.5) {$T_d$};
\draw [->]  (v1) -- (v2);
\draw [->]  (v2) -- (v3);
\draw [->]  (v2) -- (v4);
\draw [->]  (v3) -- (v13);
\draw [->]  (v13) -- (v6);
\draw [->]  (v6) -- (v14);
\draw [->]  (v14) -- (v9);
\draw [->]  (v9) -- (v15);
\draw [->]  (v11) -- (v16);
\draw [->]  (v12) -- (v17);
\draw [->]  (v15) -- (v16);
\draw [->]  (v16) -- (v17);
\draw [->]  (v17) -- (v18);
\draw [->]  (v3) -- (v5);
\draw [->]  (v4) -- (v5);
\draw [->]  (v5) -- (v6);
\draw [->]  (v5) -- (v7) node[midway,above right] {\tiny{$f$}};
\draw [->]  (v7) -- (v8);
\draw [->]  (v6) -- (v8);
\draw [->]  (v8) -- (v9);
\draw [->]  (v8) -- (v10);
\draw [->]  (v9) -- (v11);
\draw [->]  (v10) -- (v11);
\draw [->]  (v11) -- (v12);
\end{tikzpicture}
\end{tikzcd}\]



We see that in this example condition (N+) is not satisfied and the pullback construction of \cref{ssection:universal} fails to produce a $\delta$-functor.
This motivates the explicit computations in \cref{ssection:computation}, that do give negative extensions in this example.

\appendix
\section{Reminder on extriangulated categories}
\label{ssection:extricats}

Extriangulated categories, introduced in~\cite{NakaokaPalu}, axiomatize extension-closed subcategories of triangulated categories and generalize both exact categories and triangulated categories.
An extriangulated category is given by:
\begin{itemize}
 \item an additive category $\Csc$,
 \item an additive bifunctor $\Ebb:\Csc^\mathrm{op}\times\Csc\to \mathrm{Ab}$ taking values in the category of abelian groups (and to be thought of as an $\operatorname{Ext}^1$-bifunctor),
 \item an additive realization $\mathfrak{s}$ sending each element $\delta\in\Ebb(Z,X)$ to some equivalence class of diagrams $X\to Y\to Z$ playing the role of short exact sequences or of triangles.
\end{itemize}
A few axioms, recalled below, have to be satisfied.

Let $\Csc$ be an additive category, and let $\Ebb: \Csc^\mathrm{op}\times\Csc\to\mathit{Ab}$ be an additive bifunctor.

\begin{remark}
Let $\delta\in\Ebb(Z,X)$, $f\in\Csc(X,X')$ and $h\in\Csc(Z',Z)$. By functoriality there are elements $\Ebb(Z,f)(\delta)\in\Ebb(Z,X')$ and $\Ebb(h,X)(\delta)\in\Ebb(Z',X)$, which are written $f_\ast\delta$ and $h^\ast\delta$ for short.
Using those notations, we have, in $\Ebb(Z',X')$
\[ \Ebb(h,f)(\delta)=h^\ast f_\ast\delta=f_\ast h^\ast\delta. \]
\end{remark}

\begin{definition}
A \defn{morphism} $(f,h):\delta\to\delta'$, where $\delta\in\Ebb(Z,X)$ and $\delta'\in\Ebb(Z',X')$, is a pair of morphisms $f\in\Csc(X,X')$ and $h\in\Csc(Z,Z')$ in $\Csc$, satisfying $f_\ast\delta=h^\ast\delta'$.
\end{definition}

\begin{definition}\label{DefSqEquiv}
Let $X,Z\in\Csc$. Two diagrams 
\[ X \overset{x}{\longrightarrow}Y\overset{y}{\longrightarrow}Z \ \text{ and } \ X\overset{x'}{\longrightarrow}Y'\overset{y'}{\longrightarrow}Z \]
in $\Csc$ are \defn{equivalent} if there exists an isomorphism $g\in\Csc(Y,Y')$ such that
\[
\xy
(-20,0)*+{X}="X";
(5,0)*+{}="1";
(0,6)*+{Y}="Y";
(0,-6)*+{Y'}="Y'";
(-5,0)*+{}="2";
(20,0)*+{Z}="Z";
{\ar^{x} "X";"Y"};
{\ar^{y} "Y";"Z"};
{\ar_{x^{\prime}} "X";"Y'"};
{\ar_{y^{\prime}} "Y'";"Z"};
{\ar^{g}_{\cong} "Y";"Y'"};
{\ar@{}|{} "X";"1"};
{\ar@{}|{} "2";"Z"};
\endxy
\]
commutes.
The equivalence class of $X \xrightarrow{x} Y \xrightarrow{y} Z$ is denoted by $[X \xrightarrow{x} Y \xrightarrow{y} Z]$.
\end{definition}

\begin{notation}
For any $X,Y,Z,A,B,C\in\Csc$, and any $[X \xrightarrow{x} Y \xrightarrow{y} Z]$, $[A \xrightarrow{a} B \xrightarrow{b} C]$, we let
\[ 0=[X\overset{\left[\bsm1\\0\esm\right]}{\longrightarrow} X\oplus Y \overset{\left[\bsm0\;1\esm\right]}{\longrightarrow} Y]
\]
and
\[
[X \xrightarrow{x} Y \xrightarrow{y} Z]\oplus [A \xrightarrow{a} B \xrightarrow{b} C] = [X\oplus A \overset{\left[\bsm x \; 0 \\ 0\; a\esm\right]}{\longrightarrow} Y\oplus B \overset{\left[\bsm y\;0\\0\;b \esm\right]}{\longrightarrow} Z\oplus C].
\]
\end{notation}

\begin{definition}
An \defn{additive realization} $\mathfrak{s}$ is a correspondence associating, with any $X,Z,$ and $\delta\in\Ebb(Z,X)$, an equivalence class $\mathfrak{s}(\delta)=[X\xrightarrow{x}Y\xrightarrow{y}Z]$ and satisfying the following three conditions: 
\begin{itemize}
\item[$(\ast)$] Let $\delta\in\Ebb(Z,X)$ and $\delta'\in\Ebb(Z',X')$ be such that
\[\mathfrak{s}(\delta)=[X\xrightarrow{x}Y\xrightarrow{y}Z] \text{ and } \mathfrak{s}(\delta')=[X'\xrightarrow{x'}Y'\xrightarrow{y'}Z'].\]
Then, for any morphism $(f,h):\delta\to\delta'$, there exists $g\in\Csc(Y,Y')$ making
\[
\xy
(-12,6)*+{X}="0";
(0,6)*+{Y}="2";
(12,6)*+{Z}="4";
(-12,-6)*+{X'}="10";
(0,-6)*+{Y'}="12";
(12,-6)*+{Z'}="14";
{\ar^{x} "0";"2"};
{\ar^{y} "2";"4"};
{\ar_{f} "0";"10"};
{\ar^{g} "2";"12"};
{\ar^{h} "4";"14"};
{\ar_{x'} "10";"12"};
{\ar_{y'} "12";"14"};
{\ar@{}|{} "0";"12"};
{\ar@{}|{} "2";"14"};
{\ar@{}|\circlearrowright "0";"12"};
{\ar@{}|\circlearrowright "2";"14"};
\endxy
\]
commute.
\end{itemize}

The sequence $X\xrightarrow{x}Y\xrightarrow{y}Z$ \defn{realizes} $\delta$ if $\mathfrak{s}(\delta)=[X\xrightarrow{x}Y\xrightarrow{y}Z]$.

\begin{enumerate}
\item For any $X,Z\in\Csc$, the realization of $0\in\Ebb(Z,X)$ is given by $\mathfrak{s}(0)=0$.
\item For any $\delta\in\Ebb(Z,X)$ and $\delta'\in\Ebb(Z',X')$, the realization of $\delta\oplus\delta'\in\Ebb(Z\oplus Z',X\oplus X')$ is given by $\mathfrak{s}(\delta\oplus\delta')=\mathfrak{s}(\delta)\oplus\mathfrak{s}(\delta')$.
\end{enumerate}
\end{definition}

\begin{definition}(\cite[Definition 2.12]{NakaokaPalu})
A triple $(\Csc,\Ebb,\mathfrak{s})$ is called an \defn{extriangulated category} if the following holds:
\begin{itemize}
\item[{\rm (ET1)}] $\Ebb:\Csc^{\mathrm{op}}\times\Csc\to\mathit{Ab}$ is an additive bifunctor;
\item[{\rm (ET2)}] $\mathfrak{s}$ is an additive realization of $\Ebb$;
\item[{\rm (ET3)}] Let $\delta\in\Ebb(Z,X)$ and $\delta'\in\Ebb(Z',X')$ be respectively realized by $X\xrightarrow{x}Y\xrightarrow{y}Z$ and $X'\xrightarrow{x'}Y'\xrightarrow{y'}Z'$. Then, for any commutative square
\[
\xy
(-12,6)*+{X}="0";
(0,6)*+{Y}="2";
(12,6)*+{Z}="4";
(-12,-6)*+{X'}="10";
(0,-6)*+{Y'}="12";
(12,-6)*+{Z'}="14";
{\ar^{x} "0";"2"};
{\ar^{y} "2";"4"};
{\ar_{f} "0";"10"};
{\ar^{g} "2";"12"};
{\ar_{x'} "10";"12"};
{\ar_{y'} "12";"14"};
{\ar@{}|{} "0";"12"};
{\ar@{}|\circlearrowright "0";"12"};
\endxy 
\]
in $\Csc$, there exists a morphism $(f,h):\delta\to\delta'$ satisfying $h\circ y=y'\circ g$.
\item[{\rm (ET3)}]$\hspace*{-5pt}^{\mathrm{op}}$ Dual of {\rm (ET3)}.
\item[{\rm (ET4)}] Let $\delta\in\Ebb(Z',X)$ and $\delta'\in\Ebb(X',Y)$ be realized respectively by
\[
X\overset{f}{\longrightarrow}Y\overset{f'}{\longrightarrow}Z'\ \text{ and } \ Y\overset{g}{\longrightarrow}Z\overset{g'}{\longrightarrow}X'.
\]
Then there exist an object $Y'\in\Csc$, a commutative diagram in $\Csc$
\[
\xy
(-18,6)*+{X}="0";
(-6,6)*+{Y}="2";
(6,6)*+{Z'}="4";
(18,6)="6";
(-18,-6)*+{X}="10";
(-6,-6)*+{Z}="12";
(6,-6)*+{Y'}="14";
(18,-6)="16";
(-6,-18)*+{X'}="22";
(6,-18)*+{X'}="24";
(-6,-28)="32";
(6,-28)="34";
{\ar^{f} "0";"2"};
{\ar^{f'} "2";"4"};
{\ar@{=} "0";"10"};
{\ar_{g} "2";"12"};
{\ar^{d} "4";"14"};
{\ar^{h} "10";"12"};
{\ar^{h'} "12";"14"};
{\ar_{g'} "12";"22"};
{\ar^{e} "14";"24"};
{\ar@{=} "22";"24"};
{\ar@{}|{} "0";"12"};
{\ar@{}|{} "2";"14"};
{\ar@{}|{} "12";"24"};
{\ar@{-->}^\delta "4";"6"};
{\ar@{-->}^{\delta''} "14";"16"};
{\ar@{-->}^{\delta'} "22";"32"};
{\ar@{-->}^{f'_\ast \delta'} "24";"34"};
{\ar@{}|\circlearrowright "0";"12"};
{\ar@{}|\circlearrowright "2";"14"};
{\ar@{}|\circlearrowright "12";"24"};
\endxy
\]
and $\delta''\in\Ebb(Y',X)$ realized by $X\overset{h}{\longrightarrow}Z\overset{h'}{\longrightarrow}Y'$, which satisfy the following compatibilities.
  \begin{enumerate}[(i)]
    \item $Z'\overset{d}{\longrightarrow}Y'\overset{e}{\longrightarrow}X'$ realizes $f'_\ast\delta'$,
    \item $d^\ast\delta''=\delta$,
    \item $f_\ast\delta''=e^\ast\delta'$. 
  \end{enumerate}
\item[{\rm (ET4)}]$\hspace*{-5pt}^{\mathrm{op}}$ Dual of {\rm (ET4)}.
\end{itemize}
\end{definition}

We use the following terminology.
\begin{notation}
Let $(\Csc,\Ebb,\mathfrak{s})$ be an extriangulated category.
\begin{enumerate}
\item A sequence $X\xrightarrow{f}Y\xrightarrow{g}Z$ is called a \defn{conflation} if it realizes some $\delta$ in $\Ebb(Z,X)$. In that case the morphism $X\xrightarrow{f}Y$ is called an \defn{inflation}, written $X\infl Y$, and the morphism $Y\xrightarrow{g}Z$ is called a \defn{deflation}, written $Y\defl Z$.
\item For an inflation $X \overset{f}\infl Y,$ the \defn{cone} $\Cone(f)$ is an object $Z$ appearing as the last term in some conflation $X \overset{f}\infl Y \defl Z.$ It is unique up to isomorphism. Dually, for a deflation $Y \overset{g}\defl Y,$ the \defn{cocone}, or the \defn{fiber} $\mbox{CoCone}(f) = \Fib(f)$ is an object $X$ appearing as the first term in some conflation $X \infl Y \overset{g}\defl Z.$ It is unique up to isomorphism. 
\item An \defn{$\sfr$-triangle} is a diagram $X\overset{f}{\infl} Y\overset{g}{\defl} Z\overset{\delta}{\dashrightarrow}$ where $X\overset{f}{\infl} Y\overset{g}{\defl} Z$ is a conflation realizing $\delta\in\Ebb(Z,X)$.
\item A \defn{morphism of $\sfr$-triangles} is a commutative diagram
\[
\xy
(-12,6)*+{X}="0";
(0,6)*+{Y}="2";
(12,6)*+{Z}="4";
(24,6)*+{}="6";
(-12,-6)*+{X'}="10";
(0,-6)*+{Y'}="12";
(12,-6)*+{Z'}="14";
(24,-6)*+{}="16";
{\ar^{f} "0";"2"};
{\ar^{g} "2";"4"};
{\ar@{-->}^{\delta} "4";"6"};
{\ar_{x} "0";"10"};
{\ar^{y} "2";"12"};
{\ar^{z} "4";"14"};
{\ar_{f'} "10";"12"};
{\ar_{g'} "12";"14"};
{\ar@{-->}_{\delta'} "14";"16"};
{\ar@{}|{} "0";"12"};
{\ar@{}|{} "2";"14"};
\endxy
\]
where $x_\ast\delta=z^\ast\delta'$.
\end{enumerate}
\end{notation}

The axioms above imply the following:

\begin{proposition}
\label{Appendix: l.e.s.}
Assume that $(\Csc,\Ebb,\mathfrak{s})$ is an extriangulated category, and let $X\xrightarrow{f}Y\xrightarrow{g}Z\overset{\delta}{\dashrightarrow}$ be an $\sfr$-triangle.
Then the following sequences of natural transformations are exact:
\[ \Csc(Z,-)\overset{-\circ g}{\longrightarrow}\Csc(Y,-)\overset{-\circ f}{\longrightarrow}\Csc(X,-)\overset{\delta^\sharp}{\longrightarrow}\Ebb(Z,-)\overset{g^\ast}{\longrightarrow}\Ebb(Y,-)\overset{f^\ast}{\longrightarrow}\Ebb(X,-), \]
\[ \Csc(-,X)\overset{f\circ-}{\longrightarrow}\Csc(-,Y)\overset{g\circ-}{\longrightarrow}\Csc(-,Z)\overset{\delta_\sharp}{\longrightarrow}\Ebb(-,X)\overset{f_\ast}{\longrightarrow}\Ebb(-,Y)\overset{g_\ast}{\longrightarrow}\Ebb(-,Z), \]
where $\delta^\sharp(x)=x_\ast\delta$ and $\delta_\sharp(y)=y^\ast\delta$.
\end{proposition}

It follows that, in a conflation, the inflation is a weak kernel of the deflation, and the deflation is a weak cokernel of the inflation.

\begin{proposition}
 The following are examples of extriangulated categories (for more details, see~\emph{\cite{NakaokaPalu}}):
 \begin{itemize}
  \item Triangulated categories;
  \item Exact categories;
  \item Extension-closed full subcategories of triangulated (or extriangulated) categories;
  \item Quotients of exact (or extriangulated) categories by ideals generated by projective-injective objects (this was used in~\cref{rk:reduced});
  \item Homotopy categories of exact infinity-categories (see~\emph{\cite{NakaokaPalu2,Klemenc}}).
 \end{itemize}
\end{proposition}

The following shifted version of the octahedron axiom (and its dual) is often used in our proofs.
\begin{lemma}
\label{lemma:shiftedOctahedron}
Let $X\infl Y\defl Z\dashrightarrow$ and $X\infl Y'\defl Z'\dashrightarrow$ be two $\sfr$-triangles whose inflations have the same domain.
Then there exists a diagram made of morphisms of $\sfr$-triangles of the form
\[\begin{tikzcd}
	X & Y & Z & {} \\
	{Y'} & E & Z & {} \\
	{Z'} & {Z'} \\
	{} & {}
	\arrow[tail, from=1-1, to=1-2]
	\arrow[two heads, from=1-2, to=1-3]
	\arrow[tail, from=1-1, to=2-1]
	\arrow[two heads, from=2-1, to=3-1]
	\arrow[color={rgb,255:red,92;green,92;blue,214}, tail, from=2-1, to=2-2]
	\arrow[color={rgb,255:red,92;green,92;blue,214}, tail, from=1-2, to=2-2]
	\arrow[color={rgb,255:red,92;green,92;blue,214}, two heads, from=2-2, to=2-3]
	\arrow[color={rgb,255:red,92;green,92;blue,214}, two heads, from=2-2, to=3-2]
	\arrow[color={rgb,255:red,92;green,92;blue,214}, Rightarrow, no head, from=3-1, to=3-2]
	\arrow[color={rgb,255:red,92;green,92;blue,214}, Rightarrow, no head, from=1-3, to=2-3]
	\arrow[dashed, from=1-3, to=1-4]
	\arrow[color={rgb,255:red,92;green,92;blue,214}, dashed, from=2-3, to=2-4]
	\arrow[color={rgb,255:red,92;green,92;blue,214}, dashed, from=3-2, to=4-2]
	\arrow[dashed, from=3-1, to=4-1]
	\end{tikzcd}\]
\end{lemma}

We also make use of \defn{relative extriangulated structures}, as defined in~\cite[Section 3.2]{HerschendLiuNakaoka-I}.
For any full subcategory $\Xcal\subseteq \cat$ there are two associated subfunctors $\Ebb_\Xcal$ and $\Ebb^\Xcal$ of $\Ebb$, defined as follows:
\begin{definition}
\label{definition:relative structures}
\begin{eqnarray*}
   \Ebb_\Xcal(C,A) & = & \{\delta\in\Ebb(C,A)\;|\;g^\ast\delta = 0 \text{ for any $g$ with domain in }\Xcal \} \\
   \Ebb_\Xcal(C,A) & = & \{\delta\in\Ebb(C,A)\;|\;f_\ast\delta = 0 \text{ for any $f$ with codomain in }\Xcal \} \\
\end{eqnarray*}
\end{definition}

Those subfunctors are shown to induce extriangulated substructures in~\cite[Proposition 3.19]{HerschendLiuNakaoka-I}, where the first (resp. second) is the maximal extriangulated substructure making all objects in $\Xcal$ projective (resp. injective).

\bibliographystyle{alpha}
\bibliography{MutationExtricats}

\newcommand{\etalchar}[1]{$^{#1}$}
\begin{thebibliography}{ABCJP10}

\bibitem[AAG08]{Avella-AlaminosGeiss}
Diana Avella-Alaminos and Christof Gei\ss.
\newblock Combinatorial derived invariants for gentle algebras.
\newblock {\em J. Pure Appl. Algebra}, 212(1):228--243, 2008.

\bibitem[ABCJP10]{AssemBrustleCharbonneau-JodoinPlamondon}
Ibrahim Assem, Thomas Br\"ustle, Gabrielle Charbonneau-Jodoin, and Pierre-Guy
  Plamondon.
\newblock Gentle algebras arising from surface triangulations.
\newblock {\em Algebra Number Theory}, 4(2):201--229, 2010.

\bibitem[AET21]{AdachiEnomotoTsukamoto}
Takahide Adachi, Haruhisa Enomoto, and Mayu Tsukamoto.
\newblock Intervals of $s$-torsion pairs in extriangulated categories with
  negative first extensions.
\newblock Preprint,
  \href{https://arxiv.org/abs/2103.09549}{\texttt{arXiv:2103.09549}}, 2021.

\bibitem[AI12]{AiharaIyama}
Takuma Aihara and Osamu Iyama.
\newblock Silting mutation in triangulated categories.
\newblock {\em Journal of the London Mathematical Society}, 85(3):633--668,
  2012.

\bibitem[AIR14]{AdachiIyamaReiten}
Takahide Adachi, Osamu Iyama, and Idun Reiten.
\newblock {$\tau$}-tilting theory.
\newblock {\em Compos. Math.}, 150(3):415--452, 2014.

\bibitem[Ami09]{Amiot-ClusterCats}
Claire Amiot.
\newblock Cluster categories for algebras of global dimension 2 and quivers
  with potential.
\newblock {\em Ann. Inst. Fourier (Grenoble)}, 59(6):2525--2590, 2009.

\bibitem[Asa12]{Asashiba-Sugaku}
Hideto Asashiba.
\newblock Derived equivalence classification of algebras.
\newblock {\em S\={u}gaku}, 64(4):357--383, 2012.

\bibitem[AT22]{AdachiTsukamoto}
Takahide Adachi and Mayu Tsukamoto.
\newblock Hereditary cotorsion pairs and silting subcategories in
  extriangulated categories.
\newblock {\em J. Algebra}, 594:109--137, 2022.

\bibitem[Bar15]{barwick2015exact}
Clark Barwick.
\newblock On exact $\infty$-categories and the theorem of the heart.
\newblock {\em Compositio Mathematica}, 151(11):2160--2186, 2015.

\bibitem[BCS21]{BaurCoelhoSimoes}
Karin Baur and Raquel Coelho~Sim{\~{o}}es.
\newblock A geometric model for the module category of a gentle algebra.
\newblock {\em Int. Math. Res. Not. IMRN}, (15):11357--11392, 2021.

\bibitem[BDM{\etalchar{+}}20]{BrustleDouvilleMousavandThomasYildirim}
Thomas Br\"{u}stle, Guillaume Douville, Kaveh Mousavand, Hugh Thomas, and Emine
  Y{\i}ld{\i}r{\i}m.
\newblock On the combinatorics of gentle algebras.
\newblock {\em Canad. J. Math.}, 72(6):1551--1580, 2020.

\bibitem[BGMS23]{BarnardGunawanMeehanSchiffler}
Emily Barnard, Emily Gunawan, Emily Meehan, and Ralf Schiffler.
\newblock Cambrian combinatorics on quiver representations (type {$\Bbb{A}$}).
\newblock {\em Adv. in Appl. Math.}, 143:Paper No. 102428, 37, 2023.

\bibitem[BMR{\etalchar{+}}06]{BuanMarshReinekeReitenTodorov}
Aslak~Bakke Buan, Robert Marsh, Markus Reineke, Idun Reiten, and Gordana
  Todorov.
\newblock Tilting theory and cluster combinatorics.
\newblock {\em Adv. Math.}, 204(2):572--618, 2006.

\bibitem[Bon10]{Bondarko}
M.~V. Bondarko.
\newblock Weight structures vs. {$t$}-structures; weight filtrations, spectral
  sequences, and complexes (for motives and in general).
\newblock {\em J. K-Theory}, 6(3):387--504, 2010.

\bibitem[BR87]{ButlerRingel}
M.~C.~R. Butler and Claus~Michael Ringel.
\newblock Auslander-{R}eiten sequences with few middle terms and applications
  to string algebras.
\newblock {\em Comm. Algebra}, 15(1-2):145--179, 1987.

\bibitem[BS23]{BrustleSchiffler-Oberwolfach}
Thomas Br\"ustle and Ralf Schiffler.
\newblock Private communication in {O}berwolfach.
\newblock 2023.

\bibitem[BSSZ05]{BautistaSalorioZuazua}
Raymundo Bautista, Maria~Jose Souto~Salorio, and Rita Zuazua.
\newblock Almost split sequences of complexes of fixed size.
\newblock {\em J. Algebra}, 287(1):140--168, 2005.

\bibitem[BV20]{bondarko2020hearts}
Mikhail~Vladimirovich Bondarko and Sergei~Vladimirovich Vostokov.
\newblock The hearts of weight structures are the weakly idempotent complete
  categories.
\newblock {\em Chebyshevski\v{i} Sb.}, 21(3):29--38, 2020.

\bibitem[BY13]{BrustleYang}
Thomas Br\"{u}stle and Dong Yang.
\newblock Ordered exchange graphs.
\newblock In {\em Advances in representation theory of algebras}, EMS Ser.
  Congr. Rep., pages 135--193. Eur. Math. Soc., Z\"{u}rich, 2013.

\bibitem[BZ21]{BuanZhou}
Aslak~Bakke Buan and Yu~Zhou.
\newblock Weak cotorsion, $\tau$-tilting and two-term categories.
\newblock Preprint,
  \href{https://arxiv.org/abs/2111.10995}{\texttt{arXiv:2111.10995}}, 2021.

\bibitem[CB96]{CrawleyBoevey}
William Crawley-Boevey.
\newblock Rigid integral representations of quivers.
\newblock In {\em Can. Math. Soc. Conf. Proc. 18}, pages 155--163, 1996.

\bibitem[CCS06]{CalderoChapotonSchiffler}
Philippe Caldero, Fr\'ed\'eric Chapoton, and Ralf Schiffler.
\newblock Quivers with relations arising from clusters ({$A_n$} case).
\newblock {\em Trans. Amer. Math. Soc.}, 358(3):1347--1364, 2006.

\bibitem[CLZ22]{chen2022extensions}
Xiao-Wu Chen, Zengqiang Lin, and Yu~Zhou.
\newblock The extensions of t-structures.
\newblock Preprint,
  \href{https://arxiv.org/abs/2205.10831}{\texttt{arXiv:2205.10831}}, 2022.

\bibitem[DF15]{DerksenFei}
Harm Derksen and Jiarui Fei.
\newblock General presentations of algebras.
\newblock {\em Adv. Math.}, 278:210--237, 2015.

\bibitem[DIJ19]{DemonetIyamaJasso}
Laurent Demonet, Osamu Iyama, and Gustavo Jasso.
\newblock {$\tau$}-tilting finite algebras, bricks, and {$g$}-vectors.
\newblock {\em Int. Math. Res. Not. IMRN}, (3):852--892, 2019.

\bibitem[FGL19]{FuGengLiu}
Changjian Fu, Shengfei Geng, and Pin Liu.
\newblock Relative rigid objects in triangulated categories.
\newblock {\em J. Algebra}, 520:171--185, 2019.

\bibitem[FK10]{FuKeller}
Changjian Fu and Bernhard Keller.
\newblock On cluster algebras with coefficients and 2-{C}alabi-{Y}au
  categories.
\newblock {\em Transactions of the American Mathematical Society},
  362(2):859--895, 2010.

\bibitem[GLS08]{GeissLeclerclSchroer-survey}
Christof Gei{\ss}, Bernard Leclerc, and Jan Schr\"{o}er.
\newblock Preprojective algebras and cluster algebras.
\newblock In {\em Trends in representation theory of algebras and related
  topics}, EMS Ser. Congr. Rep., pages 253--283. Eur. Math. Soc., Z\"{u}rich,
  2008.

\bibitem[GNP21]{GNP1}
Mikhail Gorsky, Hiroyuki Nakaoka, and Yann Palu.
\newblock Positive and negative extensions in extriangulated categories.
\newblock Preprint, \href{https://arxiv.org/abs/2103.12482}
  {\texttt{arXiv:2103.12482}}, 2021.

\bibitem[HJ15]{HolmJorgensen-modifiedCC}
Thorsten Holm and Peter J{\o}rgensen.
\newblock Generalized friezes and a modified {C}aldero-{C}hapoton map depending
  on a rigid object.
\newblock {\em Nagoya Math. J.}, 218:101--124, 2015.

\bibitem[HJ16]{HolmJorgensen-modifiedCC2}
Thorsten Holm and Peter J{\o}rgensen.
\newblock Generalised friezes and a modified {C}aldero-{C}hapoton map depending
  on a rigid object, {II}.
\newblock {\em Bull. Sci. Math.}, 140(4):112--131, 2016.

\bibitem[HLN21]{HerschendLiuNakaoka-I}
Martin Herschend, Yu~Liu, and Hiroyuki Nakaoka.
\newblock {$n$}-exangulated categories ({I}): {D}efinitions and fundamental
  properties.
\newblock {\em J. Algebra}, 570:531--586, 2021.

\bibitem[HLN22]{HLNII}
Martin Herschend, Yu~Liu, and Hiroyuki Nakaoka.
\newblock n-{E}xangulated categories ({I}{I}): {C}onstructions from n-cluster
  tilting subcategories.
\newblock {\em Journal of Algebra}, 594:636--684, 2022.

\bibitem[IJ17]{IyamaJasso}
Osamu Iyama and Gustavo Jasso.
\newblock Higher {A}uslander correspondence for dualizing {R}-varieties.
\newblock {\em Algebras and Representation Theory}, 20(2):335--354, 2017.

\bibitem[IJY14]{IyamaJorgensenYang}
Osamu Iyama, Peter J{\o}rgensen, and Dong Yang.
\newblock Intermediate co-{$t$}-structures, two-term silting objects,
  {$\tau$}-tilting modules, and torsion classes.
\newblock {\em Algebra Number Theory}, 8(10):2413--2431, 2014.

\bibitem[INP18]{IyamaNakaokaPalu}
Osamu Iyama, Hiroyuki Nakaoka, and Yann Palu.
\newblock Auslander--{R}eiten theory in extriangulated categories.
\newblock Preprint,
  \href{https://arxiv.org/abs/1805.03776}{\texttt{arXiv:1805.03776}}, 2018.

\bibitem[IO13]{IyamaOppermann}
Osamu Iyama and Steffen Oppermann.
\newblock Stable categories of higher preprojective algebras.
\newblock {\em Advances in Mathematics}, 244:23--68, 2013.

\bibitem[IR14]{iyama2014introduction}
Osamu Iyama and Idun Reiten.
\newblock Introduction to $\tau$-tilting theory.
\newblock {\em Proceedings of the National Academy of Sciences},
  111(27):9704--9711, 2014.

\bibitem[IY08]{IyamaYoshino}
Osamu Iyama and Yuji Yoshino.
\newblock Mutation in triangulated categories and rigid {C}ohen--{M}acaulay
  modules.
\newblock {\em Inventiones mathematicae}, 172(1):117--168, 2008.

\bibitem[Iya11]{Iyama}
Osamu Iyama.
\newblock Cluster tilting for higher {A}uslander algebras.
\newblock {\em Advances in Mathematics}, 226(1):1--61, 2011.

\bibitem[Jas15]{Jasso-Reduction}
Gustavo Jasso.
\newblock Reduction of {$\tau$}-tilting modules and torsion pairs.
\newblock {\em Int. Math. Res. Not.}, (16):7190--7237, 2015.

\bibitem[JKS16]{JensenKingSu}
Bernt~Tore Jensen, Alastair~D. King, and Xiuping Su.
\newblock A categorification of {G}rassmannian cluster algebras.
\newblock {\em Proc. Lond. Math. Soc. (3)}, 113(2):185--212, 2016.

\bibitem[J{\o}r18]{jorgensen2018co}
Peter J{\o}rgensen.
\newblock Co-t-structures: The first decade.
\newblock {\em Surveys in Representation Theory of Algebras}, 716:25--36, 2018.

\bibitem[JS21]{JorgensenShah-modifiedCC}
Peter J{\o}rgensen and Amit Shah.
\newblock Grothendieck groups of $d$-exangulated categories and a modified
  {C}aldero--{C}hapoton map.
\newblock Preprint,
  \href{https://arxiv.org/abs/2106.02142}{\texttt{arXiv:2106.02142}}, 2021.

\bibitem[JS22]{JorgensenShah-Index}
Peter J{\o}rgensen and Amit Shah.
\newblock The index with respect to a rigid subcategory of a triangulated
  category.
\newblock Preprint,
  \href{https://arxiv.org/abs/2201.00740}{\texttt{arXiv:2201.00740}}, 2022.

\bibitem[JY20]{JorgensenYakimov}
Peter J{\o}rgensen and Milen Yakimov.
\newblock {$c$}-vectors of 2-{C}alabi-{Y}au categories and {B}orel subalgebras
  of {$\mathfrak{sl}_\infty$}.
\newblock {\em Selecta Math. (N.S.)}, 26(1):Paper No. 1, 46, 2020.

\bibitem[Kel05]{keller2005triangulated}
Bernhard Keller.
\newblock On triangulated orbit categories.
\newblock {\em Documenta Mathematica}, 10:551--581, 2005.

\bibitem[Kim20]{Kimura}
Yuta Kimura.
\newblock Tilting theory of noetherian algebras.
\newblock Preprint, \href{https://arxiv.org/abs/2006.01677}
  {\texttt{arXiv:2006.01677}}, 2020.

\bibitem[Kle22]{Klemenc}
Jona Klemenc.
\newblock The stable hull of an exact $\infty$-category.
\newblock {\em Homology, Homotopy and Applications}, 24(2):195--220, 2022.

\bibitem[KR07]{KellerReiten}
Bernhard Keller and Idun Reiten.
\newblock Cluster-tilted algebras are {G}orenstein and stably {C}alabi--{Y}au.
\newblock {\em Advances in Mathematics}, 211(1):123--151, 2007.

\bibitem[KS12]{KellerScherotzke-Integral}
Bernhard Keller and Sarah Scherotzke.
\newblock The integral cluster category.
\newblock {\em Int. Math. Res. Not. IMRN}, (12):2867--2887, 2012.

\bibitem[KY14]{KoenigYang}
Steffen Koenig and Dong Yang.
\newblock Silting objects, simple-minded collections, t-structures and
  co-t-structures for finite-dimensional algebras.
\newblock {\em Documenta Mathematica}, 19:403--438, 2014.

\bibitem[KZ08]{KoenigZhu}
Steffen Koenig and Bin Zhu.
\newblock From triangulated categories to abelian categories: cluster tilting
  in a general framework.
\newblock {\em Math. Z.}, 258(1):143--160, 2008.

\bibitem[LN19]{LiuNakaoka}
Yu~Liu and Hiroyuki Nakaoka.
\newblock Hearts of twin cotorsion pairs on extriangulated categories.
\newblock {\em Journal of Algebra}, 528:96--149, 2019.

\bibitem[LZ20]{liu2020hereditary}
Yu~Liu and Panyue Zhou.
\newblock Hereditary cotorsion pairs on extriangulated subcategories.
\newblock Preprint,
  \href{https://arxiv.org/abs/2012.06997}{\texttt{arXiv:2012.06997}}, 2020.

\bibitem[LZ22]{LiuZhou-relative}
Yu~Liu and Panyue Zhou.
\newblock Relative rigid objects in extriangulated categories.
\newblock {\em Journal of Pure and Applied Algebra}, 226(5):106923, 2022.

\bibitem[McC17]{McConville}
Thomas McConville.
\newblock Lattice structure of {G}rid-{T}amari orders.
\newblock {\em J. Combin. Theory Ser. A}, 148:27--56, 2017.

\bibitem[MSSSS13]{hernandez2013auslander}
Octavio Mendoza, Edith~Corina S{\'a}enz, Valente Santiago, and
  Mar{\'\i}a~Jos{\'e} Souto~Salorio.
\newblock {A}uslander--{B}uchweitz context and co-t-structures.
\newblock {\em Applied Categorical Structures}, 5(21):417--440, 2013.

\bibitem[MZ22]{marks2022lifting}
Frederik Marks and Alexandra Zvonareva.
\newblock Lifting and restricting t-structures.
\newblock {\em Bulletin of the London Mathematical Society}, 2022.

\bibitem[NOS22]{NakaokaOgawaSakai}
Hiroyuki Nakaoka, Yasuaki Ogawa, and Arashi Sakai.
\newblock Localization of extriangulated categories.
\newblock {\em Journal of Algebra}, 611:341--398, 2022.

\bibitem[NP19]{NakaokaPalu}
Hiroyuki Nakaoka and Yann Palu.
\newblock Extriangulated categories, {H}ovey twin cotorsion pairs and model
  structures.
\newblock {\em Cah. Topol. G\'{e}om. Diff\'{e}r. Cat\'{e}g.}, {L}{X}, 2019.

\bibitem[NP20]{NakaokaPalu2}
Hiroyuki Nakaoka and Yann Palu.
\newblock External triangulation of the homotopy category of an exact
  quasi-category.
\newblock Preprint,
  \href{https://arxiv.org/abs/2004.02479}{\texttt{arXiv:2004.02479}}, 2020.

\bibitem[OPS18]{OpperPlamondonSchroll}
Sebastian Opper, Pierre-Guy Plamondon, and Sibylle Schroll.
\newblock A geometric model for the derived category of gentle algebras.
\newblock Preprint,
  \href{http://arxiv.org/abs/1801.09659}{\texttt{arXiv:1801.09659}}, 2018.

\bibitem[Pau08]{Pauksztello}
David Pauksztello.
\newblock Compact corigid objects in triangulated categories and
  co-{$t$}-structures.
\newblock {\em Cent. Eur. J. Math.}, 6(1):25--42, 2008.

\bibitem[PPP19]{PaluPilaudPlamondon-surfaces}
Yann Palu, Vincent Pilaud, and Pierre-Guy Plamondon.
\newblock Non-kissing and non-crossing complexes for locally gentle algebras.
\newblock {\em Journal of Combinatorial Algebra}, 3(4):401--438, 2019.

\bibitem[PPP21]{PaluPilaudPlamondon-nonkissing}
Yann Palu, Vincent Pilaud, and Pierre-Guy Plamondon.
\newblock Non-kissing complexes and tau-tilting for gentle algebras.
\newblock {\em Mem. Amer. Math. Soc.}, 274(1343):vii+110, 2021.

\bibitem[PPPP19]{PadrolPaluPilaudPlamondon}
Arnau Padrol, Yann Palu, Vincent Pilaud, and Pierre-Guy Plamondon.
\newblock Associahedra for finite type cluster algebras and minimal relations
  between g-vectors.
\newblock Preprint,
  \href{http://arxiv.org/abs/1906.06861}{\texttt{arXiv:1906.06861}}, 2019.

\bibitem[Pre20]{Pressland}
Matthew Pressland.
\newblock Mutation of frozen {J}acobian algebras.
\newblock {\em J. Algebra}, 546:236--273, 2020.

\bibitem[Pre21]{Pressland-corrigendum}
Matthew Pressland.
\newblock Corrigendum to ``{M}utation of frozen {J}acobian algebras'' [{J}.
  {A}lgebra 546 (2020) 236-273].
\newblock {\em J. Algebra}, 588:533--537, 2021.

\bibitem[PZ20]{PauksztelloZvonareva}
David Pauksztello and Alexandra Zvonareva.
\newblock Co-{$t$}-structures, cotilting and cotorsion pairs.
\newblock Preprint,
  \href{https://arxiv.org/abs/2007.06536}{\texttt{arXiv:2007.06536}}, 2020.

\bibitem[Rea06]{reading2006cambrian}
Nathan Reading.
\newblock Cambrian lattices.
\newblock {\em Advances in Mathematics}, 205(2):313--353, 2006.

\bibitem[Tac64]{Tachikawa}
Hiroyuki Tachikawa.
\newblock On dominant dimensions of {Q}{F}-3 algebras.
\newblock {\em Transactions of the American Mathematical Society},
  112(2):249--266, 1964.

\bibitem[Tre21]{treffinger2021tau}
Hipolito Treffinger.
\newblock $\tau$-tilting theory--an introduction.
\newblock Preprint,
  \href{https://arxiv.org/abs/2106.00426}{\texttt{arXiv:2106.00426}}, 2021.

\bibitem[Wu21]{Wu}
Yilin Wu.
\newblock Relative cluster categories and higgs categories.
\newblock Preprint,
  \href{https://arxiv.org/abs/2109.03707}{\texttt{arXiv:2109.03707}}, 2021.

\bibitem[YZ19]{YangZhu}
Wuzhong Yang and Bin Zhu.
\newblock Relative cluster tilting objects in triangulated categories.
\newblock {\em Trans. Amer. Math. Soc.}, 371(1):387--412, 2019.

\bibitem[YZZ19]{YangZhouZhu-Ghost}
Wuzhong Yang, Panyue Zhou, and Bin Zhu.
\newblock Triangulated categories with cluster tilting subcategories.
\newblock {\em Pacific J. Math.}, 301(2):703--740, 2019.

\bibitem[ZZ20]{ZhouZhu-2termRelative}
Panyue Zhou and Bin Zhu.
\newblock Two-term relative cluster tilting subcategories, {$\tau$}-tilting
  modules and silting subcategories.
\newblock {\em J. Pure Appl. Algebra}, 224(9):106365, 22, 2020.

\end{thebibliography}

\end{document}